\documentclass[12pt]{article}
\usepackage{amsmath, amsthm, amssymb}
\usepackage{enumerate}
\usepackage{pdflscape}
\usepackage{caption}

\usepackage{ifpdf}
\ifpdf
\usepackage[pdftex]{graphicx}
\else
\usepackage[dvips]{graphicx}
\fi
\usepackage{tikz}
 	 \usetikzlibrary{arrows,backgrounds}
\usepackage[all]{xy}

\usepackage{multicol}

\usepackage{tocvsec2}

\usepackage{bbm}

\input xy
\xyoption{all}

\usepackage[pdftex,plainpages=false,hypertexnames=false,pdfpagelabels]{hyperref}
\newcommand{\arxiv}[1]{\href{http://arxiv.org/abs/#1}{\tt arXiv:\nolinkurl{#1}}}
\newcommand{\arXiv}[1]{\href{http://arxiv.org/abs/#1}{\tt arXiv:\nolinkurl{#1}}}
\newcommand{\doi}[1]{\href{http://dx.doi.org/#1}{{\tt DOI:#1}}}

\newcommand{\googlebooks}[1]{(preview at \href{http://books.google.com/books?id=#1}{google books})}

\usepackage{xcolor}
\definecolor{dark-red}{rgb}{0.7,0.25,0.25}
\definecolor{dark-blue}{rgb}{0.15,0.15,0.55}
\definecolor{medium-blue}{rgb}{0,0,.8}
\definecolor{DarkGreen}{RGB}{0,150,0}
\definecolor{rho}{named}{red}
\hypersetup{
   colorlinks, linkcolor={purple},
   citecolor={medium-blue}, urlcolor={medium-blue}
}

\usepackage{longtable}
\usepackage{fullpage}

\theoremstyle{plain}
\newtheorem{thm}{Theorem}[section]
\newtheorem*{thm*}{Theorem}
\newtheorem{thmalpha}{Theorem}

\newtheorem{cor}[thm]{Corollary}

\newtheorem*{cor*}{Corollary}

\newtheorem*{conj*}{Conjecture}
\newtheorem{lem}[thm]{Lemma}

\newtheorem{facts}[thm]{Facts}
\newtheorem{prop}[thm]{Proposition}

\newtheorem*{quest*}{Question}
\newtheorem*{claim*}{Claim}

\theoremstyle{definition}
\newtheorem{defn}[thm]{Definition}
\newtheorem{construction}[thm]{Construction}

\newtheorem{assumption}[thm]{Assumption}
\newtheorem{nota}[thm]{Notation}

\newtheorem{ex}[thm]{Example}
\newtheorem{sub-ex}[thm]{Sub-Example}
\newtheorem{rem}[thm]{Remark}
\newtheorem*{rem*}{Remark}

\DeclareMathOperator{\Ad}{Ad}

\DeclareMathOperator{\coev}{coev}
\DeclareMathOperator{\End}{End}
\DeclareMathOperator{\ev}{ev}
\DeclareMathOperator{\Hom}{Hom}

\DeclareMathOperator{\op}{op}
\DeclareMathOperator{\ONB}{ONB}
\DeclareMathOperator{\Ob}{Ob}

\DeclareMathOperator{\spann}{span}
\DeclareMathOperator{\id}{id}
\DeclareMathOperator{\ind}{ind}

\DeclareMathOperator{\Irr}{Irr}

\DeclareMathOperator{\tr}{tr}

\newcommand{\comment}[1]{}

\newcommand{\be}{\begin{enumerate}[label=(\arabic*)]}
\newcommand{\ee}{\end{enumerate}}

\newcommand{\C}{\mathbb{C}}

\newcommand{\set}[2]{\left\{#1 \middle| #2\right\}}

\def\semicolon{;}
\def\applytolist#1{
    \expandafter\def\csname multi#1\endcsname##1{
        \def\multiack{##1}\ifx\multiack\semicolon
            \def\next{\relax}
        \else
            \csname #1\endcsname{##1}
            \def\next{\csname multi#1\endcsname}
        \fi
        \next}
    \csname multi#1\endcsname}

\def\calc#1{\expandafter\def\csname c#1\endcsname{{\mathcal #1}}}
\applytolist{calc}QWERTYUIOPLKJHGFDSAZXCVBNM;
\def\bbc#1{\expandafter\def\csname bb#1\endcsname{{\mathbb #1}}}
\applytolist{bbc}QWERTYUIOPLKJHGFDSAZXCVBNM;
\def\bfc#1{\expandafter\def\csname bf#1\endcsname{{\mathbf #1}}}
\applytolist{bfc}QWERTYUIOPLKJHGFDSAZXCVBNM;
\def\sfc#1{\expandafter\def\csname s#1\endcsname{{\sf #1}}}
\applytolist{sfc}QWERTYUIOPLKJHGFDSAZXCVBNM;
\def\fc#1{\expandafter\def\csname f#1\endcsname{{\mathfrak #1}}}
\applytolist{fc}QWERTYUIOPLKJHGFDSAZXCVBNM;

\newcommand{\Rep}{{\sf Rep}}

\newcommand{\FreeMod}{{\sf FreeMod}}
\newcommand{\Bim}{{\sf Bim}}
\newcommand{\bfBim}{{\sf Bim_{b.f.}}}
\renewcommand{\Vec}{{\sf Vec}}
\newcommand{\fdVec}{{\sf Vec_{f.d.}}}
\newcommand{\Hilb}{{\sf Hilb}}
\newcommand{\fdHilb}{{\sf Hilb_{f.d.}}}

\newcommand{\noshow}[1]{}
\newcommand{\MR}[1]{}

\usetikzlibrary{shapes}
\usetikzlibrary{backgrounds}
\usetikzlibrary{decorations,decorations.pathreplacing,decorations.markings}
\usetikzlibrary{fit,calc,through}
\usetikzlibrary{external}
\tikzset{
	super thick/.style={line width=3pt}
}
\tikzset{
    quadruple/.style args={[#1] in [#2] in [#3] in [#4]}{
        #1,preaction={preaction={preaction={draw,#4},draw,#3}, draw,#2}
    }
} 
\tikzstyle{shaded}=[fill=red!10!blue!20!gray!30!white]
\tikzstyle{unshaded}=[fill=white]
\tikzstyle{empty box}=[circle, draw, thick, fill=white, opaque, inner sep=2mm]
\tikzstyle{annular}=[scale=.7, inner sep=1mm, baseline]
\tikzstyle{rectangular}=[scale=.75, inner sep=1mm, baseline=-.1cm]
\tikzstyle{mid>}=[decoration={markings, mark=at position 0.5 with {\arrow{>}}}, postaction={decorate}]
\tikzstyle{mid<}=[decoration={markings, mark=at position 0.5 with {\arrow{<}}}, postaction={decorate}]
\tikzstyle{over}=[double, draw=white, super thick, double=]

\newcommand{\roundNbox}[6]{
	\draw[rounded corners=5pt, very thick, #1] ($#2+(-#3,-#3)+(-#4,0)$) rectangle ($#2+(#3,#3)+(#5,0)$);
	\coordinate (ZZa) at ($#2+(-#4,0)$);
	\coordinate (ZZb) at ($#2+(#5,0)$);
	\node at ($1/2*(ZZa)+1/2*(ZZb)$) {#6};
}

\begin{document}
\title{Operator algebras in rigid C*-tensor categories}
\author{Corey Jones and David Penneys}
\date{\today}
\maketitle
\begin{abstract} 
In this article, we define operator algebras internal to a rigid C*-tensor category $\cC$.
A C*/W*-algebra object in $\cC$ is an algebra object $\bfA$ in $\ind$-$\cC$ whose category of free modules $\FreeMod_{\cC}(\bfA)$ is a $\cC$-module C*/W*-category respectively.
When $\cC=\fdHilb$, the category of finite dimensional Hilbert spaces, we recover the usual notions of operator algebras.

We generalize basic representation theoretic results, such as the Gelfand-Naimark and von Neumann bicommutant theorems, along with the GNS construction.
We define the notion of completely positive maps between C*-algebra objects in $\cC$ and prove the analog of the Stinespring dilation theorem.

As an application, we discuss approximation and rigidity properties, including amenability, the Haagerup property, and property (T) for a connected W*-algebra $\bfM$ in $\cC$.
Our definitions simultaneously unify the definitions of analytic properties for discrete quantum groups and rigid C*-tensor categories.
\end{abstract}

\setcounter{tocdepth}{2}
\tableofcontents

\settocdepth{section}

\section{Introduction}\label{sec:Introduction}

Algebras of operators on Hilbert space were first introduced to give a rigorous mathematical understanding of quantum mechanics.
Of particular importance are von Neumann algebras (W*-algebras) and C*-algebras, introduced by von Neumann \cite{MR0009096} and Gelfand-Naimark \cite{MR0009426} respectively.
Later on, abstract algebraic characterizations of C* and W*-algebras were given, which make no reference to the underlying Hilbert space.
Operator algebras have seen important applications to many branches of mathematics, including representation theory, conformal and quantum field theory, and most recently topological phases of matter.

Classically, the symmetries of a mathematical object form a group.
In recent decades, we have seen the emergence of quantum mathematical objects whose symmetries form a group-like object called a \emph{tensor category}.
Two important examples of such objects are quantum groups and subfactors, which are said to encode \emph{quantum symmetry}.

The modern theory of subfactors began with Jones' landmark result in \cite{MR0696688} showing that the index of a ${\rm II}_1$ subfactor lies in the set
$
\set{4\cos^2(\pi/n)}{n\geq 3}\cup[4,\infty]
$.
We study a finite index subfactor $N\subseteq M$ by analyzing its standard invariant, which has a number of different axiomatizations.
In finite depth, we have Ocneanu' paragroups \cite{MR996454,MR1642584}, and in the general case, we have Popa's $\lambda$-lattices \cite{MR1334479} and Jones' planar algebras \cite{math.QA/9909027}.
We may also view the standard invariant as the rigid C*-tensor category whose objects are the bifinite $N-N$ Hilbert bimodules generated by $L^2(M)$ 
and whose morphisms are bounded $N-N$ bilinear intertwiners, together with the distinguished Frobenius algebra object $L^2(M)$.

We are currently seeing the emergence of new mathematical objects which encode \emph{enriched} quantum symmetry, including superfusion categories \cite{MR2609644,1603.05928,1603.09294,1606.03466}, tensor categories enriched in braided tensor categories \cite{MR2177301,Enriched}, para planar algebras \cite{1602.02662}, and anchored planar algebras in braided pivotal categories \cite{1607.06041}.
To understand these notions from an operator algebraic framework, we must develop a theory of \emph{enriched operator algebras}, or operator algebras \emph{internal} to a rigid C*-tensor category.

In turn, these enriched operator algebras give a uniform approach to analytic properties, like amenability, the Haagerup property, and property (T), for discrete groups, (discrete) quantum groups, subfactors \cite{MR1729488}, and rigid C*-tensor categories \cite{MR3406647}.

\subsection{Algebras in monoidal categories and module categories}

Suppose $\cC$ is a semi-simple monoidal category enriched in $\fdVec$, the category of finite dimensional vector spaces.
A unital algebra in $\cC$ is an object $a\in \cC$ together with morphisms $i: 1_\cC\to a$ and $m: a\otimes a\to a$ which satisfy unit and associativity axioms.
It is easy to see that the category of algebras in $\cC$ with algebra maps is equivalent to the category of lax monoidal functors $\cC^{\text{op}}\to \fdVec$ with lax monoidal natural transformations via the Yoneda embedding $a\mapsto \cC(\,\cdot\,, a)$.

However, we want to generalize all algebras, including infinite dimensional ones.
To do so, we replace $\cC$ with $\Vec(\cC)$, the category of linear functors $\cC^{\op}\to \Vec$, where $\Vec$ denotes the category of all vector spaces.
There is a natural notion of tensor product of two such functors akin to the Day convolution product \cite{MR0272852}.
(In fact, $\Vec(\cC)$ is equivalent to the $\ind$-category of $\cC$.)
Again, we have an equivalence of categories between algebra objects in $\Vec(\cC)$ and lax monoidal functors $\cC^{\text{op}}\to \Vec$. 
An ordinary complex algebra $A$ now corresponds to the involutive lax monoidal functor $\bfA:\cC=\Vec_{\sf f.d.}^{\text{op}} \to \Vec$ sending $\bbC$ to $A$. 

A common theme in mathematics is trading an object for its representation theory.
For example, the Gelfand transform allows us to trade a unital commutative C*-algebra $A$ for its compact Hausdorff topological space of algebra representations $A\to \bbC$.
In this sense, we think of C*-algebras as encoding non-commutative topology.

For each $c\in \cC$, we write $\mathbf{c} = \cC(\,\cdot\,, c)$ for its image under the Yoneda embedding.
We may trade our algebra object $\bfA\in \Vec(\cC)$ for its category $\FreeMod_\cC(\bfA)$ of \emph{free modules} in $\Vec(\cC)$, whose objects are right $\bfA$-modules of the form $\mathbf{c}\otimes \bfA$ for $c\in\cC$, and whose morphisms are right $\bfA$-module maps.
References for the free module functor include  \cite{MR1936496,MR2863377,1509.02937}.
Notice that $\FreeMod_\cC(\bfA)$ carries the structure of a \emph{left} $\cC$-\emph{module category}.
We may now recover the algebra structure on $\bfA$ from the category $\FreeMod_\cC(\bfA)$ with its distinguished base-point $\bfA$ by the Yoneda lemma, or the internal hom contruction, since for each $c\in \cC$, we have a canonical isomorphism 
$$
\bfA(c) \underset{\text{Yoneda}}{\cong} \Hom_{\Vec(\cC)}(\mathbf{c}, \bfA)\cong \Hom_{\FreeMod_\cC(\bfA)}(\mathbf{c}\otimes \bfA , \bfA).
$$
Thus the category algebra objects $\bfA\in \Vec(\cC)$ is equivalent to the category of left $\cC$-module categories $\cM$ with distinguished \emph{basepoint} $m\in\cM$ and whose objects are of the form $c\otimes m$ for $c\in\cC$.
We call such a left $\cC$-module category \emph{cyclic}.
Usually this result is stated as a correspondence between Morita classes of algebras and left $\cC$-module categories without basepoints.
Remembering the basepoint allows us to recover the actual algebra, not just its Morita class.

\begin{figure}[!ht]
\begin{tabular}{|l|l|}
\hline
Algebras in $\bfA\in\Vec(\cC)$ 
& 
Cyclic $\cC$-module categories
\\\hline
Trivial algebra $\mathbf{1}=\cC(\,\cdot\,,1_\cC)\in \Vec(\cC)$
&
Trivial cyclic $\cC$-module category $(\cC,1_\cC)$
\\
Algebra natural transformation $\bfA\Rightarrow \bfB$ 
& 
Cyclic $\cC$-module functor $(\cM,m)\to (\cN,n)$
\\
Endomorphism algebra $\bfL(\bfV)$
&
Cyclic $\cC$-module generated by $\bfV\in \Vec(\cC)$
\\
Representation $\bfA \Rightarrow \bfL(\bfV)$
&
Cyc.~$\cC$-mod.~functor $(\cM,m) \to (\Vec(\cC),\bfV)$
\\\hline
\end{tabular}
\caption{Morita correspondence between algebra objects and cyclic module categories}
\label{fig:MoritaCorrespondence1}
\end{figure}

\subsection{\texorpdfstring{$*$}{*}-Algebras and operator algebras}

More structure on $\cC$ is needed to define a $*$-structure on an algebra object $\bfA\in \Vec(\cC)$.
An \emph{involutive structure} on $\cC$ is an anti-tensor functor $(\overline{\,\cdot\,},\nu,\varphi, r): \cC\to \cC$ where
$\overline{\,\cdot\,}$ is an anti-linear functor,
$\nu_{a,b}: \overline{a}\otimes \overline{b} \to \overline{b\otimes a}$ is a family of natural isomorphisms,
$\varphi:\overline{\overline{\,\cdot\,}}\Rightarrow \id$ is a monoidal natural isomorphism, 
and $r: 1_\cC\to \overline{1_\cC}$ is a \emph{real structure}, all of which are compatible (see Section \ref{sec:Involutions} for more details).
Given an involutive structure on $\cC$, 
we define a $*$-algebra object in $\Vec(\cC)$ to be an involutive lax monoidal functor $\cC^{\text{op}} \to \Vec$.

Hilbert spaces provide a natural context to discuss representations of $*$-algebras and operator algebras.
This has two important consequences.

First, to define an operator algebra internal to an involutive tensor category, 
we restrict our attention to bi-involutive tensor categories \cite{1511.05226} which have a dagger structure compatible with the involutive structure.
(In the graphical calculus for tensor categories, one thinks of the involutive structure as reflection about the $y$-axis and the dagger structure as reflection about the $x$-axis.
The bi-involutive condition is that these reflections commute.)
In this article, we focus completely on the case when $\cC$ is a \emph{rigid} C*-\emph{tensor category} with simple unit object.
While there should be interesting results using arbitrary bi-involutive categories, our most important results rely on objects of $\cC$ being \emph{dualizable}, and on the fact tha we may equip each morphism space $\cC(a,b)$ with the structure of a finite dimensional Hilbert space.

Our Morita correspondence above in Figure \ref{fig:MoritaCorrespondence1} restricts to an equivalence of categories between $*$-algebras in $\Vec(\cC)$ 
and cyclic $\cC$-module \emph{dagger} categories. 
Using this finer correspondence, we say that a $*$-algebra object $\bfA\in \Vec(\cC)$ is a C*-\emph{algebra object} precisely when the dagger category $\FreeMod_\cC(\bfA)$ is a C*-category.
This definition is similar in spirit to the fact that being a C*-algebra is merely a property of a complex $*$-algebra, not extra structure. 
There is an obvious analogous definition for a W*-algebra object, which again is a property of a $*$-algebra object.

The second consequence is that to represent a $*$-algebra object $\bfA\in\Vec(\cC)$, we need the notion of a Hilbert space object in $\cC$.
As with vector spaces, we replace the category of Hilbert spaces $\Hilb$ with $\Hilb(\cC)$, the category of linear dagger functors $\bfH: \cC^{\op} \to \Hilb$ with bounded linear natural transformations.
This means if $\bfH, \bfK\in \Hilb(\cC)$, a natural transformation $\theta=(\theta_c)_{c\in\cC}: \bfH\Rightarrow \bfK$ is a morphism in $\Hom_{\Hilb(\cC)}(\bfH, \bfK)$ if and only if
$$
\sup_{c\in\cC}\|\theta_c : \bfH(c)\to \bfK(c)\|<\infty.
$$
Notice that in this case, the adjoint $\theta^* : \bfK \Rightarrow \bfH$ exists, and $\Hilb(\cC)$ is a W*-category.
Whereas $\Vec(\cC)$ is merely involutive, $\Hilb(\cC)$ is bi-involutive, as the dagger structure is compatible with the involutive structure.

Given a Hilbert space object $\bfH\in \Hilb(\cC)$, there is a canonical W*-algebra object $\bfB(\bfH)\in \Vec(\cC)$ given by the internal hom;
we have natural isomorphisms
$$
\bfB(\bfH)(a)
\underset{\text{Yoneda}}{\cong}
\Hom_{\Vec(\cC)}(\mathbf{a}, \bfB(\bfH)) 
\cong 
\Hom_{\Hilb(\cC)}(\mathbf{a}\otimes \bfH , \bfH).
$$

\begin{figure}[!ht]
\begin{tabular}{|l|l|}
\hline
$*$-Algebras in $\bfA\in\Vec(\cC)$ 
& 
Cyclic $\cC$-module dagger categories
\\\hline
$*$-Algebra nat.~trans.
&
Cyclic $\cC$-module dagger functor
\\
C*-algebra object 
&
Cyclic $\cC$-module C*-category
\\
W*-algebra object
&
Cyclic $\cC$-module W*-category
\\
W*-algebra object $\bfB(\bfH)$
&
Cyc.~$\cC$-mod.~W*-cat.~generated by $\bfH\in \Hilb(\cC)$
\\
$*$-Algebra representation into $\bfB(\bfH)$
&
Cyclic $\cC$-module dagger functor into $(\Hilb(\cC), \bfH)$
\\
Normal $*$-algebra nat.~trans.
&
Normal cyclic $\cC$-module dagger functor
\\\hline
\end{tabular}
\caption{Correspondence between $*$-algebra objects and cyclic module dagger categories}
\end{figure}

We prove many of the basic operator algebra theorems for operator algebras in $\cC$.

\begin{thmalpha}
[Gelfand-Naimark]
Given any \emph{C*}-algebra object $\bfA\in \Vec(\cC)$,
there is a Hilbert space object $\bfH\in \Hilb(\cC)$ and 
a faithful $*$-algebra natural transformation $\theta: \bfA\Rightarrow \bfB(\bfH)$.
\end{thmalpha}

We also have a notion of a state on $\bfA$ and a version of the GNS construction.

Given two C*-algebra objects $\bfA,\bfB\in \Vec(\cC)$, a $*$-natural transformation $\theta: \bfA\Rightarrow \bfB$ is called \emph{completely positive} if for every $c\in\cC$, the induced linear $*$-map 
$$
\theta_{\overline{c}\otimes c} : 
\End_{\FreeMod_\cC(\bfA)}(\mathbf{c}\otimes \bfA , \mathbf{c}\otimes \bfA)
\longrightarrow
\End_{\FreeMod_\cC(\bfB)}(\mathbf{c}\otimes \bfB , \mathbf{c}\otimes \bfB)
$$
maps positive elements to positive elements (note that both of these endomorphism algebras are C*-algebras).
Since the categories of free $\bfA/\bfB$-modules admits finite direct sums, positivity on $\cC$ implies complete positivity of the maps $\theta_{\overline{c}\otimes c}$.

\begin{thmalpha}
[Stinespring dilation]
Suppose $\theta: \bfA \Rightarrow \bfB(\bfH)$ is unital completely positive.
Then there is a Hilbert space object $\bfK\in\Hilb(\cC)$, a $*$-algebra natural transformation $\pi: \bfA\Rightarrow \bfB(\bfK)$, and an isometry $v: \bfH\Rightarrow \bfK$ such that $\theta = \Ad(v)\circ \pi$ as $*$-natural transformations.
\end{thmalpha}

We can also define the notion of the commutant of a $*$-algebra object $\bfA\in \Vec(\cC)$ in a given representation.
To do so, we pass to the corresponding $\cC$-module dagger categories.
A representation $\pi: \bfA \Rightarrow \bfB(\bfH)$ gives us a $\cC$-module dagger functor $\check{\pi}: \FreeMod_\cC(\bfA) \to \Hilb(\cC)$ sending $\bfA$ to $\bfH$.
We define the commutant of a cyclic left $\cC$-module dagger subcategory $\cM\subset \Hilb(\cC)$ with basepoint $\bfH$ to be the cyclic \emph{right} $\cC$-module dagger subcategory $\cM'\subset \Hilb(\cC)$ with basepoint $\bfH$ whose morphism space $\cM'(\bfH\otimes \mathbf{c}, \bfH\otimes \mathbf{d})$ is the space of all $f\in \Hom_{\Hilb(\cC)}(\bfH\otimes \mathbf{c}, \bfH\otimes \mathbf{d})$ which commute with all morphisms from $\cM$:
$$
\set{f:\bfH\otimes \mathbf{c} \to \bfH\otimes \mathbf{d}\,}{\,
\text{for all $a,b\in\cC$ and $g\in \cM(\mathbf{a}\otimes \bfH, \mathbf{b}\otimes \bfH)$, }
\begin{tikzpicture}[baseline=-.1cm]
	\draw (0,-1.1) -- (0,1.1);
	\draw[thick, red] (-.4,-1.1) -- (-.4,-.5);
	\draw[thick, blue] (-.4,-.5) -- (-.4,1.1);
	\draw[thick, DarkGreen] (.4,-1.1) -- (.4,.5);
	\draw[thick, orange] (.4,.5) -- (.4,1.1);
	\roundNbox{unshaded}{(.1,.5)}{.3}{0}{.2}{$f$}
	\roundNbox{unshaded}{(-.1,-.5)}{.3}{.2}{0}{$g$}
	\node at (0,-1.3) {\scriptsize{$\bfH$}};
	\node at (0,1.3) {\scriptsize{$\bfH$}};
	\node at (.4,-1.3) {\scriptsize{$\mathbf{c}$}};
	\node at (.4,1.3) {\scriptsize{$\mathbf{d}$}};
	\node at (-.4,-1.3) {\scriptsize{$\mathbf{a}$}};
	\node at (-.4,1.3) {\scriptsize{$\mathbf{b}$}};
\end{tikzpicture}
=
\begin{tikzpicture}[baseline=-.1cm]
	\draw (0,-1.1) -- (0,1.1);
	\draw[thick, red] (-.4,-1.1) -- (-.4,.5);
	\draw[thick, blue] (-.4,.5) -- (-.4,1.1);
	\draw[thick, DarkGreen] (.4,-1.1) -- (.4,-.5);
	\draw[thick, orange] (.4,-.5) -- (.4,1.1);
	\roundNbox{unshaded}{(.1,-.5)}{.3}{0}{.2}{$f$}
	\roundNbox{unshaded}{(-.1,.5)}{.3}{.2}{0}{$g$}
	\node at (0,-1.3) {\scriptsize{$\bfH$}};
	\node at (0,1.3) {\scriptsize{$\bfH$}};
	\node at (.4,-1.3) {\scriptsize{$\mathbf{c}$}};
	\node at (.4,1.3) {\scriptsize{$\mathbf{d}$}};
	\node at (-.4,-1.3) {\scriptsize{$\mathbf{a}$}};
	\node at (-.4,1.3) {\scriptsize{$\mathbf{b}$}};
\end{tikzpicture}
}.
$$
We similarly define the commutant of a cyclic right $\cC$-module dagger subcategory of $\Hilb(\cC)$ with basepoint $\bfH$ as a left $\cC$-module dagger subcategory.
Thus the bicommutant of a left $\cC$-module dagger category is again a left $\cC$-module dagger category, which allows us to take the bicommutant of an algebra object in $\Vec(\cC)$.

\begin{thmalpha}
[von Neumann bicommutant]
Given a cyclic left $\cC$-module dagger subcategory $\cM\subset \Hilb(\cC)$ with basepoint $\bfH$, the weak*-closure of $\cM$ is equivalent to its bicommutant $\cM''$.
Thus if $\bfM\in \Vec(\cC)$ is a \emph{W*}-algebra object and $\pi: \bfM \Rightarrow \bfB(\bfH)$ is a faithful normal $*$-algebra natural transformation, we have a $*$-algebra natural isomorphism $\bfM\cong \bfM''$.
\end{thmalpha}

To prove these results, we bootstrap from the ordinary theory of operator algebras, which is the case $\cC=\fdHilb$.
In this case, $\Vec(\cC) \cong \Vec$, and we see that $*$/C*/W*-algebras $\bfA\in\Vec(\cC)$ exactly correspond to $*$/C*/W*-algebras respectively.
We expect that much of the world of ordinary operator algebras generalizes to the enriched setting.
We summarize the dictionary between ordinary and enriched operator algebras in the table below.

\begin{figure}[!ht]
\begin{tabular}{|l|l|}
\hline
Operator algebras in $\fdHilb$ 
& 
Operator algebras in $\cC$ 
\\\hline
C*-algebra $A$ 
& 
C*-algebra object $\bfA\in\Vec(\cC)$
\\
$*$-homomorphism $A\to B$ 
& 
$*$-algebra n.t.~$\bfA \Rightarrow \bfB$
\\
Hilbert space $H$ 
& 
Hilbert space object $\bfH\in\Hilb(\cC)$
\\
C*-algebra $B(H)$ 
& 
C*-algebra object $\bfB(\bfH)\in \Vec(\cC)$ 
\\
Representation $A\to B(H)$ 
& 
$*$-Algebra n.t.~$\bfA \Rightarrow \bfB(\bfH)$
\\
c.p.~map $A\to B$ 
& 
$*$-n.t.~$\theta: \bfA\Rightarrow \bfB$
such that $\theta_{\overline{c}\otimes c}$ is positive for all $c\in\cC$
\\
State $\varphi: A\to \bbC$
&
State $\varphi: \bfA(1_\cC) \to \bbC$
\\\hline
von Neumann algebra $M$
&
W*-algebra object $\bfM\in \Vec(\cC)$
\\
Trivial algebra $\bbC$
&
Trivial algebra object $\mathbf{1}\in \Vec(\cC)$
\\
Representation $M\to B(H)$
&
normal $*$-algebra n.t.~$\bfM \Rightarrow \bfB(\bfH)$
\\
Commutant $M'\cap B(H)$
&
Commutant W*-algebra object $\bfM'\in \Vec(\cC)$
\\
$B(H)'=\bbC$
&
$\bfB(\bfH)' = \mathbf{1}$
\\
W*-completion $A''\subseteq B(H)$
&
W*-algebra object $\bfA''\in \Vec(\cC)$
\\\hline
\end{tabular}
\caption{
Analogy between ordinary operator algebras and operator algebras in $\cC$.
We use the shorthand `n.t.' for natural transformation and `c.p.' for completely positive.}
\end{figure}

\subsection{Application to approximation and rigidity properties}

Analytic properties such as amenability, the Haageruup property, and property (T) were first defined for countable discrete groups $G$ to characterize the local topological behavior of the  \emph{unitary dual}, i.e., the C*-tensor category of unitary representations, with respect to the Fell topology \cite{MR2415834}.
One can replace this unitary dual with the state space of the universal C*-algebra of $G$, and a state may be viewed as an element of $\ell^\infty(G)$.
In turn, approximation properties, like amenability and the Haagerup property can be defined by the ability to approximate an arbitrary state in $\ell^\infty(G)$ with states that decay sufficiently nicely in $\ell^{\infty}(G)$.
Rigidity properties, like property (T), can be defined by the inability to approximate the identity with \textit{any} class that is small at infinity.

These analytic properties were later generalized to ${\rm II}_1$-factors and quantum groups.
More recently, Popa and Vaes \cite{MR3406647} introduced these notions for rigid C*-tensor categories, generalizing Popa's definitions for subfactors \cite{MR1302385,MR1729488}.
This has led to a great deal recent interest in understanding these properties from different perspectives \cite{MR3509018, MR3447719,1511.06332}.
It has become clear that in this setting, the unitary dual of a rigid C*-tensor category $\cC$ is the tensor category $\cZ(\Hilb(\cC))$.
This category was first introduced by Neshveyev and Yamashita \cite{MR3509018} as the unitary ind-category of $\cC$, and it was shown to be equivalent to the representation category of Ocneanu's tube algebra \cite{1511.07329}.

In this article, we define analytic properties of W*-algebra objects $\bfM\in \Vec(\cC)$ for which $\bfM(1_\cC)$ is trivial.
Our definitions simultaneously unify the definitions of analytic properties for countable discrete groups, quantum groups, subfactors, and rigid C*-tensor categories.  The basic idea is to find a canonical correspondence between states (in the ordinary sense) on the algebra of interest (quantum group algebra, tube algebra, etc.) and ucp maps (in our categorical sense) on an appropriately defined W*-algebra object in the appropriate category.

\subsection{Application to enriched quantum symmetry}

As mentioned earlier, one of the main motivations of this project is the emergence of objects which can be thought to encode enriched quantum symmetries, like superfusion categories, tensor categories enriched in braided tensor categories, and planar algebras internal to a braided pivotal category. 
Now that we have a notion of an enriched fusion category and planar algebra, we should look for the notion of an enriched subfactor.
We would then like to prove that enriched categories and planar algebras arise as the standard invariants of these enriched subfactors.
Indeed, in \cite[Rem.~2.22]{1602.02662}, Jaffe and Liu tell us the theory works in full generality in the $G$-graded setting for $G$-graded subfactors acting on $G$-graded Hilbert spaces for $G$ an abelian group.

The anchored planar algebras defined in \cite{1607.06041} allow the enriching category $\cC$ to be an arbitrary braided pivotal category, whose objects are not necessarily Hilbert spaces.
Thus a more abstract version of operator algebra internal to a rigid C*-tensor category is needed for this more abstract $\cC$-graded subfactor theory.
Indeed, our present notion of C*/W*-algebra object in $\Vec(\cC)$ will be used in \cite{uAPA} to define the notion of a unitary anchored planar algebra in $\Vec(\cC)$ where $\cC$ is a unitary ribbon category.
We anticipate a future generalization of subfactor theory to the $\cC$-graded setting.

\subsection{Acknowledgements}

The authors would like to thank
Marcel Bischoff,
Shamindra Ghosh,
Andr\'{e} Henriques,
Zhengwei Liu,
Thomas Sinclair, 
James Tener and
Makoto Yamashita
for helpful conversations.
Corey Jones was supported by Discovery Projects `Subfactors and symmetries' DP140100732 and `Low dimensional categories' DP160103479 from the Australian Research Council.
David Penneys was supported by DMS NSF grants 1500387 and 1655912.
Both authors would like to thank the Hausdorff Institute for Mathematics for their hospitality during the 2016 trimester on von Neumann algebras, along with the organizers Dietmar Bisch, Vaughan Jones, Sorin Popa, and Dima Shlyakhtenko.

\settocdepth{subsection}

\section{Background}\label{sec:Background}

\begin{nota}
In this article, as much as possible, we use the following types of letters to denote the following types of things.
Categories are denoted by $\cC, \cM, \cN$.
Objects in $\cC$ are denoted by lower case roman letters starting from the beginning of the alphabet $a,b,c$, 
and morphisms in $\cC(c,a\otimes b)$ will be denoted by greek letters $\alpha,\beta,\gamma$ and $\phi,\psi$.

Functors will be denoted by bold-face capital letters, e.g., $\bfA,\bfF$.
For $a\in \cC$, the representable functor $\cC(\,\cdot\,,a)$ is denoted by the bold-face lower case letter $\mathbf{a}$.
Natural transformations are usually denoted by $\theta$, $\sigma$, and $\tau$.

Vectors in vector spaces or Hilbert spaces are denoted by $f,g$ or by $\eta, \xi, \zeta$ depending on the context.
\end{nota}

We begin with a review of the notions of categories that arise in this article.
Sections \ref{sec:Categories}--\ref{sec:RigidC*TensorCategories} are adapted from \cite[\S2.1-2.2]{1511.05226}.

\subsection{Linear, dagger, and C*-categories}
\label{sec:Categories}

In this article, a linear category will mean a $\Vec$-enriched category, where $\Vec$ denotes the category of complex vector spaces, which may be infinite dimensional.
We denote by $\fdVec$ the category of finite dimensional complex vector spaces.

A \emph{dagger category} is a linear category $\cC$ equipped with an anti-linear map $\cC(a,b)\to \cC(b,a)$ for all $a,b\in\cC$ called the \emph{adjoint}.
It must satisfy the axioms $\psi^{**} = \psi$ and $(\psi\circ \phi)^* = \phi^*\circ \psi^*$ for composable $\psi, \phi$, which implies $\id_a^*=\id_a$ for all $a\in\cC$.
An invertible morphism $\psi\in \cC(a,b)$ is called \emph{unitary} if $\psi^*=\psi^{-1}$.

We denote by $\Hilb$ the category of complex Hilbert spaces, and $\fdHilb$ is the category of finite dimensional complex Hilbert spaces.

\begin{rem}
\label{rem:PropertyStarAlgebraC*}
Recall that a $*$-algebra being a C*-algebra is a property, not extra structure.
Given a $*$-algebra $A$, $A$ is a C*-algebra if and only if
$$
\|a\|^2 = \sup\set{|\lambda|\geq 0\,}{\,a^*a-\lambda 1_A\text{ is not invertible}}
$$
defines a genuine norm on $A$, $A$ is complete in this norm, the norm is sub-multiplicative $\|ab\|\leq \|a\|\cdot\|b\|$, and the norm satisfies the C*-axiom $\|a^*a\|=\|a\|^2$.
\end{rem}

\begin{defn}
A dagger category $\cC$ is called a C*-\emph{category} if it satisfies the following two properties (which are not extra structure):
\begin{itemize}
\item 
For every $a,b\in\cC$ and $\psi\in \cC(a,b)$, there is a $\phi\in \cC(a, a)$ such that $\psi^*\circ \psi = \phi^*\circ \phi$.
\item 
For each $a,b\in\cC$, the function $\|\cdot\|: \cC(a,b)\to [0,\infty]$ given by
$$
\|\psi\|^2 = \sup\set{|\lambda|\geq 0\,}{\,\psi^*\circ \psi - \lambda\id_a \text{ is not invertible}}
$$
is a norm on $\cC(a,b)$, $\cC(a,b)$ is complete in this norm, the norm is sub-multiplicative $\|\psi\circ \phi\|\leq \|\psi\|\cdot \|\phi\|$, and these norms satisfy the C*-axiom $\|\psi^*\circ \psi\|=\|\psi\|^2$.
\end{itemize}
Equivalently, by \cite{MR808930}, a dagger category $\cC$ is a C*-category if it admits a faithful dagger functor $\cC\to \Hilb$ which is norm-closed on the level of hom spaces.
When $\cC$ admits direct sums, a dagger category is a C*-category if and only if $\cC(a,a)$ is a C*-algebra for every $a\in \cC$
(see \cite[Def.~1.1]{MR808930}, and use Roberts' $2\times2$ trick to see each $\cC(a,b)$ as a closed subspace of the C*-algebra $\End_\cC(a\oplus b)$).

A C*-category $\cC$ is called a W*-\emph{category} if each Banach space $\cC(a,b)$ has a predual $\cC(a,b)_*$.
Again, this is a property, and not extra structure.
When $\cC$ admits direct sums, we recall from \cite[Lem.\,2.6]{MR808930} that a C*-category is a W*-category if and only if $\cC(a,a)$ is a von Neumann algebra for every $a\in \cC$.
\end{defn}

A functor between linear categories $\bfF: \cC\to \cD$ is called \emph{linear} if the map $\bfF: \cC(a,b) \to \cD(\bfF(a), \bfF(b))$ is linear for all $a,b\in\cC$.
If $\cC$ and $\cD$ are dagger categories, then a linear functor is called a \emph{dagger functor} if $\bfF(\psi^*) = \bfF(\psi)^*$ for all morphisms $\psi$ in $\cC$.
A dagger functor between W*-categories $\bfF: \cC\to \cD$ is called \emph{normal} if for all $a,b\in \cC$, the map $\cC(a,b) \to \cD(\bfF(a), \bfF(b))$ is continuous with respect to the weak*-topologies.
Equivalently, if $\cC, \cD$ admit direct sums, $\bfF$ is normal if for all $c\in\cC$, the induced $*$-homomorphism $\cC(c,c) \to \cD(\bfF(c), \bfF(c))$ is normal.

\subsection{Involutions on tensor categories}
\label{sec:Involutions}

A \emph{tensor category} is a linear monoidal category $(\cC, \otimes, \alpha, \lambda, \rho, 1_\cC)$, where $\otimes: \cC\times \cC \to \cC$ is the tensor product bifunctor, which is associative up to the associator natural isomorphisms $\alpha$, which satisfy the pentagon axiom.
We denote the unit object by $1_\cC$, and $\lambda, \rho$ are the unitor natural isomorphisms which satisfy the triangle axioms.
Wherever possible in the sequel, we suppress the associators and unitors to ease the notation.

A \emph{dagger tensor category} is a dagger category and a tensor category such that the associator and unitor natural isomorphisms are unitary, and which satisfies the compatibility condition $(\psi\otimes \phi)^* = \psi^*\otimes \phi^*$. 
A C*-\emph{tensor category} is a dagger tensor category whose underlying dagger category is a C*-category.

\begin{defn}
A tensor category is called \emph{involutive} if there is a covariant anti-linear functor $\overline{\,\cdot\,}: \cC\to \cC$ called the \emph{conjugate}.
This functor is involutive, meaning there are natural isomorphisms $\varphi_c: c\to \overline{\overline{c}}$ for all $c\in \cC$ satisfying $\overline{\varphi_c}=\varphi_{\overline{c}}$, and anti-monoidal, meaning there are natural isomorphisms $\nu_{a,b}: \overline{a}\otimes \overline{b}\to \overline{b\otimes a}$ and an isomorphism $r: 1\to \overline{1}$ satisfying the following axioms:
\begin{itemize}
\item
(associativity)
$\nu_{a,c\otimes b} \circ (\id_{\overline{a}}\otimes \nu_{b,c}) = \nu_{b\otimes a, c} \circ (\nu_{a,b}\otimes \id_{\overline{c}})$
\item
(unitality)
$\nu_{1,a}\circ (r\otimes \id_{\overline{a}}) = \id_{\overline{a}} = \nu_{a,1} \circ (\id_{\overline{a}} \otimes r)$.
\end{itemize}
We require $\varphi$, $\nu$, and $r$ to be compatible:
$\varphi_1 = \overline{r}\circ r$
and
$\varphi_{a\otimes b} = \overline{\nu_{b\otimes a}} \circ \nu_{\overline{a}\otimes \overline{b}} \circ (\varphi_a \otimes \varphi_b)$.
\end{defn}

\begin{ex}
The category $\Vec$ is involutive where the conjugate is defined to be taking the complex conjugate vector space.
We denote the conjugate of $V\in \Vec$ by $\overline{V}=\set{\overline{v}}{v\in V}$, where $\lambda \overline{v} = \overline{\overline{\lambda}v}$ for all $\lambda\in \bbC$.
Given a map $T: V\to W$, we define $\overline{T} : \overline{V}\to \overline{W}$ by $\overline{T} \overline{v} = \overline{Tv}$.
Note that $\overline{T}$ is linear, while $T\mapsto \overline{T}$ is anti-linear.
\end{ex}

The notions of dagger category and involutive tensor category can be combined to form the notion of a \emph{bi-involutive} tensor category, defined in \cite[\S2.1]{1511.05226}.

\begin{defn}
\label{defn:bi-involutive}
A \emph{bi-involutive tensor category} is an involutive dagger tensor category such that $\overline{\,\cdot\,}$ is a dagger functor, and the natural isomorphisms $\varphi$, $\nu$, and $r$ are all unitary.

Note that being bi-involutive is a property of an involutive tensor category which is a dagger tensor category, not extra structure.
\end{defn}

\begin{ex}
The category $\Hilb$ is bi-involutive with complex conjugation and adjoints.
\end{ex}

\begin{rem}
\label{rem:CopInvolutive}
If $\cC$ is (bi-)involutive, then $\cC^{\op}$ is too by defining $\varphi_{\cC^{\op}} = \varphi^{-1}_\cC$, $\nu_{\cC^{\op}} = \nu^{-1}_\cC$, and $r_{\cC^{\op}} = r^{-1}_\cC$.
\end{rem}

A lax tensor functor between tensor categories $\cC$ and $\cD$ consists of a triple $(\bfF, \mu, i)$ where $\bfF: \cC\to \cD$ is a functor equipped with a morphism $\iota: 1_\cD\to \bfF(1_\cC)$ and a natural transformation $\mu_{a,b}: \bfF(a)\otimes \bfF(b) \to \bfF(a\otimes b)$ for $a,b\in\cC$ which satisfies the following axioms:
\begin{itemize}
\item
(associativity)
$\mu_{a,b\otimes c} \circ (\id_{\bfF(a)}\otimes \mu_{b,c}) = \mu_{a\otimes b,c}\circ(\mu_{a,b}\otimes \id_{\bfF(c)})$
\item
(unitality)
$\mu_{1,a}\circ  (\iota \otimes \id_{\bfF(a)}) = \id_{\bfF(a)} = \mu_{a,1} \circ (\id_{\bfF(a)}\otimes \iota)$.
\end{itemize}
A lax tensor functor is called a (strong) tensor functor if $\iota$ and all $\mu_{a,b}$ are isomorphisms.
A natural transformation of (lax) tensor functors $\theta: (\bfF, \mu^\bfF, \iota^\bfF) \Rightarrow (\bfG, \mu^\bfG, \iota^\bfG)$ is a natural transformation $\theta: \bfF\Rightarrow \bfG$ such that $\theta_{a\otimes b}\circ \mu^\bfF_{a,b} = \mu^\bfG_{a,b}\circ (\theta_a\otimes \theta_b)$ for all $a,b\in \cC$ and $\theta_{1_\cC}\circ \iota^\bfF = \iota^\bfG$.

If $\cC, \cD$ are dagger tensor categories, then a lax tensor functor is called a dagger lax tensor functor if $\bfF$ is a dagger functor.
A dagger lax tensor functor is called a dagger (strong) tensor functor if $\iota$ and all $\mu_{a,b}$ are unitary isomorphisms.

\begin{defn}
\label{defn:InvolutiveFunctor}
Suppose now $\cC$ and $\cD$ are involutive tensor categories.
A (lax or strong) tensor functor $(\bfF, \mu,\iota): \cC\to \cD$ is called \emph{involutive} if it is equipped with a natural isomorphism $\chi_a : \bfF(\overline{a})\to \overline{\bfF(a)}$ for all $a\in \cC$ satisfying the axioms
\begin{itemize}
\item
(involutive)
$
\varphi_{\bfF(a)}
=
\overline{\chi_{a}}\circ \chi_{\overline{a}} \circ \bfF(\varphi_a)
$
\item
(unital)
$\chi_{1_\cC} \circ \bfF(r_\cC) \circ \iota = \overline{\iota}\circ r_\cD$
\item
(monoidal)
$\chi_{a\otimes b} \circ \bfF(\nu_{b, a}) \circ \mu_{\overline{b}, \overline{a}} 
=
\overline{\mu_{a,b}} \circ \nu_{\bfF(b),\bfF(a)}\circ (\chi_b \otimes \chi_a)$. 
\end{itemize} 

Now suppose $\cC$ and $\cD$ are bi-involutive tensor categories.
An involutive (lax or strong) tensor functor $(\bfF, \mu,\iota, \chi):\cC\to\cD$ is called \emph{bi-involutive} if $\bfF$ is also a dagger tensor functor and all $\chi_a$ are unitary isomorphisms.
\end{defn}

\subsection{Rigid C*-tensor categories}
\label{sec:RigidC*TensorCategories}

\begin{defn}
A tensor category $\cC$ is called \emph{rigid} if every $c\in \cC$ admits 
\begin{itemize}
\item
a dual object $c^{\vee}$ together with maps $\ev_c: c^\vee \otimes c\to 1_\cC$ and $\coev_c: 1_\cC\to c\otimes c^\vee$ which satisfy the zig-zag axioms 
$(\id_c \otimes \ev_c)\circ (\coev_c \otimes \id_c) = \id_c$ and
$(\ev_{c} \otimes \id_{c^\vee})\circ (\id_{c^\vee}\otimes \coev_c) = \id_{c^\vee}$, and
\item
a predual object $c_{\vee}$ such that $(c_\vee)^\vee \cong c$.
\end{itemize}
The dual of a morphism $\psi\in \cC(a,b)$ is given by 
$$
\psi^\vee 
= 
(\ev_b \otimes \id_{a^\vee})\circ (\id_{b^\vee}\otimes \psi \otimes \id_{a^\vee}) \circ (\id_{b^\vee}\otimes \coev_a):b^\vee \to a^\vee.
$$
\end{defn}

A tensor category is called \emph{semi-simple} if it is equivalent \emph{as a category} to a direct sum of copies of $\fdVec$.
Equivalently, $\cC$ is semi-simple if it admits finite direct sums (including the empty direct sum, i.e., the zero object), and every object is a direct sum of finitely many (maybe zero) simple objects, which satisfy $\cC(a,a)\cong \bbC$.

\begin{assumption}
All our rigid C*-tensor categories in this article are assumed to have simple unit object and admit finite direct sums and subobjects.
However, it is important to note that there are interesting examples without simple unit objects, e.g., unitary multifusion categories.
\end{assumption}

\begin{ex}[{\cite[\S2.1]{1511.05226}}]
Let $\cC$ be a rigid C*-tensor category.
By \cite{MR1444286}, $\cC$ is semi-simple, since all endomorphism algebras are finite dimensional C*-algebras.
By \cite[Thm.\,4.7]{MR2091457} and \cite[\S4]{MR3342166}, a rigid C*-tensor category $\cC$ has a canonical bi-involutive structure.
The conjugation is determined up to unique unitary isomorphism by 
$\ev_a: \overline{a}\otimes a\to 1_\cC$ and $\coev_a : 1_\cC\to a\otimes \overline{a}$ which satisfy the zig-zag axioms, together with the \emph{balancing condition}
\begin{equation}
\label{eq:BalancingCondition}
\coev_{a}^* \circ (\psi\otimes \id_{\overline{a}}) \circ \coev_a 
=
\ev_a \circ (\id_{\overline{a}}\otimes \psi)\circ \ev_a^* 
\end{equation}
for all $\psi\in \cC(a\to a)$.
The above equation gives a scalar multiple of $\id_{1_\cC}$ for each $a\in \cC$ by setting $\psi=\id_a$ called the \emph{quantum dimension} of $a$, which is denoted $\dim_\cC(a)$ or $d_a$.

The conjugate of a morphism $\psi\in \cC(a,b)$ is $\overline{\psi} = (\psi^*)^\vee: \overline{a}\to \overline{b}$.
The coherence structure isomorphims $j, \nu, \varphi$ are given by:
\begin{equation}
\label{eq:Bi-involutiveStructureOnC}
\begin{split}
r&=\coev_1
\\
\nu_{a,b} 
&= 
(\ev_a \otimes \id_{\overline{b\otimes a}}) \circ (\id_{\overline{a}} \otimes \ev_b \otimes \id_{a\otimes \overline{b\otimes a}}) \circ (\id_{\overline{a}\otimes \overline{b}}\otimes \coev_{b\otimes a})
\\
\varphi_a &= (\id_{\overline{\overline{a}}} \otimes \ev_a) \circ (\ev_{\overline{a}}^* \otimes \id_a) = (\coev_a^*\otimes \id_{\overline{\overline{a}}})\circ (\id_a \otimes \coev_{\overline{a}})
\end{split}
\end{equation}
where the second equivalent definition of $\varphi$ above follows from the balancing condition.
The unitary natural isomorphims $\varphi$ equip $\cC$ with a canonical spherical pivotal structure.
\end{ex}

An important class of examples of rigid C*-tensor categories are unitary fusion categories, which only have finitely many isomorphims classes of simple objects.

\begin{ex}
From our perspective, one of the most important examples of a rigid C*-tensor category is the category of bifinite bimodules over a ${\rm II}_1$ factor $(N,\tr_N)$, denoted $\bfBim(N)$.
The tensor product is given by the Connes fusion tensor product, denoted $\boxtimes_N$ (see \cite{MR1424954} for more details).

Let $\Omega\in L^2(N)$ be the image of $1\in N$.
Given an $H\in \bfBim(N)$, we call a $\xi\in H$ \emph{right} $N$-\emph{bounded} if the map $N\Omega \to H$ given by $n\Omega \mapsto \xi n$ is bounded.
We denote the extension of this map by $L_\xi:L^2(N)\to H$.
There is a similar notion of a vector being right bounded, and by \cite[Prop.~1.5 and Rem.~1.6]{MR1424954}, $\xi\in H$ is left $N$-bounded if and only if it is right $N$-bounded.
We denote the subspace of bi-bounded vectors of $H$ by $H^\circ$.
It is easy to see that $L_\xi$ is right $N$-linear, and thus so is its adjoint.
Thus for every $\eta, \xi\in H^\circ$, we have an $N$-valued inner product given by $\langle \eta|\xi\rangle_A = L_\eta^*L_\xi$, which defines a unique element of $N$, since it commutes with the right $N$-action.

For $H\in \bfBim(N)$, the conjugate bimodule $\overline{H}$ is the conjugate Hilbert space with $N-N$ action given by $n\cdot \overline{\xi}\cdot m = \overline{m^*\xi n^*}$.
We see that the map $\ev_H:\overline{H}\boxtimes_N H \to 1 :=L^2(N)$ defined by $\overline{\eta}\otimes \xi \mapsto \langle \eta|\xi\rangle_N$ is a bounded $N-N$ bimodule map.

By \cite{MR561983}, there is a finite subset $\cB=\{\beta\}\subset H^\circ$ called a \emph{right} $N$-\emph{basis} such that $\sum_{\beta\in \cB} L_\beta L_\beta^* = \id_H$.
This means that for all $\xi\in H^\circ$, we have 
$\xi = \sum_{\beta\in \cB} \beta \langle \beta| \xi\rangle_N$.
There is a similar notion of left $N$-basis, but we will not need it.
We see that the map $\coev_{H}:L^2(N)=1\to H\boxtimes_N \overline{H}$ given by $n\Omega \mapsto \sum_{\beta\in \cB} n\beta \otimes \overline{\beta}$ is independent of the choice of $\cB$.
Moreover, 
it is easy to verify that the maps $\ev_H$ and $\coev_H$ satisfy the zig-zag relations.

We note that $\ev_H$ and $\coev_H$ do not necessarily satisfy the balancing condition \eqref{eq:BalancingCondition}.
Standard solutions to the conjugate equations are obtained by renormalization on irreducible bimodules, and extending to direct sums in the standard way.
We refer the reader to \cite{MR1444286,MR3342166} for more details.
\end{ex}

\begin{nota}
In the remainder of this article, $\cC$ is a rigid C*-tensor category (with simple unit object $1_\cC$), and $\Irr(\cC)$ is a fixed set of representatives of the simple objects of $\cC$.
\end{nota}

\begin{defn}
For all $a,b\in\cC$, $\cC(a,b)$ is a Hilbert space with inner product\footnote{Here, we carry two notations for the inner product on a Hilbert space $H$: $\langle x,y\rangle_H = \langle y|x\rangle_H$.
The first is linear on the left, and the second is linear on the right.}
 given by
\begin{equation}
\label{eq:InnerProductInC}
\langle \phi ,\psi\rangle_{\cC(a,b)} 
=
\langle \psi |\phi\rangle_{\cC(a,b)} :=
\begin{tikzpicture}[baseline=-.1cm]
	\draw (0,1.3) arc (180:0:.3cm) -- (.6,-1.3) arc (0:-180:.3cm) -- (0,1.3);
	\roundNbox{unshaded}{(0,1)}{.3}{0}{0}{$\varphi_a$}
	\roundNbox{unshaded}{(0,0)}{.3}{0}{0}{$\psi^*$}
	\roundNbox{unshaded}{(0,-1)}{.3}{0}{0}{$\phi$}
	\node at (-.2,1.5) {\scriptsize{$\overline{\overline{a}}$}};
	\node at (-.2,.5) {\scriptsize{$a$}};
	\node at (.8,0) {\scriptsize{$\overline{a}$}};
	\node at (-.2,-.5) {\scriptsize{$b$}};
	\node at (-.2,-1.5) {\scriptsize{$a$}};
\end{tikzpicture}
\,.
\end{equation}
For all $a,b,c\in\Irr(\cC)$, we fix orthonormal bases $\ONB(c,a\otimes b)$ for $\cC(c,a\otimes b)$ with respect to the above inner product.
\end{defn}

Now taking $a,c,d,e\in\cC$, we have natural vector space isomorphisms
\begin{equation}
\label{eq:BoxtimesTensorProduct}
\bigoplus_{b\in \Irr(\cC)}
\cC(c, a\otimes b) \otimes \cC(b, d\otimes e)
\cong
\cC(c, a\otimes d\otimes e) 
\cong
\bigoplus_{f\in \Irr(\cC)}
\cC(f, a\otimes d)\otimes \cC(c, f\otimes e).
\end{equation}
These isomorphisms are \emph{not} Hilbert space isomorphisms, since the sets
\begin{equation}
\label{eq:NotONBs}
\begin{split}
&\set{(\id_a\otimes \beta)\circ \alpha}{b\in \Irr(\cC),\,\, \alpha\in \cC(c, a\otimes b),\text{ and } \beta\in \cC(b, d\otimes e)}
\\
&\set{(\gamma\otimes \id_e)\circ \delta}{f\in \Irr(\cC),\,\,\gamma\in \cC(f, a\otimes d)\text{ and } \delta\in \cC(c, f\otimes e)}
\end{split}
\end{equation}
are \emph{not} orthonormal bases for $\cC(c, a\otimes d\otimes e)$ under the inner product from \eqref{eq:InnerProductInC}.
To fix this problem we define a different inner product on the tensor product spaces above in \eqref{eq:BoxtimesTensorProduct}.

\begin{defn}
\label{defn:BoxtimesInnerProduct}
For $c,a,d,e\in \cC$ and $b\in \Irr(\cC)$, we define a new inner product on the vector space $\cC(c, a\otimes b) \otimes \cC(b, d\otimes e)$ by $\langle \alpha_1\otimes \beta_1 | \alpha_2\otimes \beta_2\rangle = d_b^{-1} \langle \alpha_1|\alpha_2\rangle_{\cC(c, a\otimes b)} \langle \beta_1|\beta_2\rangle_{\cC(b, d\otimes e)}$.
Note that since $b\in \Irr(\cC)$, this inner product agrees with the graphical inner product on $\cC(c, a\otimes d\otimes e)$ from \eqref{eq:InnerProductInC}.
We write 
$\cC(c, a\otimes b) \boxtimes \cC(b, d\otimes e)$
to denote this space with its new inner product.

Similarly, when $f\in \Irr(\cC)$, we define 
$\cC(f, a\otimes d)\boxtimes \cC(c, f\otimes e)$
to be the Hilbert space with the inner product which multiplies the ordinary tensor product inner product by $d_f^{-1}$.

As a rule of thumb, whenever we see two hom spaces separated by a $\boxtimes$ symbol, they will share one simple object in $\Irr(\cC)$ on either side. 
The inner product is \emph{balanced} by dividing by the dimension of this simple object.
\end{defn}

Now using the $\boxtimes$ Hilbert space tensor product of hom spaces, the natural isomorphisms in \eqref{eq:BoxtimesTensorProduct} above become unitary.
Thus for all $a,c,d,e\in\cC$, we have a unitary operator
\begin{equation}
\label{eq:UnitaryFrom6j}
\begin{split}
\bigoplus_{b\in \Irr(\cC)}
\cC(b,d\otimes e) \boxtimes \cC(c, a\otimes b) 
&\overset{U}{\to}
\bigoplus_{f\in \Irr(\cC)}
\cC(f,a\otimes d) \boxtimes \cC(c, f\otimes e)
\\
\begin{tikzpicture}[baseline = -1cm, xscale=-1]
    \draw (.25,-1.2) -- (.25,-1.5);
    \draw (-1,-.3) arc (-180:0:.5cm);
    \draw (-.5,-.8) .. controls ++(270:.3cm) and ++(180:.3cm) ..  (.25,-1.2)  .. controls ++(0:.3cm) and ++(270:.8cm) .. (1,-.3);
    \filldraw[fill=white] (-.5,-.8) circle (.05cm) node [above] {\scriptsize{$\beta$}};
    \filldraw[fill=white] (.25,-1.2) circle (.05cm) node [above] {\scriptsize{$\alpha$}};
    \node at (-1.15,-.5) {\scriptsize{$\mathbf{e}$}};
    \node at (.15,-.5) {\scriptsize{$\mathbf{d}$}};
    \node at (1.15,-.5) {\scriptsize{$\mathbf{a}$}};
    \node at (-.5,-1.2) {\scriptsize{$\mathbf{b}$}};
    \node at (.25,-1.7) {\scriptsize{$\mathbf{c}$}};
\end{tikzpicture}
&\mapsto
\sum_{\substack{
f\in\Irr(\cC)
\\
\gamma\in \ONB(f,a\otimes d)
\\
\delta\in \ONB(c,f\otimes e)
}}
U_{\alpha,\beta}^{\gamma,\delta}\,\,
\begin{tikzpicture}[baseline = -1cm]
    \draw (.25,-1.2) -- (.25,-1.5);
    \draw (-1,-.3) arc (-180:0:.5cm);
    \draw (-.5,-.8) .. controls ++(270:.3cm) and ++(180:.3cm) ..  (.25,-1.2)  .. controls ++(0:.3cm) and ++(270:.8cm) .. (1,-.3);
    \filldraw[fill=white] (-.5,-.8) circle (.05cm) node [above] {\scriptsize{$\gamma$}};
    \filldraw[fill=white] (.25,-1.2) circle (.05cm) node [above] {\scriptsize{$\delta$}};
    \node at (-1.15,-.5) {\scriptsize{$\mathbf{a}$}};
    \node at (.15,-.5) {\scriptsize{$\mathbf{d}$}};
    \node at (1.15,-.5) {\scriptsize{$\mathbf{e}$}};
    \node at (-.5,-1.2) {\scriptsize{$\mathbf{f}$}};
    \node at (.25,-1.7) {\scriptsize{$\mathbf{c}$}};
\end{tikzpicture}
\,.
\end{split}
\end{equation}
Although the operator $U$ is unitary, $(U_{\alpha,\beta}^{\gamma,\delta})$ is not a unitary matrix, since the bases in \eqref{eq:NotONBs} are not orthonormal bases.
Although these operators $U$ depend on the objects $a,c,d,e\in\cC$, we suppress them from the notation.
The collection of unitary operators $U$ satisfy the pentagon axiom.

\begin{rem}
\label{rem:BasisAndAdjoint}
Often, we will use that 
$$
\sum_{\beta\in \ONB(c,a\otimes b)} \beta\otimes \beta^*
\in \cC(c, a\otimes b) \otimes \cC(a\otimes b, c)
$$
is independent of the choice of $\ONB$.
\end{rem}

\subsection{The category \texorpdfstring{$\Vec(\cC)$}{Vec(C)}}

In this section, we introduce the category of $\cC$-graded vector spaces for a rigid C*-tensor category $\cC$.   
This category is equivalent to $\ind$-$\cC$ but has a simpler description and is easier to work with due to the semi-simplicity of $\cC$.

\begin{defn}Let $\cC$ be a rigid C*-tensor category.  
Let $\Vec(\cC)$ be the tensor category 
whose objects are linear functors $\bfV:\cC^{\op}\rightarrow \Vec$, and
whose morphisms from $\bfV$ to $\bfW$ are given by natural transformations of functors.
For $\bfV\in \Vec(\cC)$, the vector spaces $\bfV(a)$ are called the \emph{fibers} of $\bfV$.  

For functors $\bfV,\bfW\in \Vec(\cC)$, their tensor product is given by
$$ 
(\bfV\otimes \bfW)(\, \cdot\, )
:=
\bigoplus_{a, b\in \Irr(\cC)}\bfV(a)\otimes \cC(\, \cdot\,  , a\otimes b)\otimes \bfW(b) 
$$
The associator of $\Vec(\cC)$ is given by the standard associator in $\Vec$, along with the 6j symbols in $\cC$.
Suppose we have $\bfU, \bfV, \bfW \in \Vec(\cC)$, so that
\begin{align*}
(\bfU\otimes [\bfV\otimes \bfW])(c) 
&= 
\bigoplus_{a,b,d,e\in \Irr(\cC)} \bfU(a)\otimes \cC(c, a\otimes b) \otimes [\bfV(d)\otimes \cC(b,d\otimes e) \otimes \bfW(e)]
\\
([\bfU\otimes \bfV]\otimes \bfW)(c) 
&= 
\bigoplus_{a,d,e,f\in \Irr(\cC)} [\bfU(a) \otimes \cC(f,a\otimes d)\otimes \bfV(d) ]\otimes \cC(c, f\otimes e) \otimes \bfW(e)
\end{align*}
The associator 
$\bfU\otimes (\bfV\otimes \bfW) \Rightarrow (\bfU\otimes \bfV)\otimes \bfW$ has $c$-component for $c\in \cC$ given by
$$
\eta\otimes \alpha\otimes (\xi \otimes \beta \otimes \zeta)
\mapsto 
\sum_{\substack{
f\in\Irr(\cC)
\\
\gamma\in \ONB(f,a\otimes d)
\\
\delta\in \ONB(c,f\otimes e)
}}
U_{\alpha,\beta}^{\gamma,\delta}
(\eta\otimes \gamma \otimes \xi) \otimes \delta\otimes \zeta
$$
where $(U_{\alpha,\beta}^{\gamma,\delta})$ is the matrix representation of the operator $U$ from \eqref{eq:UnitaryFrom6j} with respect to the appropriate bases.

Given natural transformations $\sigma: \bfU \Rightarrow \bfV$ and $\tau:\bfW\Rightarrow \bfX$, we define $\sigma\otimes \tau: \bfU\otimes \bfW \Rightarrow \bfV\otimes \bfX$ by
$(\sigma\otimes\tau)_c =\bigoplus_{a,b\in \Irr(\cC)} \sigma_a \otimes \id_{\cC(c, a\otimes b)} \otimes \tau_b$.
\end{defn}

\begin{rem}
\label{rem:TensorProductCanonical}
Notice that this tensor structure of $\Vec(\cC)$ depends on the choice $\Irr(\cC)$, but picking any other choice $\Irr(\cC)'$, we get a canonical equivalence of monoidal structures.
Picking a unitary isomorphism $\gamma_a \in \cC(a, a')$ for each $a\in \Irr(\cC)$ and the corresponding $a'\in \Irr(\cC)'$, the monoidal equivalence is given on the summands of $(\bfV\otimes \bfW)(c)$ by
$\eta\otimes \alpha \otimes \xi\in \bfV(a)\otimes \cC(c, a\otimes b)\otimes \bfW(b)$
maps to
$$
\bfV(\gamma_a^*)(\xi)\otimes [(\gamma_a\otimes \gamma_b)\circ \alpha] \otimes \bfW(\gamma_b^*)(\xi)
\in
\bfV(a')\otimes \cC(c, a'\otimes b')\otimes \bfW(b')
$$
which is independent of the choice of $\gamma_a,\gamma_b$, since each appears with its adjoint in the formula.

We note that the category $\Vec(\cC)$ is monoidaly equivalent to the ind-category of $\cC$, where the monoidal structure is given by the Day convolution product \cite{MR0272852}.
In general, if $\cC$ is not semi-simple, then one cannot simply use linear functors, but rather functors that commute with filtered co-limits.
However, for our purposes, our model proves especially easy to work with. 
\end{rem}

\begin{defn}[Yoneda embedding]
We have a fully faithful, monoidal functor $\cC\rightarrow \Vec(\cC)$ which sends the object $a\in \cC$ to the functor $\mathbf{a}:=\cC(\, \cdot\, , a): \cC^{\op}\rightarrow \Vec$.
In particular, the tensor unit in $\Vec(\cC)$ is the functor $\mathbf{1}=\cC(\, \cdot\, , 1_{\cC})$. 
We will call functors equivalent to those of the form $\mathbf{a}$ \textit{compact objects} in $\Vec(\cC)$.  
It is an easy exercise to see that these are precisely the dualizable objects in the tensor category $\Vec(\cC)$.
\end{defn}

\begin{rem}
\label{rem:VecCNaturalTransformations}
Suppose $\theta=(\theta_c)_{c\in\cC}\in \Hom_{\Vec(\cC)}(\bfV, \bfW)$ for $\bfV,\bfW\in \Vec(\cC)$.
First, for every $c\in \Irr(\cC)$, we have $\theta_c: \bfV(c)\to \bfW(c)$ is some linear transformation.
Conversely, given such a family of maps $\theta_c: \bfV(c)\to \bfW(c)$ for all $c\in \Irr(\cC)$, we see that naturality is satisfied for all maps between simples $c\to c'$.
Hence we may extend $\theta$ to all objects in $\cC$ by additivity to get a natural transformation.
In other words, 
$$
\Hom_{\Vec(\cC)}(\bfV, \bfW) \cong \prod_{c\in \Irr(\cC)} L(\bfV(c), \bfW(c)).
$$
\end{rem}

We now give an easy way to construct vector space objects $\bfV\in \Vec(\cC)$.

\begin{defn}
\label{defn:EasyConstructionOfVectorSpaceObjects}
Suppose we have a family of vector spaces $\set{V_c}{c\in \Irr(\cC)}$.
We may construct a canonical vector space object $\bfV\in \Vec(\cC)$ as follows
(although it depends on the choice of $\Irr(\cC)$).

First, we define $\bfV(c) = V_c$ for all $c\in\Irr(\cC)$.
For arbitrary $b\in \cC$, we define $\bfV(b) = \bigoplus_{a\in \Irr(\cC)} \bfV(a)\otimes \cC(b,a)$.
Notice that the spaces $\bfV(b)$ only depended on the choice of $\Irr(\cC)$.
Suppose now we have a map $\psi\in \cC(c,b)$ for $b,c\in\cC$.
We get a natural transformation $\bfV(\psi): \bfV(b)\to \bfV(c)$ by $\bfV(\psi) = \bigoplus_{a\in \Irr(\cC)} \id_{\bfV(a)}\otimes \mathbf{a}(\psi)$.
It is now easy to verify that $\bfV$ is a linear functor $\cC^{\text{op}}\to \Vec$.
\end{defn}

\begin{defn}
\label{defn:VecCInvolutive}
We now endow $\Vec(\cC)$ with an involutive structure.
First, we define for $\bfV\in \Vec(\cC)$ a linear functor $\overline{\bfV}: \cC^{\op}\to \Vec$ by $\overline{\bfV}(c) := \overline{\bfV(\overline{c})}$ on objects, and on morphisms $\psi\in \cC(a, b)$, we define $\overline{\bfV}(\psi) = \overline{\bfV(\overline{\psi})}$.

Note that $\overline{\,\cdot\,}$ is an anti-linear functor.
If $\theta=(\theta_c)_{c\in\Irr(\cC)}: \bfV\Rightarrow \bfW$ is a natural transformation, we get a natural transformation $\overline{\theta}: \overline{\bfV}\Rightarrow \overline{\bfW}$ by $\overline{\theta}_c=\overline{\theta_{\overline{c}}}$.
We thus have a natural identification between $\Hom_{\Vec(\cC)}(\overline{\bfV}, \overline{\bfW})$ and the complex conjugate vector space $\overline{\Hom_{\Vec(\cC)}(\bfV,\bfW)}$.

The natural isomorphism $\varphi_{\bfV}: \bfV \to \overline{\overline{\bfV}}$ is given by $(\varphi_\bfV)_c = \bfV(\varphi_c^{-1})$ where $\varphi_c: c\to \overline{\overline{c}}$ is the canonical pivotal structure of $\cC$, since $\varphi_\Vec$ is the identity.
The unit map $r:\mathbf{1}\to \overline{\mathbf{1}}$ has components $r_c : \mathbf{1}(c) \to \overline{\mathbf{1}}(c)$ given by post-composing with $r_\cC$:
$$
\mathbf{1}(c) = \cC(c, 1_\cC) \xrightarrow{-\circ r_\cC} \cC(c, \overline{1}_\cC) = \overline{\mathbf{1}}(c).
$$
The natural isomorphisms $\nu_{\bfV,\bfW}:\overline{\bfV}\otimes \overline{\bfW} \to \overline{\bfW\otimes \bfV}$ are constructed as follows.
First, we note that for all $a,b,c\in\cC$, there is a canonical linear (unitary) isomorphism 
$$
\overline{\cC(\overline{c}, \overline{b}\otimes \overline{a})}
\xrightarrow{\overline{\,\,\cdot\,\,}^\Vec}
\cC(\overline{c}, \overline{b}\otimes \overline{a})
\xrightarrow{\nu_{b,a}\circ-}
\cC(\overline{c}, \overline{a\otimes b})
\xrightarrow{\overline{\,\,\cdot\,\,}^{\cC}}
\cC(c, a\otimes b)
$$
Now for all $a\in \Irr(\cC)$, there is a unique $a_*\in \Irr(\cC)$ such that $a_*\cong \overline{a}$.
We choose a unitary isomorphism $\gamma_a : a_* \to \overline{a}$ for all $a\in \Irr(\cC)$.
Now we construct a map between
\begin{align*}
(\overline{\bfV}\otimes \overline{\bfW})(c) 
&=
\bigoplus_{a,b\in \Irr(\cC)} \overline{\bfV}(a) \otimes \cC(c, a\otimes b) \otimes \overline{\bfW}(b)
\qquad\text{ and}
\\
(\overline{\bfW\otimes \bfV})(c) 
&=
\overline{(\bfW\otimes \bfV)(\overline{c})} 
=
\overline{\bigoplus_{a_*,b_*\in \Irr(\cC)} \bfW(b_*) \otimes \cC(\overline{c}, b_*\otimes a_*) \otimes \bfV(a_*)}
\\&\overset{\nu_{\, \Vec}}{\cong}
\bigoplus_{a_*,b_*\in \Irr(\cC)} \overline{\bfV(a_*)} \otimes \overline{\cC(\overline{c}, b_*\otimes a_*)} \otimes \overline{\bfW(b_*)}
\end{align*}
as follows.
Note that we have isomorphisms $\overline{\bfV(\gamma_a^*)}: \overline{\bfV(a_*)}\to \overline{\bfV(\overline{a})}$, $\overline{\bfW(\gamma_b^*)} :\overline{\bfW(b_*)}\to \overline{\bfW(\overline{b})}$,  and
$\overline{\cC(\overline{c}, b_*\otimes a_*)} \cong \overline{\cC(\overline{c}, \overline{b}\otimes \overline{a})}$ by $\overline{\psi}\mapsto \overline{(\gamma_b\otimes \gamma_a)\circ \psi}$, where these conjugates are taken in $\Vec$.
Using these isomorphisms, we get an isomorphism from the summands of $(\overline{\bfW\otimes \bfV})(c)$
to the summands of $(\overline{\bfV}\otimes \overline{\bfW})(c)$ which is independent of the choice of $\gamma_a, \gamma_b$ by Remark \ref{rem:BasisAndAdjoint}, since every time $\gamma_a,\gamma_b$ appear, we also see their adjoints $\gamma_a^*,\gamma_b^*$.
\end{defn}

\begin{defn}
A \emph{real structure} on a vector space object $\bfV\in \Vec(\cC)$ is a natural isomorphism $\sigma: \bfV \Rightarrow \overline{\bfV}$ satisfying $\overline{\sigma}\circ \sigma= \varphi_\bfV$.
\end{defn}

\subsection{Graphical calculus for the Yoneda embedding}

The Yoneda lemma gives a natural identification $\bfV(a)\cong \Hom_{\Vec(\cC)}(\mathbf{a}, \bfV)$.
In turn, this allows us to use the tensor category graphical calculus to represent elements of $\bfV(a)$ as morphisms in $\Vec(\cC)$.
Given $\xi\in \bfV(a)$, we denote the corresponding natural transformation $\mathbf{a}\Rightarrow \bfV$ as a coupon with strings and labels:
\begin{equation}
\label{eq:YonedaDiagram}
\bfV(a) 
\ni
\xi
\,\,\longleftrightarrow\,\,
\begin{tikzpicture}[baseline = -.1cm]
    \draw (0,-.6) -- (0,.6);
    \roundNbox{unshaded}{(0,0)}{.3}{0}{0}{$\xi$}
    \node at (.2,-.5) {\scriptsize{$\mathbf{a}$}};
    \node at (.2,.5) {\scriptsize{$\bfV$}};
\end{tikzpicture}
\in
\Hom_{\Vec(\cC)}(\mathbf{a},\bfV).
\end{equation}
If $\psi\in \cC(b,a)$, by the naturality condition of the Yoneda Lemma, the map $\bfV(\psi): \bfV(a)\to \bfV(b)$ is given diagrammatically by
\begin{equation}
\label{eq:NaturalityOfYoneda1}
\begin{tikzpicture}[baseline = -.1cm]
    \draw (0,-.6) -- (0,.6);
    \roundNbox{unshaded}{(0,0)}{.3}{.5}{.5}{$\bfV(\psi)(\xi)$}
    \node at (.2,-.5) {\scriptsize{$\mathbf{a}$}};
    \node at (.2,.5) {\scriptsize{$\bfV$}};
\end{tikzpicture}
=
\begin{tikzpicture}[baseline = -.6cm]
    \draw (0,-1.6) -- (0,.6);
    \roundNbox{unshaded}{(0,-1)}{.3}{0}{0}{$\psi$}
    \roundNbox{unshaded}{(0,0)}{.3}{0}{0}{$\xi$}
    \node at (.2,-1.5) {\scriptsize{$\mathbf{b}$}};
    \node at (.2,-.5) {\scriptsize{$\mathbf{a}$}};
    \node at (.2,.5) {\scriptsize{$\bfV$}};
\end{tikzpicture}\,.
\end{equation}
Naturality of the Yoneda Lemma in $\bfV$ says that if $\theta: \bfV \Rightarrow \bfW$ is a natural transformation and $\xi\in \bfV(a)$, we have
\begin{equation}
\label{eq:NaturalityOfYoneda2}
\begin{tikzpicture}[baseline = -.1cm]
    \draw (0,-.6) -- (0,.6);
    \roundNbox{unshaded}{(0,0)}{.3}{.2}{.2}{$\theta_a(\xi)$}
    \node at (.2,-.5) {\scriptsize{$\mathbf{a}$}};
    \node at (.2,.5) {\scriptsize{$\bfW$}};
\end{tikzpicture}
=
\begin{tikzpicture}[baseline = -.6cm]
    \draw (0,-1.6) -- (0,.6);
    \roundNbox{unshaded}{(0,-1)}{.3}{0}{0}{$\xi$}
    \roundNbox{unshaded}{(0,0)}{.3}{0}{0}{$\theta$}
    \node at (.2,-1.5) {\scriptsize{$\mathbf{a}$}};
    \node at (.2,-.5) {\scriptsize{$\bfV$}};
    \node at (.2,.5) {\scriptsize{$\bfW$}};
\end{tikzpicture}\,.
\end{equation}

\begin{lem}
\label{lem:Yoneda}
Given $\bfV,\bfW\in \Vec(\cC)$, $a,b\in\Irr(\cC)$,  $\xi\in \bfV(a)$, $\eta\in \bfW(b)$, and a morphism $\alpha\in\cC(c,a\otimes b)$,
the natural transformation corresponding to
$$
\xi\otimes \alpha\otimes \eta
\in
\bfV(a)\otimes \cC(c,a\otimes b)\otimes \bfW(b)
\subset (\bfV\otimes \bfW)(c)
$$
is given by
\begin{equation}
\label{eq:SimpleTensors}
\begin{tikzpicture}[baseline = -.4cm]
    \draw (0,0) -- (0,.6);
    \draw (1,0) -- (1,.6);
    \draw (.5,-.8) -- (.5,-1.2);
    \draw (0,-.3) arc (-180:0:.5cm);
    \filldraw[fill=white] (.5,-.8) circle (.05cm) node [above] {\scriptsize{$\alpha$}};
    \roundNbox{unshaded}{(0,0)}{.3}{0}{0}{$\xi$}
    \roundNbox{unshaded}{(1,0)}{.3}{0}{0}{$\eta$}
    \node at (-.15,-.5) {\scriptsize{$\mathbf{a}$}};
    \node at (1.15,-.5) {\scriptsize{$\mathbf{b}$}};
    \node at (.35,-1) {\scriptsize{$\mathbf{c}$}};
    \node at (.2,.5) {\scriptsize{$\bfV$}};
    \node at (1.2,.5) {\scriptsize{$\bfW$}};
\end{tikzpicture}\,
\in
\Hom_{\Vec(\cC)}(\mathbf{c}, \bfV\otimes \bfW).
\end{equation}
where we draw an open circle to denote the morphism $\alpha\in\cC(c,a\otimes b)$.
\end{lem}
\begin{proof}
By the naturality condition in the Yoneda lemma \eqref{eq:NaturalityOfYoneda1}, it suffices to show that the natural transformation corresponding to 
$$
\xi\otimes \id_{a\otimes b}\otimes \eta
\in
(\bfV\otimes \bfW)(a\otimes b)
=
\bfV(a)\otimes \cC(a\otimes b, a\otimes b) \otimes \bfW(b)
$$
is given by
$$
\begin{tikzpicture}[baseline = -.1cm]
    \draw (0,-.6) -- (0,.6);
    \draw (1,-.6) -- (1,.6);
    \roundNbox{unshaded}{(0,0)}{.3}{0}{0}{$\xi$}
    \roundNbox{unshaded}{(1,0)}{.3}{0}{0}{$\eta$}
    \node at (.15,-.5) {\scriptsize{$\mathbf{a}$}};
    \node at (1.15,-.5) {\scriptsize{$\mathbf{b}$}};
    \node at (.2,.5) {\scriptsize{$\bfV$}};
    \node at (1.2,.5) {\scriptsize{$\bfW$}};
\end{tikzpicture}
\in\Hom_{\Vec(\cC)}(\mathbf{a}\otimes \mathbf{b}, \bfV\otimes \bfW)
$$
But this picture is represented by the natural transformation in $\Hom_{\Vec(\cC)}(\mathbf{a}\otimes \mathbf{b}, \bfV\otimes \bfW)$ defined by sending 
$$
\cC(c, a\otimes b)
\cong
(\mathbf{a}\otimes \mathbf{b})(c)
\ni
f
\longmapsto
\xi\otimes f\otimes \eta
\in 
\bfV(a)\otimes \cC(c,a\otimes b)\otimes \bfW(b)\subseteq (\bfV\otimes \bfW)(c).
$$  
From the Yoneda lemma, the canonical bijection between natural transformations from $\mathbf{a}\otimes\mathbf{b}$ to $\bfV\otimes \bfW$  and $(\bfV\otimes \bfW)(a\otimes b)$ is given by evaluating on the identity morphism $\id_{a\otimes b}\in \cC(a\otimes b, a\otimes b)$.
\end{proof}

\begin{rem}
Using the string diagram notation, we easily see why the matrix representation $(U_{\alpha,\beta}^{\gamma,\delta})$ of the operator $U$ from \eqref{eq:UnitaryFrom6j} appears in the associator for $\Vec(\cC)$:
\begin{equation}
\label{eq:Associativity}
\begin{tikzpicture}[baseline = -.4cm, xscale=-1]
    \draw (-1,0) -- (-1,.6);
    \draw (0,0) -- (0,.6);
    \draw (1,0) -- (1,.6);
    \draw (.25,-1.2) -- (.25,-1.5);
    \draw (-1,-.3) arc (-180:0:.5cm);
    \draw (-.5,-.8) .. controls ++(270:.3cm) and ++(180:.3cm) ..  (.25,-1.2)  .. controls ++(0:.3cm) and ++(270:.8cm) .. (1,-.3);
    \filldraw[fill=white] (-.5,-.8) circle (.05cm) node [above] {\scriptsize{$\beta$}};
    \filldraw[fill=white] (.25,-1.2) circle (.05cm) node [above] {\scriptsize{$\alpha$}};
    \roundNbox{unshaded}{(-1,0)}{.3}{0}{0}{$\zeta$}
    \roundNbox{unshaded}{(0,0)}{.3}{0}{0}{$\xi$}
    \roundNbox{unshaded}{(1,0)}{.3}{0}{0}{$\eta$}
    \node at (-1.15,-.5) {\scriptsize{$\mathbf{e}$}};
    \node at (.15,-.5) {\scriptsize{$\mathbf{d}$}};
    \node at (1.15,-.5) {\scriptsize{$\mathbf{a}$}};
    \node at (-.5,-1.2) {\scriptsize{$\mathbf{b}$}};
    \node at (.25,-1.7) {\scriptsize{$\mathbf{c}$}};
    \node at (-1,.8) {\scriptsize{$\bfW$}};
    \node at (0,.8) {\scriptsize{$\bfV$}};
    \node at (1,.8) {\scriptsize{$\bfU$}};
\end{tikzpicture}
=
\sum_{\substack{
f\in\Irr(\cC)
\\
\gamma\in \ONB(f,a\otimes d)
\\
\delta\in \ONB(c,f\otimes e)
}}
U_{\alpha,\beta}^{\gamma,\delta}\,\,
\begin{tikzpicture}[baseline = -.4cm]
    \draw (-1,0) -- (-1,.6);
    \draw (0,0) -- (0,.6);
    \draw (1,0) -- (1,.6);
    \draw (.25,-1.2) -- (.25,-1.5);
    \draw (-1,-.3) arc (-180:0:.5cm);
    \draw (-.5,-.8) .. controls ++(270:.3cm) and ++(180:.3cm) ..  (.25,-1.2)  .. controls ++(0:.3cm) and ++(270:.8cm) .. (1,-.3);
    \filldraw[fill=white] (-.5,-.8) circle (.05cm) node [above] {\scriptsize{$\gamma$}};
    \filldraw[fill=white] (.25,-1.2) circle (.05cm) node [above] {\scriptsize{$\delta$}};
    \roundNbox{unshaded}{(-1,0)}{.3}{0}{0}{$\eta$}
    \roundNbox{unshaded}{(0,0)}{.3}{0}{0}{$\xi$}
    \roundNbox{unshaded}{(1,0)}{.3}{0}{0}{$\zeta$}
    \node at (-1.15,-.5) {\scriptsize{$\mathbf{a}$}};
    \node at (.15,-.5) {\scriptsize{$\mathbf{d}$}};
    \node at (1.15,-.5) {\scriptsize{$\mathbf{e}$}};
    \node at (-.5,-1.2) {\scriptsize{$\mathbf{f}$}};
    \node at (.25,-1.7) {\scriptsize{$\mathbf{c}$}};
    \node at (-1,.8) {\scriptsize{$\bfU$}};
    \node at (0,.8) {\scriptsize{$\bfV$}};
    \node at (1,.8) {\scriptsize{$\bfW$}};
\end{tikzpicture}
\,.
\end{equation}
\end{rem}

\subsection{The category \texorpdfstring{$\Hilb(\cC)$}{Hilb(C)}}

Ordinary operator algebras are realized as operators acting on Hilbert spaces, motivating us to study a `Hilbert space' version of of $\Vec(\cC)$.
We note that this category is equivalent to the Neshveyev-Yamashita unitary ind category \cite{MR3509018}.
We prefer to use our model to emphasize the similarities and differences between $\Vec(\cC)$ and $\Hilb(\cC)$.

\begin{defn}
Let $\cC$ be a rigid C*-tensor category.  
Let $\Hilb(\cC)$ be the tensor category whose objects are linear dagger functors $\bfH:\cC^{\op}\rightarrow \Hilb$ and whose morphisms from $\bfH$ to $\bfK$ are given by bounded natural transformations.
(Recall a natural transformation $\theta=(\theta_{c})_{c\in \cC}$ is bounded if $\sup_{c\in \cC} \|\theta_c\|<\infty$.)

We define a dagger structure on $\Hilb(\cC)$ by defining the adjoint of a natural transformation locally.
If $\theta: \bfH \Rightarrow \bfK$ is a natural transformation, $\theta^*: \bfK\Rightarrow \bfH$ is defined by $(\theta^*)_c =(\theta_c)^*$ for $c\in \cC$.
Similar to Remark \ref{rem:VecCNaturalTransformations}, we see that for all $\bfH,\bfK\in \Hilb(\cC)$,
$$
\Hom_{\Hilb(\cC)}(\bfH, \bfK) \cong \bigoplus_{c\in\Irr(\cC)} B(\bfH(c), \bfK(c))
$$
as von Neumann algebras, which witnesses that $\Hilb(\cC)$ is in fact a W*-category.
\end{defn}

Similar to $\Vec(\cC)$, using the Yoneda embedding, 
each $\xi \in \bfH(a)$ can be viewed as a bounded natural transformation $\Theta_\xi\in \Hom_{\Hilb(\cC)}(\mathbf{a}, \bfH)$ by
defining for $\psi\in \cC(b,a)$, $(\Theta_\xi)_b(\psi) = \bfH(\psi)(\xi)$ as in \eqref{eq:NaturalityOfYoneda1}.
Again, we denote $\Theta_\xi$ by a coupon with an $\mathbf{a}$ string on the bottom and an $\bfH$ string on the top as in \eqref{eq:YonedaDiagram}.
Notice this means that $\xi = (\Theta_\xi)_a(\id_a)$.
We denote the adjoint $\Theta_\xi^* : \bfH \Rightarrow \mathbf{a}$ graphically by
\begin{equation}
\label{eq:YonedaDiagram2}
\Theta_\xi^*
=
\begin{tikzpicture}[baseline = -.1cm]
    \draw (0,-.6) -- (0,.6);
    \roundNbox{unshaded}{(0,0)}{.3}{0}{0}{$\xi^*$}
    \node at (.2,.5) {\scriptsize{$\mathbf{a}$}};
    \node at (.2,-.5) {\scriptsize{$\bfH$}};
\end{tikzpicture}
\in
\Hom_{\Vec(\cC)}(\bfH,\mathbf{a}).
\end{equation}
By the second naturality criterion of the Yoneda embedding \eqref{eq:NaturalityOfYoneda2},
for $\xi\in \bfH(a)\cong \Hom_{\Hilb(\cC)}(\mathbf{a}, \bfH)$ and $\eta\in \bfH(b)\cong \Hom_{\Hilb(\cC)}(\mathbf{b}, \bfH)$, we have
$$
\begin{tikzpicture}[baseline = -.1cm]
    \draw (0,-.6) -- (0,.6);
    \roundNbox{unshaded}{(0,0)}{.3}{.5}{.5}{$(\Theta_\eta)_a^*(\xi)$}
    \node at (.2,-.5) {\scriptsize{$\mathbf{a}$}};
    \node at (.2,.5) {\scriptsize{$\mathbf{b}$}};
\end{tikzpicture}
=
\begin{tikzpicture}[baseline = -.6cm]
    \draw (0,-1.6) -- (0,.6);
    \roundNbox{unshaded}{(0,-1)}{.3}{0}{0}{$\xi$}
    \roundNbox{unshaded}{(0,0)}{.3}{0}{0}{$\eta^*$}
    \node at (.2,-1.5) {\scriptsize{$\mathbf{a}$}};
    \node at (.2,-.5) {\scriptsize{$\bfV$}};
    \node at (.2,.5) {\scriptsize{$\mathbf{b}$}};
\end{tikzpicture}
\,.
$$

Recall that for all $a,b\in\cC$, $\cC(a,b)$ carries the Hilbert space structure from \eqref{eq:InnerProductInC}.
Thus for $\eta, \xi \in \bfH(a)$, since $\bfH(a)\ni \eta = (\Theta_\eta)_a(\id_a)$ for $\id_a\in \mathbf{a}(a)=\cC(a,a)$,
we obtain the following identity between the inner product on $\bfH(a)$ and the graphical calculus:
\begin{equation}
\label{eq:DiagrammaticInnerProduct}
\langle \eta| \xi\rangle_{\bfH(a)} =
\langle (\Theta_{\eta})_a(\id_a)| \xi\rangle_{\bfH(a)} =
\langle \id_a | (\Theta_{\eta}^*)_a(\xi)\rangle_{\cC(a,a)}=
\begin{tikzpicture}[baseline=-.1cm]
	\draw (0,1.3) arc (180:0:.3cm) -- (.6,-1.3) arc (0:-180:.3cm) -- (0,1.3);
	\roundNbox{unshaded}{(0,1)}{.3}{0}{0}{$\varphi_a$}
	\roundNbox{unshaded}{(0,0)}{.3}{0}{0}{$\eta^*$}
	\roundNbox{unshaded}{(0,-1)}{.3}{0}{0}{$\xi$}
	\node at (-.2,1.5) {\scriptsize{$\overline{\overline{\mathbf{a}}}$}};
	\node at (-.2,.5) {\scriptsize{$\mathbf{a}$}};
	\node at (.8,0) {\scriptsize{$\overline{\mathbf{a}}$}};
	\node at (-.2,-.5) {\scriptsize{$\bfH$}};
	\node at (-.2,-1.5) {\scriptsize{$\mathbf{a}$}};
\end{tikzpicture}
\,.
\end{equation}

\begin{rem}
We demonstrate the utility and consistency of the graphical calculus by the following calculation, which verifies that $\bfH$ is a dagger functor.
For all $\psi\in \cC(b,a)$, $\xi\in \bfH(a)$, and  $\eta\in \bfH(b)$, we have
$$
\langle \eta|\bfH(\psi)\xi\rangle_{\bfH(b)}
=
\begin{tikzpicture}[baseline=-.1cm]
	\draw (0,1.8) arc (180:0:.3cm) -- (.6,-1.8) arc (0:-180:.3cm) -- (0,1.8);
	\roundNbox{unshaded}{(0,1.5)}{.3}{0}{0}{$\varphi_b$}
	\roundNbox{unshaded}{(0,.5)}{.3}{0}{0}{$\eta^*$}
	\roundNbox{unshaded}{(0,-.5)}{.3}{0}{0}{$\xi$}
	\roundNbox{unshaded}{(0,-1.5)}{.3}{0}{0}{$\psi$}
	\node at (-.2,2) {\scriptsize{$\overline{\overline{\mathbf{b}}}$}};
	\node at (-.2,1) {\scriptsize{$\mathbf{b}$}};
	\node at (.8,0) {\scriptsize{$\overline{\mathbf{b}}$}};
	\node at (-.2,0) {\scriptsize{$\bfH$}};
	\node at (-.2,-1) {\scriptsize{$\mathbf{a}$}};
	\node at (-.2,-2) {\scriptsize{$\mathbf{b}$}};
\end{tikzpicture}
=
\begin{tikzpicture}[baseline=-.1cm]
	\draw (0,1.3) arc (180:0:.5cm) -- (1,-1.3) arc (0:-180:.5cm) -- (0,1.3);
	\roundNbox{unshaded}{(0,1)}{.3}{0}{0}{$\varphi_b$}
	\roundNbox{unshaded}{(0,0)}{.3}{0}{0}{$\eta^*$}
	\roundNbox{unshaded}{(0,-1)}{.3}{0}{0}{$\xi$}
	\roundNbox{unshaded}{(1,0)}{.3}{0}{0}{$\psi^{\vee}$}
	\node at (-.2,1.5) {\scriptsize{$\overline{\overline{\mathbf{b}}}$}};
	\node at (-.2,.5) {\scriptsize{$\mathbf{b}$}};
	\node at (1.2,.5) {\scriptsize{$\overline{\mathbf{b}}$}};
	\node at (1.2,-.5) {\scriptsize{$\overline{\mathbf{a}}$}};
	\node at (-.2,-.5) {\scriptsize{$\bfH$}};
	\node at (-.2,-1.5) {\scriptsize{$\mathbf{a}$}};
\end{tikzpicture}
=
\begin{tikzpicture}[baseline=-.1cm]
	\draw (0,1.8) arc (180:0:.3cm) -- (.6,-1.8) arc (0:-180:.3cm) -- (0,1.8);
	\roundNbox{unshaded}{(0,1.5)}{.3}{.1}{.1}{$\psi^{\vee\vee}$}
	\roundNbox{unshaded}{(0,.5)}{.3}{0}{0}{$\varphi_b$}
	\roundNbox{unshaded}{(0,-.5)}{.3}{0}{0}{$\eta^*$}
	\roundNbox{unshaded}{(0,-1.5)}{.3}{0}{0}{$\xi$}
	\node at (-.2,2) {\scriptsize{$\overline{\overline{\mathbf{a}}}$}};
	\node at (-.2,1) {\scriptsize{$\overline{\overline{\mathbf{b}}}$}};
	\node at (.8,0) {\scriptsize{$\overline{\mathbf{a}}$}};
	\node at (-.2,0) {\scriptsize{$\mathbf{b}$}};
	\node at (-.2,-1) {\scriptsize{$\bfH$}};
	\node at (-.2,-2) {\scriptsize{$\mathbf{a}$}};
\end{tikzpicture}
=
\begin{tikzpicture}[baseline=-.1cm]
	\draw (0,1.8) arc (180:0:.3cm) -- (.6,-1.8) arc (0:-180:.3cm) -- (0,1.8);
	\roundNbox{unshaded}{(0,1.5)}{.3}{0}{0}{$\varphi_a$}
	\roundNbox{unshaded}{(0,.5)}{.3}{0}{0}{$\psi$}
	\roundNbox{unshaded}{(0,-.5)}{.3}{0}{0}{$\eta^*$}
	\roundNbox{unshaded}{(0,-1.5)}{.3}{0}{0}{$\xi$}
	\node at (-.2,2) {\scriptsize{$\overline{\overline{\mathbf{a}}}$}};
	\node at (-.2,1) {\scriptsize{$\mathbf{a}$}};
	\node at (.8,0) {\scriptsize{$\overline{\mathbf{a}}$}};
	\node at (-.2,0) {\scriptsize{$\mathbf{b}$}};
	\node at (-.2,-1) {\scriptsize{$\bfH$}};
	\node at (-.2,-2) {\scriptsize{$\mathbf{a}$}};
\end{tikzpicture}
=
\langle \bfH(\psi^*) \eta| \xi\rangle_{\bfH(a)}.
$$ 
\end{rem}

We now endow $\Hilb(\cC)$ with the structure of a C*-tensor category.

\begin{defn}
The tensor product of dagger functors $\bfH,\bfK\in \Hilb(\cC)$ is given similarly to that in $\Vec(\cC)$ by
$$
(\bfH\otimes \bfK)(\, \cdot\, )
:=
\bigoplus_{a, b\in \Irr(\cC)}\bfH(a)\boxtimes \cC(\, \cdot\, , a\otimes b)\boxtimes \bfK(b),
$$ 
where we take direct sum in $\Hilb$.
Here, we use the $\boxtimes$ symbol as in Definition \eqref{defn:BoxtimesInnerProduct} above.
The underlying space is the ordinary tensor product of vector spaces, but we equip it with a new inner product, which is not the ordinary Hilbert space tensor product inner product.
For $c\in\cC$ and $a,b\in\Irr(\cC)$, we define the inner product on the space $\bfH(a)\boxtimes \cC(c,a\otimes b) \boxtimes \bfK(b)$ using the graphical calculus \eqref{eq:SimpleTensors} as a guide:
$$
\langle
\xi_2\boxtimes \alpha_2\boxtimes \eta_2|
\xi_1\boxtimes \alpha_1\boxtimes \eta_1
\rangle
:=
\begin{tikzpicture}[baseline = -.1cm]
    \draw (0,-.5) -- (0,.5);
    \draw (1,-.5) -- (1,.5);
    \draw (.5,1.3) -- (.5,1.7);
    \draw (0,.8) arc (180:0:.5cm);
    \draw (0,-.8) arc (-180:0:.5cm);
    \draw (.5,2.3)  .. controls ++(90:.5cm) and ++(90:.5cm) .. (1.7, 2.3) -- (1.7,-1.3)  .. controls ++(270:.5cm) and ++(270:.5cm) .. (.5,-1.3);
    \filldraw[fill=white] (.5,1.3) circle (.05cm) node [below] {\scriptsize{$\alpha^*_2$}};
    \filldraw[fill=white] (.5,-1.3) circle (.05cm) node [above] {\scriptsize{$\alpha_1$}};
    \roundNbox{unshaded}{(0,-.5)}{.3}{0}{0}{$\xi_1$}
    \roundNbox{unshaded}{(1,-.5)}{.3}{0}{0}{$\eta_1$}
    \roundNbox{unshaded}{(0,.5)}{.3}{0}{0}{$\xi_2^*$}
    \roundNbox{unshaded}{(1,.5)}{.3}{0}{0}{$\eta_2^*$}
    \roundNbox{unshaded}{(.5,2)}{.3}{0}{0}{$\varphi_c$}
    \node at (-.15,1) {\scriptsize{$\mathbf{a}$}};
    \node at (1.15,1) {\scriptsize{$\mathbf{b}$}};
    \node at (-.15,-1) {\scriptsize{$\mathbf{a}$}};
    \node at (1.15,-1) {\scriptsize{$\mathbf{b}$}};
    \node at (.35,1.5) {\scriptsize{$\mathbf{c}$}};
    \node at (.35,2.5) {\scriptsize{$\overline{\overline{\mathbf{c}}}$}};
    \node at (1.85,0) {\scriptsize{$\overline{\mathbf{c}}$}};
    \node at (.35,-1.5) {\scriptsize{$\mathbf{c}$}};
    \node at (.2,0) {\scriptsize{$\bfH$}};
    \node at (1.2,0) {\scriptsize{$\bfK$}};
\end{tikzpicture}
=
\frac{1}{d_ad_b} \langle \xi_2|\xi_1\rangle_{\bfH(a)} \langle \eta_2|\eta_1\rangle_{\bfK(b)}\langle \alpha_2|\alpha_1\rangle_{\cC(c,a\otimes b)}.
$$
As before, each $\boxtimes$ symbol requires a \emph{balancing} factor determined by the simple object in $\Irr(\cC)$ appearing on either side of the $\boxtimes$.

As the operator $U$ from \eqref{eq:UnitaryFrom6j} is unitary, we see that the associator isomorphism from $\Vec(\cC)$ depicted in \eqref{eq:Associativity} is unitary.
We suppress the rest of the definition as it is parallel to that of $\Vec(\cC)$.
\end{defn}

We endow $\Hilb(\cC)$ with an involutive structure the same way we endowed $\Vec(\cC)$ with an involutive structure.
One checks that $\Hilb(\cC)$ is actually bi-involutive under this involutive structure.
(Recall from Definition \ref{defn:bi-involutive} that being bi-involutive is a property of an involutive tensor category which is also a dagger category.)
There is a similar notion of a real structure for a Hilbert space object $\bfH\in \Hilb(\cC)$.
Notice that the full subcategory of compact objects is again equivalent to $\cC$ as a bi-involutive tensor category via $a\rightarrow \mathbf{a}:=\cC(\,\cdot\,, a)$.

The following results will be useful in what follows.

\begin{lem}
\label{lem:ZeroInHilbC}
Suppose $f\in \Hom_{\Hilb(\cC)}(\mathbf{a}\otimes \bfH , \bfH)$.
We have $f=0$ if and only if $f_c (\alpha\boxtimes \xi) = 0$ for all $c\in \cC$, $\alpha\in \cC(c, a\otimes b)$,  and $\xi\in \bfH(b)$.
\end{lem}
\begin{proof}
The forward direction is trivial.
For the reverse direction, just note that linear combinations of vectors of the form $\alpha\boxtimes \xi$ are dense in $(\mathbf{a}\otimes \bfH)(c)$.
Thus $f_c=0$ for all $c\in \cC$, and $f=0$.
\end{proof}

There is also a Hilbert space object version of Definition \ref{defn:EasyConstructionOfVectorSpaceObjects}.

\begin{defn}
\label{defn:EasyConstructionOfHilbertSpaceObject}
Suppose we have a family of Hilbert spaces $\set{H_c}{c\in\Irr(\cC)}$.
We may construct a canonical Hilbert space object $\bfH\in \Hilb(\cC)$ (which depends on $\Irr(\cC)$) by first setting $\bfH(c) = H_c$ for all $c\in\Irr(\cC)$.
For arbitrary $b\in \cC$, we define $\bfH(b) = \bigoplus_{a\in \Irr(\cC)} \bfH(a)\boxtimes \cC(b,a)$, the direct sum of Hilbert spaces.

Given a map $\psi\in \cC(c,b)$ for $b,c\in\cC$, for all $a\in \Irr(\cC)$, the map $\mathbf{a}(\psi): \cC(b,a) \to \cC(c,a)$ which precomposes with $\psi$ has norm bounded above by the same universal constant $\|\psi\circ \psi^*\|^{1/2}_{\infty}$ in the C*-algebra $\cC(b,b)$, which does not depend on $a\in \Irr(\cC)$.
Thus we get a bounded map $\bfH(\psi): \bfH(b)\to \bfH(c)$ by $\bfH(\psi) = \bigoplus_{a\in \Irr(\cC)} \id_{\bfH(a)}\boxtimes \mathbf{a}(\psi)$.
It is now easy to verify that $\bfH$ preserves identity maps, composition, and adjoints, so $\bfH\in\Hilb(\cC)$.
\end{defn}

\begin{cor}
\label{cor:CompletionVecToHilb}
Suppose we have a $\bfV\in \Vec(\cC)$ together with positive semi-definite sesquilinear forms $\langle \,\cdot\,,\,\cdot\,\rangle_c$ on $\bfV(c)$ for all $c\in \Irr(\cC)$.
Define $\bfH(c)$ to be the Hilbert space completion in $\|\,\cdot\,\|_2$ of $\bfV(c)/N_{\langle \,\cdot\,,\,\cdot\,\rangle_c}$, where $N_{\langle \,\cdot\,,\,\cdot\,\rangle_c}$ is the length zero vectors in $\bfV(c)$.
There is a canonical extension of $\bfH$ to an element of $\Hilb(\cC)$.
\end{cor}


\section{\texorpdfstring{$*$}{*}-Algebras in \texorpdfstring{$\Vec(\cC)$}{Vec(C)} }
\label{sec:AlgebrasInC}

The aim of this section is to define the notion of a $*$-algebra in $\Vec(\cC)$ for $\cC$ a rigid C*-tensor category with simple unit object.
As a reminder, we suppress the associators and unitors of $\cC$ whenever possible.

First, we review the notions of algebras and module categories and the correspondence between them.
We then define a $*$-algebra, and we discuss the correspondence between $*$-algebras and dagger module categories.
This correspondence affords an elegant definition of a C*-algebra object as a $*$-algebra object whose corresponding dagger module category is a C*-category.
W*-algebra objects are defined similarly.

\subsection{Algebras and module categories}

\begin{defn} 
An \textit{algebra object} in a tensor category $\cT$ is an object $A\in \cT$, and a pair of maps $m:A\otimes A\rightarrow A$ and $i: 1_\cC\rightarrow A$ such that
\begin{enumerate}[(1)]
\item
$m\circ (\id_{A}\otimes m)=m\circ(m\otimes \id_{A})$ as maps $A\otimes A\otimes A\rightarrow A$, and
\item
$m\circ (i\otimes \id_{A})=\id_{A}=m\circ(\id_{A}\otimes i)$ as maps $A\to A$.
\end{enumerate}

Given two algebra objects $(A,m_A,i_A)$ and $(B,m_B,i_B)$ in $\cT$, an algebra homomorphism $\theta: A\to B$ is a map in $\cT(A,B)$ that is compatible with the units and multiplications of $A$ and $B$, i.e.,
$m_B\circ(\theta\otimes \theta) = \theta\circ m_A$
and
$\theta\circ i_A = i_B$.
\end{defn}

\begin{nota}
Given an algebra $(A,m,i)$, we represent $m$ and $i$ by a trivalent and univalent vertex respectively:
$$
m = 
\begin{tikzpicture}[baseline=-.1cm]
	\coordinate (a) at (0,0);
	\draw (a) -- (0,.3);
    \draw (-.3,-.3) arc (180:0:.3cm);
    \filldraw (a) circle (.05cm);
    \node at (-.3,-.5) {\scriptsize{$A$}};
	\node at (.3,-.5) {\scriptsize{$A$}};
	\node at (0,.5) {\scriptsize{$A$}};
\end{tikzpicture}
\qquad\qquad
i =
\begin{tikzpicture}[baseline=-.1cm]
	\coordinate (a) at (0,-.1);
	\draw (a) -- (0,.3);
    \filldraw (a) circle (.05cm);
    \node at (0,.5) {\scriptsize{$A$}};
\end{tikzpicture}\,.
$$
\end{nota}

The following proposition is straightforward, and we leave the details to the reader.

\begin{prop}
\label{prop:AlgebrasAndLaxTensorFunctors}
We have an equivalence of categories
\[
\left\{\,\text{\rm Algebra objects in $\Vec(\cC)$}\,\right\}
\,\,\cong\,\,
\left\{\,\text{\rm Lax tensor functors $\cC^{\op} \to \Vec$}\,\right\}.
\]
\end{prop}

The content of the above proposition is that the data of a multiplication map $m$ for algebra object $\bfA \in \Vec(\cC)$ is equivalent to the data of a laxitor $\mu$ for $\bfA: \cC^{\op}\to \Vec$, and the unit $i \in \bfA(1_\cC)$ can be identified with the image of $1_\bbC$ under the unit map $\iota: 1_\Vec \to \bfA(1_\cC)$.

We will use the equivalence of categories in Proposition \ref{prop:AlgebrasAndLaxTensorFunctors} to freely pass between algebra objects in $\Vec(\cC)$ and lax tensor functors $\cC^{\op}\to \Vec$.

\begin{defn} 
\label{defn:NonStarBaseAlgebra}
The \textit{base algebra} associated to $(\bfA, m, i)$ is the unital associative algebra $\bfA(1_\cC)$ with multiplication $\mu_{1_\cC,1_\cC}:\bfA(1_\cC)\otimes \bfA(1_\cC)\to \bfA(1_\cC)$ and unit $i\in \bfA(1_\cC)$.
\end{defn}

We now discuss $\cC$-module categories and $\cC$-module functors.
In Section \ref{sec:AlgebrasAndModuleCategories} below, we will review the equivalence between algebra objects and cyclic $\cC$-module categories, whose objects are labeled by the objects of $\cC$.

\begin{defn}
A $\cC$-module category is a linear category $\cM$ with a bilinear bifunctor $-\otimes -: \cC\times \cM \to \cM$ together with unitor natural isomorphisms $\lambda_m : 1_\cC\otimes m \to m$ for $m\in \cM$ and associator natural isomorphisms $\alpha_{c,d,m}: c\otimes (d\otimes m ) \to (c\otimes_\cC d)\otimes m$ satisfying 
the pentagon axiom
$$
\alpha_{a,b\otimes_\cC c, m}\circ \alpha_{a,b,c\otimes m}
=
(\alpha_{a,b,c}\otimes \id_m)\circ \alpha_{a\otimes_\cC b, c, m}\circ (\id_a \otimes \alpha_{b,c,m})
\qquad
a,b,c\in \cC; m\in \cM
$$
and the triangle axioms
$$
(\rho_c \otimes \id_m) \circ \alpha_{c,1_\cC,m}
=
\id_c \otimes \lambda_m
\qquad\text{ and }\qquad
(\lambda_c \otimes \id_m) \circ \alpha_{1_\cC,c,m}
=
\lambda_{c\otimes m}
\qquad c\in\cC; m\in \cM.
$$
As with tensor categories, we will suppress the associator and unitor isomorphisms for module categories whenever possible to ease the notation.

A $\cC$-module category $\cM$ is called \emph{cyclic} with basepoint $m\in\cM$ if the objects of $\cM$ are exactly the $c\otimes m$ for $c\in \cC$.

A $\cC$-\emph{module dagger category} is a $\cC$-module category which is also a dagger category such that the isomorphisms $\alpha_{a,b,m}$ and $\lambda_m$ are unitary isomorphisms, and for all $\psi\in \cC(a,b)$ and $\phi\in\cM(m,n)$, we have $(\psi\otimes \phi)^* = \psi^*\otimes \phi^*$, where $\psi^*\in \cC(b,a)$ comes from the dagger structure of $\cC$.
A $\cC$-\emph{module} C*-\emph{category} is a $\cC$-module dagger category whose underlying dagger category is a C*-category.
A $\cC$-module W*-category is defined similarly, but we must require that the action maps $\phi\mapsto \psi\otimes \phi$ are weak*-continuous.
\end{defn}

\begin{rem}
Some of our readers may want to take idempotent completions of cyclic $\cC$-module categories, but we will not need it for our purposes.
We note that one can freely pass back and forth between a cyclic $\cC$-module category and its idempotent completion, since we are remembering the basepoint.
\end{rem}

\begin{defn}
Suppose $\cM,\cN$ are two $\cC$-module categories.
A $\cC$-\emph{module functor} $(\bfF,\omega): \cM\to \cN$ is a functor $\bfF: \cM\to \cN$ together with a family $\omega$ of natural isomorphisms $(\omega_{c,m}: c\otimes \bfF(m)\xrightarrow{\cong} \bfF(c\otimes m))_{c\in\cC, m\in\cM}$ such that the following axioms hold:
\begin{itemize}
\item 
(assocaitivity)
For all $c\in \cC$ and $m\in\cM$, the following diagram commutes:
$$
\xymatrix{
c\otimes (d \otimes \bfF(m)) 
\ar[rr]^{\id_c \otimes \omega_{d,m}}
\ar[d]^\alpha
&&
c\otimes \bfF(d\otimes m)
\ar[rr]^{\omega_{c,d\otimes m}}
&&
\bfF(c\otimes (d\otimes m))
\ar[d]^{\bfF(\alpha)}
\\
(c\otimes d) \otimes \bfF(m)
\ar[rrrr]^{\omega_{c\otimes d,m}}
&&&&
\bfF((c\otimes d)\otimes m)
}
$$
\item
(unitality)
For all $c\in \cC$ and $m\in \cM$, the following diagram commutes:
\[
\xymatrix{
1_\cC \otimes \bfF(m) 
\ar[dr]_{\lambda_{\bfF(m)}}
\ar[rr]^{\omega_{c,m}}
&&
\bfF(1_\cC \otimes m)
\ar[dl]^{\bfF(\lambda_m)}
\\
&\bfF(m)
}
\]
\end{itemize}
A \emph{cyclic} $\cC$-module functor between cyclic $\cC$-module categories $(\bfF, \omega, \varpi):(\cM,m) \to (\cN,n)$ is a $\cC$-module functor $(\bfF,\omega)$ together with an isomorphism $\varpi: n \to \bfF(m)$.

If $\cM, \cN$ are $\cC$-module dagger categories, we call a $\cC$-module functor $(\bfF,\omega)$ a $\cC$-module \emph{dagger} functor if all the isomorphisms $\omega_{c,m}$ are unitary, and $\bfF(f^*)=\bfF(f)^*$ for all morphisms $f\in\cM(m,n)$.
If $\cM,\cN$ are $\cC$-module W*-categories, we call a $\cC$-module dagger functor \emph{normal} if each map $\cM(m, n) \to \cN(\bfF(m), \bfF(n))$ is continuous with respect to the weak* topologies.
Finally, if $(\cM,m)$ and $(\cN,n)$ are cyclic $\cC$-module dagger categories, $(\bfF,\omega, \varpi)$ is called a \emph{cyclic} $\cC$-module dagger functor if $(\bfF,\omega)$ is a dagger $\cC$-module functor and $\varpi$ is unitary.
\end{defn}

\subsection{Equivalence between algebras and cyclic module categories}
\label{sec:AlgebrasAndModuleCategories}

The following theorem is well-known to experts, so we will only sketch the proof.

\begin{thm}
\label{thm:AlgebraModuleCategoryCorrespondence}
There is an equivalence of categories
$$
\left\{\,\text{\rm Algebra objects in $\Vec(\cC)$}\,\right\}
\,\,\cong\,\,
\left\{\,\text{\rm Cyclic $\cC$-module categories}\,\right\}.
$$
\end{thm}

\begin{rem}
In certain situations, subject to appropriate adjectives, the above correspondence is usually stated as a correspondence between algebra objects up to Morita equivalence and $\cC$-module categories without basepoints.
(In the semi-simple case, see \cite[Rem.\,3.5(ii)]{MR1976459}.)
Picking a basepoint for $\cM$ allows us to recover the algebra, and not just its Morita class.
\end{rem}

\begin{proof}[Sketch of proof]
To go from algebras objects to cyclic $\cC$-module categories, 
we take the category $\FreeMod_{\Vec(\cC)}(\bfA)$ of free right $\bfA$-modules in $\Vec(\cC)$, whose objects are of the form
$\mathbf{c}\otimes \bfA$ for $c\in\cC$.
(Note this is the image of the free module functor described in \cite{MR1936496,MR2863377,1509.02937}.)

To go from cyclic $\cC$-module categories to algebras, we get an algebra $\bfA\in \Vec(\cC)$ by taking the $\Vec(\cC)$-valued internal hom: $\bfA:=\underline{\Hom}_{\Vec(\cC)}(m,m)$, which satisfies the universal property
$
\Hom_{\Vec(\cC)}(\mathbf{c},\bfA) \cong \cM( c\otimes m , m)
$
for all $c\in \cC$.
\end{proof}

We discuss in more detail the correspondence between algebra objects $\bfA\in \Vec(\cC)$ and cyclic $\cC$-module categories, as it will be useful for our discussion of operator algebras in $\Vec(\cC)$, 

\begin{construction}[Module $\cM_\bfA$ from algebra $\bfA$]
\label{construction:Module}
Suppose that $\bfA\in \Vec(\cC)$ is an algebra object.
The cyclic left $\cC$-module category $\cM_\bfA$ has objects given by $\mathbf{c}\otimes \bfA \in \Vec(\cC)$ for $c\in \cC$, and morphisms given by right $\bfA$-module maps.
Recall from \cite[Lem.\,3.2]{MR1976459} (see also \cite[Fig.\,4]{MR1936496}) that there is a natural equivalence
$$
\Hom_{\cM_\bfA}(\mathbf{a}\otimes \bfA , \mathbf{b}\otimes \bfA) \cong \Hom_{\Vec(\cC)}(\mathbf{a}, \mathbf{b}\otimes \bfA).
$$
Using this natural equivalence, we give a concrete, equivalent description of $\cM_\bfA$ which is even easier to work with.

The objects are the symbols
$c_\bfA$ for $c\in \cC$, and the morphism spaces are defined by
$$
\cM_\bfA(a_\bfA,b_\bfA):=
\Hom_{\Vec(\cC)}(\mathbf{a},\mathbf{b}\otimes \bfA)
\cong
\bfA(\overline{b}\otimes a)\in\Vec
$$
by Frobenius reciprocity.
Graphically, we denote our morphisms as follows:
$$
\begin{tikzpicture}[baseline = -.1cm]
    \draw (0,0) -- (.7,0);
    \draw (0,-.7) -- (0,.7);
    \roundNbox{unshaded}{(0,0)}{.3}{0}{0}{$f$}
    \node at (-.2,.5) {\scriptsize{$\mathbf{b}$}};
    \node at (.5,.15) {\scriptsize{$\bfA$}};
    \node at (-.2,-.5) {\scriptsize{$\mathbf{a}$}};
\end{tikzpicture}
:=
\begin{tikzpicture}[baseline = -.1cm]
    \draw (.15,0) -- (.15,.7);
    \draw (-.15,0) -- (-.15,.7);
    \draw (0,-.7) -- (0,0);
    \roundNbox{unshaded}{(0,0)}{.3}{0}{0}{$f$}
    \node at (.3,.5) {\scriptsize{$\bfA$}};
    \node at (.15,-.5) {\scriptsize{$\mathbf{a}$}};
    \node at (-.3,.5) {\scriptsize{$\mathbf{b}$}};
\end{tikzpicture}
$$
Composition is defined by stacking and using the multiplication map for $\bfA$:
$$
g\circ f
:=
\begin{tikzpicture}[baseline = .4cm]
    \draw (.3,1) arc (90:-90:.5cm);
    \draw (.8,.5) -- (1.2,.5);
    \draw (0,-.7) -- (0,1.7);
    \filldraw (.8,.5) circle (.05cm);
    \roundNbox{unshaded}{(0,0)}{.3}{0}{0}{$f$}
    \roundNbox{unshaded}{(0,1)}{.3}{0}{0}{$g$}
    \node at (-.2,1.5) {\scriptsize{$\mathbf{c}$}};
    \node at (-.2,.5) {\scriptsize{$\mathbf{b}$}};
    \node at (1,.7) {\scriptsize{$\bfA$}};
    \node at (-.2,-.5) {\scriptsize{$\mathbf{a}$}};
\end{tikzpicture}
.
$$
It is straightforward to check that composition is associative.
The identity morphism for $a_\bfA$ is given by 
$$
\id_{a_\bfA} =
\begin{tikzpicture}[baseline = -.1cm]
    \draw (0,-.7) -- (0,.7);
    \draw (.3,0) -- (1,0);
    \filldraw (.3,0) circle (.05cm);
    \roundNbox{dashed}{(0,0)}{.4}{0}{.3}{}
    \node at (-.15,0) {\scriptsize{$\mathbf{a}$}};
    \node at (.5,.15) {\scriptsize{$\bfA$}};
\end{tikzpicture}
\,.
$$

The $\cC$-module structure on $\cM_\bfA$ is defined as follows.
The object $c\in\cC$ acts on $a_\bfA$ by $(c\otimes a)_\bfA$.
Given morphisms $\psi\in \cC(c, d)$ and $f\in \cM_\bfA(a_\bfA,b_\bfA)=\Hom_{\Vec(\cC)}(\mathbf{a},\mathbf{b}\otimes \bfA)$, we define
$$
\psi\otimes f 
=
\begin{tikzpicture}[baseline = -.1cm]
    \draw (-.8,-.7) -- (-.8,.7);
    \draw (0,0) -- (.7,0);
    \draw (0,-.7) -- (0,.7);
    \roundNbox{unshaded}{(-.8,0)}{.3}{0}{0}{$\psi$}
    \roundNbox{unshaded}{(0,0)}{.3}{0}{0}{$f$}
    \node at (-.2,.5) {\scriptsize{$\mathbf{b}$}};
    \node at (.5,.15) {\scriptsize{$\bfA$}};
    \node at (-.2,-.5) {\scriptsize{$\mathbf{a}$}};
    \node at (-1,-.5) {\scriptsize{$\mathbf{c}$}};
    \node at (-1,.5) {\scriptsize{$\mathbf{d}$}};
\end{tikzpicture}
\in 
\Hom_{\Vec(\cC)}(\mathbf{c}\otimes \mathbf{a},\mathbf{d}\otimes \mathbf{b}\otimes \bfA)
=
\cM_\bfA((c\otimes a)_\bfA, (d\otimes b)_\bfA).
$$
It is straightforward to show this defines a bifunctor using the graphical calculus as a guide.
We see that $\cM_\bfA$ is cyclic, with basepoint $(1_\cC)_\bfA$.

If $\theta: \bfA\Rightarrow \bfB$ is an algebra natural transformation, we define $\check{\theta}:\cM_\bfA \to \cM_\bfB$ by
$$
\begin{tikzpicture}[baseline = -.1cm]
    \draw (0,0) -- (.7,0);
    \draw (0,-.7) -- (0,.7);
    \roundNbox{unshaded}{(0,0)}{.3}{0}{0}{$f$}
    \node at (-.2,.5) {\scriptsize{$\mathbf{b}$}};
    \node at (.5,.15) {\scriptsize{$\bfA$}};
    \node at (-.2,-.5) {\scriptsize{$\mathbf{a}$}};
\end{tikzpicture}
\longmapsto
\begin{tikzpicture}[baseline = -.1cm]
    \draw (0,0) -- (1.7,0);
    \draw (0,-.7) -- (0,.7);
    \roundNbox{unshaded}{(0,0)}{.3}{0}{0}{$f$}
    \roundNbox{unshaded}{(1,0)}{.3}{0}{0}{$\theta$}
    \node at (-.2,.5) {\scriptsize{$\mathbf{b}$}};
    \node at (.5,.15) {\scriptsize{$\bfA$}};
    \node at (-.2,-.5) {\scriptsize{$\mathbf{a}$}};
    \node at (1.5,.15) {\scriptsize{$\bfB$}};
\end{tikzpicture}
.
$$
It is straightforward to verify that $\check{\theta}$ is a functor.
The fact that $\theta$ preserves units means $\check{\theta}$ preserves identity morphisms, and the fact that $\theta$ intertwines the multiplication means that $\check{\theta}$ preserves composition.
\end{construction}

\begin{construction}[Algebra $\bfA_{m}$ from module $(\cM,m)$]
\label{construction:Algebra}
Suppose that $(\cM,m)$ is a cyclic left $\cC$-module category.
Concretely, for $a\in \cC^{\text{op}}$, we define $\bfA(a) := \cM(a\otimes m, m) \in \Vec$.
We use the usual diagrammatic calculus to denote $f\in \bfA(a)$ as
$$
\begin{tikzpicture}[baseline = -.1cm]
    \draw (0,0) -- (-.7,0);
    \draw (0,-.7) -- (0,.7);
    \roundNbox{unshaded}{(0,0)}{.3}{0}{0}{$f$}
    \node at (.2,-.5) {\scriptsize{$m$}};
    \node at (.2,.5) {\scriptsize{$m$}};
    \node at (-.5,.2) {\scriptsize{$a$}};
\end{tikzpicture}
:=
\begin{tikzpicture}[baseline = -.1cm]
    \draw (0,0) -- (0,.7);
    \draw (-.15,-.7) -- (-.15,0);
    \draw (.15,-.7) -- (.15,0);
    \roundNbox{unshaded}{(0,0)}{.3}{0}{0}{$f$}
    \node at (.4,-.5) {\scriptsize{$m$}};
    \node at (.2,.5) {\scriptsize{$m$}};
    \node at (-.4,-.5) {\scriptsize{$a$}};
\end{tikzpicture}
$$
Given $\psi\in \cC(b,a)$, we get a map $\bfA(\psi): \bfA(a) \to \bfA(b)$ by
$$
\begin{tikzpicture}[baseline = -.1cm]
    \draw (0,0) -- (-.7,0);
    \draw (0,-.7) -- (0,.7);
    \roundNbox{unshaded}{(0,0)}{.3}{0}{0}{$f$}
    \node at (.2,-.5) {\scriptsize{$m$}};
    \node at (.2,.5) {\scriptsize{$m$}};
    \node at (-.5,.2) {\scriptsize{$a$}};
\end{tikzpicture}
\mapsto 
\begin{tikzpicture}[baseline = -.1cm]
    \draw (0,0) -- (-1.7,0);
    \draw (0,-.7) -- (0,.7);
    \roundNbox{unshaded}{(-1,0)}{.3}{0}{0}{$\psi$}
    \roundNbox{unshaded}{(0,0)}{.3}{0}{0}{$f$}
    \node at (.2,-.5) {\scriptsize{$m$}};
    \node at (.2,.5) {\scriptsize{$m$}};
    \node at (-.5,.2) {\scriptsize{$a$}};
    \node at (-1.5,.2) {\scriptsize{$b$}};
\end{tikzpicture}
\,.
$$
We define the laxitor $\mu_{a,b}: \bfA(a)\otimes \bfA(b) \to \bfA(a\otimes b)$ for $a,b\in\cC$ on $f\otimes g\in \bfA(a)\otimes \bfA(b)$ by
$$
\mu_{a,b}(f\otimes g)
:=
f\circ (\id_{a}\otimes g)
=
\begin{tikzpicture}[baseline = .4cm]
    \draw (-.4,1.3) -- (-.4,1.7);
    \draw (0,0) -- (0,.7);
    \draw (-.65,-.7) -- (-.65,1);
    \draw (-.15,-.7) -- (-.15,0);
    \draw (.15,-.7) -- (.15,0);
    \roundNbox{unshaded}{(-.4,1)}{.3}{.3}{.3}{$f$}
    \roundNbox{unshaded}{(0,0)}{.3}{0}{0}{$g$}
    \node at (.4,-.5) {\scriptsize{$m$}};
    \node at (.2,.5) {\scriptsize{$m$}};
    \node at (-.2,1.5) {\scriptsize{$m$}};
    \node at (-.3,-.5) {\scriptsize{$b$}};
    \node at (-.8,-.5) {\scriptsize{$a$}};
\end{tikzpicture}
=
\begin{tikzpicture}[baseline = .4cm]
    \draw (0,0) -- (-.7,0);
    \draw (0,1) -- (-.7,1);
    \draw (0,-.7) -- (0,1.7);
    \roundNbox{unshaded}{(0,0)}{.3}{0}{0}{$g$}
    \roundNbox{unshaded}{(0,1)}{.3}{0}{0}{$f$}
    \node at (.2,-.5) {\scriptsize{$m$}};
    \node at (.2,.5) {\scriptsize{$m$}};
    \node at (.2,1.5) {\scriptsize{$m$}};
    \node at (-.5,1.2) {\scriptsize{$a$}};
    \node at (-.5,.2) {\scriptsize{$b$}};
\end{tikzpicture}
\in\bfA(a\otimes b).
$$
The unit $i\in \bfA(1_\cC)=\cM(1_\cC\otimes m, m)$ is given by the identity $\id_m$.

Now suppose $\bfA$ and $\bfB$ are the algebras corresponding to the cyclic $\cC$-module categories $(\cM,m)$ and $(\cN,n)$.
Given a cyclic $\cC$-module functor $(\Phi,\omega, \varpi): (\cM,m) \to (\cN,n)$, we get a natural transformation $\hat{\Phi}: \bfA\Rightarrow \bfB$ by defining $\hat{\Phi}_a$ on $f\in \bfA(a)=\cM(a\otimes m ,m)$ by
$$
\hat{\Phi}_a (f)
=
\begin{tikzpicture}[baseline=-.1cm]
	\draw (0,.5) -- (0,2);
	\draw[double] (0,.5) -- (0,-.5);
	\draw (.2,-.5) -- (.2,-2);
	\draw (-.2,-.5) -- (-.2,-2);
	\roundNbox{unshaded}{(0,1.5)}{.25}{.2}{.2}{$\varpi^{-1}$}
	\roundNbox{unshaded}{(0,.5)}{.3}{.2}{.2}{$\Phi(f)$}
	\roundNbox{unshaded}{(0,-.5)}{.3}{.2}{.2}{$\omega_a$}
	\roundNbox{unshaded}{(.2,-1.5)}{.25}{0}{0}{$\varpi$}
	\node at (.2,1.9) {\scriptsize{$n$}};
	\node at (.4,1) {\scriptsize{$\Phi(m)$}};
	\node at (.7,0) {\scriptsize{$\Phi(a\otimes m)$}};
	\node at (.6,-1) {\scriptsize{$\Phi(m)$}};
	\node at (.4,-1.9) {\scriptsize{$n$}};
	\node at (-.4,-1.5) {\scriptsize{$a$}};
\end{tikzpicture}
\in\cN(a\otimes n, n)
=\bfB(a).
$$
Since $\Phi(\id_m) = \id_{\Phi(m)}$, $\hat{\Phi}(i_\bfA) = i_\bfB$.
Since $\Phi$ preserves composition, we see $\hat{\Phi}$ intertwines $\mu^\bfA$ and $\mu^\bfB$.
\end{construction}

\begin{rem}
\label{rem:RightVersionOfConstructionAlgebra}
Just as we defined the left cyclic $\cC$-module category $\cM_\bfA$ from our algebra object $\bfA\in \Vec(\cC)$, we can also define the right cyclic $\cC$-module category ${}_\bfA\cM$ as follows.
The objects are symbols of the form ${}_\bfA c$ for $c\in\cC$, and 
$$
\Hom_{{}_\bfA\cM}({}_\bfA a,{}_\bfA b)
=
\Hom_{\Hilb(\cC)}(\mathbf{a}, \bfH\otimes \mathbf{b})
\cong
\bfA(a\otimes \overline{b}).
$$
We have a similar diagrammatic calculus for the morphisms:
$$
\begin{tikzpicture}[baseline = -.1cm, xscale=-1]
    \draw (0,0) -- (.7,0);
    \draw (0,-.7) -- (0,.7);
    \roundNbox{unshaded}{(0,0)}{.3}{0}{0}{$f$}
    \node at (-.2,.5) {\scriptsize{$\mathbf{b}$}};
    \node at (.5,.15) {\scriptsize{$\bfA$}};
    \node at (-.2,-.5) {\scriptsize{$\mathbf{a}$}};
\end{tikzpicture}
:=
\begin{tikzpicture}[baseline = -.1cm]
    \draw (.15,0) -- (.15,.7);
    \draw (-.15,0) -- (-.15,.7);
    \draw (0,-.7) -- (0,0);
    \roundNbox{unshaded}{(0,0)}{.3}{0}{0}{$f$}
    \node at (-.3,.5) {\scriptsize{$\bfA$}};
    \node at (.15,-.5) {\scriptsize{$\mathbf{a}$}};
    \node at (.3,.5) {\scriptsize{$\mathbf{b}$}};
\end{tikzpicture}
$$
There is a similar version of Theorem \ref{thm:AlgebraModuleCategoryCorrespondence} which gives an equivalence of categories between cyclic right $\cC$-module categories and algebra objects in $\Vec(\cC)$.
\end{rem}

\begin{defn}
Given a tensor category $\cC$, we define $\cC^{\text{op}}$ to be the opposite category with the reverse composition.
We define $\cC^{\text{mp}}$ to be the category with the reverse monoidal structure.
We define $\cC^{\text{mop}}$ to be the tensor category with both the reverse composition and reverse monoidal structure.
(This notation is from \cite[p.~26]{1312.7188}.)
When $\cC$ is rigid, we get monoidal equivalences $\cC\cong \cC^{\text{mop}}$ and $\cC^{\text{op}} \cong \cC^{\text{mp}}$.
\end{defn}

\begin{rem}
The data of $\cM$ as a left $\cC$-module category can also be viewed as:
\begin{itemize}
\item
$\cM$ as a right $\cC^{\text{mp}}$-module category,
\item
$\cM^{\op}$ as a left $\cC^{\op}$ module category, and
\item
$\cM^{\op}$ as a right $\cC^{\text{mop}}$ module category.
\end{itemize}
When $\cM$ is cyclic with basepoint $m$, these four module categories correspond to the following four algebras:
$\bfA\in \Vec(\cC)$, $\bfA^{\text{mp}}\in \Vec(\cC^{\text{mp}})$, $\bfA^{\op}\in \Vec(\cC^{\op})$, and $\bfA^{\text{mop}}\in \Vec(\cC^{\text{mop}})$.

One shows that on objects $a\in \cC$, $\bfA^{\text{mp}}(a) = \bfA(a)$ and $\bfA^{\op}(a) = \bfA^{\text{mop}}(a) = \bfA(a^{\vee})$.
The multiplication maps for $\bfA^{\text{mp}}$,  $\bfA^{\op}$, and $\bfA^{\text{mop}}$ are given in terms of the multiplication map  $\mu$ for $\bfA$ respectively as follows:
\begin{align*}
\bfA^{\text{mp}}(a) \otimes \bfA^{\text{mp}}(b) &= \bfA(a) \otimes \bfA(b) \xrightarrow{\beta^\Vec} \bfA(b)\otimes \bfA(a) \xrightarrow{\mu} \bfA(b\otimes a) = \bfA^{\text{mp}}(a\otimes_{\cC^{\text{mp}}} b)
\\
\bfA^{\text{op}}(a) \otimes \bfA^{\text{op}}(b) &= \bfA(a^\vee) \otimes \bfA(b^\vee) \xrightarrow{\beta^\Vec} \bfA(b^\vee) \otimes \bfA(a^\vee)  \xrightarrow{\mu} \bfA(b^\vee\otimes a^\vee) \cong  \bfA^{\text{op}}(a\otimes_{\cC^{\text{op}}} b)
\\
\bfA^{\text{mop}}(a) \otimes \bfA^{\text{mop}}(b) &= \bfA(a^\vee) \otimes \bfA(b^\vee) \xrightarrow{\mu} \bfA(a^\vee\otimes b^\vee) \cong \bfA((b\otimes a)^\vee)  =  \bfA^{\text{mop}}(a\otimes_{\cC^{\text{mop}}} b)
\end{align*}
where $\beta^\Vec$ is the braiding in $\Vec$.
We see that the monoidal equivalence $\cC\to \cC^{\text{mop}}$ given by $a\mapsto a^\vee$ takes the algebra $\bfA^{\text{mop}}\in \Vec(\cC^{\text{mop}})$ to the algebra $\bfA\in \Vec(\cC)$, and similarly for $\bfA^{\text{op}}$ and $\bfA^{\text{mp}}$.
\end{rem}

\subsection{\texorpdfstring{$*$}{star}-Algebra objects and dagger module categories}

We now introduce the new notion of a $*$-algebra object in $\Vec(\cC)$.
We will give three equivalent definitions.

Let $(\cC, \varphi, \nu, r)$ be a rigid C*-tensor category with its canonical bi-involutive structure.
We warn the reader that $\cC^{\op}$ has the bi-involutive structure defined in Remark \ref{rem:CopInvolutive}.
In particular, for the material in this section, the coherence axioms for the involutive structure on $\Vec(\cC)$
have inverses on all instances of $\varphi, \nu, r$.

\begin{defn}
\label{defn:ConjugateLinearNaturalTransformation}
Let $\bfF,\bfG\in \Vec(\cC)$.
A conjugate linear natural transformation $\theta : \bfF \Rightarrow \bfG$ is a family of conjugate linear maps $\theta_a : \bfF(a)\to \bfG(\overline{a})$ for $a\in \cC$ which is conjugate natural, i.e., for all $\psi\in \cC(b,a)$, the following diagram commutes:
\[
\xymatrix{
\bfF(a) 
\ar[d]^{\bfF(\psi)}
\ar[rr]^{\theta_a}
&&
\bfG(\overline{a})
\ar[d]^{\bfF(\overline{\psi})}
\\
\bfF(b)
\ar[rr]^{\theta_b}
&&
\bfG(\overline{b})
}
\]
Given two conjugate linear natural transformations $\theta : \bfF \Rightarrow \bfG$ and $\rho: \bfG \Rightarrow \bfH$, their composite $\rho \circ \theta :\bfF \Rightarrow \bfH$ is the \emph{linear} natural transformation defined by
$$
(\rho\circ \theta)_a :=\bfH(\varphi_a) \circ \rho_{\overline{a}} \circ \theta_a : \bfF(a) \to \bfH(a).
$$
\end{defn}

\begin{rem}
The reader should be careful not to confuse a conjugate linear natural transformation with the conjugate natural transformation from Definition \ref{defn:VecCInvolutive}, which is linear!
\end{rem}

\begin{defn} 
\label{def:StarAlgebra}
A $*$-\emph{structure} on an algebra object $\bfA\in \Vec(\cC)$ is a conjugate linear natural transformation $j: \bfA \Rightarrow \bfA$ which satisfies the following axioms:
\begin{itemize}
\item 
(involutive) 
$j\circ j = \id_\bfA$, i.e., $\id_{\bfA(a)} = \bfA(\varphi_a) \circ j_{\overline{a}} \circ j_a$ for all $a\in \cC$
\item
(unital)
$j_{1_\cC} = \bfA(r^{-1})$
\item
(monoidal)
$j_{a\otimes b}(\mu_{a,b}(f, g))=\bfA(\nu_{b,a}^{-1})(\mu_{\overline{b},\overline{a}}(j_b(g), j_a(f)))$ for all $a,b\in\cC$.
\end{itemize}
A $*$-algebra object in $\Vec(\cC)$ is an algebra object together with a $*$-structure.

Suppose $(\bfA,j^\bfA),(\bfB,j^\bfB)\in\Vec(\cC)$ are $*$-algebras.
We call a natural transformation $\theta: \bfA \Rightarrow \bfB$ a $*$-\emph{natural transformation} if
$\theta_{\overline{a}}(j_a^\bfA(f))=j_a^\bfB(\theta_a(f))$ for all $f\in\bfA(a)$.
\end{defn}

We already saw in Definition \ref{defn:VecCInvolutive} that $\Vec(\cC)$ has the structure of an involutive category.
We now define an involutive structure on the category of algebra objects in $\Vec(\cC)$.

\begin{defn}
The \emph{conjugate algebra} $(\overline{\bfA},\overline{\mu}, \overline{\iota})\in \Vec(\cC)$ is defined by first taking the conjugate functor $\overline{\bfA}\in \Vec(\cC)$ from Definition \ref{defn:VecCInvolutive}.
The laxitor $\overline{\mu}_{a,b}: \overline{\bfA}(a) \otimes \overline{\bfA}(b) \to \overline{\bfA}(a\otimes b)$ is given by
$$
\overline{\bfA(\overline{a})}\otimes \overline{\bfA(\overline{b})}
\xrightarrow{\nu_{\bfA(\overline{a}), \bfA(\overline{b})}}
\overline{\bfA(\overline{b}) \otimes \bfA(\overline{a})}
\xrightarrow{\overline{\mu_{\overline{b}, \overline{a}}}}
\overline{\bfA(\overline{b} \otimes \overline{a})}
\xrightarrow{\overline{\bfA(\nu_{a,b}^{-1})}}
\overline{\bfA(\overline{a\otimes b})}.
$$
The unit map $\overline{\iota} : 1_{\cC} \to \overline{\bfA}(1_\cC)$ is given by  
$$
1_\cC 
\xrightarrow{r}
\overline{1_\cC}
\xrightarrow{\overline{\iota}}
\overline{\bfA(1_\cC)}
\xrightarrow{\overline{\bfA(r^{-1})}}
\overline{\bfA(\overline{1_\cC})}.
$$
It is straightforward to check that this defines a lax monoidal functor.

Suppose now $\theta: \bfA\Rightarrow \bfB$ is an algebra natural transformation, 
we let $\overline{\theta} : \overline{\bfA} \Rightarrow \overline{\bfB}$ be the conjugate natural transformation from Definition \ref{defn:VecCInvolutive}.
We see that $\overline{\theta}$ is an algebra natural transformation.

It is straightforward to verify that the category of algebra objects in $\Vec(\cC)$ with algebra natural transformations is involutive.
\end{defn}

\begin{prop}
\label{prop:StarAlgebraEquivalentDefinitions}
The following pieces of data are equivalent for an algebra object $(\bfA, \mu, \iota) \in \Vec(\cC)$:
\begin{enumerate}[(1)]
\item
a $*$-structure $j$ on $\bfA$,
\item
an involutive structure $\chi$ on $\bfA$ (thought of as a lax monoidal functor c.f.~Prop.~\ref{prop:AlgebrasAndLaxTensorFunctors}), and
\item
an algebra natural isomorphism $\sigma: \bfA \Rightarrow \overline{\bfA}$ which is involutive, i.e., $\overline{\sigma}\circ \sigma = \varphi_\bfA$.
\end{enumerate}
\end{prop}
\begin{proof}
We tell how to construct the maps, and we leave the details to the reader.
To go from (1) to (2), we define $\chi_a(f) = \overline{j_a^{-1}(f)}$ for $f\in \bfA(\overline{a})$.
To go from (2) to (3), we define $\sigma_a := \chi_a \circ \bfA(\varphi_a^{-1})$.
To go from (3) to (1), we define $j_a(f) := \overline{\sigma_a(f)}$ for all $f\in \bfA(a)$.
\end{proof}

Similar to Theorem \ref{thm:AlgebraModuleCategoryCorrespondence}, we have the following equivalence.
Again, we only sketch the proof, as the details are similar to that of Theorem \ref{thm:AlgebraModuleCategoryCorrespondence}.
We do, however, provide explicit details on the correspondence between the $*$-structure on an algebra object $\bfA
\in \Vec(\cC)$ and a dagger structure on a cyclic $\cC$-module category $(\cM,m)$.

\begin{thm}
\label{thm:StarAlgebraDaggerModuleCategoryCorrespondence}
There is an equivalence of categories
$$
\left\{\,\text{\rm $*$-Algebra objects in $\Vec(\cC)$}\,\right\}
\,\,\cong\,\,
\left\{\,\text{\rm Cyclic $\cC$-module dagger categories}\,\right\}.
$$
\end{thm}

\begin{proof}[Sketch of proof]
If $(\bfA,\mu,\iota, j) :\cC^{\op}\to \Vec$ is a $*$-algebra object, we define a dagger structure on $\cM_\bfA$ from Construction \ref{construction:Module} as follows.
For $f\in \bfA(\overline{b}\otimes a)\cong \cM_\bfA(a_\bfA,b_\bfA)$, we define 
$$
f^* 
=
(\bfA(\id_{\overline{a}}\otimes \varphi_b) \circ \bfA(\nu_{a,\overline{b}})\circ j_{\overline{b}\otimes a}) (f)
 \in\bfA(\overline{a}\otimes b)\cong \cM_\bfA(b_\bfA,a_\bfA).
$$
To prove that this defines a dagger structure on $\cM_\bfA$, we must show $f^{**}=f$ and $(g\circ f)^* = f^*\circ g^*$ for composable $f$ and $g$.
The first of these properties is verified as follows.
Suppose $f\in \bfA(\overline{b}\otimes a)\cong \cM_\bfA(a_\bfA, b_\bfA)$. 
Using naturality of $j$, followed by the compatibility of $\nu$ and $\varphi$ in an involutive tensor category, we have
\begin{align*}
f^{**} = 
\begin{tikzpicture}[baseline=-.1cm]
	\draw (-.3,-3.1) -- (-.3,-1.2);
	\draw (.3,-3.1) -- (.3,-1.2);
	\draw[double] (0,-1.2) -- (0,2.4);
	\draw (0,2.4) -- (0,3.1);
	\roundNbox{unshaded}{(0,2.4)}{.3}{1.2}{1.2}{$(j_{\overline{\overline{b}\otimes a}}\circ j_{\overline{b}\otimes a}) (f)$}
	\roundNbox{unshaded}{(0,1.2)}{.3}{.2}{.2}{$\overline{\nu_{a, \overline{b}}}$}
	\roundNbox{unshaded}{(0,0)}{.3}{.6}{.6}{$\overline{\id_{\overline{a}}\otimes \varphi_b}$}
	\roundNbox{unshaded}{(0,-1.2)}{.3}{.2}{.2}{$\nu_{b,\overline{a}}$}
	\roundNbox{unshaded}{(.3,-2.4)}{.3}{0}{0}{$\varphi_a$}
	\node at (.2,2.9) {\scriptsize{$\bfA$}};
	\node at (.5,1.8) {\scriptsize{$\overline{\overline{\overline{\mathbf{b}}\otimes \mathbf{a}}}$}};
	\node at (.5,.6) {\scriptsize{$\overline{\overline{\mathbf{a}}\otimes \overline{\overline{\mathbf{b}}}}$}};
	\node at (.5,-.6) {\scriptsize{$\overline{\overline{\mathbf{a}}\otimes \mathbf{b}}$}};
	\node at (.5,-1.8) {\scriptsize{$\overline{\overline{\mathbf{a}}}$}};
	\node at (-.5,-2.4) {\scriptsize{$\overline{\mathbf{b}}$}};
	\node at (.5,-2.9) {\scriptsize{$\mathbf{a}$}};
\end{tikzpicture}
=
\begin{tikzpicture}[baseline=-.1cm]
	\draw (-.4,-3.1) -- (-.4,-1.2);
	\draw (.4,-3.1) -- (.4,-1.2);
	\draw[double] (0,-1.2) -- (0,1.2);
	\draw (.15,1.2) -- (.15,2.4);
	\draw (-.15,1.2) -- (-.15,2.4);
	\draw (0,2.4) -- (0,3.1);
	\roundNbox{unshaded}{(0,2.4)}{.3}{0}{0}{$f$}
	\roundNbox{unshaded}{(0,1.2)}{.35}{.2}{.2}{$\varphi_{\overline{b}\otimes a}^{-1}$}
	\roundNbox{unshaded}{(0,0)}{.3}{.2}{.2}{$\overline{\nu_{a, \overline{b}}}$}
	\roundNbox{unshaded}{(0,-1.2)}{.3}{.3}{.3}{$\nu_{\overline{\overline{b}},\overline{a}}$}
	\roundNbox{unshaded}{(.4,-2.4)}{.3}{0}{0}{$\varphi_a$}
	\roundNbox{unshaded}{(-.4,-2.4)}{.3}{0}{0}{$\varphi_{\overline{b}}$}
	\node at (.2,2.9) {\scriptsize{$\bfA$}};
	\node at (-.35,1.8) {\scriptsize{$\overline{\mathbf{b}}$}};
	\node at (.35,1.8) {\scriptsize{$\mathbf{a}$}};
	\node at (.5,.6) {\scriptsize{$\overline{\overline{\overline{\mathbf{b}}\otimes \mathbf{a}}}$}};
	\node at (.5,-.6) {\scriptsize{$\overline{\overline{\mathbf{a}}\otimes \overline{\overline{\mathbf{b}}}}$}};
	\node at (.6,-1.8) {\scriptsize{$\overline{\overline{\mathbf{a}}}$}};
	\node at (-.6,-1.8) {\scriptsize{$\overline{\overline{\overline{\mathbf{b}}}}$}};
	\node at (-.6,-2.9) {\scriptsize{$\overline{\mathbf{b}}$}};
	\node at (.6,-2.9) {\scriptsize{$\mathbf{a}$}};
\end{tikzpicture}
=
f.
\end{align*}
Suppose now $f\in \bfA(\overline{b}\otimes a)$ and $g\in \bfA(\overline{c}\otimes b)$.
Using naturality and monoidality of $j$,  
$$
(g\circ f)^* = 
\begin{tikzpicture}[baseline=-.1cm]
	\draw (1,2.7) .. controls ++(90:.4cm) and ++(0:.3cm) .. (0,3.2);
	\draw (-1,2.7) .. controls ++(90:.4cm) and ++(180:.3cm) .. (0,3.2);
	\filldraw (0,3.2) circle (.05cm);
	\draw (0,3.2) -- (0,3.6);
	\draw[double] (1,2.4) -- (1,1.2);
	\draw[double] (-1,2.4) -- (-1,1.2);
	\draw[quadruple={[line width=.15mm,white] in
      [line width=.45mm,black] in
      [line width=.75mm,white] in
      [line width=1.05mm,black]}] (0,0) -- (0,1.2);
	\draw[double] (0,-1.2) -- (0,0);
	\draw (.5,-3.1) -- (.5,-1.2);
	\draw (-.5,-3.1) -- (-.5,-1.2);
	\roundNbox{unshaded}{(-1,2.4)}{.3}{.4}{.4}{$j_{\overline{b}\otimes a}(f)$}
	\roundNbox{unshaded}{(1,2.4)}{.3}{.4}{.4}{$j_{\overline{c}\otimes b}(g)$}
	\roundNbox{unshaded}{(0,1.2)}{.35}{1.3}{1.3}{$\nu^{-1}_{(\overline{b}\otimes a), (\overline{c}\otimes b)}$}
	\roundNbox{unshaded}{(0,0)}{.3}{1.3}{1.3}{$\overline{\id_{\overline{c}}\otimes \coev_b \otimes \id_a}$}
	\roundNbox{unshaded}{(0,-1.2)}{.3}{1.2}{1.2}{$\nu_{a,\overline{c}}$}
	\roundNbox{unshaded}{(.5,-2.4)}{.3}{0}{0}{$\varphi_c$}
	\node at (.2,3.4) {\scriptsize{$\bfA$}};
	\node at (1.2,2.9) {\scriptsize{$\bfA$}};
	\node at (-1.2,2.9) {\scriptsize{$\bfA$}};
	\node at (1.5,1.8) {\scriptsize{$\overline{\overline{\mathbf{c}}\otimes\mathbf{b}}$}};
	\node at (-1.5,1.8) {\scriptsize{$\overline{\overline{\mathbf{b}}\otimes\mathbf{a}}$}};
	\node at (1,.6) {\scriptsize{$\overline{\overline{\mathbf{c}}\otimes \mathbf{b} \otimes \overline{\mathbf{b}} \otimes \mathbf{a}}$}};
	\node at (.5,-.6) {\scriptsize{$\overline{\overline{\mathbf{c}}\otimes \mathbf{a}}$}};
	\node at (.7,-1.8) {\scriptsize{$\overline{\overline{\mathbf{c}}}$}};
	\node at (-.7,-2.4) {\scriptsize{$\overline{\mathbf{a}}$}};
	\node at (.7,-2.9) {\scriptsize{$\mathbf{c}$}};
\end{tikzpicture}
=
\begin{tikzpicture}[baseline=-.1cm]
	\draw (1,2.7) .. controls ++(90:.4cm) and ++(0:.3cm) .. (0,3.2);
	\draw (-1,2.7) .. controls ++(90:.4cm) and ++(180:.3cm) .. (0,3.2);
	\filldraw (0,3.2) circle (.05cm);
	\draw (0,3.2) -- (0,3.6);
	\draw[double] (1,2.4) -- (1,1.2);
	\draw[double] (-1,2.4) -- (-1,1.2);
	\draw[quadruple={[line width=.15mm,white] in
      [line width=.45mm,black] in
      [line width=.75mm,white] in
      [line width=1.05mm,black]}] (0,0) -- (0,1.2);
	\draw[double] (0,-1.2) -- (0,0);
	\draw (1,-1.9) -- (1,0);
	\draw (-1,-1.9) -- (-1,0);
	\roundNbox{unshaded}{(-1,2.4)}{.3}{.4}{.4}{$j_{\overline{b}\otimes a}(f)$}
	\roundNbox{unshaded}{(1,2.4)}{.3}{.4}{.4}{$j_{\overline{c}\otimes b}(g)$}
	\roundNbox{unshaded}{(0,1.2)}{.35}{1.3}{1.3}{$\nu^{-1}_{(\overline{b}\otimes a), (\overline{c}\otimes b)}$}
	\roundNbox{unshaded}{(0,0)}{.3}{1.3}{1.3}{$\nu$}
	\roundNbox{unshaded}{(0,-1.2)}{.3}{.3}{.3}{$\overline{\coev_b}$}
	\roundNbox{unshaded}{(1,-1.2)}{.3}{0}{0}{$\varphi_c$}
	\node at (.2,3.4) {\scriptsize{$\bfA$}};
	\node at (1.2,2.9) {\scriptsize{$\bfA$}};
	\node at (-1.2,2.9) {\scriptsize{$\bfA$}};
	\node at (1.5,1.8) {\scriptsize{$\overline{\overline{\mathbf{c}}\otimes\mathbf{b}}$}};
	\node at (-1.5,1.8) {\scriptsize{$\overline{\overline{\mathbf{b}}\otimes\mathbf{a}}$}};
	\node at (1,.6) {\scriptsize{$\overline{\overline{\mathbf{c}}\otimes \mathbf{b} \otimes \overline{\mathbf{b}} \otimes \mathbf{a}}$}};
	\node at (.5,-.6) {\scriptsize{$\overline{\mathbf{b}\otimes \overline{\mathbf{b}}}$}};
	\node at (1.2,-.6) {\scriptsize{$\overline{\overline{\mathbf{c}}}$}};
	\node at (-1.2,-1.2) {\scriptsize{$\overline{\mathbf{a}}$}};
	\node at (1.2,-1.7) {\scriptsize{$\mathbf{c}$}};
\end{tikzpicture}
$$
Here, since $\nu$ is associative, we simply write one $\nu$ with three inputs for a composite of $\nu$'s to simplify the notation.
Again using associativity of $\nu$, followed by \eqref{eq:Bi-involutiveStructureOnC}, the right hand side above is equal to 
$$
\begin{tikzpicture}[baseline=.5cm]
	\draw (1,2.7) .. controls ++(90:.4cm) and ++(0:.3cm) .. (0,3.2);
	\draw (-1,2.7) .. controls ++(90:.4cm) and ++(180:.3cm) .. (0,3.2);
	\filldraw (0,3.2) circle (.05cm);
	\draw (0,3.2) -- (0,3.6);
	\draw[double] (1,2.4) -- (1,1.2);
	\draw[double] (-1,2.4) -- (-1,1.2);
	\draw[double] (0,-1.2) -- (0,0);
	\draw (1.2,-1.9) -- (1.2,1.2);
	\draw (-1.2,-1.9) -- (-1.2,1.2);
	\draw (.6,.3) -- (.6,.9);
	\draw (-.6,.3) -- (-.6,.9);
	\roundNbox{unshaded}{(-1,2.4)}{.3}{.4}{.4}{$j_{\overline{b}\otimes a}(f)$}
	\roundNbox{unshaded}{(1,2.4)}{.3}{.4}{.4}{$j_{\overline{c}\otimes b}(g)$}
	\roundNbox{unshaded}{(-1,1.2)}{.3}{.4}{.4}{$\nu_{a\otimes \overline{b}}$}
	\roundNbox{unshaded}{(1,1.2)}{.3}{.4}{.4}{$\nu_{b\otimes \overline{c}}$}
	\roundNbox{unshaded}{(0,0)}{.35}{.5}{.5}{$\nu^{-1}_{\overline{b}\otimes b}$}
	\roundNbox{unshaded}{(0,-1.2)}{.3}{.3}{.3}{$\overline{\coev_b}$}
	\roundNbox{unshaded}{(1.2,-1.2)}{.3}{0}{0}{$\varphi_c$}
	\node at (.2,3.4) {\scriptsize{$\bfA$}};
	\node at (1.2,2.9) {\scriptsize{$\bfA$}};
	\node at (-1.2,2.9) {\scriptsize{$\bfA$}};
	\node at (1.5,1.8) {\scriptsize{$\overline{\overline{\mathbf{c}}\otimes\mathbf{b}}$}};
	\node at (-1.5,1.8) {\scriptsize{$\overline{\overline{\mathbf{b}}\otimes\mathbf{a}}$}};
	\node at (-.8,.6) {\scriptsize{$\overline{\overline{\mathbf{b}}}$}};
	\node at (.8,.6) {\scriptsize{$\overline{\mathbf{b}}$}};
	\node at (.5,-.6) {\scriptsize{$\overline{\mathbf{b}\otimes \overline{\mathbf{b}}}$}};
	\node at (1.4,-.6) {\scriptsize{$\overline{\overline{\mathbf{c}}}$}};
	\node at (-1.4,-1.2) {\scriptsize{$\overline{\mathbf{a}}$}};
	\node at (1.4,-1.8) {\scriptsize{$\mathbf{c}$}};
\end{tikzpicture}
=
\begin{tikzpicture}[baseline=.5cm]
	\draw (1,2.7) .. controls ++(90:.4cm) and ++(0:.3cm) .. (0,3.2);
	\draw (-1,2.7) .. controls ++(90:.4cm) and ++(180:.3cm) .. (0,3.2);
	\filldraw (0,3.2) circle (.05cm);
	\draw (0,3.2) -- (0,3.6);
	\draw[double] (1,2.4) -- (1,1.2);
	\draw[double] (-1,2.4) -- (-1,1.2);
	\draw (1.2,-1) -- (1.2,1.2);
	\draw (-1.2,-1) -- (-1.2,1.2);
	\draw (-.6,.9) -- (-.6,-.3) .. controls ++(270:.6cm) and ++(270:.6cm) .. (.6,-.3) -- (.6,.9);
	\roundNbox{unshaded}{(-1,2.4)}{.3}{.4}{.4}{$j_{\overline{b}\otimes a}(f)$}
	\roundNbox{unshaded}{(1,2.4)}{.3}{.4}{.4}{$j_{\overline{c}\otimes b}(g)$}
	\roundNbox{unshaded}{(-1,1.2)}{.3}{.4}{.4}{$\nu_{a\otimes \overline{b}}$}
	\roundNbox{unshaded}{(1,1.2)}{.3}{.4}{.4}{$\nu_{b\otimes \overline{c}}$}
	\roundNbox{unshaded}{(-.6,0)}{.3}{0}{0}{$\varphi_b$}
	\roundNbox{unshaded}{(1.2,0)}{.3}{0}{0}{$\varphi_c$}
	\node at (.2,3.4) {\scriptsize{$\bfA$}};
	\node at (1.2,2.9) {\scriptsize{$\bfA$}};
	\node at (-1.2,2.9) {\scriptsize{$\bfA$}};
	\node at (-1.5,1.8) {\scriptsize{$\overline{\overline{\mathbf{b}}\otimes\mathbf{a}}$}};
	\node at (1.5,1.8) {\scriptsize{$\overline{\overline{\mathbf{c}}\otimes\mathbf{b}}$}};
	\node at (-1.4,0) {\scriptsize{$\overline{\mathbf{a}}$}};
	\node at (1.4,.6) {\scriptsize{$\overline{\overline{\mathbf{c}}}$}};
	\node at (1.4,-.6) {\scriptsize{$\mathbf{c}$}};
	\node at (-.8,.6) {\scriptsize{$\overline{\overline{\mathbf{b}}}$}};
	\node at (-.8,-.5) {\scriptsize{$\mathbf{b}$}};
	\node at (.4,0) {\scriptsize{$\overline{\mathbf{b}}$}};
\end{tikzpicture}
=
f^*\circ g^*.
$$
It is similarly easy to prove $(\psi\otimes f)^* = \psi^*\otimes f^*$ for all $\psi\in \cC(c,d)$ and $f\in \cM_\bfA(a_\bfA, b_\bfA)$.

When $\theta: \bfA \Rightarrow \bfB$ is a $*$-algebra natural transformation between $*$-algebras, the induced cyclic $\cC$-module functor $\check{\theta}$ from Construction \ref{construction:Module} is a dagger functor.

For the other direction, suppose $(\cM, m)$ is a cyclic $\cC$-module dagger category.
Taking $\bfA=\bfA_m\in \Vec(\cC)$ corresponding to $(\cM,m)$ as in Construction \ref{construction:Algebra}, we get a $*$-structure on $\bfA$ as follows.
For $f\in \bfA(a)=\cM(a\otimes m, m)$, we define 
$$
\bfA(\overline{a})
\ni
\begin{tikzpicture}[baseline = -.1cm]
    \draw (0,0) -- (-.9,0);
    \draw (0,-.7) -- (0,.7);
    \roundNbox{unshaded}{(0,0)}{.3}{.25}{.25}{$j_a(f)$}
    \node at (.2,-.5) {\scriptsize{$m$}};
    \node at (.2,.5) {\scriptsize{$m$}};
    \node at (-.7,.2) {\scriptsize{$\overline{a}$}};
\end{tikzpicture}
:=
(\ev_{a}\otimes \id_{m})\circ (\id_{\overline{a}}\otimes f^{*})
=
\begin{tikzpicture}[baseline = -.1cm]
    \draw (0,0) -- (0,-.7);
    \draw (-.15,.3) arc (0:180:.15cm) -- (-.45,-.7);
    \draw (.15,.7) -- (.15,0);
    \roundNbox{unshaded}{(0,0)}{.3}{0}{0}{$f^*$}
    \node at (.4,.5) {\scriptsize{$m$}};
    \node at (.2,-.5) {\scriptsize{$m$}};
    \node at (-.6,-.5) {\scriptsize{$\overline{a}$}};
\end{tikzpicture}
\in\cM(\overline{a}\otimes m, m).
$$
It is straightforward to show $j$ satisfies the appropriate axioms.

If $\Phi: (\cM,m) \to (\cN,n)$ is a cyclic $\cC$-module dagger functor, then $\hat{\Phi}$ defined as in Construction \ref{construction:Algebra} is a $*$-algebra natural transformation.
\end{proof}

\begin{defn} 
Given a $*$-algebra $(\bfA, \mu, \iota, j)$, the \emph{base algebra} is the $*$-algebra $\bfA(\overline{1_\cC}\otimes 1_\cC)\cong \End_{\cM_\bfA}(1_\bfA)$.
(Here, the multiplication is the same as in Definition \ref{defn:NonStarBaseAlgebra}, but to define the $*$-structure, we must think of $1_\cC \cong \overline{1_\cC}\otimes 1_\cC$.)
\end{defn}

\subsection{C* and W*-algebra objects}

The equivalence between $*$-algebras objects in $\Vec(\cC)$ and $\cC$-module dagger categories afforded by Theorem \ref{thm:StarAlgebraDaggerModuleCategoryCorrespondence} gives us an elegant definition of a C*-algebra object in $\Vec(\cC)$.

\begin{defn} 
A $*$-algebra object $\bfA\in \Vec(\cC)$ is a C*-\emph{algebra object} if the $\cC$-module dagger category $\cM_\bfA$ is a C*-category.
Similarly, a $*$-algebra object is a W*-algebra object if $\cM_\bfA$ is a W*-category. 
\end{defn}

\begin{rem}
For C*-algebras, any $*$-algebra homomorphism is automatically continuous.
Given any $*$-algebra natural transformation $\theta: \bfA\Rightarrow \bfB$,
the components $\theta_{\overline{c}\otimes c}$ give $*$-algebra homomorphisms between C*-algebras $\bfA(\overline{c}\otimes c) \to \bfB(\overline{c}\otimes c)$, and are thus always bounded.
When $\bfA$ and $\bfB$ are W*-algebra objects, we should consider \emph{normal} $*$-algebra natural transformations.
This means each component $\theta_c : \bfA(c)\to \bfB(c)$ is continuous with respect to the weak*-topology coming from the identification $\bfA(c) \cong \cM_\bfA(c_\bfA, 1_\bfA)$ and $\bfB(c)\cong \cM_\bfB(c_\bfB, 1_\bfB)$.
\end{rem}

The equivalence of categories for $*$-algebra objects in Theorem \ref{thm:StarAlgebraDaggerModuleCategoryCorrespondence} gives us the following result for W*-algebra objects.

\begin{thm}
\label{thm:W*-algebraCorrespondence}
There is an equivalence of categories
$$
\left\{\,\text{\rm W*-algebra objects in $\Vec(\cC)$}\,\right\}
\,\,\cong\,\,
\left\{\,\text{\rm Cyclic $\cC$-module W*-categories}\,\right\}
$$
where the morphisms on both sides are required to be normal/weak* continuous.
\end{thm}

The following fact gives an abstract characterization of C*/W*-algebra objects which does not explicitly use cyclic $\cC$-module dagger categories.

\begin{prop}
A $*$-algebra $\bfA\in \Vec(\cC)$ is a \emph{C*}/\emph{W*}-algebra object if and only if for every $c\in\cC$, the $*$-algebra $\bfA(\overline{c}\otimes c)$ is a \emph{C*}/\emph{W*}-algebra.
Here, the multiplication and $*$-structure are those pulled back from identifying $\bfA(\overline{c}\otimes c) \cong \cM_\bfA(c_\bfA, c_\bfA)$.
\end{prop}
\begin{proof}
Recall from Section \ref{sec:Categories} that a dagger category that admits direct sums is a C*-category if and only if each endomorphism $*$-algebra is a C*-algebra. 
Note $\cM_\bfA$ admits direct sums, with $a_\bfA \oplus b_\bfA \cong (a\oplus b)_\bfA$.
Thus we need only check $\cM_\bfA(c_\bfA, c_\bfA)\cong \bfA(\overline{c}\otimes c)$ is a C*-algebra for each $c\in \cC$.
The result for W*-algebra objects is similar.
\end{proof}

It is important to note that as in Remark \ref{rem:PropertyStarAlgebraC*}, being a C*/W*-algebra object is a property of a $*$-algebra object, not extra structure.
\bigskip

We now give examples of C* and W*-algebra objects in $\Vec(\cC)$ for various rigid C*-tensor categories $\cC$.  Recall that to define a C*/W*-algebra object in $\cC$, we simply need to produce a $\cC$-module C*/W*-category, and choose a distinguished object. 

\begin{ex}[C*/W*-algebras]
Let $\cC = \fdHilb$, the category of finite dimensional Hilbert spaces.
Then $\Vec(\cC) \cong \Vec$, the category of vector spaces (not necessarly finite dimensional).
To see this, we just note that any $\bfF\in \Vec(\cC)$ is completely determined by $\bfF(1_\cC=\bbC)\in \Vec$, which is some vector space.
Hence if $H\in \fdHilb$, we have $\bfF(H)=\bfF(\bbC)\otimes_\bbC H$.

We then see that *-algebras $\bfA\in\Vec(\cC)$ are exactly $*$-algebras, C*-algebras in $\Vec(\cC)$ are just C*-algebras, and W*-algebras in $\Vec(\cC)$ are just W*-algebras.

Given a $*$-algebra $\bfA\in \Vec(\cC)$, we see that $\cM_\bfA$ has finite dimensional Hilbert spaces as objects, and the morphism space $\cM_\bfA(H,K)=\bfA(\bbC)\otimes B(H,K)$ with composition given by
$(f\otimes S) \circ (g\otimes T) = \mu(f\otimes g)\otimes  (S\circ T)$.
\end{ex}

\begin{ex}[Discrete groups]
Suppose $G$ is a discrete group.
The group algebra $\bfG=\bbC[G]$ is a connected W*-algebra object in the rigid C*-tensor category $\fdHilb(G)$ of finite dimensional $G$-graded Hilbert spaces. 
Given an outer action of $G$ on a ${\rm II}_1$ factor $N$, we get a tensor functor $\bfF: \fdHilb(G) \to \bfBim(N)$, and we may view the crossed product as a W*-algebra object in $\Vec(\Hilb_{f.d.}(G))\cong \Vec(G)$, the category of (not neccessarily finite dimensional) G-graded vector spaces.
\end{ex}

\begin{ex}
Let $\cC$ be a rigid C*-tensor category, $\cD$ be a W*-tensor category, and  $\bfF:\cC\rightarrow \cD$ be a dagger tensor functor.  
Then we can view $\cD$ as
\begin{enumerate}[(1)]
\item
A $\cC$-module category. 
Every object $m\in \cD$ gives a W*-algebra $\bfA_{\bfF,m}\in \cC$.  
\item
A $\cC\boxtimes \cC^{\text{mp}}$ module category.  
Every $m\in \cD$ gives a W*-algebra $\bfD(\bfA_{\bfF,m})$ called the \textit{quantum double} of $\bfA_{\bfF,m}$.
\end{enumerate}
Below, we mention classes of sub-examples that we feel are particularly interesting, and merit individual study.
\end{ex}

\begin{sub-ex}[Symmetric enveloping algebra]
Let $\mathbf{id}: \cC\rightarrow \cC$ be the identity functor.  
Then the algebra $\bfA_{\mathbf{id}, 1_\cC}$ is trivial, while the quantum double $\bfD(\bfA_{\mathbf{id}, 1_\cC})$ is called the \textit{symmetric enveloping algebra} \cite{MR1729488,MR1966525}.  
We will give more details on this algebra in Section \ref{sec:AnalyticProperties}.
\end{sub-ex}

\begin{sub-ex}[Quantum group]
Given a rigid C*-tensor category $\cC$ and a dagger tensor functor $\bfF:\cC\rightarrow \fdHilb$, one can apply the Tannaka-Krein-Woronowicz construction to obtain a \textit{discrete quantum group} $\bbG$, whose representation category is canonically equivalent to $\cC$, and whose forgetful functor is equivalent to the original functor $\bfF$.  
We discuss this construction in greater detail in Section \ref{sec:AnalyticProperties}.  
Let $\bbG$ be a discrete quantum group and let $\cC=\Rep(\bbG)$ be its associated rigid C*-tensor category, with forgetful functor $\bfF: \cC\rightarrow \fdHilb$.
Then $\bfG=\bfA_{\bfF.\bbC}$ is called the \textit{quantum group algebra object}.  
We call $\bfD(\bfG)$ the \textit{Drinfeld double algebra object}, which is actually Mortia equivalent to $\bfD(\bfA_{\mathbf{id},1_{\cC}})$ in an appropriate sense, though explaining this would take us too far afield for now.   
We will investigate the quantum group algebras further in Section \ref{sec:AnalyticProperties}.
\end{sub-ex}

\begin{sub-ex}[{${\rm II}_1$ factor bimodules}]
Let $\bfF:\cC\rightarrow \Bim(N)$ be a dagger tensor functor where $N$ is a ${\rm II}_{1}$ factor. 
If $\bfF$ is full, then choosing $m$ to be $L^{2}(N)$ yields the trivial algebra: $\bfA_{\bfF,L^2(N)}\cong \bfA_{\mathbf{id},1_{\cC}}$.  
In general, choosing an bifinite bimodule $H\in \bfBim(N)$ necessarily yields a locally finite algebra, and if $H$ is irreducible and $\bfF$ is full, $\bfA_{\bfF,H}$ will be connected.  
\end{sub-ex}

\begin{ex}[Quasi-regular subfactors]
An irredicible inclusion of ${\rm II}_{1}$ factors $N\subseteq M$ is called \textit{quasi-regular} if the $N-N$ bimodule ${}_{N} L^{2}(M)_{N}$ decomposes as $L^2(N)\oplus \bigoplus_{i\geq 1} n_i H_i$ where for each $i\in \bbN$, $n_i\in \bbN$ and $H_i \in \bfBim(N)$ is bifinite.
(See \cite{MR1622812,MR1729488,1511.07329} for more details.)
We define $\bfM \in \Vec(\bfBim(N))$ by $\bfM(H):=\Hom_{N-N}(H, {}_N L^{2}(M)_{N})$ for $H\in \bfBim(N)$.
Since $M$ is a von Neumann algebra, $\bfM$ is a W*-algebra object, and since $N\subseteq M$ is irreducible, $\bfM$ is connected.
\end{ex}

\begin{ex}[Planar algebras and graphs]
Consider $\cT\cL\cJ_\delta$, the Temperley-Lieb-Jones category with loop parameter $\delta$.
Recall that $\Ad(\cT\cL\cJ_\delta) \cong \Ad(\Rep(SU_{q}(2)))$ for $\delta=q+q^{-1}$, where $\Ad$ denotes taking the adjoint subcategory, a.k.a., the even half.

A planar algebra $\cP_\bullet$ is an algebra object in $\Vec(\Ad(\cT\cL\cJ_\delta))$ which lifts to a commutative algebra in the center $\cZ(\Vec(\Ad(\cT\cL\cJ_\delta)))$.
The corresponding cyclic module category for $\cP_\bullet$ is the projection category ${\sf Proj}(\cP_\bullet)$ from \cite{MR2559686,MR3405915} with basepoint the empty diagram.

This is really a special case of De Commer and Yamashita's classification of module categories of $\Rep(SU_{q}(2))$ by weighted graphs with balanced cost functions \cite{MR3420332}.  
A planar algebra with its principal graph and dimension function gives such a module category when we restrict to the adjoint subcategory $\Ad(\Rep(SU_{q}(2)))$.
\end{ex}

\begin{ex}[{$\bfB(\bfH)$}]
\label{ex:BofH}
Note that $\Hilb(\cC)$ is a $\cC$-module W*-category, since $\cC$ embeds in $\Hilb(\cC)$.
Given $\bfH\in \Hilb(\cC)$, take $(\cM_\bfH, \bfH)$ to be the cyclic left  $\cC$-module W*-category generated by $\bfH$.
We define $\bfB(\bfH)$ to be the algebra in $\Vec(\cC)$ obtained from Construction \ref{construction:Algebra} applied to $(\cM_\bfH,\bfH)$.
Notice that by Theorem \ref{thm:AlgebraModuleCategoryCorrespondence}, 
$$
\Hom_{\cM_{\bfB(\bfH)}}(a, b) \cong \bfB(\bfH)(\overline{b}\otimes a)
=
\Hom_{\Hilb(\cC)}(\overline{\mathbf{b}}\otimes \mathbf{a}\otimes \bfH, \bfH)
\cong
\Hom_{\Hilb(\cC)}(\mathbf{a}\otimes \bfH, \mathbf{b}\otimes\bfH).
$$
Similarly, we define ${}_\bfH\cM$ to be the cyclic right $\cC$-module W*-category generated by $\bfH$. 
Note that $\cM_\bfH \cong \cM_{\bfB(\bfH)}$, but in general,  ${}_{\bfB(\bfH)}\cM\ncong {}_\bfH\cM$, as we see from the following proposition.
\end{ex}

\begin{prop}
We have the following correspondence between algebras in $\Vec(\cC)$ and left and right $\cC$-module \emph{W*}-categories:
\begin{align*}
\bfB(\bfH) &\longleftrightarrow \cM_\bfH
&
\overline{\bfB(\overline{\bfH})} &\longleftrightarrow {}_\bfH\cM
\\
\bfB(\overline{\bfH}) &\longleftrightarrow \cM_{\overline{\bfH}}
&
\overline{\bfB(\bfH)} &\longleftrightarrow {}_{\overline{\bfH}}\cM
\end{align*}
\end{prop}

Note that $\overline{\bfB(\bfH)} \cong \bfB(\bfH)$ as algebras in $\Vec(\cC)$, but is helpful to think of these two algebras as four algebras to see the above correspondences more easily.

\begin{proof}
It suffices to prove ${}_{\overline{\bfH}} \cM$ corresponds to $\overline{\bfB(\bfH)}$.
Let ${}_{\overline{\bfH}}\bfA\in \Vec(\cC)$ be the algebra corresponding to ${}_{\overline{\bfH}}\cM$.
We see we have a $*$-algebra natural transformation $\theta:\overline{\bfB(\bfH)}\Rightarrow {}_{\overline{\bfH}}\bfA$, where for $c\in \cC$, $\theta_c$ is the composite map
\begin{align*}
\overline{\bfB(\bfH)}(c) 
&= 
\overline{\bfB(\bfH)(\overline{c})} 
=
\overline{\Hom_{\Vec(\cC)}(\overline{\mathbf{c}} \otimes \bfH, \bfH) }
=
\Hom_{\Vec(\cC)}(\overline{\overline{\mathbf{c}} \otimes \bfH}, \overline{\bfH})
\\&\overset{\nu\circ -}{\cong}
\Hom_{\Vec(\cC)}(\overline{\bfH}\otimes\overline{\overline{\mathbf{c}}}, \overline{\bfH})
\overset{(\id\otimes \varphi)\circ -}{\cong}
\Hom_{\Vec(\cC)}(\overline{\bfH}\otimes\mathbf{c}, \overline{\bfH})
=
{}_{\overline{\bfH}}\bfA(c).
\end{align*}
We leave the rest of the details to the reader.
\end{proof}


\section{C*-algebra objects in \texorpdfstring{$\Vec(\cC)$}{Vec(C)}}
\label{sec:CStarAlgebras}

We now focus our attention on C*-algebra objects in $\Vec(\cC)$.
We prove analogues of theorems for ordinary C*-algebras in $\fdHilb$.

From this point on, whenever possible, we suppress the involutive structure $(\nu, \varphi, r)$ on $\Vec(\cC)$ and $\Hilb(\cC)$ to ease the notation.
As a consequence, we will consider $\bfA(1_\cC)$ as the base algebra instead of $\bfA(\overline{1_\cC}\otimes 1_\cC)$.

\subsection{Conditional expectations}

Let $a\in\cC$ and consider a $\cC$-module C*-category $\cM$. 
For every $m\in\cM$, the map  $\iota_{a}:\End_{\cM}(m)\rightarrow \End_\cM(a\otimes m)$ given by $\iota_{a}(f)=\id_{a}\otimes f$ is a unital C*-algebra isometry.
The standard left inverse for $m\in \cM$, denoted $E_{a}: \End_{\cM}(a\otimes m)\rightarrow \End_{\cM}(m)$, is given by 
\begin{equation}
\label{eq:ConditionalExpectation}
E_{a}(f)
=
\frac{1}{d_a}\,
\begin{tikzpicture}[baseline = -.1cm]
    \draw (.15,-.7) -- (.15,.7);
    \draw (-.15,.3) arc (0:180:.15cm) -- (-.45,-.3) arc (-180:0:.15cm);
    \roundNbox{unshaded}{(0,0)}{.3}{0}{0}{$f$}
    \node at (.4,-.5) {\scriptsize{$m$}};
    \node at (.4,.5) {\scriptsize{$m$}};
    \node at (-.6,0) {\scriptsize{$\overline{a}$}};
\end{tikzpicture}
\end{equation}
An easy calculation shows $E_{a}\circ {\iota_{a}}=\id_m$.  
The map $E_a$ is also called a \emph{partial trace} or a \emph{conditional expectation}.

We  make use of the following result, due to Longo-Roberts in the abstract setting \cite{MR1444286}, which is a version of the Pimsner-Popa inequality for subfactors \cite{MR860811}.
We give an easy diagrammatic proof for the convenience of the reader.

\begin{lem}
\label{lem:PimsnerPopa}
If $a\in\cC$ and $m\in \cM$, then
$f\leq d_a^2 (\id_a\otimes E_a(f))$
for all $f\in \End_\cM(a\otimes m)^+$.
\end{lem}
\begin{proof}
For all $f\geq 0$, since the tensor product of positive operators is positive,
\begin{align*}
\begin{tikzpicture}[baseline = -.1cm]
    \draw (.15,-.7) -- (.15,.7);
    \draw (-.15,-.7) -- (-.15,.7);
    \roundNbox{unshaded}{(0,0)}{.3}{0}{0}{$f$}
    \node at (.4,-.5) {\scriptsize{$m$}};
    \node at (.4,.5) {\scriptsize{$m$}};
    \node at (-.4,-.5) {\scriptsize{$a$}};
    \node at (-.4,.5) {\scriptsize{$a$}};
\end{tikzpicture}
&=
\begin{tikzpicture}[baseline=-.1cm]
    \draw (.15,-.9) -- (.15,.9);
    \draw (-.15,-.3) -- (-.15,-.6) arc (0:-180:.15cm) -- (-.45,-.3) arc (0:180:.15cm) -- (-.75,-.9);
    \draw (-.15,.3) -- (-.15,.6) arc (0:180:.15cm) -- (-.45,.3) arc (0:-180:.15cm) -- (-.75,.9);
    \roundNbox{unshaded}{(0,0)}{.3}{0}{0}{$f$}
    \roundNbox{dotted}{(-.3,0)}{.5}{.2}{.3}{}
    \node at (.4,-.8) {\scriptsize{$m$}};
    \node at (.4,.8) {\scriptsize{$m$}};
    \node at (-.9,-.8) {\scriptsize{$a$}};
    \node at (-.9,.8) {\scriptsize{$a$}};
\end{tikzpicture}
\leq
d_a\,
\begin{tikzpicture}[baseline=-.1cm]
    \draw (.15,-.9) -- (.15,.9);
    \draw (-.15,-.3) -- (-.15,-.6) arc (0:-180:.15cm) --(-.45,.6) arc (180:0:.15cm) -- (-.15,.3);
    \draw (-.75,-.9) -- (-.75,.9);
    \roundNbox{unshaded}{(0,0)}{.3}{0}{0}{$f$}
    \roundNbox{dotted}{(-.3,0)}{.5}{.2}{.3}{}
    \node at (.4,-.8) {\scriptsize{$m$}};
    \node at (.4,.8) {\scriptsize{$m$}};
    \node at (-.9,-.8) {\scriptsize{$a$}};
    \node at (-.9,.8) {\scriptsize{$a$}};
\end{tikzpicture}
=
d_a^2 (\id_a\otimes E_a(f)).
\qedhere
\end{align*}
\end{proof}

\begin{cor}
\label{cor:NormInequalities}
For all $f\in \End_\cM(a\otimes m)^+$, we have
$$
d_a^{-2}\|f\| \leq  \| E_a(f)\| \leq \|f\|.
$$
In particular, $E_a$ is faithful.
\end{cor}
\begin{proof}
The first inequality follows from Lemma \ref{lem:PimsnerPopa} by dividing by $d_a^2$.
The second inequality follows from the fact that $\id_a\otimes E_a$ is positive, together with $f\leq \|f\|(\id_a\otimes \id_m)$.
\end{proof}

We now discuss the probabilistic index of a conditional expectation due to \cite{MR860811}.
For a broader discussion of index of inclusions, conditional expectations, and dualizability, see \cite{MR3342166}.

\begin{defn}  
Suppose $A\subseteq B$ is a unital inclusion of C*-algebras. 
A conditional expectation $E:B\rightarrow A\subseteq B$ is said to have \emph{finite index} if there exists a number $K\ge1$ such that $K\cdot E- 1_{B}$ is a positive map.   
The \textit{index} $\ind(E)$ of $E$ is the infimum of such $K$.
\end{defn}

Suppose we have a $\cC$-module C*-category $\cM$.
Lemma \ref{lem:PimsnerPopa} above shows that the conditional expectation $E_{a}: \End_\cM(a\otimes m)\rightarrow \End_\cM(m)$ has finite index, and in particular $1\le \ind(E_{a})\le d^{2}_{a}$.  
Using the following proposition, this inequality has particularly nice consequences when $\End_\cM(m)$ is finite dimensional.

\begin{prop}[{\cite[Cor.\,4.4]{MR1642530}}]
If $E:B\rightarrow A\subseteq B$ is a faithful conditional expectation and $A$ is finite dimensional, then $B$ is finite dimensional if and only if $\ind(E)<\infty$.
\end{prop}

\begin{cor} 
Let $\cM$ be a $\cC$-module \emph{C*}-category.
Suppose we have $m\in \cM$ such that $\dim(\End_\cM(m))<\infty$.
Then $\End_\cM(a\otimes m)$ is a finite dimensional \emph{C*}-algebra for all $a\in \cC$.
In particular, the cyclic $\cC$-module \emph{C*}-category generated by $m$ is pre-semi-simple.\footnote{Here, pre-semi-simple is used in the sense of \cite{Enriched}, meaning that all the endomorphism algebras are finite dimensional semi-simple algebras.}
\end{cor}

\begin{cor}\label{cor:FiniteDimension} 
Suppose $\bfA\in \Vec(\cC)$ is a \emph{C*}-algebra object and $\dim(\bfA(1_{\cC}))<\infty$.
Then $\dim(\bfA(a))<\infty$ for all $a\in \cC$.
\end{cor}

\begin{defn}
A C*-algebra object $\bfA\in\Vec(\cC)$ is called \emph{locally finite} if either of the following equivalent conditions hold.
\begin{enumerate}[(1)]
\item
$\dim(\bfA(c))<\infty$ for all $c\in \cC$.
\item
$\dim(\bfA(1_{\cC}))<\infty$
\end{enumerate}
\end{defn}

Locally finite algebra objects bear many similarities to finite dimensional algebras.  
If $\cC=\fdHilb$, then locally finite algebras are just finite dimensional C*-algebras.
Moreover, each locally finite C*-algebra object is also a W*-algebra object, which generalizes the fact that finite dimensional C*-algebras are also W*-algebras.

Among locally finite algebras are a distinguished class: the so-called connected algebras.

\begin{defn} 
An algebra object $\bfA\in \Vec(\cC)$ is \emph{connected} if $\dim( \bfA(1_{\cC}))=1$.
\end{defn}

Note that every connected C*/W*-algebra has a unique state.

\begin{rem}
Given a C*/W*-algebra object $\bfA\in \Vec(\cC)$, there is a \emph{forgetful} functor which maps $\bfA$ to the C*/W*-algebra $\bfA(1_\cC)$.
This functor has a left adjoint; the \emph{free} functor maps a C*/W*-algebra $A$ to the C*/W*-algebra object in $\Vec(\cC)$ which has $\bfA(1_\cC) = A$, and $\bfA(c) = (0)$ for all $c\in \Irr(\cC)$ with $c\ncong 1_\cC$.

These free C*-algebra objects contain the least categorical content, since they have nothing to do with the category $\cC$.
On the other hand, connected algebras only contain categorical content, since looking at the fiber over $1_\cC$ tells us nothing about the algebra.
The unit fiber is trivial, while globally, the structure can be highly non-trivial.
\end{rem}

The following facts are straightforward.

\begin{facts}
Under the correspondence between \emph{W*}-algebra objects and cyclic $\cC$-module \emph{W*}-categories from Theorem \ref{thm:StarAlgebraDaggerModuleCategoryCorrespondence}, we have the following correspondences.
\begin{itemize}
\item
Locally finite algebras correspond to pre-semi-simple cyclic $\cC$-module \emph{W*}-categories.
\item
Connected algebra objects correspond to cyclic $\cC$-module \emph{W*}-categories with irreducible basepoint.
\end{itemize}
\end{facts}

\subsection{Operator valued inner products and equivalence of norms}
\label{sec:OperatorValuedInnerProducts}

We now use our previous discussion on conditional expectations in the special case of the cyclic $\cC$-module C*-category $\cM_\bfA$ associated to a C*-algebra object $\bfA\in \Vec(\cC)$.

\begin{defn}
Given a C*-algebra object $\bfA\in \Vec(\cC)$, we have a left $\bfA(1_\cC)$-linear inner product
${}_a\langle\,\cdot\, ,\,\cdot\,\rangle: \bfA(a)\times \bfA(a)\rightarrow \bfA(1_\cC)$ 
and a right $\bfA(1_\cC)$-linear inner product
$\langle\,\cdot\, |\,\cdot\,\rangle_{a}: \bfA(a)\times \bfA(a)\rightarrow \bfA(1_\cC)$ defined by
$$
{}_a\langle f, g\rangle
=
\begin{tikzpicture}[baseline = -.1cm]
    \draw (-.5,-.3) arc (-180:0:.5cm);
    \draw (-.5,.3) arc (180:0:.5cm);
    \filldraw (0,.8) circle (.05cm);
    \draw (0,.8) -- (0,1.2);
    \roundNbox{unshaded}{(-.5,0)}{.3}{0}{0}{$f$}
    \roundNbox{unshaded}{(.5,0)}{.3}{.2}{.2}{$j_a(g)$}
    \node at (-.65,-.5) {\scriptsize{$\mathbf{a}$}};
    \node at (-.65,.5) {\scriptsize{$\bfA$}};
    \node at (.65,-.5) {\scriptsize{$\overline{\mathbf{a}}$}};
    \node at (.65,.5) {\scriptsize{$\bfA$}};
    \node at (.15,1) {\scriptsize{$\bfA$}};
\end{tikzpicture}
\qquad\qquad
\langle f| g\rangle_{a}
=
\begin{tikzpicture}[baseline = -.1cm, xscale=-1]
    \draw (-.5,-.3) arc (-180:0:.5cm);
    \draw (-.5,.3) arc (180:0:.5cm);
    \filldraw (0,.8) circle (.05cm);
    \draw (0,.8) -- (0,1.2);
    \roundNbox{unshaded}{(-.5,0)}{.3}{0}{0}{$g$}
    \roundNbox{unshaded}{(.5,0)}{.3}{.2}{.2}{$j_a(f)$}
    \node at (-.65,-.5) {\scriptsize{$\mathbf{a}$}};
    \node at (-.65,.5) {\scriptsize{$\bfA$}};
    \node at (.65,-.5) {\scriptsize{$\overline{\mathbf{a}}$}};
    \node at (.65,.5) {\scriptsize{$\bfA$}};
    \node at (.15,1) {\scriptsize{$\bfA$}};
\end{tikzpicture}
$$
Notice that for $f\in \bfA(1_\cC)$ and $g,h\in \bfA(a)$, 
${}_a\langle \mu_{1_\cC,a}(f \otimes g), h\rangle= \mu_{1_\cC, 1_\cC}(f\otimes {}_a\langle  g, h\rangle)$ 
and similarly 
$\langle g| \mu_{a, 1_\cC}(h\otimes f)\rangle_{a}=\mu_{1_\cC, 1_\cC}(\langle g| h\rangle_{a} \otimes f)$.
These forms are anti-symmetric, and they are positive definite by Corollary \ref{cor:NormInequalities}.
\end{defn}

\begin{lem}
\label{lem:CompletelyPositive}
Suppose $\bfA\in \Vec(\cC)$ is a \emph{C*}-algebra object.
The right $\bfA(1_\cC)$-valued inner product is completely positive in the sense that for all $f_1,\dots, f_n\in \bfA(a)$, the matrix $(\langle f_i | f_j\rangle_{a})_{i,j}$ is positive in $M_n(\bfA(1_\cC))$.
A similar result holds for the left $\bfA(1_\cC)$-valued inner product.
\end{lem}
\begin{proof}
Let $\cM_{\bf{A}}$ be the cyclic $\cC$-module C*-category from Construction \ref{construction:Module}.  
The conditional expectation 
$
E_a:\bfA(\overline{a}\otimes a)\cong \End_{\cM_{\bfA}}(a_\bfA)\to \End_{\cM_\bfA}(1_\bfA)\cong  \bfA(1_{\cC})
$
from \eqref{eq:ConditionalExpectation} is given by 
$$
E_{a}(f)
=
\frac{1}{d_a}\,
\begin{tikzpicture}[baseline = -.1cm]
    \draw (0,0) -- (.7,0);
    \draw (0,.3) arc (0:180:.3cm) -- (-.6,-.3) arc (-180:0:.3cm);
    \roundNbox{unshaded}{(0,0)}{.3}{0}{0}{$f$}
    \node at (.15,.5) {\scriptsize{$\mathbf{a}$}};
    \node at (.5,.15) {\scriptsize{$\bfA$}};
    \node at (.15,-.5) {\scriptsize{$\mathbf{a}$}};
\end{tikzpicture}\,.
$$
The right $\bfA(1_\cC)$-valued inner product is exactly given by $\langle f|g\rangle_a = d_a E_a(\mu_{\overline{a},a}(j_{a}(f)\otimes g))$, and $E_a$ is completely positive.
\end{proof}

\begin{lem}
\label{lem:PositiveOperator}
Suppose $f_1,\dots, f_n\in \bfA(a)$ and $\psi\in \cC(b,a)$. 
As operators in $M_n(\bfA(1_\cC))$,
$$
0
\leq
\bigg(
\langle
\bfA(\psi)(f_k)|\bfA(\psi)(f_i)
\rangle_a
\bigg)_{i,k}
\leq
\| \psi\|_\cC^2 
\bigg(
\langle f_k|f_i \rangle_a
\bigg)_{i,k}.
$$
\end{lem}
\begin{proof}
Using the module category as in Lemma \ref{lem:CompletelyPositive}, we see that
\begin{align*}
\bigg(
\langle 
\bfA(\psi)(f_k)|\bfA(\psi)(f_i)
\rangle_a
\bigg)_{i,k}
&=
\left(\,
\begin{tikzpicture}[baseline = -.1cm]
    \draw (-.8,-1.8) arc (0:-180:.3cm) -- (-1.4,1.8) arc (180:0:.3cm);
    \draw (.6,.5) arc (90:-90:.5cm) -- (.3,-.5);
    \filldraw (1.1,0) circle (.05cm);
    \draw (1.1,0) -- (1.5,0);
    \draw (-.3,-.5) arc (90:180:.5cm) -- (-.8,-1.2);
    \draw (-.5,.5) arc (270:180:.3cm) -- (-.8,1.2);
    \roundNbox{unshaded}{(0,.5)}{.3}{.2}{.3}{$j_a(f_k)$}
    \roundNbox{unshaded}{(0,-.5)}{.3}{0}{0}{$f_i$}
    \roundNbox{unshaded}{(-.8,-1.5)}{.3}{0}{0}{$\psi$}
    \roundNbox{unshaded}{(-.8,1.5)}{.3}{0}{0}{$\psi^*$}
    \node at (-.95,-1) {\scriptsize{$\mathbf{a}$}};
    \node at (-.95,1) {\scriptsize{$\mathbf{a}$}};
    \node at (-.65,-2) {\scriptsize{$\mathbf{b}$}};
    \node at (-.65,2) {\scriptsize{$\mathbf{b}$}};
    \node at (1.3,.15) {\scriptsize{$\bfA$}};
    \node at (.8,.65) {\scriptsize{$\bfA$}};
    \node at (.5,-.65) {\scriptsize{$\bfA$}};
\end{tikzpicture}
\right)_{i,j}
\\&=
\left(
\begin{tikzpicture}[baseline = -.1cm]
    \draw (-.8,-.8) -- (-.8,-1) arc (0:-180:.3cm) -- (-1.4,1) arc (180:0:.3cm) -- (-.8,.8);
    \draw (.6,.5) arc (90:-90:.5cm) -- (.3,-.5);
    \filldraw (1.1,0) circle (.05cm);
    \draw (1.1,0) -- (1.5,0);
    \draw (-.3,-.5) -- (-.5,-.5) arc (90:180:.3cm) -- (-.8,-.8);
    \draw (-.5,.5) arc (270:180:.3cm) -- (-.8,.8);
    \roundNbox{unshaded}{(0,.5)}{.3}{.2}{.3}{$j_a(f_k)$}
    \roundNbox{unshaded}{(0,-.5)}{.3}{0}{0}{$f_i$}
    \roundNbox{unshaded}{(-1.4,-.5)}{.3}{0}{0}{$\psi^\vee$}
    \roundNbox{unshaded}{(-1.4,.5)}{.3}{.2}{.2}{$(\psi^*)^\vee$}
    \node at (-.65,-1) {\scriptsize{$\mathbf{a}$}};
    \node at (-.65,1) {\scriptsize{$\mathbf{a}$}};
    \node at (-1.55,-1) {\scriptsize{$\overline{\mathbf{a}}$}};
    \node at (-1.55,1) {\scriptsize{$\overline{\mathbf{a}}$}};
    \node at (-1.55,0) {\scriptsize{$\overline{\mathbf{b}}$}};
    \node at (1.3,.15) {\scriptsize{$\bfA$}};
    \node at (.8,.65) {\scriptsize{$\bfA$}};
    \node at (.5,-.65) {\scriptsize{$\bfA$}};
\end{tikzpicture}
\right)_{i,k}
\leq
\|\psi^\vee\|_\cC^2
\bigg(
\langle f_k|f_i \rangle_a
\bigg)_{i,k}.
\end{align*}
Now we use that fact that $\psi\mapsto \psi^\vee$ is an isometry.
\end{proof}

\begin{rem}
A priori, we get two different norms on each $\bfA(a)$ from ${}_a\langle \,\cdot\,,\,\cdot\,\rangle$ and $\langle \,\cdot\,|\,\cdot\,\rangle_a$ by composing with $\|\,\cdot\,\|_{\bfA(1_\cC)}$, and $\bfA(a)$ is complete with respect to both of them.

Suppose $\bfA\in \Vec(\cC)$ is a C*-algebra object and $c\in\cC$.
An element $f\in\bfA(c)$ can be embedded into $\cM_\bfA$ in many ways.
Any time we write $c\cong\overline{b}\otimes a$ for $a,b\in\cC$, we get a vector space isomorphism $\bfA(c)\cong \bfA(\overline{b}\otimes a)$.
We get many norms on $\bfA(c)$ from these identifications.
\end{rem}

\begin{defn}
When we can write $c\cong \overline{b}\otimes a$, we define  $\|\,\cdot\,\|_{a,b}$ on $\bfA(c)$ by $\|\,\cdot\,\|_{\cM_\bfA(a_\bfA,b_\bfA)}$.
\end{defn}

All these norms induce the same topology on $\bfA(c)$.

\begin{prop}
Suppose $c\cong \overline{b}\otimes a$.  
Then $d_b^{-1} \|f\|_{a,b}\le \|f\|_{c,1_\cC}\le d_{b} \|f\|_{a,b}$.
\end{prop}
\begin{proof}
This follows directly from Corollary \ref{cor:NormInequalities} by drawing diagrams.
\end{proof}

\subsection{The Gelfand-Naimark theorem in \texorpdfstring{$\cC$}{C}}

First, we prove the analog of the Gelfand-Naimark theorem: every C*-algebra is a norm closed C*-subalgebra of $B(H)$.
To do so, we need to discuss what an embedding of C*-algebra objects should be.

\begin{defn}
An embedding of $\cC$-module C*-categories $\cM \to \cN$ is a faithful $\cC$-module $*$-functor $\Phi:\cM \to \cN$ which is norm closed on the level of hom spaces.
An embedding of C*-algebra objects $\theta:\bfA\to \bfB$ is a C*-algebra natural transformation $\theta: \bfA\Rightarrow \bfB$ such that $\check{\theta}: \cM_\bfA \to \cM_\bfB$ is an embedding of $\cC$-module C*-categories.
\end{defn}

\begin{thm}[Gelfand-Naimark in $\cC$]
\label{thm:GelfandNaimark}
A \emph{C*}-algebra object $\bfA\in \Vec(\cC)$ has an embedding into  $\mathbf{B(H)}$ for some object $\bfH\in \Hilb(\cC)$.
\end{thm}
\begin{proof}
Let $(\pi, K)$ be a faithful, non-degenrate representation of the C*-algebra $\bfA(1_\cC)$.
(Such a representation exists by the usual Gelfand-Naimark theorem.)
We'll use Definition \ref{defn:EasyConstructionOfHilbertSpaceObject} to define a Hilbert space object $\bfH\in \Hilb(\cC)$.
First, we define the vector space object $\bfV\in \Vec(\cC)$ by $\bfV(a)=\bfA(a)\otimes_\bbC K$ for $a\in\cC$.  
We define a sesquilinear form on each $\bfV(a)$ by 
\begin{equation}
\label{eq:SesquilinearFormForGN}
\langle f\otimes \xi, g\otimes \eta\rangle_{\bfV(a)} 
=
\langle \pi(\langle g|f  \rangle_{a}) \xi, \eta\rangle_{K}.
\end{equation}
It is clear that $\langle \,\cdot\,, \,\cdot\,\rangle_{\bfV(a)}$ is positive by Lemma \ref{lem:CompletelyPositive} and anti-symmetric.
We may now take the quotient by the length zero vectors and complete to obtain a Hilbert space $\bfH(a)$.
By Definition \ref{defn:EasyConstructionOfHilbertSpaceObject}, we get a canonical Hilbert space object $\bfH\in \Hilb(\cC)$ by only considering simple objects $a\in \Irr(\cC)$.
(Note that for arbitrary $a\in \cC$, $\bfH(a)$ from Definition \ref{defn:EasyConstructionOfHilbertSpaceObject} is isomorphic to the completion of $\bfV(a)$ with the sesquilinear form from \eqref{eq:SesquilinearFormForGN}.)

We now construct an embedding $\bfA \to \bfB(\bfH)$.
To do so, we construct a cyclic $\cC$-module dagger functor $\Phi:(\cM_\bfA,1_\bfA)\to (\cM_\bfH, \bfH)$.
For each $c\in \cC$, we define $\Phi(c_\bfA) = \mathbf{c}\otimes \bfH$.
Given a map $f\in \cM_\bfA(a,b)\cong \bfA(\overline{b}\otimes a)$, we construct a bounded natural transformation $\Phi(f): \mathbf{a}\otimes \bfH \Rightarrow \mathbf{b}\otimes \bfH$ as follows.

Suppose $c\in \cC$.
First, using the diagrammatic calculus for simple tensors \eqref{eq:SimpleTensors}, we see
$$
\spann\set{
\begin{tikzpicture}[baseline = -.4cm]
    \draw (0,.6) -- (0,-.3);
    \draw (1,0) -- (1,.6);
    \draw (.5,-.8) -- (.5,-1.2);
    \draw (0,-.3) arc (-180:0:.5cm);
    \filldraw[fill=white] (.5,-.8) circle (.05cm) node [above] {\scriptsize{$\alpha$}};
    \roundNbox{unshaded}{(1,0)}{.3}{0}{0}{$g$}
    \node at (-.15,-.5) {\scriptsize{$\mathbf{a}$}};
    \node at (1.15,-.5) {\scriptsize{$\mathbf{d}$}};
    \node at (.35,-1) {\scriptsize{$\mathbf{c}$}};
    \node at (1.2,.5) {\scriptsize{$\bfA$}};
\end{tikzpicture}
\otimes \xi
}{
d\in \Irr(\cC),\,\,
g\in \bfA(d),\,\,
\xi\in K,
\text{ and }
\alpha\in \cC(c, a\otimes d)
}
$$
is dense in $(\mathbf{a}\otimes \bfH)(c)$.
The inner product of two such vectors is given by
\begin{equation}
\label{eq:GelfandNaimarkInnerProduct}
\left\langle
\begin{tikzpicture}[baseline = -.4cm]
    \draw (0,.6) -- (0,-.3);
    \draw (1,0) -- (1,.6);
    \draw (.5,-.8) -- (.5,-1.2);
    \draw (0,-.3) arc (-180:0:.5cm);
    \filldraw[fill=white] (.5,-.8) circle (.05cm) node [above] {\scriptsize{$\alpha$}};
    \roundNbox{unshaded}{(1,0)}{.3}{0}{0}{$g$}
    \node at (-.15,-.5) {\scriptsize{$\mathbf{a}$}};
    \node at (1.15,-.5) {\scriptsize{$\mathbf{d}$}};
    \node at (.35,-1) {\scriptsize{$\mathbf{c}$}};
    \node at (1.2,.5) {\scriptsize{$\bfA$}};
\end{tikzpicture}
\otimes \xi
,
\begin{tikzpicture}[baseline = -.4cm]
    \draw (0,.6) -- (0,-.3);
    \draw (1,0) -- (1,.6);
    \draw (.5,-.8) -- (.5,-1.2);
    \draw (0,-.3) arc (-180:0:.5cm);
    \filldraw[fill=white] (.5,-.8) circle (.05cm) node [above] {\scriptsize{$\beta$}};
    \roundNbox{unshaded}{(1,0)}{.3}{0}{0}{$h$}
    \node at (-.15,-.5) {\scriptsize{$\mathbf{a}$}};
    \node at (1.15,-.5) {\scriptsize{$\mathbf{e}$}};
    \node at (.35,-1) {\scriptsize{$\mathbf{c}$}};
    \node at (1.2,.5) {\scriptsize{$\bfA$}};
\end{tikzpicture}
\otimes \eta
\right\rangle_{(\textbf{a}\otimes\bfH)(c)}
=
\delta_{d=e}
\left\langle
\pi\left(
\begin{tikzpicture}[baseline = -.1cm]
    \draw (.5,.5) arc (90:-90:.5cm) -- (.3,-.5);
    \draw (-1,.5) -- (0,.5);
    \draw (-1,-.5) -- (0,-.5);
    \draw (-1,-.5) -- (-1,.5);
    \draw (-1,.5) arc (90:270:.5cm);
    \draw (1,0) -- (1.5,0);
    \filldraw[fill=white] (-1,-.5) circle (.05cm);
    \filldraw[fill=white] (-1,.5) circle (.05cm);
    \filldraw(1,0) circle (.05cm);
    \roundNbox{unshaded}{(0,.5)}{.3}{.2}{.2}{$j_d(h)$}
    \roundNbox{unshaded}{(0,-.5)}{.3}{0}{0}{$g$}
    \node at (-.8,0) {\scriptsize{$\overline{\mathbf{a}}$}};
    \node at (-1.7,0) {\scriptsize{$\overline{\mathbf{c}}$}};
    \node at (-.75,-.7) {\scriptsize{$\mathbf{d}$}};
    \node at (-.75,.7) {\scriptsize{$\mathbf{d}$}};
    \node at (1.25,.15) {\scriptsize{$\bfA$}};
    \node at (.7,.65) {\scriptsize{$\bfA$}};
    \node at (.5,-.65) {\scriptsize{$\bfA$}};
\end{tikzpicture}
\right)
\xi,
\eta 
\right\rangle_K.
\end{equation}
Now we define our map $\Phi$ on $f\in \cM_\bfA(a_\bfA, b_\bfA)$ by 
$$
\Phi(f)\left(
\begin{tikzpicture}[baseline = -.4cm]
    \draw (0,.6) -- (0,-.3);
    \draw (1,0) -- (1,.6);
    \draw (.5,-.8) -- (.5,-1.2);
    \draw (0,-.3) arc (-180:0:.5cm);
    \filldraw[fill=white] (.5,-.8) circle (.05cm) node [above] {\scriptsize{$\alpha$}};
    \roundNbox{unshaded}{(1,0)}{.3}{0}{0}{$h$}
    \node at (-.15,-.5) {\scriptsize{$\mathbf{a}$}};
    \node at (1.15,-.5) {\scriptsize{$\mathbf{d}$}};
    \node at (.35,-1) {\scriptsize{$\mathbf{c}$}};
    \node at (1.2,.5) {\scriptsize{$\bfA$}};
\end{tikzpicture}
\otimes \xi
\right)
=
\begin{tikzpicture}[baseline = -.1cm]
    \draw (-.15,.3) -- (-.15,1.3);
    \draw (.15,.3) arc (180:0:.425cm);
    \draw (.575,.725) -- (.575,1.3);
    \draw (.5,-.8) -- (.5,-1.2);
    \draw (0,-.3) arc (-180:0:.5cm);
    \filldraw (.575,.725) circle (.05cm);
    \filldraw[fill=white] (.5,-.8) circle (.05cm) node [above] {\scriptsize{$\alpha$}};
    \roundNbox{unshaded}{(0,0)}{.3}{0}{0}{$f$}
    \roundNbox{unshaded}{(1,0)}{.3}{0}{0}{$h$}
    \node at (-.35,1.1) {\scriptsize{$\mathbf{b}$}};
    \node at (-.15,-.5) {\scriptsize{$\mathbf{a}$}};
    \node at (1.15,-.5) {\scriptsize{$\mathbf{d}$}};
    \node at (.35,-1) {\scriptsize{$\mathbf{c}$}};
    \node at (1.2,.5) {\scriptsize{$\bfA$}};
    \node at (.375,1.1) {\scriptsize{$\bfA$}};
    \node at (.1,.725) {\scriptsize{$\bfA$}};
\end{tikzpicture}
\otimes \xi\,.
$$
We see that this map is well-defined and bounded by the following calculation, where we suppress some labels to ease the notation:
\begin{equation}
\label{eq:LeftActionBounded}
\begin{split}
\left\|
\Phi(f)\left(
\sum_{k=1}^n
\begin{tikzpicture}[baseline = -.4cm]
    \draw (0,.6) -- (0,-.3);
    \draw (1,0) -- (1,.6);
    \draw (.5,-.8) -- (.5,-1.2);
    \draw (0,-.3) arc (-180:0:.5cm);
    \filldraw[fill=white] (.5,-.8) circle (.05cm) node [above] {\scriptsize{$\alpha_j$}};
    \roundNbox{unshaded}{(1,0)}{.3}{0}{0}{$h_k$}
    \node at (.15,.5) {\scriptsize{$\mathbf{a}$}};
    \node at (1.15,-.5) {\scriptsize{$\mathbf{d}$}};
    \node at (.35,-1) {\scriptsize{$\mathbf{c}$}};
    \node at (1.2,.5) {\scriptsize{$\bfA$}};
\end{tikzpicture}
\otimes \xi_k
\right)
\right\|^2_{\bfH(b)}
&=
\sum_{i,k}
\left\langle
\pi\left(
\begin{tikzpicture}[baseline=-.1cm]
    \draw (-.5,1.5) arc (90:270:.45cm);
    \draw (-.5,.4) .. controls ++(180:.4cm) and ++(180:.6cm) .. (-.2,-.4);
    \draw (-.3,-.6) arc (90:270:.45cm);
    \draw (.5,1.5) arc (90:-90:.5cm);
    \draw (.3,-.5) arc (90:-90:.5cm);
    \draw (1,1) .. controls ++(0:.5cm) and ++(90:.5cm) .. (1.6,0) .. controls ++(270:.5cm) and ++(0:.5cm) .. (.8,-1);
    \draw (1.6,0) -- (2,0);
    \draw (-.75,-1) .. controls ++(180:.8cm) and ++(180:.6cm) .. (-.95,1);
    \filldraw[fill=white] (-.75,-1) circle (.05cm) node [right] {\scriptsize{$\alpha_i$}};
    \filldraw[fill=white] (-.95,1) circle (.05cm) node [right] {\scriptsize{$\alpha_k^*$}};
    \filldraw (1.6,0) circle (.05cm);
    \filldraw (1,1) circle (.05cm);
    \filldraw (.8,-1) circle (.05cm);
    \roundNbox{unshaded}{(0,1.5)}{.3}{.2}{.2}{$j(h_k)$}
    \roundNbox{unshaded}{(0,.5)}{.3}{.2}{.2}{$j(f)$}
    \roundNbox{unshaded}{(0,-.5)}{.3}{0}{0}{$f$}
    \roundNbox{unshaded}{(0,-1.5)}{.3}{0}{0}{$h_i$}
    \node at (-1.5,0) {\scriptsize{$\overline{\mathbf{c}}$}};
    \node at (-.9,0) {\scriptsize{$\overline{\mathbf{b}}$}};
    \node at (1.8,.2) {\scriptsize{$\bfA$}};
\end{tikzpicture}
\right)
\xi_i,\xi_k\right\rangle
\\&= 
\sum_{i,k}
\left\langle
\pi\left(
\begin{tikzpicture}[baseline=-.1cm]
    \draw (-.5,1.5) arc (90:270:.45cm);
    \draw (-.5,.4) .. controls ++(180:.4cm) and ++(180:.6cm) .. (-.2,-.4);
    \draw (-.3,-.6) arc (90:270:.45cm);
    \draw (.5,.5) arc (90:-90:.5cm) -- (.3,-.5);
    \draw (.5,1.5) .. controls ++(0:.5cm) and ++(90:.4cm) .. (1.5,.75) .. controls ++(270:.3cm) and ++(0:.3cm) .. (1,0);
    \draw (1.5,.75) .. controls ++(0:.3cm) and ++(90:.3cm) .. (2,0) .. controls ++(270:.8cm) and ++(0:.8cm) .. (.3,-1.5);
    \draw (2,0) -- (2.4,0);
    \draw (-.75,-1) .. controls ++(180:.8cm) and ++(180:.6cm) .. (-.95,1);
    \filldraw[fill=white] (-.75,-1) circle (.05cm) node [right] {\scriptsize{$\alpha_i$}};
    \filldraw[fill=white] (-.95,1) circle (.05cm) node [right] {\scriptsize{$\alpha_k^*$}};
    \filldraw (1.5,.75) circle (.05cm);
    \filldraw (1,0) circle (.05cm);
    \filldraw (2,0) circle (.05cm);
    \roundNbox{unshaded}{(0,1.5)}{.3}{.2}{.2}{$j(h_k)$}
    \roundNbox{unshaded}{(0,.5)}{.3}{.2}{.2}{$j(f)$}
    \roundNbox{unshaded}{(0,-.5)}{.3}{0}{0}{$f$}
    \roundNbox{unshaded}{(0,-1.5)}{.3}{0}{0}{$h_i$}
    \node at (-1.5,0) {\scriptsize{$\overline{\mathbf{c}}$}};
    \node at (-.9,0) {\scriptsize{$\overline{\mathbf{b}}$}};
    \node at (2.2,.2) {\scriptsize{$\bfA$}};
\end{tikzpicture}
\right)
\xi_i,\xi_k\right\rangle
\\&\leq 
\|f\|_{\cM_\bfA}^2 
\left\|
\sum_{k=1}^n
\begin{tikzpicture}[baseline = -.4cm]
    \draw (0,.6) -- (0,-.3);
    \draw (1,0) -- (1,.6);
    \draw (.5,-.8) -- (.5,-1.2);
    \draw (0,-.3) arc (-180:0:.5cm);
    \filldraw[fill=white] (.5,-.8) circle (.05cm) node [above] {\scriptsize{$\alpha_j$}};
    \roundNbox{unshaded}{(1,0)}{.3}{0}{0}{$h_k$}
    \node at (.15,.5) {\scriptsize{$\mathbf{a}$}};
    \node at (1.15,-.5) {\scriptsize{$\mathbf{d}$}};
    \node at (.35,-1) {\scriptsize{$\mathbf{c}$}};
    \node at (1.2,.5) {\scriptsize{$\bfA$}};
\end{tikzpicture}
\otimes \xi_k
\right\|^2_{\bfH(a)}
\end{split}
\end{equation}
where the final inequality holds since for $f\in \cM_\bfA(a_\bfA,b_\bfA)$, $f^*\circ f \leq \|f\|_{\cM_\bfA}^2 \id_a$.

To prove functoriality of $\Phi$, if $f\in \cM_\bfA(a_\bfA,b_\bfA)$, $g\in \cM_\bfA(b_\bfA, c_\bfA)$, we have
\begin{align*}
\Phi(g)\Phi(f)\left(
\begin{tikzpicture}[baseline = -.4cm]
    \draw (0,.6) -- (0,-.3);
    \draw (1,0) -- (1,.6);
    \draw (.5,-.8) -- (.5,-1.2);
    \draw (0,-.3) arc (-180:0:.5cm);
    \filldraw[fill=white] (.5,-.8) circle (.05cm) node [above] {\scriptsize{$\alpha$}};
    \roundNbox{unshaded}{(1,0)}{.3}{0}{0}{$h$}
    \node at (.15,.5) {\scriptsize{$\mathbf{a}$}};
    \node at (1.15,-.5) {\scriptsize{$\mathbf{d}$}};
    \node at (.35,-1) {\scriptsize{$\mathbf{e}$}};
    \node at (1.2,.5) {\scriptsize{$\bfA$}};
\end{tikzpicture}
\otimes \xi
\right)
&=
\begin{tikzpicture}[baseline = -.1cm]
    \draw (-.15,.3) -- (-.15,1.3);
    \draw (.15,.3) arc (180:0:.425cm);
    \draw (.575,.725) -- (.575,1.3);
    \draw (.5,-.8) -- (.5,-1.2);
    \draw (0,-.3) arc (-180:0:.5cm);
    \filldraw (.575,.725) circle (.05cm);
    \filldraw[fill=white] (.5,-.8) circle (.05cm) node [above] {\scriptsize{$\alpha$}};
    \draw (0,1.3) arc (180:0:.2875cm);
    \draw (.2875,1.5875) -- (.2875,1.9);
    \draw (-.3,1.3) -- (-.3,1.9);
    \filldraw (.2875,1.5875) circle (.05cm);
    \roundNbox{unshaded}{(0,0)}{.3}{0}{0}{$f$}
    \roundNbox{unshaded}{(-.15,1)}{.3}{0}{0}{$g$}
    \roundNbox{unshaded}{(1,0)}{.3}{0}{0}{$h$}
    \node at (-.35,.5) {\scriptsize{$\mathbf{b}$}};
    \node at (-.15,-.5) {\scriptsize{$\mathbf{a}$}};
    \node at (1.15,-.5) {\scriptsize{$\mathbf{d}$}};
    \node at (-.1,1.75) {\scriptsize{$\mathbf{c}$}};
    \node at (.35,-1) {\scriptsize{$\mathbf{e}$}};
    \node at (1.2,.5) {\scriptsize{$\bfA$}};
    \node at (.775,1.1) {\scriptsize{$\bfA$}};
    \node at (.5,1.75) {\scriptsize{$\bfA$}};
\end{tikzpicture}
\otimes \xi
=
\begin{tikzpicture}[baseline = -.1cm]
    \draw (-.15,.3) -- (-.15,1.3);
    \draw (.5,-.8) -- (.5,-1.2);
    \draw (0,-.3) arc (-180:0:.5cm);
    \filldraw (.575,1.875) circle (.05cm);
    \draw (.575,1.875) -- (.575,2.2);
    \filldraw[fill=white] (.5,-.8) circle (.05cm) node [above] {\scriptsize{$\alpha$}};
    \draw (0,1.3) arc (180:0:.15cm);
    \draw (.3,.3) -- (.3,1.3);
    \draw (-.3,1.3) -- (-.3,2.2);
    \filldraw (.15,1.45) circle (.05cm);
    \draw (.15,1.45) arc (180:0:.425) -- (1,.3);
    \roundNbox{unshaded}{(0,0)}{.3}{.15}{.25}{$f$}
    \roundNbox{unshaded}{(-.15,1)}{.3}{0}{0}{$g$}
    \roundNbox{unshaded}{(1,0)}{.3}{0}{0}{$h$}
    \node at (-.35,.5) {\scriptsize{$\mathbf{b}$}};
    \node at (-.15,-.5) {\scriptsize{$\mathbf{a}$}};
    \node at (1.15,-.5) {\scriptsize{$\mathbf{d}$}};
    \node at (-.1,2.1) {\scriptsize{$\mathbf{c}$}};
    \node at (.35,-1) {\scriptsize{$\mathbf{e}$}};
    \node at (.775,2.1) {\scriptsize{$\bfA$}};
\end{tikzpicture}
\otimes \xi
\\&=
\begin{tikzpicture}[baseline = -.1cm]
    \draw (-.15,.3) -- (-.15,1.3);
    \draw (.15,.3) arc (180:0:.425cm);
    \draw (.575,.725) -- (.575,1.3);
    \draw (.5,-.8) -- (.5,-1.2);
    \draw (0,-.3) arc (-180:0:.5cm);
    \filldraw (.575,.725) circle (.05cm);
    \filldraw[fill=white] (.5,-.8) circle (.05cm) node [above] {\scriptsize{$\alpha$}};
    \roundNbox{unshaded}{(0,0)}{.3}{.2}{.2}{$g\circ f$}
    \roundNbox{unshaded}{(1,0)}{.3}{0}{0}{$h$}
    \node at (-.35,1.1) {\scriptsize{$\mathbf{b}$}};
    \node at (-.15,-.5) {\scriptsize{$\mathbf{a}$}};
    \node at (1.15,-.5) {\scriptsize{$\mathbf{d}$}};
    \node at (.35,-1) {\scriptsize{$\mathbf{c}$}};
    \node at (1.2,.5) {\scriptsize{$\bfA$}};
    \node at (.375,1.1) {\scriptsize{$\bfA$}};
    \node at (.1,.725) {\scriptsize{$\bfA$}};
\end{tikzpicture}
\otimes \xi
=
\Phi(g\circ f) \left(
\begin{tikzpicture}[baseline = -.4cm]
    \draw (0,.6) -- (0,-.3);
    \draw (1,0) -- (1,.6);
    \draw (.5,-.8) -- (.5,-1.2);
    \draw (0,-.3) arc (-180:0:.5cm);
    \filldraw[fill=white] (.5,-.8) circle (.05cm) node [above] {\scriptsize{$\alpha$}};
    \roundNbox{unshaded}{(1,0)}{.3}{0}{0}{$h$}
    \node at (.15,.5) {\scriptsize{$\mathbf{a}$}};
    \node at (1.15,-.5) {\scriptsize{$\mathbf{d}$}};
    \node at (.35,-1) {\scriptsize{$\mathbf{e}$}};
    \node at (1.2,.5) {\scriptsize{$\bfA$}};
\end{tikzpicture}
\otimes \xi
\right).
\end{align*}
To show that $\Phi$ is a dagger functor, if $f\in \cM_\bfA(a_\bfA, b_\bfA)$, then
$$
\Phi(f^*)
\left(
\begin{tikzpicture}[baseline = -.4cm]
    \draw (0,.6) -- (0,-.3);
    \draw (1,0) -- (1,.6);
    \draw (.5,-.8) -- (.5,-1.2);
    \draw (0,-.3) arc (-180:0:.5cm);
    \filldraw[fill=white] (.5,-.8) circle (.05cm) node [above] {\scriptsize{$\alpha$}};
    \roundNbox{unshaded}{(1,0)}{.3}{0}{0}{$h$}
    \node at (.15,.5) {\scriptsize{$\mathbf{b}$}};
    \node at (1.15,-.5) {\scriptsize{$\mathbf{d}$}};
    \node at (.35,-1) {\scriptsize{$\mathbf{c}$}};
    \node at (1.2,.5) {\scriptsize{$\bfA$}};
\end{tikzpicture}
\otimes \xi
\right)
=
\begin{tikzpicture}[baseline = -.1cm]
    \draw (-.15,.3) -- (-.15,1.3);
    \draw (.15,.3) arc (180:0:.425cm);
    \draw (.575,.725) -- (.575,1.3);
    \draw (.5,-.8) -- (.5,-1.2);
    \draw (0,-.3) arc (-180:0:.5cm);
    \filldraw (.575,.725) circle (.05cm);
    \filldraw[fill=white] (.5,-.8) circle (.05cm) node [above] {\scriptsize{$\alpha$}};
    \roundNbox{unshaded}{(0,0)}{.3}{0}{0}{$f^*$}
    \roundNbox{unshaded}{(1,0)}{.3}{0}{0}{$h$}
    \node at (-.35,1.1) {\scriptsize{$\mathbf{a}$}};
    \node at (-.15,-.5) {\scriptsize{$\mathbf{b}$}};
    \node at (1.15,-.5) {\scriptsize{$\mathbf{d}$}};
    \node at (.35,-1) {\scriptsize{$\mathbf{c}$}};
    \node at (1.2,.5) {\scriptsize{$\bfA$}};
    \node at (.375,1.1) {\scriptsize{$\bfA$}};
    \node at (.1,.725) {\scriptsize{$\bfA$}};
\end{tikzpicture}
\otimes \xi
=
\begin{tikzpicture}[baseline = -.1cm]
    \draw (-.15,.3) -- (-.15,1.3);
    \draw (.15,.3) arc (180:0:.425cm);
    \draw (.575,.725) -- (.575,1.3);
    \draw (.5,-.8) -- (.5,-1.2);
    \draw (0,-.3) arc (-180:0:.5cm);
    \filldraw (.575,.725) circle (.05cm);
    \filldraw[fill=white] (.5,-.8) circle (.05cm) node [above] {\scriptsize{$\alpha$}};
    \roundNbox{unshaded}{(0,0)}{.3}{.3}{.3}{\scriptsize{$j_{\overline{b}\otimes a}(f)$}}
    \roundNbox{unshaded}{(1,0)}{.3}{0}{0}{$h$}
    \node at (-.35,1.1) {\scriptsize{$\mathbf{a}$}};
    \node at (-.15,-.5) {\scriptsize{$\mathbf{b}$}};
    \node at (1.15,-.5) {\scriptsize{$\mathbf{d}$}};
    \node at (.35,-1) {\scriptsize{$\mathbf{c}$}};
    \node at (1.2,.5) {\scriptsize{$\bfA$}};
    \node at (.375,1.1) {\scriptsize{$\bfA$}};
    \node at (.1,.725) {\scriptsize{$\bfA$}};
\end{tikzpicture}
\otimes \xi
\,,
$$
which is easily seen to be $\Phi(f)^*$ using the inner product \eqref{eq:GelfandNaimarkInnerProduct} above.

Finally, we want to show that $\Phi$ is an isometry on the level of hom spaces.
It suffices to check that $\Phi$ is injective on the endomorphism C*-algebras $\End_{\cM_\bfA}(a_\bfA)$ for $a\in\cC$, since injective C*-algebra homomorphisms are isometric.
To do so, let $f\in \End_{\cM_\bfA}(a_\bfA)\cong \bfA(\overline{a}\otimes a)$ such that $f\neq 0$.
Choose a vector $\xi\in K$ such that $\langle \pi\langle f | f\rangle_{a}\xi, \xi\rangle_{K}\ne 0$, which is possible since $\pi$ is faithful.
Then consider the vector 
$$
(\id_a\otimes i_\bfA)\otimes \xi
=
\begin{tikzpicture}[baseline=-.1cm]
    \draw (0,-.7) -- (0,-.3);
    \draw (-.25,.3) -- (-.25,.7);
    \draw (.25,0) -- (.25,.7);
    \draw (0,-.3) .. controls ++(90:.2cm) and ++(270:.2cm) .. (-.25,.3);
    \filldraw (.25,0) circle (.05cm);
    \roundNbox{dashed}{(0,0)}{.3}{.2}{.2}{}
    \node at (-.2,-.5) {\scriptsize{$\mathbf{a}$}};
    \node at (-.45,.5) {\scriptsize{$\mathbf{a}$}};
    \node at (.45,.5) {\scriptsize{$\bfA$}};
\end{tikzpicture}
\otimes \xi
\in
(\mathbf{a}\otimes \bfH)(a)
$$
where $i_\bfA \in \bfA(1_\cC)$ is the unit morphism.
Then $\Phi(f)((\id_a\otimes i_\bfA)\otimes \xi)\ne 0$, so $\Phi(f)\neq 0$.
Thus $\Phi|_{\End_{\cM_{\bf{A}}}(a_\bfA)}$ has trivial kernel.
\end{proof}

From our version of the Gelfand-Naimak theorem, we get  notion of a representation of a C*-algebra object.

\begin{defn}
A representation of the C*-algebra object $(\bfA,m,i) \in \Vec(\cC)$ is a Hilbert space object $\bfH\in \Hilb(\cC)$ and a $*$-algebra natural transformation $\theta: \bfA\Rightarrow \bfB(\bfH)$.

Notice that for all $c\in\cC$, we get a $*$-algebra homomorphism $\theta_{\overline{c}\otimes c}: \bfA(\overline{c}\otimes c) \to \bfB(\bfH)(\overline{c}\otimes c)$, which is automatically bounded.
A representation $\theta: \bfA\Rightarrow \bfB(\bfH)$ is called \emph{faithful} if each $\theta_{\overline{c}\otimes c}$ is injective.
\end{defn}

\subsection{Completely positive maps and Stinespring dilation}

Just as completely positive maps between C*-algebras are positive $*$-maps which are not necessarily algebra homomorphisms, completely positive maps between C*-algebra objects in $\Vec(\cC)$ are positive $*$-natural transformations which are not necessarily algebra natural transformations.
The major difference is that since $\cC$ admits finite direct sums, positivity is sufficient for complete positivity!

\begin{defn}
Suppose $\bfA, \bfB\in \Vec(\cC)$ are C*-algebra objects.
A $*$-natural transformation $\theta : \bfA \Rightarrow \bfB$ is called a \emph{completely positive map} (cp map) if for all $c\in \cC$, the bounded linear transformation $\theta_{\overline{c}\otimes c}:\bfA(\overline{c}\otimes c) \to \bfB(\overline{c}\otimes c)$ maps positive operators in the C*-algebra $\bfA(\overline{c}\otimes c)\cong \cM_\bfA(c_\bfA,c_\bfA)$ to positive operators in the C*-algebra $\bfB(\overline{c}\otimes c)\cong \cM_\bfB(c_\bfB,c_\bfB)$.

A completely positive map $\theta: \bfA \Rightarrow \bfB$ is called \emph{unital} (ucp) if $\theta_{1_\cC}(i_\bfA) = i_\bfB$.
\end{defn}

\begin{ex}
A $*$-algebra natural transformation $\theta : \bfA \Rightarrow \bfB$ is completely positive, since every $\theta_{\overline{c}\otimes c}:\cM_\bfA(c_\bfA,c_\bfA) \to \cM_\bfB(c_\bfB,c_\bfB)$ is a $*$-algebra homomorphism.
\end{ex}

\begin{ex}
Notice that when $\cC=\fdHilb$, a positive $*$-natural transformation $\theta: \bfA \Rightarrow \bfB$ is equivalent to an ordinary completely positive map $\theta_{1_\cC} : \bfA(1_\cC) \to \bfB(1_\cC)$.
Notice that the positivity of the $n$-fold amplification of $\theta_{1_\cC}$ follows from the positivity of $\theta_{\overline{(\oplus_{i=1}^n 1_\cC)}\otimes (\oplus_{i=1}^n 1_\cC)}$, since $M_n(\bfA(1_\cC))\cong \bfA(\overline{(\oplus_{i=1}^n 1_\cC)}\otimes (\oplus_{i=1}^n 1_\cC))$.
\end{ex}

We now include a formulation of cp-multiplier on the module category side which corresponds to completely positive $*$-natural transformations between algebra objects.

\begin{defn}
\label{defn:multiplier} 
Let $(\cM, m)$ and $(\cN, n)$ be cyclic $\cC$-modules categories.  
A \textit{multiplier} is a collection of maps $\Theta_{a,b}: \cM(a\otimes m,b\otimes m)\rightarrow \cN(a\otimes n,b\otimes n)$ such that for all $\psi\in \cC(a,b)$, $\phi\in \cC(c,d)$, $f\in \cM(b\otimes m,c\otimes m)$, and $e\in\cC$,
$$
\Theta_{a,d}\left(
\begin{tikzpicture}[baseline=-.1cm]
	\draw (-.6,-1.5) -- (-.6,1.5);
	\draw (-.2,-1.5) -- (-.2,1.5);
	\draw (.2,-1.5) -- (.2,1.5);
	\roundNbox{unshaded}{(-.2,1)}{.25}{0}{0}{$\phi$}
	\roundNbox{unshaded}{(-.2,-1)}{.25}{0}{0}{$\psi$}
	\roundNbox{unshaded}{(0,0)}{.3}{.1}{.1}{$f$}
	\node at (-.4,-1.4) {\scriptsize{$a$}};
	\node at (-.4,-.5) {\scriptsize{$b$}};
	\node at (-.4,.5) {\scriptsize{$c$}};
	\node at (-.4,1.4) {\scriptsize{$d$}};
	\node at (-.8,0) {\scriptsize{$e$}};
	\node at (.4,1) {\scriptsize{$m$}};
	\node at (.4,-1) {\scriptsize{$m$}};
\end{tikzpicture}
\right)
=
\begin{tikzpicture}[baseline=-.1cm]
	\draw (-1,-1.5) -- (-1,1.5);
	\draw (-.2,-1.5) -- (-.2,1.5);
	\draw (.2,-1.5) -- (.2,1.5);
	\roundNbox{unshaded}{(-.2,1)}{.25}{0}{0}{$\phi$}
	\roundNbox{unshaded}{(-.2,-1)}{.25}{0}{0}{$\psi$}
	\roundNbox{unshaded}{(0,0)}{.3}{.4}{.4}{$\Theta_{b,c}(f)$}
	\node at (-.4,-1.4) {\scriptsize{$a$}};
	\node at (-.4,-.5) {\scriptsize{$b$}};
	\node at (-.4,.5) {\scriptsize{$c$}};
	\node at (-.4,1.4) {\scriptsize{$d$}};
	\node at (-1.2,0) {\scriptsize{$e$}};
	\node at (.4,1) {\scriptsize{$m$}};
	\node at (.4,-1) {\scriptsize{$m$}};
\end{tikzpicture}
\,.
$$
If $(\cM,m)$ and $(\cN,n)$ are $\cC$-module C*-categories, we define a \emph{cp-multiplier} as a multiplier $\{\Theta_{a,b}\}$ such that every $\Theta_{c,c}$ maps positive elements to positive elements.

In Example \ref{ex:SymEA}, we compare our definition of (cp-)multiplier with \cite[Def.~3.4]{MR3406647}.
\end{defn}

Recall that a natural transformation $\theta: \bfA\Rightarrow\bf B$ can be defined by a collection of maps $(\theta_{c}:\bfA(c)\rightarrow \bfB(c))_{c\in\Irr(\cC)}$, which corresponds to a collection of maps $(\check{\theta}_{c}: \cM_{\bfA}(c_\bfA, 1_\bfA)\rightarrow \cM_{\bfB}(c_\bfB, 1_\bfB))_{c\in\Irr(\cC)}$.  
Conversely, any such collection of maps $\cM_\bfA\to \cM_\bfB$ induces a natural transformation $\bfA\Rightarrow \bfB$.

\begin{defn} 
Suppose $\bfA,\bfB\in \Vec(\cC)$ are algebra objects, and let $\theta: \bfA\Rightarrow\bf B$ be a natural transformation (which is not necessarily an algebra natural transformation).
We define the \textit{amplification} of $\theta$ to be the multiplier $\set{\Theta_{a,b}}{ a,b\in \cC}: \cM_{\bfA}\rightarrow \cM_{\bfB}$ given by 
\begin{equation}
\label{eq:AmplifiedMultiplier}
\Theta_{a,b}\left(
\begin{tikzpicture}[baseline=-.1cm]
	\draw (0,-.6) -- (0,.6);
	\roundNbox{unshaded}{(0,0)}{.3}{0}{0}{$f$}
	\node at (.3,-.5) {\scriptsize{$a_\bfA$}};
	\node at (.3,.5) {\scriptsize{$b_\bfA$}};
\end{tikzpicture}
\right)
=
\sum_{\substack{c\in \Irr(\cC)\\ 
\alpha\in \ONB(c,\overline{b}\otimes a)}} 
d_c
\begin{tikzpicture}[baseline=-.4cm]
	\draw (0,.3) arc (0:180:.3cm) -- (-.6,-.3) arc (-180:0:.3cm);
	\draw (-.3,-.6) -- (-.3,-1.2);
	\draw (0,-2.2) -- (0,-1.5) arc (0:180:.3cm) arc (0:-180:.6cm) -- (-1.8,1.1);
	\filldraw[fill=white] (-.3,-.6) circle (.05cm) node [above] {\scriptsize{$\alpha$}};
	\filldraw[fill=white] (-.3,-1.2) circle (.05cm) node [below] {\scriptsize{$\alpha^*$}};
	\roundNbox{}{(0,0)}{.8}{.8}{0}{}
	\roundNbox{unshaded}{(0,0)}{.3}{0}{0}{$f$}
	\node at (-1.2,0) {$\theta_c\Bigg($};
	\node at (.55,0) {$\Bigg)$};
	\node at (-2,.9) {\scriptsize{$b_\bfB$}};
	\node at (.2,.5) {\scriptsize{$b_\bfA$}};
	\node at (.2,-.5) {\scriptsize{$a_\bfA$}};
	\node at (.3,-2) {\scriptsize{$a_\bfB$}};
	\node at (0,-1) {\scriptsize{$c_\bfB$}};
	\node at (-.8,0) {\scriptsize{$\overline{b_\bfA}$}};
	\node at (-.9,-1.5) {\scriptsize{$\overline{b_\bfB}$}};
\end{tikzpicture}\,.
\end{equation}
\end{defn}

It is straightforward to check that $\{\Theta_{a,b}\}$ is a multiplier.  
The word ``amplification" is motivated by the case $\cC=\fdHilb$, in which algebra objects are ordinary associative algebras $A,B\in \Vec$, and natural transformations are simply characterized by linear maps $A\to B$.
A map $\theta: A\rightarrow B$ induces amplified maps $M_{n\times m}(A)\rightarrow M_{n\times m}(B)$.
It is easy to check that our definition of amplification above coincides with the usual notion when $\cC=\fdHilb$.

In fact, all multipliers arise uniquely as amplifications.
The proof of the following proposition is straightforward.

\begin{prop} 
Let $\Theta=\set{\Theta_{a,b}}{ a,b\in \cC}:\cM_{\bfA}\to \cM_{\bfB} $ be a multiplier.  
Then $\Theta$ is the amplification of the natural transformation $\theta:\bfA\Rightarrow \bfB$ corresponding to the family $\set{\Theta_{c,1}: \cM_{\bfA}(c_\bfA, 1_\bfA)\rightarrow \cM_{\bfB}(c_\bfB, 1_\bfB)}{c\in\Irr(\cC)}$.
\end{prop}

\begin{cor} 
A multiplier $\cM_\bfA \to \cM_\bfB$ is completely positive if and only if it is the amplification of a completely positive map $\bfA\Rightarrow \bfB$.
\end{cor}

We now prove the analog of the Stinespring dilation theorem.
We begin with a lemma.

\begin{lem}
Suppose $\bfH, \bfK\in \Hilb(\cC)$ and $v\in \Hom_{\Hilb(\cC)}(\bfH, \bfK)$.
Then the map $\Ad(v): \bfB(\bfK)\to \bfB(\bfH)$ defined by
$$
\Ad(v)(f) 
= 
(\id_c\otimes v^*)\circ f \circ (\id_c\otimes v)
=
\begin{tikzpicture}[baseline=-.1cm]
	\draw (-.2,-1.5) -- (-.2,1.5);
	\draw (.2,-1.5) -- (.2,1.5);
	\roundNbox{unshaded}{(.25,1)}{.25}{0}{0}{$v^*$}
	\roundNbox{unshaded}{(0,0)}{.3}{.1}{.1}{$f$}
	\roundNbox{unshaded}{(.25,-1)}{.25}{0}{0}{$v$}
    \node at (-.2,-1.7) {\scriptsize{$\mathbf{c}$}};
    \node at (-.2,1.7) {\scriptsize{$\mathbf{c}$}};
    \node at (.2,-1.7) {\scriptsize{$\bfH$}};
    \node at (.2,1.7) {\scriptsize{$\bfH$}};
    \node at (.35,-.525) {\scriptsize{$\bfK$}};
    \node at (.35,.525) {\scriptsize{$\bfK$}};
\end{tikzpicture}
\in
\Hom_{\Hilb(\cC)}(\mathbf{c}\otimes \bfH, \mathbf{c}\otimes \bfH)
\cong
\bfB(\bfH)(\overline{c}\otimes c)
$$
for $f\in \bfB(\bfK)(\overline{c}\otimes c)$ is completely positive.
\end{lem}
\begin{proof}
Let $c\in \cC$ and let $f=g^*\circ g$ for some $g\in \bfB(\bfK)(\overline{c}\otimes c) \cong \Hom_{\Hilb(\cC)}(\mathbf{c}\otimes \bfK, \mathbf{c}\otimes \bfK)$.
Then we have $\Ad(v)(f) = h^*\circ h$ for $h=g\circ (\id_c\otimes v) \in \Hom_{\Hilb(\cC)}(\mathbf{c}\otimes \bfH, \mathbf{c}\otimes \bfK)$. 
Now since $\Hilb(\cC)$ is a C*-category, we may write $h^*\circ h = k^*\circ k$ for some $k \in \Hom_{\Hilb(\cC)}(\mathbf{c}\otimes \bfH, \mathbf{c}\otimes \bfH) \cong \bfB(\bfH)(\overline{c}\otimes c)$.
\end{proof}

\begin{thm}[Stinespring dilation in $\cC$]
Let $\bfH\in \Hilb(\cC)$, and let $\bfB(\bfH)\in \Vec(\cC)$ be the corresponding \emph{C*}-algebra object.
Let $\bfA\in \Vec(\cC)$ be another \emph{C*}-algebra object, and suppose we have a unital completely positive map $\theta:\bfA \to \bfB(\bfH)$.
Then there is a $\bfK\in \Hilb(\cC)$, a $*$-algebra natural transformation $\pi: \bfA \to \bfB(\bfK)$, and an isometry $v\in \Hom_{\Hilb(\cC)}(\bfH, \bfK)$ such that $\theta = \Ad(v)\circ \pi$ as $*$-natural transformations.
\end{thm}
\begin{proof}
We follow the usual proof of the Stinespring dilation theorem. 
We begin by defining $\bfV=\bfA\otimes \bfH\in \Vec(\cC)$.  
We define a sesquilinear form on $\bfV(c)=(\bfA\otimes \bfH)(c)=\bigoplus_{a,b\in \Irr(\cC)} \bfA(a)\otimes \cC(c, a\otimes b)\otimes \bfH(b)$ for $c\in \Irr(\cC)$ by
$$
\langle
f\otimes \alpha \otimes \xi, 
g\otimes \beta \otimes \eta
\rangle
\id_{\mathbf{c}}
=
\begin{tikzpicture}[baseline=-.1cm]
	\draw (.25,1.35) arc (0:180:.25cm) -- (-.25,-1.35) arc (-180:0:.25cm) -- (.25,1.35);
	\draw (0,1.6) -- (0,2.6)  .. controls ++(90:.5cm) and ++(90:.5cm) .. (1.4,2.6) -- (1.4,-2)  .. controls ++(270:.5cm) and ++(270:.5cm) .. (0,-2) -- (0,-1.6);
	\roundNbox{unshaded}{(.25,1)}{.25}{0}{0}{$\eta^*$}
	\roundNbox{unshaded}{(0,0)}{.3}{.6}{.6}{$\theta(g^*\circ f)$}
	\roundNbox{unshaded}{(0,2.3)}{.3}{0}{0}{$\varphi_c$}
	\roundNbox{unshaded}{(.25,-1)}{.25}{0}{0}{$\xi$}
	\filldraw[fill=white] (0,1.6) circle (.05cm) node [below] {\scriptsize{$\beta^*$}};
	\filldraw[fill=white] (0,-1.6) circle (.05cm) node [above] {\scriptsize{$\alpha$}};
	\node at (-.2,1.8) {\scriptsize{$\mathbf{c}$}};
	\node at (-.2,-1.8) {\scriptsize{$\mathbf{c}$}};
	\node at (-.2,2.8) {\scriptsize{$\overline{\overline{\mathbf{c}}}$}};
	\node at (1.6,0) {\scriptsize{$\overline{\mathbf{c}}$}};
	\node at (.4,-1.4) {\scriptsize{$\mathbf{b}$}};
	\node at (-.4,-1.4) {\scriptsize{$\mathbf{a}$}};
	\node at (.4,1.4) {\scriptsize{$\mathbf{d}$}};
	\node at (-.4,1.4) {\scriptsize{$\mathbf{e}$}};
	\node at (.45,-.5) {\scriptsize{$\bfH$}};
	\node at (.45,.5) {\scriptsize{$\bfH$}};
\end{tikzpicture}
\,,
$$
Here, for $f\in \bfA(a)$ and $g\in \bfA(e)$,  $g^*\circ f\in \bfA(\overline{e}\otimes a)\cong \cM_\bfA(a_\bfA, e_\bfA)$ is the composite of $f\in \cM_\bfA(a_\bfA, 1_\cC)$ and $g^*\in \cM_\bfA(1_\cC, e_\bfA)$. 
Equivalently, $g^*\circ f = \mu_{1_\cC, 1_\cC}(j_e(g)\otimes f)$, where we view $f\in \bfA(1_\cC\otimes a)$ and $j_e(g)\in \bfA(e\otimes 1_\cC)$.
Thus $\theta(g^*\circ f)\in \Hom_{\Hilb(\cC)}(\mathbf{a}\otimes \bfH, \mathbf{e}\otimes \bfH)$.

Note that this sesquilinear form is positive semi-definite since $\theta$ is completely positive.
We now define $\bfK\in \Hilb(\cC)$ as the completion of $\bfV$ (modulo the length zero vectors) using Corollary \ref{cor:CompletionVecToHilb}.

There is an obvious map $v:\bfH \Rightarrow\bfA\otimes \bfH=\bfV$ given for $c\in \Irr(\cC)$ by
$$
\bfH(c) \ni \xi \longmapsto i_\bfA\otimes \id_{1_\cC} \otimes \xi \in \bfA(1_\cC) \otimes \cC(c, 1_\cC\otimes c) \otimes \bfH(c).
$$
Note that $v$ is an isometry, since $i_\bfA^*\circ i_\bfA = \mu_{1_\cC, 1_\cC}(j_{1_\cC}(i_\bfA)\otimes i_\bfA) = i_\bfA$ and $\theta$ is unital.
A straightforward computation shows that $v^*: \bfK \Rightarrow \bfH$ is given on $(\bfA\otimes \bfH)(c)$ by
$$
\bfA(a) \otimes \cC(c, a\otimes b) \otimes \bfH(b)
\ni
g\otimes \alpha \otimes \eta \longmapsto 
\begin{tikzpicture}[baseline=-.8cm]
	\draw (-.25,0) -- (-.25,-1.35) arc (-180:0:.25cm) -- (.25,0);
	\draw (0,-1.6) -- (0,-2);
	\draw (0,.3) -- (0,.7);
	\roundNbox{unshaded}{(0,0)}{.3}{.2}{.2}{$\theta(g)$}
	\roundNbox{unshaded}{(.25,-1)}{.25}{0}{0}{$\eta$}
	\filldraw[fill=white] (0,-1.6) circle (.05cm) node [above] {\scriptsize{$\alpha$}};
	\node at (.4,-1.4) {\scriptsize{$\mathbf{b}$}};
	\node at (-.4,-1.4) {\scriptsize{$\mathbf{a}$}};
	\node at (.45,-.5) {\scriptsize{$\bfH$}};
	\node at (.2,.5) {\scriptsize{$\bfH$}};
	\node at (.2,-1.8) {\scriptsize{$\mathbf{c}$}};
\end{tikzpicture}
\,.
$$

We now get a $*$-representation $\pi:\bfA \Rightarrow \bfB(\bfK)$ as follows.
If  $h\in \bfA(d)$ for $d\in\Irr(\cC)$, we define the map $\pi_d(h)\in \bfB(\bfK)(d)=\Hom_{\Hilb(\cC)}(\mathbf{d}\otimes \bfK , \bfK)$ by defining its $c$-component $(\pi_d(h))_c$ for $\beta \otimes g\otimes \alpha\otimes \xi \in \cC(c, d\otimes e)\otimes \bfA(a)\otimes \cC(e, a\otimes b)\otimes \bfH(b)\subset (\mathbf{d}\otimes (\bfA\otimes \bfH))(c)$ by
$$
\beta\otimes g\otimes \alpha\otimes \xi  \longmapsto 
\begin{tikzpicture}[baseline = -.4cm]
    \draw (0,.3) arc (0:180:.5cm);
    \draw (1,0) -- (1,1.2);
    \draw (.5,-.8) arc (0:-180:.75cm) -- (-1,-.3);
    \draw (0,-.3) arc (-180:0:.5cm);
    \draw (-.25,-1.55) -- (-.25,-1.9);
    \draw (-.5,.8) -- (-.5,1.2);
    \filldraw[fill=white] (.5,-.8) circle (.05cm) node [above] {\scriptsize{$\alpha$}};
    \filldraw[fill=white] (-.25,-1.55) circle (.05cm) node [above] {\scriptsize{$\beta$}};
    \filldraw (-.5,.8) circle (.05cm);
    \roundNbox{unshaded}{(-1,0)}{.3}{0}{0}{$h$}
    \roundNbox{unshaded}{(0,0)}{.3}{0}{0}{$g$}
    \roundNbox{unshaded}{(1,0)}{.3}{0}{0}{$\xi$}
    \node at (-.15,-.5) {\scriptsize{$\mathbf{a}$}};
    \node at (1.15,-.5) {\scriptsize{$\mathbf{b}$}};
    \node at (.65,-1) {\scriptsize{$\mathbf{e}$}};
    \node at (-1.2,-.5) {\scriptsize{$\mathbf{d}$}};
    \node at (0,-1.7) {\scriptsize{$\mathbf{c}$}};
    \node at (-1.2,.5) {\scriptsize{$\bfA$}};
    \node at (.2,.5) {\scriptsize{$\bfA$}};
    \node at (-.7,1) {\scriptsize{$\bfA$}};
    \node at (1.2,.5) {\scriptsize{$\bfH$}};
\end{tikzpicture}
\in 
(\bfA\otimes \bfH)(c).
$$
The fact that this map is bounded and thus well-defined is similar to \eqref{eq:LeftActionBounded} above, but is a bit more involved as $\theta$ appears in the inner product.
First, we use the 6j symbols of $\cC$ to rewrite
$
\beta \otimes g\otimes \alpha\otimes \xi
\in
\cC(c, d\otimes e)\otimes \bfA(a)\otimes \cC(e, a\otimes b)\otimes \bfH(b)
$
as
$$
\sum_{\substack{
f\in\Irr(\cC)
\\
\gamma\in \ONB(f,a\otimes d)
\\
\delta\in \ONB(c,f\otimes b)
}}
U_{\gamma,\delta}\,
\begin{tikzpicture}[baseline = -.4cm]
    \draw (0,.3) arc (0:180:.5cm);
    \draw (1,0) -- (1,1.2);
    \draw (-.5,-.8) arc (-180:0:.75cm) -- (1,-.3);
    \draw (0,-.3) arc (0:-180:.5cm);
    \draw (.25,-1.55) -- (.25,-1.9);
    \draw (-.5,.8) -- (-.5,1.2);
    \filldraw[fill=white] (-.5,-.8) circle (.05cm) node [above] {\scriptsize{$\gamma$}};
    \filldraw[fill=white] (.25,-1.55) circle (.05cm) node [above] {\scriptsize{$\delta$}};
    \filldraw (-.5,.8) circle (.05cm);
    \roundNbox{dashed}{(-.5,0)}{1.2}{-.2}{-.2}{}
    \roundNbox{unshaded}{(-1,0)}{.3}{0}{0}{$h$}
    \roundNbox{unshaded}{(0,0)}{.3}{0}{0}{$g$}
    \roundNbox{unshaded}{(1,0)}{.3}{0}{0}{$\xi$}
    \node at (.15,-.5) {\scriptsize{$\mathbf{a}$}};
    \node at (1.15,-.5) {\scriptsize{$\mathbf{b}$}};
    \node at (-.65,-1) {\scriptsize{$\mathbf{f}$}};
    \node at (-1.2,-.5) {\scriptsize{$\mathbf{d}$}};
    \node at (0,-1.7) {\scriptsize{$\mathbf{c}$}};
    \node at (-1.2,.5) {\scriptsize{$\bfA$}};
    \node at (.2,.5) {\scriptsize{$\bfA$}};
    \node at (-.7,1) {\scriptsize{$\bfA$}};
    \node at (1.2,.5) {\scriptsize{$\bfH$}};
\end{tikzpicture}
\,.
$$
Treating the dashed box above as a single element $k_{f,\gamma}\in \bfA(f)$, by definition,
$$
\|(\pi_d(h))_c(\beta \otimes f\otimes \alpha\otimes \xi)\|_{\bfK(c)}^2 
=
\sum_{\substack{
f\in\Irr(\cC)
\\
\gamma\in \ONB(f,a\otimes d)
\\
\delta\in \ONB(c,f\otimes b)
}}
\sum_{\substack{
f'\in\Irr(\cC)
\\
\gamma'\in \ONB(f',a\otimes d)
\\
\delta'\in \ONB(c,f'\otimes b)
}}
U_{\gamma,\delta}
U_{\gamma',\delta'}
\begin{tikzpicture}[baseline=-.1cm]
	\draw (.5,1.35) .. controls ++(90:.34cm) and ++(90:.34cm) .. (-.5,1.35) -- (-.5,-1.35) .. controls ++(270:.34cm) and ++(270:.34cm) .. (.5,-1.35) -- (.5,1.35);
	\draw (0,1.6) -- (0,2.6)  .. controls ++(90:.5cm) and ++(90:.5cm) .. (1.6,2.6) -- (1.6,-2)  .. controls ++(270:.5cm) and ++(270:.5cm) .. (0,-2) -- (0,-1.6);
	\roundNbox{unshaded}{(0,2.3)}{.3}{0}{0}{$\varphi_c$}
	\node at (-.2,2.8) {\scriptsize{$\overline{\overline{\mathbf{c}}}$}};
	\node at (1.8,0) {\scriptsize{$\overline{\mathbf{c}}$}};
	\roundNbox{unshaded}{(.5,1)}{.25}{0}{0}{$\xi^*$}
	\roundNbox{unshaded}{(0,0)}{.3}{1}{1}{$\theta(k_{f',\gamma'}^*\circ k_{f,\gamma})$}
	\roundNbox{unshaded}{(.5,-1)}{.25}{0}{0}{$\xi$}
	\filldraw[fill=white] (0,1.6) circle (.05cm) node [below] {\scriptsize{$(\delta')^*$}};
	\filldraw[fill=white] (0,-1.6) circle (.05cm) node [above] {\scriptsize{$\delta$}};
	\node at (-.2,1.8) {\scriptsize{$\mathbf{c}$}};
	\node at (-.2,-1.8) {\scriptsize{$\mathbf{c}$}};
	\node at (.7,-1.4) {\scriptsize{$\mathbf{b}$}};
	\node at (-.7,-1.4) {\scriptsize{$\mathbf{f}$}};
	\node at (.7,1.4) {\scriptsize{$\mathbf{b}$}};
	\node at (-.7,1.4) {\scriptsize{$\mathbf{f}'$}};
	\node at (.7,-.5) {\scriptsize{$\bfH$}};
	\node at (.7,.5) {\scriptsize{$\bfH$}};
\end{tikzpicture}
\,.
$$
Now using that $\theta$ is natural, we may pull the $\gamma$ and $\gamma'$ outside of the $\theta$, and again use the 6j symbols to obtain
\begin{equation}
\label{eq:Stinespring}
\begin{tikzpicture}[baseline=-.1cm]
	\draw (.5,1.35) .. controls ++(90:.34cm) and ++(90:.34cm) .. (-.5,1.35) -- (-.5,-1.35) .. controls ++(270:.34cm) and ++(270:.34cm) .. (.5,-1.35) -- (.5,1.35);
	\draw (0,1.6) .. controls ++(90:.6cm) and ++(90:.6cm) .. (-2,1.6) -- (-2,-1.6) .. controls ++(270:.6cm) and ++(270:.6cm) .. (0,-1.6);
	\draw (-1,2.05) -- (-1,3)  .. controls ++(90:.7cm) and ++(90:.7cm) .. (2,3) -- (2,-2.4)  .. controls ++(270:.7cm) and ++(270:.7cm) .. (-1,-2.4) -- (-1,-2.05);
	\roundNbox{unshaded}{(-1,2.7)}{.3}{0}{0}{$\varphi_c$}
	\node at (-1.2,3.2) {\scriptsize{$\overline{\overline{\mathbf{c}}}$}};
	\node at (2.2,0) {\scriptsize{$\overline{\mathbf{c}}$}};
	\roundNbox{unshaded}{(.5,1)}{.25}{0}{0}{$\xi^*$}
	\roundNbox{unshaded}{(-1,0)}{.3}{2.5}{2.5}{$\theta[(\mu_{d,a}(h\otimes g))^*\circ \mu_{d,a}(h\otimes g)]$}
	\roundNbox{unshaded}{(.5,-1)}{.25}{0}{0}{$\xi$}
	\filldraw[fill=white] (-1,2.05) circle (.05cm) node [below] {\scriptsize{$\beta^*$}};
	\filldraw[fill=white] (-1,-2.05) circle (.05cm) node [above] {\scriptsize{$\beta$}};
	\filldraw[fill=white] (0,1.6) circle (.05cm) node [below] {\scriptsize{$\alpha^*$}};
	\filldraw[fill=white] (0,-1.6) circle (.05cm) node [above] {\scriptsize{$\alpha$}};
	\node at (-1.2,2.2) {\scriptsize{$\mathbf{c}$}};
	\node at (-1.2,-2.2) {\scriptsize{$\mathbf{c}$}};
	\node at (-2.2,1) {\scriptsize{$\mathbf{d}$}};
	\node at (-2.2,-1) {\scriptsize{$\mathbf{d}$}};
	\node at (.2,1.8) {\scriptsize{$\mathbf{e}$}};
	\node at (.2,-1.8) {\scriptsize{$\mathbf{e}$}};
	\node at (.7,-1.4) {\scriptsize{$\mathbf{b}$}};
	\node at (-.7,-1) {\scriptsize{$\mathbf{a}$}};
	\node at (.7,1.4) {\scriptsize{$\mathbf{b}$}};
	\node at (-.7,1) {\scriptsize{$\mathbf{a}$}};
	\node at (.7,-.5) {\scriptsize{$\bfH$}};
	\node at (.7,.5) {\scriptsize{$\bfH$}};
\end{tikzpicture}
\end{equation}
Now viewing $h\in \cM_\bfA(d_\bfA, 1_\bfA)$, similar to the calculation \eqref{eq:LeftActionBounded} above, we see that
\begin{align*}
(\mu_{d,a}(h\otimes g))^*\circ \mu_{d,a}(h\otimes g) 
&=
\begin{tikzpicture}[baseline = -.1cm]
	\draw (1.15,0) -- (1.5,0);
	\draw (.3,1.2) arc (90:-90:.4cm);
	\draw (.3,-1.2) arc (-90:90:.4cm);
	\draw (.7,.8) .. controls ++(0:.6cm) and ++(0:.6cm) .. (.7,-.8);
	\filldraw (.7,.8) circle (.05cm);
	\filldraw (.7,-.8) circle (.05cm);
	\filldraw (1.15,0) circle (.05cm);
	\draw (-.3,.4) arc (-90:-180:1.2cm);
	\draw (-.3,1.2) arc (-90:-180:.4cm);
	\draw (-.3,-.4) arc (90:180:1.2cm);
	\draw (-.3,-1.2) arc (90:180:.4cm);
	\roundNbox{unshaded}{(0,1.2)}{.3}{0}{0}{$g^*$}
	\roundNbox{unshaded}{(0,.4)}{.3}{0}{0}{$h^*$}
	\roundNbox{unshaded}{(0,-.4)}{.3}{0}{0}{$h$}
	\roundNbox{unshaded}{(0,-1.2)}{.3}{0}{0}{$g$}
	\node at (1.325,.2) {\scriptsize{$\bfA$}};
	\node at (-1.2,1.5) {\scriptsize{$\mathbf{d}$}};
	\node at (-.5,1.5) {\scriptsize{$\mathbf{a}$}};
	\node at (-1.2,-1.5) {\scriptsize{$\mathbf{d}$}};
	\node at (-.5,-1.5) {\scriptsize{$\mathbf{a}$}};
\end{tikzpicture}
=
\begin{tikzpicture}[baseline = -.1cm]
	\draw (1.5,.3) -- (1.9,.3);
	\draw (.3,.4) arc (90:-90:.4cm);
	\draw (.7,0) .. controls ++(0:.3cm) and ++(90:.2cm) .. (1.1,-.6);
	\draw (.3,-1.2) .. controls ++(0:.6cm) and ++(270:.3cm) .. (1.1,-.6);
	\draw (.3,1.2) .. controls ++(0:.7cm) and ++(90:.5cm) .. (1.5,.3);
	\draw (1.1,-.6) .. controls ++(0:.3cm) and ++(270:.3cm) .. (1.5,.3);
	\filldraw (.7,0) circle (.05cm);
	\filldraw (1.1,-.6) circle (.05cm);
	\filldraw (1.5,.3) circle (.05cm);
	\draw (-.3,.4) arc (-90:-180:1.2cm);
	\draw (-.3,1.2) arc (-90:-180:.4cm);
	\draw (-.3,-.4) arc (90:180:1.2cm);
	\draw (-.3,-1.2) arc (90:180:.4cm);
	\roundNbox{unshaded}{(0,1.2)}{.3}{0}{0}{$g^*$}
	\roundNbox{unshaded}{(0,.4)}{.3}{0}{0}{$h^*$}
	\roundNbox{unshaded}{(0,-.4)}{.3}{0}{0}{$h$}
	\roundNbox{unshaded}{(0,-1.2)}{.3}{0}{0}{$g$}
	\node at (1.7,.5) {\scriptsize{$\bfA$}};
	\node at (-1.2,1.5) {\scriptsize{$\mathbf{d}$}};
	\node at (-.5,1.5) {\scriptsize{$\mathbf{a}$}};
	\node at (-1.2,-1.5) {\scriptsize{$\mathbf{d}$}};
	\node at (-.5,-1.5) {\scriptsize{$\mathbf{a}$}};
\end{tikzpicture}
\\&\leq
\|h\|^2_{\cM_\bfA}
\begin{tikzpicture}[baseline = -.1cm]
	\draw (-.9,-.8) -- (-.9,.8);
	\draw (.7,0) -- (1.1,0);
	\draw (.3,.4) arc (90:-90:.4cm);
	\filldraw (.7,0) circle (.05cm);
	\draw (-.3,.4) arc (-90:-180:.4cm);
	\draw (-.3,-.4) arc (90:180:.4cm);
	\roundNbox{unshaded}{(0,.4)}{.3}{0}{0}{$g^*$}
	\roundNbox{unshaded}{(0,-.4)}{.3}{0}{0}{$g$}
	\node at (.9,.2) {\scriptsize{$\bfA$}};
	\node at (-.5,.7) {\scriptsize{$\mathbf{a}$}};
	\node at (-.5,-.7) {\scriptsize{$\mathbf{a}$}};
	\node at (-.7,0) {\scriptsize{$\mathbf{d}$}};
\end{tikzpicture}
= 
\|h\|^2_{\cM_\bfA} \bfA(\id_{\overline{a}}\otimes \ev_d\otimes \id_a)(g^*\circ g),
\end{align*}
where we view $g\in \bfA(\overline{1_\cC}\otimes a)\cong \cM_\bfA(a_\bfA, 1_\cC)$.
Now using positivity and naturality of $\theta$, we see that \eqref{eq:Stinespring} above is bounded above by
$$
\|h\|_{\cM_\bfA}^2
\,\,
\begin{tikzpicture}[baseline=-.1cm]
	\draw (.5,1.35) .. controls ++(90:.34cm) and ++(90:.34cm) .. (-.5,1.35) -- (-.5,-1.35) .. controls ++(270:.34cm) and ++(270:.34cm) .. (.5,-1.35) -- (.5,1.35);
	\draw (0,1.6) .. controls ++(90:.6cm) and ++(90:.6cm) .. (-2,1.6) -- (-2,-1.6) .. controls ++(270:.6cm) and ++(270:.6cm) .. (0,-1.6);
	\draw (-1,2.05) -- (-1,3)  .. controls ++(90:.5cm) and ++(90:.5cm) .. (1.4,3) -- (1.4,-2.4)  .. controls ++(270:.5cm) and ++(270:.5cm) .. (-1,-2.4) -- (-1,-2.05);
	\roundNbox{unshaded}{(-1,2.7)}{.3}{0}{0}{$\varphi_c$}
	\node at (-1.2,3.2) {\scriptsize{$\overline{\overline{\mathbf{c}}}$}};
	\node at (1.6,0) {\scriptsize{$\overline{\mathbf{c}}$}};
	\roundNbox{unshaded}{(.5,1)}{.25}{0}{0}{$\xi^*$}
	\roundNbox{unshaded}{(0,0)}{.3}{.5}{.5}{$\theta(g^*g)$}
	\roundNbox{unshaded}{(.5,-1)}{.25}{0}{0}{$\xi$}
	\filldraw[fill=white] (-1,2.05) circle (.05cm) node [below] {\scriptsize{$\beta^*$}};
	\filldraw[fill=white] (-1,-2.05) circle (.05cm) node [above] {\scriptsize{$\beta$}};
	\filldraw[fill=white] (0,1.6) circle (.05cm) node [below] {\scriptsize{$\alpha^*$}};
	\filldraw[fill=white] (0,-1.6) circle (.05cm) node [above] {\scriptsize{$\alpha$}};
	\node at (-2.2,0) {\scriptsize{$\mathbf{d}$}};
	\node at (-1.2,2.2) {\scriptsize{$\mathbf{c}$}};
	\node at (-1.2,-2.2) {\scriptsize{$\mathbf{c}$}};
	\node at (.2,1.8) {\scriptsize{$\mathbf{e}$}};
	\node at (.2,-1.8) {\scriptsize{$\mathbf{e}$}};
	\node at (.7,-1.4) {\scriptsize{$\mathbf{b}$}};
	\node at (-.7,-1.4) {\scriptsize{$\mathbf{a}$}};
	\node at (.7,1.4) {\scriptsize{$\mathbf{b}$}};
	\node at (-.7,1.4) {\scriptsize{$\mathbf{a}$}};
	\node at (.7,-.5) {\scriptsize{$\bfH$}};
	\node at (.7,.5) {\scriptsize{$\bfH$}};
\end{tikzpicture}
\leq
\|h\|_{\cM_\bfA}^2
\|\beta\otimes g\otimes \alpha \otimes \xi\|^2_{(\mathbf{d}\otimes \bfK)(d)}.
$$
Hence $\pi_d(h): \mathbf{d}\otimes \bfK\Rightarrow \bfK$ is a bounded natural transformation with norm at most $\|h\|_{\cM_\bfA}$.
We omit the proof that $\pi$ is a $*$-algebra natural transformation, which is similar to the proof of the same fact in Theorem \eqref{thm:GelfandNaimark}.

Finally, we verify that $\Ad(v)\circ \pi = \theta$.
Indeed, for $h\in \bfA(a)$ and $\xi\in \bfH(b)$, we have $\pi_a(h)\in \bfB(\bfK)(a)=\Hom_{\Hilb(\cC)}(\mathbf{a}\otimes \bfK , \bfK)$ and $v(\xi)\in \bfK(b)$.
For $c\in \Irr(\cC)$ and $\alpha \in \cC(c, a\otimes b)$, we have
\begin{align*}
[(\Ad(v)\circ \pi)_a(h)]_c(\alpha\otimes \xi) 
&= 
v^*[\pi_a(h)_c (\alpha\otimes v(\xi))] 
=
v^*(\pi_a(h)(\alpha\otimes i_\bfA\otimes \id_{1_\cC}\otimes \xi))
\\&=
v^*\left(
\begin{tikzpicture}[baseline = -.4cm]
    \draw (0,.6) -- (0,-.3);
    \draw (1,0) -- (1,.6);
    \draw (.5,-.8) -- (.5,-1.2);
    \draw (0,-.3) arc (-180:0:.5cm);
    \filldraw[fill=white] (.5,-.8) circle (.05cm) node [above] {\scriptsize{$\alpha$}};
    \roundNbox{unshaded}{(0,0)}{.3}{0}{0}{$h$}
    \roundNbox{unshaded}{(1,0)}{.3}{0}{0}{$\xi$}
    \node at (-.15,-.5) {\scriptsize{$\mathbf{a}$}};
    \node at (1.15,-.5) {\scriptsize{$\mathbf{b}$}};
    \node at (.35,-1) {\scriptsize{$\mathbf{c}$}};
    \node at (1.2,.5) {\scriptsize{$\bfH$}};
    \node at (-.2,.5) {\scriptsize{$\bfA$}};
\end{tikzpicture}
\right)
=
\begin{tikzpicture}[baseline=-.8cm]
	\draw (-.25,0) -- (-.25,-1.35) arc (-180:0:.25cm) -- (.25,0);
	\draw (0,-1.6) -- (0,-2);
	\draw (0,.3) -- (0,.7);
	\roundNbox{unshaded}{(0,0)}{.3}{.2}{.2}{$\theta(h)$}
	\roundNbox{unshaded}{(.25,-1)}{.25}{0}{0}{$\xi$}
	\filldraw[fill=white] (0,-1.6) circle (.05cm) node [above] {\scriptsize{$\alpha$}};
	\node at (.4,-1.4) {\scriptsize{$\mathbf{b}$}};
	\node at (-.4,-1.4) {\scriptsize{$\mathbf{a}$}};
	\node at (.45,-.5) {\scriptsize{$\bfH$}};
	\node at (.2,.5) {\scriptsize{$\bfH$}};
	\node at (.2,-1.8) {\scriptsize{$\mathbf{c}$}};
\end{tikzpicture}
=
\theta_a(h) (\alpha\otimes \xi).
\end{align*}
By Lemma \ref{lem:ZeroInHilbC} applied to $(\Ad(v)\circ \pi)_a(h) - \theta_a(h)$, we are finished.
\end{proof}

\subsection{States and the GNS construction}
\label{sec:StateAndGNS}

\begin{defn}
A state on $\bfA$ is a ucp map $\phi:\bfA\to {\mathbf{1}}$ where ${\mathbf{1}}=\cC(\,\cdot\,, 1_\cC): \cC^{\op}\to \Hilb$.
Notice that a state on $\bfA$ correspond uniquely to states on $\bfA(1_\cC)$, since for all $c\in \Irr(\cC)$, $\mathbf{1}(c) = \delta_{c=1_\cC} \bbC$.
Thus $\phi$ is completely determined by the map $\phi_{1_\cC}:\bfA(1_\cC) \to \bbC$, which must be a state since $\phi$ is ucp.
\end{defn}

We now give two equivalent definitions of the GNS Hilbert space representation of $\bfA$ with respect to a state $\phi$.

First, given a state $\phi$ on $\bfA$, we can define Hilbert space objects $L^{2}(\mathbf{A})_\phi$ and ${}_\phi L^{2}(\mathbf{A})\in \Hilb(\cC)$ by defining the fiber Hilbert spaces at $a\in \cC$ to be the completion of $\bfA(a)$ with the inner products on $\bfA(a)$ given respecively by
$$
\langle f, g\rangle_{L^2(\bfA)_\phi}=\phi(\langle g| f\rangle_a)
\qquad\qquad
\langle f, g\rangle_{{}_\phi L^2(\bfA)}=\phi({}_a\langle f, g\rangle).
$$
Notice that we have a left and a right $L^2(\bfA)$, similar to the left and right $\cC$-module categories associated to $\bfA$.

Equivalently, when we identify a state $\phi$ on $\bfA$ with a state on $\bfA(1_\cC)$, we can first apply the GNS construction for C*-algebras to obtain a representation of $\mathbf{A}(1_\cC)$ on $L^{2}(\mathbf{A}(1_\cC),\phi)$, and then apply our Gelfand-Naimark Theorem \ref{thm:GelfandNaimark} in $\cC$ to obtain a Hilbert space object $L^{2}(\mathbf{A})_\phi\in \Hilb(\cC)$ and a representation $\mathbf{A}\Rightarrow \bfB(L^{2}(\mathbf{A}))_\phi$.
In other words, we get a cyclic $\cC$-module dagger functor $\bfF_{\phi}:\cM_{\mathbf{A}}\rightarrow \cM_{L^{2}(\mathbf{A})_\phi}$, analogous to the GNS construction. 

Notice that the second description gives an isomorphic Hilbert space object.
If $\bfH$ is as in our Gelfand-Naimark Theorem \ref{thm:GelfandNaimark} in $\cC$ for the $\bfA(1_\cC)$-module $L^2(\bfA(1_\cC), \phi)$, we have $\bfH(a)$ is the completion of $\bfA(a) \otimes_{\bfA(1_\cC)} L^2(\bfA(1_\cC), \phi)$ using the relative tensor product inner product.

\begin{rem}
The right $L^2(\bfA)_\phi$ carries a left $\bfA(1_\cC)$ action, whereas the left ${}_\phi L^2(\bfA)$ carries a right $\bfA(1_\cC)$ action.
In general, the left and right $L^2(\bfA)$ are different, and their relationship is complicated.
We anticipate they are related by a version of the Tomita-Takesaki theory in our $\cC$-graded setting.
We defer this exploration to a future article.
\end{rem}

We can, however, relate the left and right $L^2(\bfA)$ by taking conjugates.

\begin{prop}
\label{prop:StateConjugateLeftRight}
The map $\bfA(a)\ni f\mapsto \overline{j_a(f)}\in \overline{\bfA(\overline{a})}$ extends to an isomorphism of Hilbert space objects ${}_\phi L^2(\bfA) \cong \overline{L^2(\bfA)_\phi}$ in $\Hilb(\cC)$.
\end{prop}
\begin{proof}
It suffices to show the map is isometric.
Given $f,g\in \bfA(a)$, we have
\begin{align*}
\langle \overline{j_a(f)}, \overline{j_a(g)}\rangle_{\overline{L^2(\bfA)_\phi}}
&=
\langle \overline{j_a(f)}, \overline{j_a(g)}\rangle_{\overline{L^2(\bfA)_\phi(\overline{a})}}
=
\langle j_a(g), j_a(f)\rangle_{L^2(\bfA)_\phi(\overline{a})}
\\&=
\phi(\langle j_a(f)| j_a(g)\rangle_{\overline{a}})
=
\phi
\left(
\begin{tikzpicture}[baseline = .1cm]
    \draw (-.5,-.3) arc (-180:0:.5cm);
    \draw (-.5,.3) arc (180:0:.5cm);
    \filldraw (0,.8) circle (.05cm);
    \draw (0,.8) -- (0,1.2);
    \roundNbox{unshaded}{(-.5,0)}{.3}{0}{0}{$f$}
    \roundNbox{unshaded}{(.5,0)}{.3}{.2}{.2}{$j_a(g)$}
    \node at (-.65,-.5) {\scriptsize{$\mathbf{a}$}};
    \node at (-.65,.5) {\scriptsize{$\bfA$}};
    \node at (.65,-.5) {\scriptsize{$\overline{\mathbf{a}}$}};
    \node at (.65,.5) {\scriptsize{$\bfA$}};
    \node at (.15,1) {\scriptsize{$\bfA$}};
\end{tikzpicture}
\right)
=
\phi({}_{a}\langle f, g\rangle)
=
\langle f, g\rangle_{{}_\phi L^2(\bfA)}.
\qedhere
\end{align*}

\end{proof}

We now focus on tracial states.

\begin{defn}
A \emph{trace} on the C*-algebra object $\bfA\in \Vec(\cC)$ is a state $\tau$ on $\bfA$ such that for all $a\in\cC$ and $f,g\in\bfA(a)$,
$$
\tau({}_a\langle f, g \rangle)
=
\tau
\left(
\begin{tikzpicture}[baseline = .1cm]
    \draw (-.5,-.3) arc (-180:0:.5cm);
    \draw (-.5,.3) arc (180:0:.5cm);
    \filldraw (0,.8) circle (.05cm);
    \draw (0,.8) -- (0,1.2);
    \roundNbox{unshaded}{(-.5,0)}{.3}{0}{0}{$f$}
    \roundNbox{unshaded}{(.5,0)}{.3}{.2}{.2}{$j_a(g)$}
    \node at (-.65,-.5) {\scriptsize{$\mathbf{a}$}};
    \node at (-.65,.5) {\scriptsize{$\bfA$}};
    \node at (.65,-.5) {\scriptsize{$\overline{\mathbf{a}}$}};
    \node at (.65,.5) {\scriptsize{$\bfA$}};
    \node at (.15,1) {\scriptsize{$\bfA$}};
\end{tikzpicture}
\right)
=
\tau\left(
\begin{tikzpicture}[baseline = .1cm, xscale=-1]
    \draw (-.5,-.3) arc (-180:0:.5cm);
    \draw (-.5,.3) arc (180:0:.5cm);
    \filldraw (0,.8) circle (.05cm);
    \draw (0,.8) -- (0,1.2);
    \roundNbox{unshaded}{(-.5,0)}{.3}{0}{0}{$f$}
    \roundNbox{unshaded}{(.5,0)}{.3}{.2}{.2}{$j_a(g)$}
    \node at (-.65,-.5) {\scriptsize{$\mathbf{a}$}};
    \node at (-.65,.5) {\scriptsize{$\bfA$}};
    \node at (.65,-.5) {\scriptsize{$\overline{\mathbf{a}}$}};
    \node at (.65,.5) {\scriptsize{$\bfA$}};
    \node at (.15,1) {\scriptsize{$\bfA$}};
\end{tikzpicture}
\right)
=
\tau(\langle g|f \rangle_a).
$$
\end{defn}

The following corollaries are now immediate.

\begin{cor}
\label{cor:TraceLeftRight}
If $\tau$ is a trace on $\bfA$, then $L^2(\bfA)_\tau = {}_\tau L^2(\bfA)$ in $\Hilb(\cC)$, since the left and right inner products on each $\bfA(a)$ agree on the nose.
\end{cor}

\begin{cor}
If $\tau$ is a trace on $\bfA$, then $L^2(\bfA)_\tau$ is symmetric with $\sigma: L^2(\bfA)_\tau \Rightarrow \overline{L^2(\bfA)_\tau}$ given by the map from Proposition \ref{prop:StateConjugateLeftRight} together with Corollary \ref{cor:TraceLeftRight}.
\end{cor}
\begin{proof}
It remains to show that $\sigma$ is involutive, i.e., $\varphi = \overline{\sigma}\circ \sigma$.
Indeed, since complex conjugation on $\Vec$ is strictly involutive, we see that for $f\in \bfA(a)$, 
$$
\overline{\sigma}(\sigma(f)) 
= 
\overline{\sigma}(\overline{j_a(f)}) 
=
(j_{\overline{a}}\circ j_a) (f) 
= 
\bfA(\varphi^{-1}_a)(f)
=
(\varphi_{L^2(\bfA)_\tau})_a (f),
$$
where $\varphi_{L^2(\bfA)_\tau}$ was given in Definition \ref{defn:VecCInvolutive}.
\end{proof}

Given a trace $\tau$ on $\bfA$, we get left and right GNS representations of $\bfA$ on $L^2(\bfA)$, where by Corollary \ref{cor:TraceLeftRight}, we do not need to specify the left or the right $L^2(\bfA)$.
First, for each $a\in\cC$, we denote the image of $\bfA(a)$ in $L^2(\bfA)(a)$ by $\widehat{\bfA}(a)$.
For $f\in \bfA(a)$, we get bounded natural transformations 
$\lambda(f): \mathbf{a}\otimes L^2(\bfA)\Rightarrow L^2(\bfA)$ 
and
$\rho(f): L^2(\bfA)\otimes \mathbf{a} \Rightarrow L^2(\bfA)$
given on their $c$-components for $c\in \Irr(\cC)$ by
\begin{equation}
\label{eq:LeftAndRightActionOnL2A}
\begin{split}
\cC(c, a\otimes b)\boxtimes \widehat{\bfA}(b)
\ni 
\alpha\boxtimes \xi
&\overset{\lambda(f)_c}{\longmapsto}
\begin{tikzpicture}[baseline=-.1cm]
	\draw (0,.8) -- (0,1.2);
	\draw (.5,.3) arc (0:180:.5cm);
	\draw (.5,-.3) arc (0:-180:.5cm);
	\draw (0,-.8) -- (0,-1.2);
	\filldraw (0,.8) circle (.05cm);
	\filldraw[fill=white] (0,-.8) circle (.05cm) node [above] {\scriptsize{$\alpha$}};
	\roundNbox{unshaded}{(-.5,0)}{.3}{0}{0}{$f$}
	\roundNbox{unshaded}{(.5,0)}{.3}{0}{0}{$\xi$}
	\node at (.2,1) {\scriptsize{$\widehat{\bfA}$}};
	\node at (.7,.5) {\scriptsize{$\widehat{\bfA}$}};
	\node at (-.7,.5) {\scriptsize{$\bfA$}};
	\node at (.7,-.5) {\scriptsize{$\mathbf{b}$}};
	\node at (-.7,-.5) {\scriptsize{$\mathbf{a}$}};
	\node at (.2,-1) {\scriptsize{$\mathbf{c}$}};
\end{tikzpicture}
\in \widehat{\bfA}(c)
\\
\widehat{\bfA}(b)\boxtimes \cC(c, b\otimes a)
\ni
\xi\boxtimes \beta
&\overset{\rho(f)_c}{\longmapsto}
\begin{tikzpicture}[baseline=-.1cm]
	\draw (0,.8) -- (0,1.2);
	\draw (.5,.3) arc (0:180:.5cm);
	\draw (.5,-.3) arc (0:-180:.5cm);
	\draw (0,-.8) -- (0,-1.2);
	\filldraw (0,.8) circle (.05cm);
	\filldraw[fill=white] (0,-.8) circle (.05cm) node [above] {\scriptsize{$\beta$}};
	\roundNbox{unshaded}{(-.5,0)}{.3}{0}{0}{$\xi$}
	\roundNbox{unshaded}{(.5,0)}{.3}{0}{0}{$f$}
	\node at (.2,1) {\scriptsize{$\widehat{\bfA}$}};
	\node at (-.7,.5) {\scriptsize{$\widehat{\bfA}$}};
	\node at (.7,.5) {\scriptsize{$\bfA$}};
	\node at (.7,-.5) {\scriptsize{$\mathbf{a}$}};
	\node at (-.7,-.5) {\scriptsize{$\mathbf{b}$}};
	\node at (.2,-1) {\scriptsize{$\mathbf{c}$}};
\end{tikzpicture}
\in \widehat{\bfA}(c)
\end{split}
\end{equation}

\begin{prop}
\label{prop:TwoCommutingActionsOnL2A}
The left and right actions of $\bfA$ on $L^2(\bfA)$ commute up to the associator in $\Hilb(\cC)$.
\end{prop}
\begin{proof}
First, note that the associativity of $(\bfA,\mu, \iota)$ involves the associator $\alpha_{\bfA,\bfA,\bfA}$ in $\Vec(\cC)$ described in \eqref{eq:Associativity}, which gives a unitary $\alpha_{L^2(\bfA), L^2(\bfA), L^2(\bfA)}$ in $\Hilb(\cC)$.
This means for all $\eta\in \bfA(a)$, $\xi\in \bfA(d)$, $\zeta\in \bfA(e)$ and $\alpha\in \cC(c, a\otimes b)$, $\beta\in \cC(b, d\otimes e)$, we have
\begin{equation}
\begin{tikzpicture}[baseline = -.1cm, xscale=-1]
	\draw (-1,-.3) arc (-180:0:.5cm);
	\draw (-.5,-.8) .. controls ++(270:.6cm) and ++(270:.6cm) .. (1,-.8) -- (1,-.3);
	\draw (-1,.3) arc (180:0:.5cm);
	\draw (-.5,.8) .. controls ++(90:.6cm) and ++(90:.6cm) .. (1,.8) -- (1,.3);
	\draw (.25,-1.25) -- (.25,-1.6);
	\draw (.25,1.25) -- (.25,1.6);
	\filldraw[fill=white] (-.5,-.8) circle (.05cm) node [above] {\scriptsize{$\beta$}};
	\filldraw[fill=white] (.25,-1.25) circle (.05cm) node [above] {\scriptsize{$\alpha$}};
	\filldraw (-.5,.8) circle (.05cm);
	\filldraw (.25,1.25) circle (.05cm);
    \roundNbox{unshaded}{(-1,0)}{.3}{0}{0}{$\zeta$}
    \roundNbox{unshaded}{(0,0)}{.3}{0}{0}{$\xi$}
    \roundNbox{unshaded}{(1,0)}{.3}{0}{0}{$\eta$}
    \node at (-1.15,-.5) {\scriptsize{$\mathbf{e}$}};
    \node at (.15,-.5) {\scriptsize{$\mathbf{d}$}};
    \node at (.85,-.5) {\scriptsize{$\mathbf{a}$}};
    \node at (-1.15,.5) {\scriptsize{$\bfA$}};
    \node at (.15,.5) {\scriptsize{$\bfA$}};
    \node at (.85,.5) {\scriptsize{$\bfA$}};
    \node at (.4,1.45) {\scriptsize{$\bfA$}};
    \node at (.4,-1.45) {\scriptsize{$\mathbf{c}$}};
    \node at (-.65,1) {\scriptsize{$\bfA$}};
    \node at (-.65,-1) {\scriptsize{$\mathbf{b}$}};
\end{tikzpicture}
=
\sum_{\substack{
f\in\Irr(\cC)
\\
\gamma\in \ONB(f,a\otimes d)
\\
\delta\in \ONB(c,f\otimes e)
}}
U_{\alpha,\beta}^{\gamma,\delta}\,\,
\begin{tikzpicture}[baseline = -.1cm]
	\draw (-1,-.3) arc (-180:0:.5cm);
	\draw (-.5,-.8) .. controls ++(270:.6cm) and ++(270:.6cm) .. (1,-.8) -- (1,-.3);
	\draw (-1,.3) arc (180:0:.5cm);
	\draw (-.5,.8) .. controls ++(90:.6cm) and ++(90:.6cm) .. (1,.8) -- (1,.3);
	\draw (.25,-1.25) -- (.25,-1.6);
	\draw (.25,1.25) -- (.25,1.6);
	\filldraw[fill=white] (-.5,-.8) circle (.05cm) node [above] {\scriptsize{$\gamma$}};
	\filldraw[fill=white] (.25,-1.25) circle (.05cm) node [above] {\scriptsize{$\delta$}};
	\filldraw (-.5,.8) circle (.05cm);
	\filldraw (.25,1.25) circle (.05cm);
    \roundNbox{unshaded}{(-1,0)}{.3}{0}{0}{$\eta$}
    \roundNbox{unshaded}{(0,0)}{.3}{0}{0}{$\xi$}
    \roundNbox{unshaded}{(1,0)}{.3}{0}{0}{$\zeta$}
    \node at (-1.15,-.5) {\scriptsize{$\mathbf{a}$}};
    \node at (.15,-.5) {\scriptsize{$\mathbf{d}$}};
    \node at (.85,-.5) {\scriptsize{$\mathbf{e}$}};
    \node at (-1.15,.5) {\scriptsize{$\bfA$}};
    \node at (.15,.5) {\scriptsize{$\bfA$}};
    \node at (.85,.5) {\scriptsize{$\bfA$}};
    \node at (.4,1.45) {\scriptsize{$\bfA$}};
    \node at (.4,-1.45) {\scriptsize{$\mathbf{c}$}};
    \node at (-.65,1) {\scriptsize{$\bfA$}};
    \node at (-.65,-1) {\scriptsize{$\mathbf{f}$}};
\end{tikzpicture}\,.
\end{equation}
Consider 
$\lambda(\eta)\in \Hom_{\Hilb(\cC)}(\mathbf{a} \otimes L^2(\bfA), L^2(\bfA))$ for $\eta\in \bfA(a)$ 
and 
$\rho(\zeta)\in \Hom_{\Hilb(\cC)}(L^2(\bfA)\otimes \mathbf{e}, L^2(\bfA))$ for $\zeta\in \bfA(e)$.
On the $c$-component
$\cC(c,a\otimes b)\boxtimes (\widehat{\bfA}(d)\boxtimes \cC(b, d\otimes e))$
of $\mathbf{a}\otimes (L^2(\bfA) \otimes \mathbf{e})$
for $c\in\Irr(\cC)$, we have
$$
(\lambda(\eta)\otimes \id_{\mathbf{e}})_{c}\circ (\id_{\mathbf{a}}\otimes\rho(\eta))_c
=
(\id_{\mathbf{a}}\otimes\rho(\eta))_c \circ(\lambda(\eta)\otimes \id_{\mathbf{e}})_{c}\circ(\alpha_{L^2(\bfA), L^2(\bfA), L^2(\bfA)})_c.
$$
In string diagrams, this is exactly the commutation relation
\begin{equation}
\label{eq:Commutation}
\begin{tikzpicture}[baseline=-.1cm]
	\draw (0,-1.1) -- (0,1.1);
	\draw[thick, red] (-.6,-1.1) -- (-.6,-.5);
	\draw[thick, blue] (.6,-1.1) -- (.6,.5);
	\roundNbox{unshaded}{(.1,.5)}{.3}{0}{.4}{$\zeta$}
	\roundNbox{unshaded}{(-.1,-.5)}{.3}{.4}{0}{$\eta$}
	\node at (0,-1.3) {\scriptsize{$L^2(\bfA)$}};
	\node at (0,1.3) {\scriptsize{$L^2(\bfA)$}};
	\node at (.6,-1.3) {\scriptsize{$\mathbf{e}$}};
	\node at (-.6,-1.3) {\scriptsize{$\mathbf{a}$}};
\end{tikzpicture}
=
\begin{tikzpicture}[baseline=-.1cm]
	\draw (0,-1.1) -- (0,1.1);
	\draw[thick, red] (-.6,-1.1) -- (-.6,.5);
	\draw[thick, blue] (.6,-1.1) -- (.6,-.5);
	\roundNbox{unshaded}{(.1,-.5)}{.3}{0}{.4}{$\zeta$}
	\roundNbox{unshaded}{(-.1,.5)}{.3}{.4}{0}{$\eta$}
	\node at (0,-1.3) {\scriptsize{$L^2(\bfA)$}};
	\node at (0,1.3) {\scriptsize{$L^2(\bfA)$}};
	\node at (.6,-1.3) {\scriptsize{$\mathbf{e}$}};
	\node at (-.6,-1.3) {\scriptsize{$\mathbf{a}$}};
\end{tikzpicture}
\end{equation}
where we suppress the associator in $\Hilb(\cC)$.
\end{proof}

Now that we have commuting actions of the algebra object $\bfA$ on a Hilbert space object, we can discuss bimodules and commutants.
Formalizing these notions provides a segue into Section \eqref{sec:WStarAlgebras} on W*-algebra objects and the von Neumann bicommutant theorem in $\cC$.



\section{W*-algebra objects in \texorpdfstring{$\Vec(\cC)$}{Vec(C)}}
\label{sec:WStarAlgebras}

We now focus our attention on W*-algebra objects in $\Vec(\cC)$.
We prove analogues of theorems for ordinary W*-algebras in $\fdHilb$.
Recall from Theorem \ref{thm:W*-algebraCorrespondence} that normality/weak* continuity is required for morphisms between W*-algebra objects and between cyclic $\cC$-module W*-categories.

\subsection{Commutants in bimodule categories}

\begin{rem}
One notion of a commutant for a tensor category $\cC$ and a left $\cC$-module category $\cM$ is the category $\cC'_{\cM}=\End_\cC(\cM)$, the left $\cC$-module endofunctors of $\cM$.
Notice that this is the \emph{categorified} commutant, where we think of $\cC$ as a categorified ring, and thus the commutant $\cC'_\cM$ is again a tensor category.
This is the wrong categoricial level for this article.
We want the commutant of an algebra object $\bfA$ inside $\bfB(\bfH)$, which corresponds to the cyclic $\cC$-module C*-category $\cM_\bfA$ inside $\cM_\bfH$, to be another algebra object.
Thus we want a definition of commutant which gives us another module category, not a tensor category.
\end{rem}

The above discussion, together with the commutation relation \eqref{eq:Commutation}, motivates us to take the commutant of a cyclic $\cC$-module dagger category inside a cyclic $\cC-\cC$ bimodule W*-category.

\begin{defn}
Suppose $\bfH\in \Hilb(\cC)$.
Define the cyclic $\cC-\cC$ bimodule W*-category $\cB_\bfH$ to be the full W*-subcategory of $\Hilb(\cC)$ generated by $\bfH$, i.e., the objects are of the form $\mathbf{a}\otimes \bfH \otimes \mathbf{b}$ for $a,b\in\cC$.
\end{defn}

\begin{defn}
\label{defn:CommutantModuleCategory}
Suppose $(\cM,m)$ is a cyclic left $\cC$-module dagger category, and $\Phi:(\cM, m)\to (\cM_\bfH, \bfH)$ is a cyclic $\cC$-module dagger functor.
The right commutant $\Phi(\cM)'$ is the cyclic right $\cC$-module W*-category whose objects are the left $\cC$-module dagger functors $\bfR_c:=-\otimes \mathbf{c}:\cM_\bfH \to \cB_\bfH$ for $c\in \cC$, and whose morphisms are bounded left $\cC$-module dagger natural transformations which commute with morphisms from $\Phi(\cM)$.
Here, the basepoint is $\bfR_{1_\cC}$
These morphisms are described explicitly below in Lemma \ref{lem:NTbetweenRightCreation}.
Notice $\Phi(\cM)'$ is obviously a right $\cC$-module category with $\cC$-action given by $\bfR_a \otimes c :=\bfR_{a\otimes c}$.
It follows from Corollary \ref{cor:IdentifyMPrime} below that the right commutant is a W*-category.

Similarly, there is a notion of a left commutant of a cyclic right $\cC$-module C*-category, which is a cyclic left $\cC$-module W*-category.
This is defined using left creation functors $\bfL_c:\cM_\bfH \to \cB_\bfH$ for $c\in\cC$.

Starting with a left module $(\cM,m)$ and a cyclic $\cC$-module dagger functor $\Phi: (\cM,m) \to (\cM_\bfH, \bfH)$, the right commutant $\Phi(\cM)'$ has a canonical normal cyclic right $\cC$-module dagger functor $(\Phi(\cM)', \bfH) \rightarrow ({}_\bfH\cM, \bfH)$. 
The bicommutant $\Phi(\cM)''$ is the left commutant of the right commutant of $\Phi(\cM)'$.
There is a similar notion of bicommutant of a right module.

If $\theta: \bfA\to \bfB(\bfH)$ is a representation of a C*-algebra object in $\bfA\in \Vec(\cC)$, we define its commutant $\theta(\bfA)'\in\Vec(\cC)$ to be the W*-algebra object from Remark \ref{rem:RightVersionOfConstructionAlgebra} corresponding to the right commutant $\check{\theta}(\cM_\bfA)'$ where $\check{\theta}:(\cM_\bfA,1_\bfA) \to (\cM_\bfH, \bfH)$ is the corresponding cyclic $\cC$-module dagger functor from Construction \ref{construction:Module}.
\end{defn}

\begin{lem}
\label{lem:NTbetweenRightCreation}
Any $\cC$-module natural transformation $(f_c)_{c\in \cC}: \bfR_a\Rightarrow \bfR_b$ for $a,b\in\cC$ is of the form $f_c = \id_c\otimes f$ for some fixed $f\in \Hom_{\Hilb(\cC)}(\bfH\otimes\mathbf{a}, \bfH\otimes\mathbf{b})$ independent of $c\in\cC$.
\end{lem}
\begin{proof}
Suppose we have a natural transformation $(f_c)_{c\in\cC}: \bfR_a\Rightarrow \bfR_b$.
We observe that $f_c = \id_{\mathbf{c}} \otimes f_c'$, where $f_c'$ is the normalized left partial trace of $f_c$:
$$
\begin{tikzpicture}[baseline=-.1cm]
	\draw[thick, DarkGreen] (-.1,-.5) -- (-.1,.5);
	\draw[thick, red] (.4,-.5) -- (.4,0);
	\draw[thick, blue] (.4,0) -- (.4,.5);
	\draw (.15,-.5) -- (.15,.5);
	\roundNbox{unshaded}{(0,0)}{.25}{0}{.3}{$f_c$}
	\node at (.4,.7) {\scriptsize{$\mathbf{b}$}};
	\node at (.4,-.7) {\scriptsize{$\mathbf{a}$}};
	\node at (-.1,-.7) {\scriptsize{$\mathbf{c}$}};
	\node at (.15,-.7) {\scriptsize{$\bfH$}};
\end{tikzpicture}
=
\begin{tikzpicture}[baseline=.4cm, xscale=-1]
	\draw (-.1,-.7) -- (-.1,1.7);
	\draw[thick, blue] (-.3,0) -- (-.3,1.7);
	\draw[thick, red] (-.3,-.7) -- (-.3,0);
	\draw[thick, DarkGreen] (.1,-.7) -- (.1,.7) arc (180:0:.15cm) -- (.4,-.4) arc (-180:0:.15cm) -- (.7,1.4) arc (0:180:.15cm) -- (.4,1.3) arc (0:-180:.15cm) -- (.1,1.7);
	\roundNbox{unshaded}{(0,0)}{.25}{.2}{0}{$f_c$}
	\roundNbox{dashed}{(0,0)}{.4}{.2}{.5}{}
	\roundNbox{dashed}{(0,1)}{.4}{.2}{.5}{}
\end{tikzpicture}
=
\begin{tikzpicture}[baseline=.4cm, xscale=-1]
	\draw (-.1,-.7) -- (-.1,1.7);
	\draw[thick, blue] (-.3,0) -- (-.3,1.7);
	\draw[thick, red] (-.3,-.7) -- (-.3,0);
	\draw[thick, DarkGreen] (.1,-.7) -- (.1,.7) arc (180:0:.15cm) -- (.4,-.4) arc (-180:0:.15cm) -- (.7,1.4) arc (0:180:.15cm) -- (.4,1.3) arc (0:-180:.15cm) -- (.1,1.7);
	\roundNbox{unshaded}{(0,0)}{.4}{.2}{.5}{$f_{c\otimes \overline{c}\otimes c}$}
	\roundNbox{dashed}{(0,1)}{.4}{.2}{.5}{}
\end{tikzpicture}
=
\dim_\cC(c)^{-1}\,
\begin{tikzpicture}[baseline=-.1cm, xscale=-1]
	\draw (-.1,-1.7) -- (-.1,1.7);
	\draw[thick, red] (-.3,-1.7) -- (-.3,0);
	\draw[thick, blue] (-.3,0) -- (-.3,1.7);
	\draw[thick, DarkGreen] (.1,.25) -- (.1,.7) arc (180:0:.15cm) -- (.4,-.7) arc (0:-180:.15cm) -- (.1,-.25);
	\draw[thick, DarkGreen] (.1,-1.7) -- (.1,-1.3) arc (180:0:.15cm) -- (.4,-1.4) arc (-180:0:.15cm) -- (.7,1.4) arc (0:180:.15cm) -- (.4,1.3) arc (0:-180:.15cm) -- (.1,1.7);
	\roundNbox{dashed}{(0,-1)}{.4}{.2}{.5}{}
	\roundNbox{unshaded}{(0,0)}{.4}{.2}{.5}{$f_{c\otimes \overline{c}\otimes c}$}
	\roundNbox{dashed}{(0,1)}{.4}{.2}{.5}{}
\end{tikzpicture}
=
\dim_\cC(c)^{-1}\,
\begin{tikzpicture}[baseline=-.1cm, xscale=-1]
	\draw (-.1,-1.7) -- (-.1,1.7);
	\draw[thick, red] (-.3,-1.7) -- (-.3,0);
	\draw[thick, blue] (-.3,0) -- (-.3,1.7);
	\draw[thick, DarkGreen] (.1,.25) -- (.1,.7) arc (180:0:.15cm) -- (.4,-.7) arc (0:-180:.15cm) -- (.1,-.25);
	\draw[thick, DarkGreen] (.1,-1.7) -- (.1,-1.3) arc (180:0:.15cm) -- (.4,-1.4) arc (-180:0:.15cm) -- (.7,1.4) arc (0:180:.15cm) -- (.4,1.3) arc (0:-180:.15cm) -- (.1,1.7);
	\roundNbox{unshaded}{(0,0)}{.25}{.2}{0}{$f_c$}
	\roundNbox{dashed}{(0,-1)}{.4}{.2}{.5}{}
	\roundNbox{dashed}{(0,1)}{.4}{.2}{.5}{}
\end{tikzpicture}
=
\dim_\cC(c)^{-1}\,
\begin{tikzpicture}[baseline=-.1cm]
	\draw (.1,-.5) -- (.1,.5);
	\draw[thick, red] (.3,-.5) -- (.3,0);
	\draw[thick, blue] (.3,0) -- (.3,.5);
	\draw[thick, DarkGreen] (-.1,-.25) arc (0:-180:.15cm) -- (-.4,.25) arc (180:0:.15cm);
	\draw[thick, DarkGreen] (-.55,-.5) -- (-.55,.5);
	\roundNbox{unshaded}{(0,0)}{.25}{0}{.2}{$f_c$}
\end{tikzpicture}
\,.
$$
Finally, a similar calculation shows that $f_c'=f_d'$ for all $c,d\in\cC$.
We have
\begin{align*}
\id_{\mathbf{d}} \otimes f_c' 
&=
\dim_\cC(c)^{-1}\,
\begin{tikzpicture}[baseline=-.1cm]
	\draw (.1,-.5) -- (.1,.5);
	\draw[thick, red] (.3,-.5) -- (.3,0);
	\draw[thick, blue] (.3,0) -- (.3,.5);
	\draw[thick, DarkGreen] (-.1,-.25) arc (0:-180:.15cm) -- (-.4,.25) arc (180:0:.15cm);
	\draw[thick, orange] (-.55,-.5) -- (-.55,.5);
	\roundNbox{unshaded}{(0,0)}{.25}{0}{.2}{$f_c$}
\end{tikzpicture}
= 
\dim_\cC(c)^{-1}\,
\begin{tikzpicture}[baseline=-.1cm, xscale=-1]
	\draw (-.1,-1.7) -- (-.1,1.7);
	\draw[thick, red] (-.3,-1.7) -- (-.3,0);
	\draw[thick, blue] (-.3,0) -- (-.3,1.7);
	\draw[thick, DarkGreen] (.1,.25) -- (.1,.7) arc (180:0:.15cm) -- (.4,-.7) arc (0:-180:.15cm) -- (.1,-.25);
	\draw[thick, orange] (.1,-1.7) -- (.1,-1.3) arc (180:0:.15cm) -- (.4,-1.4) arc (-180:0:.15cm) -- (.7,1.4) arc (0:180:.15cm) -- (.4,1.3) arc (0:-180:.15cm) -- (.1,1.7);
	\roundNbox{unshaded}{(0,0)}{.25}{.2}{0}{$f_c$}
	\roundNbox{dashed}{(0,-1)}{.4}{.2}{.5}{}
	\roundNbox{dashed}{(0,1)}{.4}{.2}{.5}{}
\end{tikzpicture}
=
\dim_\cC(c)^{-1}\,
\begin{tikzpicture}[baseline=-.1cm, xscale=-1]
	\draw (-.1,-1.7) -- (-.1,1.7);
	\draw[thick, red] (-.3,-1.7) -- (-.3,0);
	\draw[thick, blue] (-.3,0) -- (-.3,1.7);
	\draw[thick, DarkGreen] (.1,.25) -- (.1,.7) arc (180:0:.15cm) -- (.4,-.7) arc (0:-180:.15cm) -- (.1,-.25);
	\draw[thick, orange] (.1,-1.7) -- (.1,-1.3) arc (180:0:.15cm) -- (.4,-1.4) arc (-180:0:.15cm) -- (.7,1.4) arc (0:180:.15cm) -- (.4,1.3) arc (0:-180:.15cm) -- (.1,1.7);
	\roundNbox{dashed}{(0,-1)}{.4}{.2}{.5}{}
	\roundNbox{unshaded}{(0,0)}{.4}{.2}{.5}{$f_{d\otimes \overline{c}\otimes c}$}
	\roundNbox{dashed}{(0,1)}{.4}{.2}{.5}{}
\end{tikzpicture}
\\
&=
\begin{tikzpicture}[baseline=.4cm, xscale=-1]
	\draw (-.1,-.7) -- (-.1,1.7);
	\draw[thick, blue] (-.3,0) -- (-.3,1.7);
	\draw[thick, red] (-.3,-.7) -- (-.3,0);
	\draw[thick, orange] (.1,-.7) -- (.1,.7) arc (180:0:.15cm) -- (.4,-.4) arc (-180:0:.15cm) -- (.7,1.4) arc (0:180:.15cm) -- (.4,1.3) arc (0:-180:.15cm) -- (.1,1.7);
	\roundNbox{unshaded}{(0,0)}{.4}{.2}{.5}{$f_{d\otimes \overline{d}\otimes d}$}
	\roundNbox{dashed}{(0,1)}{.4}{.2}{.5}{}
\end{tikzpicture}
=
\begin{tikzpicture}[baseline=.4cm, xscale=-1]
	\draw (-.1,-.7) -- (-.1,1.7);
	\draw[thick, blue] (-.3,0) -- (-.3,1.7);
	\draw[thick, red] (-.3,-.7) -- (-.3,0);
	\draw[thick, orange] (.1,-.7) -- (.1,.7) arc (180:0:.15cm) -- (.4,-.4) arc (-180:0:.15cm) -- (.7,1.4) arc (0:180:.15cm) -- (.4,1.3) arc (0:-180:.15cm) -- (.1,1.7);
	\roundNbox{unshaded}{(0,0)}{.25}{.2}{0}{$f_d$}
	\roundNbox{dashed}{(0,0)}{.4}{.2}{.5}{}
	\roundNbox{dashed}{(0,1)}{.4}{.2}{.5}{}
\end{tikzpicture}
=
\begin{tikzpicture}[baseline=-.1cm]
	\draw[thick, orange] (-.1,-.5) -- (-.1,.5);
	\draw[thick, red] (.4,-.5) -- (.4,0);
	\draw[thick, blue] (.4,0) -- (.4,.5);
	\draw (.15,-.5) -- (.15,.5);
	\roundNbox{unshaded}{(0,0)}{.25}{0}{.3}{$f_d$}
	\node at (.4,.7) {\scriptsize{$\mathbf{b}$}};
	\node at (.4,-.7) {\scriptsize{$\mathbf{a}$}};
	\node at (-.1,-.7) {\scriptsize{$\mathbf{d}$}};
	\node at (.15,-.7) {\scriptsize{$\bfH$}};
\end{tikzpicture}
=
\dim_\cC(d)^{-1}\,
\begin{tikzpicture}[baseline=-.1cm]
	\draw (.1,-.5) -- (.1,.5);
	\draw[thick, red] (.3,-.5) -- (.3,0);
	\draw[thick, blue] (.3,0) -- (.3,.5);
	\draw[thick, orange] (-.1,-.25) arc (0:-180:.15cm) -- (-.4,.25) arc (180:0:.15cm);
	\draw[thick, orange] (-.55,-.5) -- (-.55,.5);
	\roundNbox{unshaded}{(0,0)}{.25}{0}{.2}{$f_d$}
\end{tikzpicture}
=
\id_{\mathbf{d}}\otimes f_d'
\end{align*}
Now using that $\id_{\mathbf{d}}\otimes -$ is faithful, we see $f_c'=f_d'$.
Calling this map $f$, we are finished.
\end{proof}

Using Lemma \ref{lem:NTbetweenRightCreation}, we get a concrete description of the right commutant $\cM'$.

\begin{cor}
\label{cor:IdentifyMPrime}
We may identify the objects of $\Phi(\cM)'$ with objects of the form $\bfH\otimes \mathbf{c} \in \Hilb(\cC)$, and for $c,d\in\cC$, the morphism space $\Phi(\cM)'(\bfR_c, \bfR_d)$ is isomorphic to 
$$
\set{f\in \Hom_{\Hilb(\cC)}(\bfH\otimes \mathbf{c}, \bfH\otimes \mathbf{d})}{\,\,\,
\begin{tikzpicture}[baseline=-.1cm]
	\draw (0,-1.1) -- (0,1.1);
	\draw[thick, red] (-.4,-1.1) -- (-.4,-.5);
	\draw[thick, blue] (-.4,-.5) -- (-.4,1.1);
	\draw[thick, DarkGreen] (.4,-1.1) -- (.4,.5);
	\draw[thick, orange] (.4,.5) -- (.4,1.1);
	\roundNbox{unshaded}{(.1,.5)}{.3}{0}{.2}{$f$}
	\roundNbox{unshaded}{(-.1,-.5)}{.3}{.2}{0}{$g$}
	\node at (0,-1.3) {\scriptsize{$\bfH$}};
	\node at (0,1.3) {\scriptsize{$\bfH$}};
	\node at (.4,-1.3) {\scriptsize{$\mathbf{c}$}};
	\node at (.4,1.3) {\scriptsize{$\mathbf{d}$}};
	\node at (-.4,-1.3) {\scriptsize{$\mathbf{a}$}};
	\node at (-.4,1.3) {\scriptsize{$\mathbf{b}$}};
\end{tikzpicture}
=
\begin{tikzpicture}[baseline=-.1cm]
	\draw (0,-1.1) -- (0,1.1);
	\draw[thick, red] (-.4,-1.1) -- (-.4,.5);
	\draw[thick, blue] (-.4,.5) -- (-.4,1.1);
	\draw[thick, DarkGreen] (.4,-1.1) -- (.4,-.5);
	\draw[thick, orange] (.4,-.5) -- (.4,1.1);
	\roundNbox{unshaded}{(.1,-.5)}{.3}{0}{.2}{$f$}
	\roundNbox{unshaded}{(-.1,.5)}{.3}{.2}{0}{$g$}
	\node at (0,-1.3) {\scriptsize{$\bfH$}};
	\node at (0,1.3) {\scriptsize{$\bfH$}};
	\node at (.4,-1.3) {\scriptsize{$\mathbf{c}$}};
	\node at (.4,1.3) {\scriptsize{$\mathbf{d}$}};
	\node at (-.4,-1.3) {\scriptsize{$\mathbf{a}$}};
	\node at (-.4,1.3) {\scriptsize{$\mathbf{b}$}};
\end{tikzpicture}
\,\,\text{ for all }
g\in \Hom_{\Phi(\cM)}(\mathbf{a}\otimes \bfH, \mathbf{b}\otimes \bfH)
}.
$$
Thus $\cM'$ is a \emph{W*}-category.
\end{cor}

Next, we show that $\bfB(\bfH)$ has the commutant that one should expect.
We begin with a discussion of creation natural transformations.

\begin{defn}[Creation natural transformations]
\label{defn:CreationNaturalTransformations}
Suppose $b,c\in\Irr(\cC)$ and $a\in\cC$.
For $\alpha\in \cC(c, a\otimes b)$, and $x\in B(\bfH(c), \bfH(b))$, we define a natural transformation $L^{\alpha, x}= (L^{\alpha, x}_d)_{d\in\cC} : \bfH \to \mathbf{a}\otimes \bfH$ as follows.
First, we recall that 
$$
(\mathbf{a}\otimes \bfH)(d)
\cong 
\bigoplus_{e,f\in \Irr(\cC)}
\textbf{a}(e)\boxtimes
\cC(d, e\otimes f)\boxtimes \bfH(f)
\cong 
\bigoplus_{f\in \Irr(\cC)}
\cC(d, a\otimes f)\boxtimes \bfH(f).
$$
via the Hilbert space isomorphism $\psi\boxtimes \beta \boxtimes \xi \mapsto  \mathbf{a}(\psi)(\beta) \boxtimes \xi$.
(That the inner products agree is a simple graphical calculation.)
For $d\in\Irr(\cC)$, we define the component $L^{\alpha, x}_d$ by
$$
\bfH(d)\ni\xi 
=
\begin{tikzpicture}[baseline = -.1cm]
    \draw (0,-.6) -- (0,.6);
    \roundNbox{unshaded}{(0,0)}{.3}{0}{0}{$\xi$}
    \node at (.15,-.5) {\scriptsize{$\mathbf{d}$}};
    \node at (.15,.5) {\scriptsize{$\bfH$}};
\end{tikzpicture}
\overset{L^{\alpha,x}_d}{\longmapsto} 
=
\delta_{c=d} [\alpha\boxtimes x(\xi)]
=
\delta_{c=d}
\begin{tikzpicture}[baseline = -.4cm]
    \draw (0,.6) -- (0,-.3);
    \draw (1,0) -- (1,.6);
    \draw (.5,-.8) -- (.5,-1.2);
    \draw (0,-.3) arc (-180:0:.5cm);
    \filldraw[fill=white] (.5,-.8) circle (.05cm) node [above] {\scriptsize{$\alpha$}};
    \roundNbox{unshaded}{(1,0)}{.3}{.2}{.2}{$x(\xi)$}
    \node at (-.15,-.5) {\scriptsize{$\mathbf{a}$}};
    \node at (1.15,-.5) {\scriptsize{$\mathbf{b}$}};
    \node at (.35,-1) {\scriptsize{$\mathbf{c}$}};
    \node at (1.2,.5) {\scriptsize{$\bfH$}};
\end{tikzpicture}
\in \cC(c, a\otimes b)\boxtimes \bfH(b).
$$
The maps $(L^{\alpha, x}_d)_{d\in\Irr(\cC)}$ are clearly natural with respect to maps between simples objects $d\to d'$, and we extend it by additivity to all objects.

If $b=c$, we define the creation natural transformation $L^\alpha = L^{\alpha,\id_c}$.
We define right creation natural transformations similarly.
It is straightforward to check that the $d$-component of $(L^{\alpha,x})^*$ for $d\in \Irr(\cC)$ is given on the summand $\cC(d, a\otimes f)\boxtimes \bfH(f)$ for $f\in \Irr(\cC)$ by 
$$
(L^{\alpha,x})^*_d(\beta\boxtimes \xi) = \delta_{d=c}\delta_{f=b} \langle \alpha|\beta\rangle_{\cC(c, a\otimes b)}x^*(\xi).
$$
\end{defn}

\begin{prop}
\label{prop:MHprime}
For an object $\bfH\in \Hilb(\cC)$, the commutant $\cM_\bfH'$ is equivalent to $(\cC,1_\cC)$ as a cyclic right $\cC$-module W*-category.
Hence $\bfB(\bfH)' \cong \mathbf{1}\in \Vec(\cC)$.
\end{prop}
\begin{proof}
We want to show that any $f=(f_c)_{c\in\cC}\in \Hom_{\Hilb(\cC)}(\bfH \otimes \mathbf{a}, \bfH\otimes \mathbf{b})$ in $\cM^{\prime}_{\bfH}$ is of the form $\id_\bfH \otimes \psi$ for some $\psi\in \cC(a,b)$.
By applying duality maps in $\Hilb(\cC)$, it suffices to consider the case that $a=1_\cC$.

Suppose $f=(f_c)_{c\in\cC}\in \Hom_{\Hilb(\cC)}(\bfH,\bfH \otimes \mathbf{b})$ is in $\cM^{\prime}_{\bfH}$.  
\\

\noindent\textbf{Claim 1:}
For $c\in\Irr(\cC)$, the image of $f_c$ is contained in the `diagonal' $\bfH(c)\boxtimes \cC(c, c\otimes b)$.
\begin{proof}[Proof of Claim 1]
Let $g\in \End_{\Hilb(\cC)}(\bfH)$ be the natural transformation which is the identity on $\bfH(c)$ and zero everywhere else.
Since $f\in \cM_\bfH'$, $f\circ g = (g\otimes \id_{\mathbf{b}})\circ f$.
We have $(f\circ g)_c =f_c$, but $((g\otimes \id_{\mathbf{b}})\circ f)_c$ has image contained in the $\bfH(c)\boxtimes \cC(c, c\otimes b)$ summand of $(\bfH\otimes \mathbf{b})(c)$.
\end{proof}

\noindent\textbf{Claim 2:}
For $c\in \Irr(\cC)$, there is a $\beta_c\in \cC(c, c\otimes b)$ such that $f_c(\xi) = \xi \boxtimes \beta_c$ for all $\xi\in \bfH(c)$.
Thus $f$ is the sum of the right creation operators $R^{\beta_c}$ over $c\in \Irr(\cC)$.
\begin{proof}[Proof of Claim 2]
We may assume there is a unit vector $\xi \in \bfH(c)$ such that $f_c(\xi) \neq 0$, as otherwise the claim is trivial setting $\beta_c=0$.
Note that we can write $f_c(\xi)$ as a finite sum
$$
f_c(\xi) =
\sum_{\beta\in \ONB(c,c\otimes b)} \eta_\beta \boxtimes \beta.
$$
Let $g\in \End_{\Hilb(\cC)}(\bfH)$ be the natural transformation whose $c$-component is the projection onto $\bbC \xi$ and zero everywhere else.
Then 
$$
\sum_{\beta} \eta_\beta \boxtimes \beta
=
f_c(\xi)
=(f\circ g)_c (\xi) 
=
((g\otimes \id_{\mathbf{b}})\circ f)_c(\xi) 
= 
\sum_\beta \langle \xi | \eta_\beta\rangle \xi \boxtimes \beta
=
\xi\boxtimes \beta_{c},
$$
where $\beta_{c} = \sum_\beta\langle \xi | \eta_\beta\rangle \beta \in \cC(c,c\otimes b)\setminus\{0\}$.

Now pick any other unit vector $\eta\in \bfH(c)$.
Let $h\in  \End_{\Hilb(\cC)}(\bfH)$ be the natural transformation whose $c$-component is the rank one operator $|\eta\rangle\langle \xi|$ and zero everywhere else.
Since $f\in \cM_\bfH'$, we have $f\circ h = (h\otimes \id_{\mathbf{b}})\circ f$, so 
$$
f_c(\eta)
=
(f_c\circ h_c)(\xi)
=
(f\circ h)_c(\xi)
=
((h\otimes \id_{\mathbf{b}})\circ f)_c(\xi)
=
(h\otimes \id_{\mathbf{b}})_c(\xi\boxtimes \beta_{c})
=
\eta\boxtimes \beta_{c}.
$$
Thus we see $f_c$ is right creation by $\beta_c$.
\end{proof}

For the next claim, we define $\operatorname{Supp}(\bfH)=\set{ c\in \Irr(\cC)}{\bfH(c) \neq (0)}$.
\\

\noindent\textbf{Claim 3}:
The collection $(\beta_c)_{c\in \operatorname{Supp}(\bfH)}$ satisfy the following `naturality condition': for all $a\in \cC$ and $\alpha\in \cC(c, a\otimes d)$, we have $(\id_{a}\otimes \beta_{d})\circ \alpha =(\alpha\otimes \id_{b})\circ  \beta_c$.
\begin{proof}[Proof of Claim 3]
Suppose $c,d\in \operatorname{Supp}(\bfH)$.  
Let $\xi\in \bfH(c)$ and $\eta\in \bfH(d)$ be unit vectors, and consider the rank one operator $x=|\eta\rangle \langle \xi|$.
Form the creation natural transformation $L^{\alpha, x} \in \Hom_{\Hilb(\cC)}(\bfH , \textbf{a}\otimes \bfH)$ from Definition \ref{defn:CreationNaturalTransformations}.
Since $f\in \cM_\bfH'$, by the definition of the tensor product of natural transformations in $\Hilb(\cC)$, we have
\begin{align*}
\eta\boxtimes [(\id_{a}\otimes \beta_d)\circ \alpha]
&=
(\id_{\mathbf{a}}\otimes f)_c\circ L^{\alpha, x}_c(\xi)
=
(L^{\alpha, x}\otimes \id_b)_c\circ f_c(\xi)
=
\eta\boxtimes [(\alpha\otimes \id_b)\circ \beta_c].
\qedhere
\end{align*}
\end{proof}

\noindent\textbf{Claim 4:}
There is a fixed $\psi\in \cC(1_\cC, b)$ such that for all $c\in\operatorname{supp}(\bfH)$, 
$
\beta_c = \id_c\otimes \psi
$.
\begin{proof}[Proof of Claim 4]
Again, we may assume there are distinct $c,d\in \operatorname{supp}(\bfH)$.
Setting $a=c\otimes \overline{d}$, the naturality condition from Claim 3 implies that
$$
\begin{tikzpicture}[baseline=-.1cm]
    \draw (-.15,-.2) .. controls ++(90:.5cm) and ++(270:1cm) .. (-1,1);
    \draw (.15,-.2) -- (.15,1);
    \draw (0,-1) -- (0,-.8);
    \draw (-.725,1) -- (-.725,.8) arc (-180:0:.3cm) -- (-.125,1);
	\roundNbox{unshaded}{(0,-.5)}{.3}{0}{0}{$\beta_c$}
	\roundNbox{dashed}{(0,.5)}{.3}{.9}{.1}{}
	\node at (.15,1.2) {\scriptsize{$b$}};
	\node at (-.15,1.2) {\scriptsize{$d$}};
	\node at (-.725,1.2) {\scriptsize{$\overline{d}$}};
	\node at (-1,1.2) {\scriptsize{$c$}};
	\node at (-1.4,.5) {\scriptsize{$\alpha$}};
\end{tikzpicture}
=
\begin{tikzpicture}[baseline=-.1cm]
    \draw (0,-1) .. controls ++(90:.4cm) and ++(270:.6cm) .. (-1,0)-- (-1,1);
    \draw (-.15,.8) -- (-.15,1);
    \draw (.15,.8) -- (.15,1);
    \draw (0,.2) -- (0,-.1) arc (0:-180:.3cm) -- (-.6,1);
	\roundNbox{unshaded}{(0,.5)}{.3}{0}{0}{$\beta_d$}
	\roundNbox{dashed}{(0,-.4)}{.3}{.9}{.1}{}
	\node at (.15,1.2) {\scriptsize{$b$}};
	\node at (-.15,1.2) {\scriptsize{$d$}};
	\node at (-.6,1.2) {\scriptsize{$\overline{d}$}};
	\node at (-1,1.2) {\scriptsize{$c$}};
	\node at (-1.4,-.4) {\scriptsize{$\alpha$}};
\end{tikzpicture}
\,.
$$
This has two important implications.
First, setting $c=d$ and applying a $c$-cap between the second and third strings on the top, we see that
$$
\beta_{c} = 
\begin{tikzpicture}[baseline=-.1cm]
    \draw (0,-.7) -- (0,-.3);
    \draw (-.25,.3) -- (-.25,.7);
    \draw (.25,0) -- (.25,.7);
    \draw (0,-.3) .. controls ++(90:.2cm) and ++(270:.2cm) .. (-.25,.3);
    \filldraw (.25,0) circle (.05cm) node [right] {\scriptsize{$\gamma_c$}};
    \roundNbox{dashed}{(0,0)}{.3}{.2}{.5}{}
    \node at (-.2,-.5) {\scriptsize{$c$}};
    \node at (-.45,.5) {\scriptsize{$c$}};
    \node at (.45,.5) {\scriptsize{$b$}};
\end{tikzpicture}
=
\id_{c}\otimes \gamma_c
$$
for some map $\gamma_c\in \cC(1_\cC, b)$.

The naturality condition above now reduces to $\id_c \otimes \ev_d^* \otimes \gamma_c = \id_c\otimes \ev_d^*\otimes \gamma_d$.
By capping off and diving by the appropriate scalars, we see that $\gamma_c = \gamma_d$.
Thus there is a single $\psi\in \cC(1_\cC, b)$ such that for all $c\in \operatorname{supp}(\bfH)$, $\beta_c = \id_c \otimes \psi$.
\end{proof}

From this series of claims, we deduce that there is a fixed $\psi\in \cC(1_\cC, b)$ such that for all $c\in \operatorname{supp}(\bfH)$, $f_c : \bfH(c) \to \bigoplus_{a\in\Irr(\cC)}\bfH(a)\boxtimes \cC(c, a\otimes b)$ is exactly the map $\xi\mapsto \xi\boxtimes (\id_c\otimes \psi) \in \bfH(c)\boxtimes \cC(c, c\otimes b)$.
Note that this formula also holds when $c\notin \operatorname{supp}(\bfH)$, since for those $c\in \Irr(\cC)$, we have $\bfH(c)=(0)$.
The result follows.
\end{proof}

\subsection{W*-categories and linking algebras}

We now recall how to pass back and forth between W*-categories and their linking von Neumann algebras.
Much of the following conversation is contained, either explicitly or implicitly, in \cite[\S2]{MR808930}.
(See also \cite[Prop.\,7.17]{MR808930}.)

Suppose we have a small W*-category $\cX$, so that $\Ob(\cX)$ is a set.
%
By \cite{MR808930}, there is a faithful normal dagger functor $\bfF: \cX \to \Hilb$, also called a faithful \emph{representation}, which is norm closed at the level of hom spaces.

\begin{defn}
We define the auxiliary Hilbert space by
$$
K = \bigoplus_{x\in \Ob(\cX)} \bfF(x).
$$
The \emph{linking algebra} of $(\cX,\bfF)$ is the von Neumann algebra
$$
\text{W*}(\cX,\bfF) =\left( \bigoplus_{x,y\in \Ob(\cX)} \bfF(\cX(x,y))\right)'' \cap B(K),
$$
together with the distinguished projections $p_x := \bfF(\id_x) \in B(\bfF(x))$ for all $x\in \Ob(\cX)$.
Note that $\sum_{x\in\Ob(x)} p_x = 1_{\text{W*}(\cX, \bfF)}$.
Also, for all $x,y\in \Ob(\cX)$, we have a Banach space isomorphism $\cX(x, y) \cong p_y\text{W*}(\cX, \bfF)p_x$.
\end{defn}

Now starting with a von Neumann algebra $X$ and a partition of unity, i.e., a collection of mutually orthogonal projections $P=\{p\}$ which sum to $1$, we get a small W*-category $\cW_{(X,P)}$ as follows.
The objects are the $p\in P$, and the morphism spaces are given by $\cW_{(X,P)}(p, q) = q X p$.
The identity morphism in $\cW_{(X,P)}(p,p)$ is $p$, and composition is exactly multiplication in $X$.

It is now easy to see that starting with the W*-category $\cX$ and the faithful representation $\bfF:\cX\to \Hilb$, we have an equivalence of categories $\cX \cong \cW_{(\text{W*}(\cX,\bfF), \set{p_x}{x\in\Ob(\cX)})}$, and the objects of both categories can be identified with $\Ob(\cX)$.
Starting with a von Neumann algebra $X$ and a distinguished partition of unity $P$, given any faithful representation $\bfF: \cW_{(X,P)} \to \Hilb$, we have an isomorphism of von Neumann algebras $\text{W*}(\cW_{(X,P)}, \bfF)\cong X$.

Moreover, these constructions are well-behaved with respect to inclusions.
Given a W*-subcategory $\cY\subset \cX$ with $\Ob(\cY)=\Ob(\cX)$, together with a faithful representation $\bfF: \cX\to \Hilb$, we get a unital inclusion of von Neumann algebras $\text{W*}(\cY,\bfF|_\cY) \subseteq \text{W*}(\cX, \bfF)$.
Note also that $\text{W*}(\cY, \bfF|_\cY)$ contains the distinguished partition of unity $\set{p_x}{x\in \Ob(\cX)}$.
Conversely, given a von Neumann subalgebra $Y\subseteq X$ such that $Y$ contains the distinguished partition of unity $P=\{p\}$ of $X$, we get a canonical W*-subcategory $\cW_{(Y,P)} \subseteq \cW_{(X,P)}$ whose objects are the elements $p\in P$ and whose morphisms spaces are given by $\cY(p,q) = qYp\subseteq qXp$.

The discussion above sketches the proof of the following theorem.

\begin{thm}
\label{thm:LinkingAlgebraCorrespondence}
Let $\cX$ be a \emph{W*}-category and let $X=\text{\emph{W*}}(\cX,\bfF)$ be its linking algebra with respect to the representation $\bfF$, together with the partition of unity $P=\set{p_x}{x\in \Ob(\cX)}$.
There is a bijective correspondence between:
\begin{enumerate}[(1)]
\item
\emph{W*}-subcategories of $\cX$ whose objects are $\Ob(\cX)$, and
\item
von Neumann subalgebras of $X$ containing the distinguished partition of unity $P$.
\end{enumerate}
\end{thm}

As an application of the above theorem, given a dagger subcategory $\cZ\subset \cX$ whose objects are $\Ob(\cX)$, we can define its W*-completion.
First, we take the unital $*$-subalgebra $Z\subset X=\text{W*}(\cX,\bfF)$ corresponding to $\cZ$, and we define the W*-completion $\cZ''$ of $\cZ$ to be the W*-subcategory of $\cX$ corresponding to $Z''\subseteq X$.
It is easy to see that $\cZ''$ can also be obtained by taking $\cZ''(x,y)$ to be the weak* closure of every $\cZ(x,y)$ in $\cX(x,y)$ for all $x,y\in \Ob(\cX)$.
(This is essentially in \cite[Thm.\,4.2]{MR808930}.)

\subsection{The linking algebra of \texorpdfstring{$\cB_\bfH$}{MH} and the bicommutant theorem}

For this section, we assume that $\cC$ is small.
Our main result holds when $\cC$ is essentially small, meaning $\cC$ is equivalent as a dagger category to a small rigid C*-tensor category.

Suppose we have $\bfH\in \Hilb(\cC)$.
We now apply the discussion of the last section to the W*-category $\cB_\bfH$, whose objects are the set $\Ob(\cB_\bfH) = \set{\mathbf{a}\otimes \bfH\otimes \mathbf{b}}{a,b\in \cC}$.
Note that we get a faithful representation $\cB_\bfH \to \Hilb$ by $\mathbf{a}\otimes \bfH \otimes \mathbf{b} \mapsto \bigoplus_{e\in \Irr(\cC)} (\mathbf{a}\otimes \bfH \otimes \mathbf{b})(e)$.
Thus the auxiliary Hilbert space is given by 
$$
K=
\bigoplus_{a,b,c,d\in \Ob(\cC)}
\bigoplus_{e\in\Irr(\cC)}
(\mathbf{a}\otimes \bfH\otimes \mathbf{b})(e),
$$
and the linking algebra W*$(\bfH)$ of $\cB_\bfH$ is the von Neumann algebra
$$
\text{W*}(\bfH)=\left(\bigoplus_{a,b\in[\Ob(\cC)]}\bigoplus_{e\in\Irr(\cC)}
B\bigg((\mathbf{a}\otimes \bfH\otimes \mathbf{b})(e), (\mathbf{c}\otimes \bfH\otimes \mathbf{d})(e)\bigg)\right)'' \cap B(K).
$$
The distinguished partition of unity is 
$$
P=\set{p_{a,b}=\sum_{e\in\Irr(\cC)}\id_{(\mathbf{a}\otimes \bfH \otimes \mathbf{b})(e)} }{a,b,\in\cC}.
$$

We now use Theorem \ref{thm:LinkingAlgebraCorrespondence} to define some canonical W*-subcategories of $\cB_\bfH$.
There is a left action $\lambda$ of $\cM_\bfH$ on $K$.
Each morphism in $f\in \cM_\bfH(\mathbf{a}\otimes \bfH , \mathbf{c}\otimes \bfH)$ includes into W*$(\bfH)$ by defining
$$
\lambda(f)=\sum_{b\in \Ob(\cC)}f\boxtimes \id_{\mathbf{b}} \in 
\left(\bigoplus_{a,b,c\in \Ob(\cC)}
\bigoplus_{e\in\Irr(\cC)}
B\bigg((\mathbf{a}\otimes \bfH\otimes \mathbf{b})(e), (\mathbf{c}\otimes \bfH\otimes \mathbf{b})(e)\bigg)\right)'' 
\subset \text{W*}(\bfH).
$$
Similarly, there is a right action $\rho$ of ${}_\bfH\cM$ on $K$ such that each $\rho(g)\in \text{W*}(\bfH)$.

\begin{defn}
We define the von Neumann subalgebra $\cL_\bfH = \lambda(\cM_\bfH)''\cap B(K)\subset \text{W*}(\bfH)$.
Similarly, there is a von Neumann subalgebra $\cR_\bfH = \rho({}_\bfH\cM)'' \cap B(K)\subset \text{W*}(\bfH)$.
\end{defn}

We now have a partition of unity of W*$(\bfH)$ which is coarser than $P$ given by
$$
Q_L=\set{q_a = \lambda(\id_{\mathbf{a}\otimes \bfH})= \sum_{b\in\Ob(\cC)} p_{a,b}}{a\in \cC}.
$$
Similar to Theorem \ref{thm:LinkingAlgebraCorrespondence}, by construction, we can recover a category equivalent to $\cM_\bfH$, since for all $a,b\in \cC$, we have a Banach space isomorphism
$$
q_b \cL_\bfH q_a \cong \cM_\bfH(\mathbf{a}\otimes \bfH,\mathbf{b}\otimes \bfH).
$$
Similarly, we can recover ${}_\bfH\cM$ from $\cR_\bfH$ from the coarse partition of unity
$$
Q_R=\set{\rho(\id_{\bfH\otimes \mathbf{b}})= \sum_{a\in\Ob(\cC)} p_{a,b}}{b\in \cC}.
$$

Now we have a normal embedding $\iota:\cM_{\mathbf{1}}\hookrightarrow \cM_\bfH$ given by embedding $\cC(a,b) \hookrightarrow \Hom_{\Hilb(\cC)}(\mathbf{a},\mathbf{b})$ by $\psi\mapsto \psi \boxtimes \id_\bfH$.
The image $\iota(\cM_{\mathbf{1}})$ inside $\cM_\bfH$ behaves like the scalars. 

\begin{defn}
We define the left `trivial' algebra $\cL_\cC = (\lambda\circ \iota(\cM_{\mathbf{1}}))''\subset \lambda(\cM_\bfH)\subset \text{W*}(\bfH)$.
Similarly, we have the right `trivial' algebra $\cR_\cC$ defined using $\rho\circ \iota$.
\end{defn}

Using this new language, the following corollary is just a restatement of Proposition \ref{prop:MHprime}.

\begin{cor} 
\label{cor:MHprime}
$\cL_{\bfH}^{\prime}\cap \cR_\bfH=\cR_{\cC}$.
Similarly, $\cL_{\bfH}\cap \cR_{\bfH}^{\prime}=\cL_{\cC}$.
\end{cor}

Now by Theorem \ref{thm:LinkingAlgebraCorrespondence}, there is a bijective correspondence between von Neumann subalgebras $N\subseteq \cL_\bfH$ containing the distinguished partition of unity $Q$, and W*-subcategories $\cW_{(N,Q_L)}\subseteq \cM_\bfH$ with the objects $\mathbf{a}\otimes \bfH$ for $a\in \cC$.
Similarly, we have a bijective correspondence between von Neumann subalgebras $N\subset \cR_\bfH$ containing $R$ and W*-subcategories ${}_{(N,Q_R)}\cW \subseteq {}_\bfH \cM$ with the objects $\bfH\otimes \mathbf{b}$ for $b\in \cC$.

\begin{rem}
\label{rem:LinkingAlgebrasAndCModules}
Note that in general, $\cW_{(N,Q_L)}$ is \emph{not} a $\cC$-module category.
If $\cW_{(N,Q_L)}$ is a $\cC$-module category, then $N$ necessarily contains $\cL_{\cC}$, since $N$ is unital.
The converse is not true, since there is no \emph{a priori} compatibility between the algebra multiplication and the left $\cC$ action.  
However, we call always obtain a W*-algebra containing $N$ which corresponds to a $\cC$-module category by taking the closure in $\cM_\bfH$ of $\cW_{(N,Q_L)}$ under the left $\cC$-action and taking the corresponding W*-algebra.
\end{rem}

We can now describe the relationship between these bijective correspondences and the commutants from Definition \ref{defn:CommutantModuleCategory}.
In our new language, the following proposition is a restatement of Corollary \ref{cor:IdentifyMPrime}.

\begin{prop}
\label{prop:CategoricalAndAlgebraicCommutant}
Suppose $\cN\subseteq \cM_\bfH$ is a cyclic $\cC$-module dagger subcategory (not necessarily \emph{W*}), so that $\cN$ has objects of the form $\mathbf{a}\otimes \bfH$ for $a\in \cC$.
Let $N\subseteq \cL_\bfH$ be the corresponding $*$-subalgebra.
Then $\cN'={}_{(N'\cap \cR_\bfH, Q_R)}\cW$.

We have a similar statement switching the role of left and right.
\end{prop}

We are now ready to prove our version of the von Neumann bicommutant theorem.

\begin{thm}[von Neumann bicommutant]
\label{thm:Bicommutant}
Suppose we have a cyclic $\cC$-module dagger functor $\Phi:(\cM, m)\to (\cM_\bfH, \bfH)$.
The bicommutant cyclic $\cC$-module \emph{W*}-category $(\Phi(\cM)'', \bfH)$ is equivalent to the weak* closure of $\Phi(\cM)$ inside $(\cM_\bfH, \bfH)$.
\end{thm}

\begin{proof}
Let $M\subseteq \cL_{H}$ be the induced subalgebra of $\Phi(\cM)\subseteq \cM_{\bfH}$.  
Then we see that the weak* closure of $\Phi(\cM)\subseteq \cM_\bfH$ is $\cW_{(M'', Q_L)}$.  
Also note that by Remark \ref{rem:LinkingAlgebrasAndCModules}, $\cL_{\cC}\subseteq M$ since $M$ is unital.
We then have 
\begin{align*}
\cM^{\prime \prime}
&=
\cW_{\left((M'\cap \cR_{\bfH})'\cap \cL_{\bfH}, Q_L\right)}
&&\text{(by Proposition \ref{prop:CategoricalAndAlgebraicCommutant})}
\\&=
\cW_{\left((M''\cap \cL_{\bfH})\vee (\cR^{\prime}_{\bfH}\cap \cL_{\bfH}), Q_L\right)}
\\&=
\cW_{\left(M''\vee \cL_{\cC} , Q_L\right)}
&&\text{(by Corollary \ref{cor:MHprime})}
\\&=
\cW_{(M'', Q_L)}
&&
\qedhere
\end{align*}
\end{proof}

\begin{rem}
In the event that $\Phi= \check{\theta}$ where $\theta: \bfA \Rightarrow \bfB(\bfH)$ is a representation of a C*-algebra object, Theorem \ref{thm:Bicommutant} gives us a W*-algebra object corresponding to $\check{\theta}(\cM_{\bfA})''$, which we denote by $\bfA''\in \Vec(\cC)$.
Note that the representation $\bfA\Rightarrow \bfB(\bfH)$ factors through $\bfA''$, since the cyclic $\cC$-module dagger functor $\check{\theta}$ factors through $\check{\theta}(\cM_{\bfA})''$.
\end{rem}

\subsection{Bimodules between algebra objects}

Equipped with our definition of the commutant and the bicommutant theorem, we now discuss the notion of a bimodule between algebra objects in $\Vec(\cC)$.

\begin{defn}
Suppose we have two $*$-algebra objects $\bfA,\bfB\in \Vec(\cC)$.
An $\bfA-\bfB$ bimodule consists of 
a Hilbert space object $\bfH\in\Hilb(\cC)$,
together with representations 
$\lambda: \cM_\bfA \to \cM_\bfH \subset \cB_\bfH$ and 
$\rho: {}_\bfB\cM \to {}_\bfH\cM \subset \cB_\bfH$
satisfying the requirement that $\rho({}_\bfB\cM) \subseteq \lambda(\cM_\bfA)'$.
\end{defn}

\begin{ex}
Suppose $\bfA\in \Vec(\cC)$ is a C*-algebra object with a trace $\tau$.
The left and right actions of $\bfA$ on $L^2(\bfA)$ from \eqref{eq:LeftAndRightActionOnL2A} make $L^2(\bfA)$ an $\bfA-\bfA$ bimodule by Proposition \ref{prop:TwoCommutingActionsOnL2A}.
\end{ex}

\begin{defn}
A finite W*-algebra object in $\Vec(\cC)$ is a pair $(\bfM, \tau)$ where $\bfM\in \Vec(\cC)$ is a W*-algebra object and $\tau$ is normal a tracial state on $\bfM$ which is faithful on $\bfM(1_\cC)$.
\end{defn}

\begin{prop}
The GNS representation of a \emph{W*}-algebra object $\bfM\in \Vec(\cC)$ with respect to a normal state $\phi$ on $\bfM(1_\cC)$ gives a normal representation $\pi: \bfM \Rightarrow \bfB(L^2(\bfM)_\phi)$.
\end{prop}
\begin{proof}
It suffices to show that the induced $\cC$-module dagger functor $\check{\pi}: \cM_\bfM \to \cM_{L^2(\bfM)}$ is normal.
Since these W*-categories admit direct sums, we need only show that for every increasing bounded net $f_i\nearrow f$ in the W*-algebra $\cM_\bfM(a_\bfM,a_\bfM)\cong \bfM(\overline{a}\otimes a)$, we have $\check{\pi}(f_i)\nearrow \check{\pi}(f)$ in the W*-algebra $\End_{\Hilb(\cC)}(\mathbf{a}\otimes L^2(\bfM))$.
Since the action of $\End_{\Hilb(\cC)}(\mathbf{a}\otimes L^2(\bfM))$ on $H_a:=\bigoplus_{c\in \Irr(\cC)} (\mathbf{a}\otimes L^2(\bfM))(c)\cong \bigoplus_{c\in \Irr(\cC)} \cC(c, a\otimes b)\boxtimes L^2(\bfM)(b)$ is faithful (and normal), we need only check that for every $\sum_{k=1}^n\alpha_k\boxtimes \xi_k \in \cC(c,a\otimes b)\boxtimes \widehat{\bfM}(b)$, 
\begin{align*}
\sum_{k,\ell=1}^n
\langle (\alpha_k\boxtimes \xi_k) |\check{\pi}(f_i) (\alpha)\ell\boxtimes \xi_\ell)\rangle_c
&=
\phi\left(
\begin{tikzpicture}[baseline=.1cm]
	\draw (.4,.7) .. controls ++(90:.6cm) and ++(90:.6cm) .. (-1,.7) -- (-1,.3);
	\draw (0,.3) arc (180:0:.4cm);
	\filldraw (.4,.7) circle (.05cm);
	\filldraw (-.3,1.15) circle (.05cm);
	\draw (.1,-.3) arc (-180:0:.35cm);
	\draw (-.3,1.15) -- (-.3,1.4);
	\draw (-.1,-.3) .. controls ++(270:.55cm) and ++(270:.55cm) .. (-1,-.3);
	\draw (.45,-.65) .. controls ++(270:.5cm) and ++(270:.5cm) .. (-.55,-.71);
	\filldraw[fill=white] (.45,-.65) circle (.05cm) node [above] {\scriptsize{$\alpha_\ell$}};
	\filldraw[fill=white] (-.55,-.71) circle (.05cm) node [above] {\scriptsize{$\overline{\alpha_k}$}};
	\roundNbox{unshaded}{(-1,0)}{.3}{.25}{.25}{$j_b(\xi_k)$}
	\roundNbox{unshaded}{(0,0)}{.3}{0}{0}{$f_i$}
	\roundNbox{unshaded}{(.8,0)}{.3}{0}{0}{$\xi_\ell$}
	\node at (.2,-.7) {\scriptsize{$\mathbf{a}$}};
	\node at (-.2,-.75) {\scriptsize{$\overline{\mathbf{a}}$}};
	\node at (.6,-.8) {\scriptsize{$\mathbf{c}$}};
	\node at (-.7,-.85) {\scriptsize{$\overline{\mathbf{c}}$}};
	\node at (.95,-.5) {\scriptsize{$\mathbf{b}$}};
	\node at (-1.15,-.5) {\scriptsize{$\overline{\mathbf{b}}$}};
	\node at (1,.5) {\scriptsize{$\bfM$}};
	\node at (-1.2,.5) {\scriptsize{$\bfM$}};
	\node at (-.2,.5) {\scriptsize{$\bfM$}};
	\node at (.6,.9) {\scriptsize{$\bfM$}};
	\node at (-.1,1.3) {\scriptsize{$\bfM$}};
\end{tikzpicture}
\right)
\\&\nearrow
\phi\left(
\begin{tikzpicture}[baseline=.1cm]
	\draw (.4,.7) .. controls ++(90:.6cm) and ++(90:.6cm) .. (-1,.7) -- (-1,.3);
	\draw (0,.3) arc (180:0:.4cm);
	\filldraw (.4,.7) circle (.05cm);
	\filldraw (-.3,1.15) circle (.05cm);
	\draw (.1,-.3) arc (-180:0:.35cm);
	\draw (-.3,1.15) -- (-.3,1.4);
	\draw (-.1,-.3) .. controls ++(270:.55cm) and ++(270:.55cm) .. (-1,-.3);
	\draw (.45,-.65) .. controls ++(270:.5cm) and ++(270:.5cm) .. (-.55,-.71);
	\filldraw[fill=white] (.45,-.65) circle (.05cm) node [above] {\scriptsize{$\alpha_\ell$}};
	\filldraw[fill=white] (-.55,-.71) circle (.05cm) node [above] {\scriptsize{$\overline{\alpha_k}$}};
	\roundNbox{unshaded}{(-1,0)}{.3}{.25}{.25}{$j_b(\xi_k)$}
	\roundNbox{unshaded}{(0,0)}{.3}{0}{0}{$f$}
	\roundNbox{unshaded}{(.8,0)}{.3}{0}{0}{$\xi_\ell$}
	\node at (.2,-.7) {\scriptsize{$\mathbf{a}$}};
	\node at (-.2,-.75) {\scriptsize{$\overline{\mathbf{a}}$}};
	\node at (.6,-.8) {\scriptsize{$\mathbf{c}$}};
	\node at (-.7,-.85) {\scriptsize{$\overline{\mathbf{c}}$}};
	\node at (.95,-.5) {\scriptsize{$\mathbf{b}$}};
	\node at (-1.15,-.5) {\scriptsize{$\overline{\mathbf{b}}$}};
	\node at (1,.5) {\scriptsize{$\bfM$}};
	\node at (-1.2,.5) {\scriptsize{$\bfM$}};
	\node at (-.2,.5) {\scriptsize{$\bfM$}};
	\node at (.6,.9) {\scriptsize{$\bfM$}};
	\node at (-.1,1.3) {\scriptsize{$\bfM$}};
\end{tikzpicture}
\right)
=
\sum_{k,\ell=1}^n
\langle (\alpha_k\boxtimes \xi_k |\check{\pi}(f) (\alpha_\ell\boxtimes \xi_\ell)\rangle_c.
\end{align*}
(Since $(f_i)$ is bounded, we need only check $\check{\pi}(f_i)\nearrow \check{\pi}(f)$ on a dense subspace of $H_a$.)
This readily follows from the fact that $\cM_\bfM$ is a W*-category, the positivity and normality of the conditional expectation $E_c$ from \eqref{eq:ConditionalExpectation}, and the normality of $\phi$ on $\bfM(1_\cC)$.
\end{proof}

\begin{cor}
The left and right $\bfM$ action on $L^2(\bfM)$ from \eqref{eq:LeftAndRightActionOnL2A} are normal.
\end{cor}

We now want to prove the analog of the fact that $JMJ = M'$ in the case of a finite von Neumann algebra $(M, \tr_M)$.
We recall the proof in the ordinary operator algebra setting so that we may adapt the technique for a finite von Neumann algebra object $(\bfM, \tau)\in \Vec(\cC)$.
This proof is certainly well-known to experts, but we are unaware if this particular proof appears in the literature.

\begin{thm}
Let $(M,\tr_M)$ be a finite von Neumann algebra, and let $J: L^2(M,\tr_M) \to L^2(M,\tr_M)$ be the conjugate linear unitary given by the extension of $x\Omega\mapsto x^*\Omega$, where $\Omega\in L^2(M,\tr_M)$ is the image of $1\in M$.
Then $JMJ=M'$.
\end{thm}
\begin{proof}
It is trivial that $JMJ\subseteq M'$ since left and right multiplication commute.
It is easy to compute that for all $x\in M'$, $Jx\Omega = x^*\Omega$.
Indeed, for all $f\in M$,
\begin{equation}
\label{eq:JxOmega}
\langle Jx\Omega, f\Omega\rangle
=
\langle Jf\Omega, x\Omega\rangle
=
\langle f^*\Omega, x\Omega\rangle
=
\langle x^*\Omega, f\Omega\rangle.
\end{equation} 
Using this, we compute $M'\subseteq JMJ$ by showing $(JMJ)' = JM'J \subseteq M$.
Indeed, for $x,y\in M'$ and $f,g\in M$, we have
\begin{equation}
\label{eq:JM'JinM}
\begin{split}
\langle xJyJ f\Omega, g\Omega \rangle
&=
\langle Jyf^*\Omega, x^*g\Omega \rangle
=
\langle Jf^*y\Omega, gx^*\Omega \rangle
=
\langle Jgx^*\Omega , f^*y\Omega \rangle
\\&=
\langle JgJx\Omega , f^*y\Omega \rangle
=
\langle f x\Omega , Jg^*Jy\Omega \rangle
=
\langle f x\Omega , Jg^*y^*\Omega \rangle
\\&=
\langle xf \Omega , Jy^*g^*\Omega \rangle
=
\langle xf \Omega , Jy^*Jg\Omega \rangle
=
\langle JyJxf \Omega , g\Omega \rangle,
\end{split}
\end{equation}
which relied on the fact that $(JzJ)^* = Jz^*J$ for $z=g,y$.
\end{proof}

To generalize this result to $(\bfM, \tau)\in \Vec(\cC)$, we first need to define the analogs of $J$ and $\Omega$, and define the analogs of the vectors $\Omega$, $f\Omega$ for $f\in \bfM(a)$, and $x\Omega$ for $x\in \bfM'(b)$.
We should then see that $J\bfM J \subseteq \bfM'$.
By construction, we should have $Jf\Omega =j_a^\bfM(f)\Omega$, and we will need to prove $Jx\Omega  = j_b^{\bfM'}(x)\Omega$.
We should then be able to follow the string of inequalities in the above proof to conclude the other inclusion.

\begin{defn}
For $a,b\in \cC$, we define a conjugate linear natural transformation $J^{a,b} : \mathbf{a}\otimes L^2(\bfM) \otimes \mathbf{b} \Rightarrow \overline{\mathbf{b}}\otimes L^2(\bfM) \otimes \overline{\mathbf{a}}$ by defining its $c$-component $J^{a,b}_c$ for $c\in \cC$ by the extension of the map
$$
\cC(d, a\otimes e)\boxtimes \widehat{\bfM}(e)\boxtimes \cC(c, d\otimes b)
\ni
\alpha \boxtimes \xi \boxtimes \beta
\mapsto
\overline{\beta} \boxtimes j_e(\xi) \boxtimes \overline{\alpha}
\in
\cC(\overline{c}, \overline{b}\otimes \overline{d})\boxtimes  \widehat{\bfM}(\overline{e})\boxtimes \cC(\overline{d}, \overline{e}\otimes \overline{a})
$$
for $d,e\in \Irr(\cC)$.
(Recall $\widehat{\bfM}(e)$ is the image of $\bfM(e)$ inside $L^2(\bfM)(e)$ as in Section \ref{sec:StateAndGNS}.)
\end{defn}

The following lemmas are straightforward calculations.

\begin{lem}
Using Definition \ref{defn:ConjugateLinearNaturalTransformation} for composition for conjugate linear natural transformations, and suppressing $\varphi$'s, the the conjugate linear natural transformations $J$ satisfy $J^{\overline{b}, \overline{a}} \circ J^{a,b} = \id_{\mathbf{a}\otimes L^2(\bfM)\otimes \mathbf{b}}$.
\end{lem}

\begin{lem}
\label{lem:ConjugateInnerProduct}
For $\alpha \boxtimes \xi \boxtimes \beta\in \cC(d, a\otimes e)\boxtimes \widehat{\bfM}(e)\boxtimes \cC(c, d\otimes b)$
and $\gamma \boxtimes \eta \boxtimes \delta\in \cC(\overline{c}, \overline{b}\otimes \overline{d})\boxtimes \widehat{\bfM}(\overline{e})\boxtimes \cC(\overline{d}, \overline{e}\otimes \overline{a})$,
$$
\langle 
J^{a,b} (\alpha \boxtimes \xi \boxtimes \beta)
| 
\gamma \boxtimes \eta \boxtimes \delta
\rangle_{(\mathbf{\overline{b}}\otimes L^2(\bfM)\otimes \mathbf{\overline{a}})(c)} 
= 
\langle 
J^{\overline{b},\overline{a}}(\gamma \boxtimes \eta \boxtimes \delta)
| 
\alpha \boxtimes \xi \boxtimes \beta 
\rangle_{(\mathbf{a}\otimes L^2(\bfM)\otimes \mathbf{b})(c)}.
$$
\end{lem}

\begin{lem}
\label{lem:C-C-bilinearJ}
Conjugation by $J$ is an anti $\cC-\cC$ bimodule natural transformation.
That is, for every $x\in \cB_\bfH(\mathbf{a}\otimes \bfH \otimes \mathbf{b} , \mathbf{c}\otimes \bfH \otimes \mathbf{d})$ and $\psi\in \cC(a', c')$ and $\phi\in \cC(b',d')$, we have
$$
\begin{tikzpicture}[baseline = -.1cm]
	\draw (0,1.5) -- (0,-1.5);
	\draw (.4,1.5) -- (.4,-1.5);
	\draw (-.4,1.5) -- (-.4,-1.5);
	\draw (1.1,1.5) -- (1.1,-1.5);
	\draw (-1.1,1.5) -- (-1.1,-1.5);
	\roundNbox{unshaded}{(0,1)}{.3}{1.1}{1.1}{$J$}
	\roundNbox{unshaded}{(1.1,0)}{.3}{0}{0}{$\overline{\psi}$}
	\roundNbox{unshaded}{(0,0)}{.3}{.3}{.3}{$x$}
	\roundNbox{unshaded}{(-1.1,0)}{.3}{0}{0}{$\overline{\phi}$}
	\roundNbox{unshaded}{(0,-1)}{.3}{1.1}{1.1}{$J$}
	\node at (0,-1.7) {\scriptsize{$\bfH$}};
	\node at (-.4,-1.7) {\scriptsize{$\mathbf{a}$}};
	\node at (.4,-1.7) {\scriptsize{$\mathbf{b}$}};
	\node at (-1.1,-1.7) {\scriptsize{$\mathbf{a}'$}};
	\node at (1.1,-1.7) {\scriptsize{$\mathbf{b}'$}};
	\node at (-1.3,-.5) {\scriptsize{$\overline{\mathbf{b}'}$}};
	\node at (1.3,-.5) {\scriptsize{$\overline{\mathbf{a}'}$}};
	\node at (-.6,-.5) {\scriptsize{$\overline{\mathbf{b}}$}};
	\node at (.6,-.5) {\scriptsize{$\overline{\mathbf{a}}$}};
	\node at (-.6,.5) {\scriptsize{$\overline{\mathbf{d}}$}};
	\node at (.6,.5) {\scriptsize{$\overline{\mathbf{c}}$}};
	\node at (-1.3,.5) {\scriptsize{$\overline{\mathbf{d}'}$}};
	\node at (1.3,.5) {\scriptsize{$\overline{\mathbf{c}'}$}};
	\node at (-1.1,1.7) {\scriptsize{$\mathbf{c}'$}};
	\node at (1.1,1.7) {\scriptsize{$\mathbf{d}'$}};
	\node at (-.4,1.7) {\scriptsize{$\mathbf{c}$}};
	\node at (.4,1.7) {\scriptsize{$\mathbf{d}$}};
	\node at (0,1.7) {\scriptsize{$\bfH$}};
\end{tikzpicture}
=
\begin{tikzpicture}[baseline = -.1cm]
	\draw (0,1.5) -- (0,-1.5);
	\draw (.4,1.5) -- (.4,-1.5);
	\draw (-.4,1.5) -- (-.4,-1.5);
	\draw (1.1,1.5) -- (1.1,-1.5);
	\draw (-1.1,1.5) -- (-1.1,-1.5);
	\roundNbox{unshaded}{(0,1)}{.3}{.3}{.3}{$J$}
	\roundNbox{unshaded}{(1.1,0)}{.3}{0}{0}{$\phi$}
	\roundNbox{unshaded}{(0,0)}{.3}{.3}{.3}{$x$}
	\roundNbox{unshaded}{(-1.1,0)}{.3}{0}{0}{$\psi$}
	\roundNbox{unshaded}{(0,-1)}{.3}{.3}{.3}{$J$}
	\node at (0,-1.7) {\scriptsize{$\bfH$}};
	\node at (-.4,-1.7) {\scriptsize{$\mathbf{a}$}};
	\node at (.4,-1.7) {\scriptsize{$\mathbf{b}$}};
	\node at (-1.1,-1.7) {\scriptsize{$\mathbf{a}'$}};
	\node at (1.1,-1.7) {\scriptsize{$\mathbf{b}'$}};
	\node at (-.6,-.5) {\scriptsize{$\overline{\mathbf{b}}$}};
	\node at (.6,-.5) {\scriptsize{$\overline{\mathbf{a}}$}};
	\node at (-.6,.5) {\scriptsize{$\overline{\mathbf{d}}$}};
	\node at (.6,.5) {\scriptsize{$\overline{\mathbf{c}}$}};
	\node at (-1.1,1.7) {\scriptsize{$\mathbf{c}'$}};
	\node at (1.1,1.7) {\scriptsize{$\mathbf{d}'$}};
	\node at (-.4,1.7) {\scriptsize{$\mathbf{c}$}};
	\node at (.4,1.7) {\scriptsize{$\mathbf{d}$}};
	\node at (0,1.7) {\scriptsize{$\bfH$}};
\end{tikzpicture}
\,.
$$
\end{lem}

\begin{prop}
\label{prop:JMJinM'}
For $f\in \bfM(a)$, we have $J^{1_\cC,1_\cC}\lambda(j_a(f))J^{1_\cC,a} =\rho(f)$ as natural transformations $L^2(\bfM)\otimes \mathbf{a}\Rightarrow L^2(\bfM)$.
\end{prop}
\begin{proof}
Fix $c\in\Irr(\cC)$ and $\xi\boxtimes \alpha \in \widehat{\bfM}(b)\boxtimes \cC(c, b\otimes a)$.
Then 
\begin{align*}
(J^{1_\cC,1_\cC}_c\circ \lambda(j_a(f))_c\circ J^{1_\cC,a}_c) (\xi\boxtimes \alpha )
&= 
(J^{1_\cC,1_\cC}_c\circ \lambda(j_a(f))_c) (\overline{\alpha} \boxtimes j_b(\xi))
\\&=
J^{1_\cC,1_\cC}_c 
\left(
\begin{tikzpicture}[baseline=-.1cm, xscale=-1]
	\draw (0,.9) -- (0,1.3);
	\draw (.6,.3) arc (0:180:.6cm);
	\draw (.6,-.3) arc (0:-180:.6cm);
	\draw (0,-.9) -- (0,-1.3);
	\filldraw (0,.9) circle (.05cm);
	\filldraw[fill=white] (0,-.9) circle (.05cm) node [above] {\scriptsize{$\overline{\alpha}$}};
	\roundNbox{unshaded}{(-.6,0)}{.3}{.2}{.2}{$j_b(\xi)$}
	\roundNbox{unshaded}{(.6,0)}{.3}{.2}{.2}{$j_a(f)$}
	\node at (.2,1.1) {\scriptsize{$\widehat{\bfM}$}};
	\node at (-.8,.5) {\scriptsize{$\widehat{\bfM}$}};
	\node at (.8,.5) {\scriptsize{$\bfM$}};
	\node at (.8,-.5) {\scriptsize{$\overline{\mathbf{a}}$}};
	\node at (-.8,-.5) {\scriptsize{$\overline{\mathbf{b}}$}};
	\node at (.2,-1.1) {\scriptsize{$\overline{\mathbf{c}}$}};
\end{tikzpicture}
\right)
=
\begin{tikzpicture}[baseline=-.1cm]
	\draw (0,.8) -- (0,1.2);
	\draw (.5,.3) arc (0:180:.5cm);
	\draw (.5,-.3) arc (0:-180:.5cm);
	\draw (0,-.8) -- (0,-1.2);
	\filldraw (0,.8) circle (.05cm);
	\filldraw[fill=white] (0,-.8) circle (.05cm) node [above] {\scriptsize{$\alpha$}};
	\roundNbox{unshaded}{(-.6,0)}{.3}{0}{0}{$\xi$}
	\roundNbox{unshaded}{(.6,0)}{.3}{0}{0}{$f$}
	\node at (.2,1.1) {\scriptsize{$\widehat{\bfM}$}};
	\node at (-.8,.5) {\scriptsize{$\widehat{\bfM}$}};
	\node at (.8,.5) {\scriptsize{$\bfM$}};
	\node at (.8,-.5) {\scriptsize{$\mathbf{a}$}};
	\node at (-.8,-.5) {\scriptsize{$\mathbf{b}$}};
	\node at (.2,-1.1) {\scriptsize{$\mathbf{c}$}};
\end{tikzpicture}
=
\rho(f)_c(\xi\boxtimes \alpha ).
\end{align*}
\end{proof}

\begin{defn}
Using the left and right actions $\lambda, \rho$ of $\bfM$ on $L^2(\bfM)$ from \eqref{eq:LeftAndRightActionOnL2A}, we define the W*-algebra object $J\bfM J\in \Vec(\cC)$ is the algebra object corresponding to the \emph{left} $\cC$-module W*-category generated by the $J\lambda(f)J=\rho(j(f))$ for $f\in \bfM(a)$.
(Note that the left $\cC$-action here is given by $\psi\triangleright \rho(f) = \rho(f)\otimes \overline{\psi}$.)
By Proposition \ref{prop:JMJinM'}, we have $J\bfM J \subseteq \bfM'$, i.e., there is a canonical injective $*$-algebra natural transformation $J\bfM J \Rightarrow \bfM'$.
\end{defn}

We now define $\Omega: \mathbf{1} \Rightarrow L^2(\bfM)$ to be the bounded natural transformation corresponding to $i_\bfM\in \bfM(1_\cC)$.
For $f\in \bfM(a)$, we see $f\Omega \in \widehat{\bfM}(a)\subset L^2(\bfM)(a)$ corresponds to $\lambda(f)(\id_\mathbf{a}\otimes \Omega): \mathbf{a} \Rightarrow L^2(\bfM)$:
$$
\begin{tikzpicture}[baseline = -.1cm]
    \draw (0,-.3) -- (0,-.7);
    \draw (0,.3) -- (0,.7);
    \roundNbox{unshaded}{(0,0)}{.3}{.2}{.2}{$f\Omega$}
    \node at (.2,-.5) {\scriptsize{$\mathbf{a}$}};
    \node at (.2,.5) {\scriptsize{$\widehat{\bfM}$}};
\end{tikzpicture}
=
\begin{tikzpicture}[baseline = -.1cm]
    \draw (0,.3) arc (180:0:.3cm);
    \draw (0,-.3) -- (0,-.7);
    \draw (.3,.6) -- (.3,1);
    \filldraw (.3,.6) circle (.05cm);
    \filldraw (.6,.3) circle (.05cm) node [right] {\scriptsize{$i_\bfM$}};
    \roundNbox{unshaded}{(0,0)}{.3}{0}{0}{$f$}
    \node at (.2,-.5) {\scriptsize{$\mathbf{a}$}};
    \node at (.5,.8) {\scriptsize{$\widehat{\bfM}$}};
    \node at (-.2,.5) {\scriptsize{$\bfM$}};
\end{tikzpicture}
=
\begin{tikzpicture}[baseline = -.1cm]
    \draw (-.2,-.3) -- (-.2,-.7);
    \draw (.2,-.3) -- (.2,-.5);
    \draw (0,.3) -- (0,.7);
    \filldraw (.2,-.5) circle (.05cm) node [right] {\scriptsize{$\Omega$}};
    \roundNbox{unshaded}{(0,0)}{.3}{.2}{.2}{$\lambda(f)$}
    \node at (-.4,-.5) {\scriptsize{$\mathbf{a}$}};
    \node at (.5,.5) {\scriptsize{$L^2(\bfM)$}};
\end{tikzpicture}
\,.
$$
Moreover, it is immediate that $Jf\Omega = J^{1_\cC,1_\cC}_a f\Omega=j_a(f)\Omega$. 

\begin{lem}
\label{lem:JxOmega=x*Omega}
Suppose $x\in \bfM'(b)\subseteq {}_\bfH\cM({}_{\bfM'}b, {}_{\bfM'}1)$.
We have $J x\Omega = j_b^{\bfM'}(x) \Omega$.
\end{lem}
\begin{proof}
The proof follows \eqref{eq:JxOmega} almost exactly.
Suppose $f\in \bfM(b)$. Then
$$
\langle f\Omega|Jx\Omega \rangle 
=
\langle x\Omega | Jf\Omega \rangle
=
\begin{tikzpicture}[baseline = -.1cm]
	\draw (-.2,1) -- (-.2,.5);
	\draw (.2,-1) -- (.2,-.5);
	\draw (0,.5) -- (0,-.5);
	\draw (.2,.8) arc (180:0:.4cm) -- (1,-.8) arc (0:-180:.6cm);
	\filldraw (-.2,1) circle (.05cm) node [left] {\scriptsize{$\Omega^*$}};
	\filldraw (.2,-1) circle (.05cm) node [right] {\scriptsize{$\Omega$}};
	\roundNbox{unshaded}{(0,.5)}{.3}{.2}{.2}{$x^*$}
	\roundNbox{unshaded}{(0,-.5)}{.3}{.5}{.5}{$\lambda(j_b(f))$}
	\node at (-.5,0) {\scriptsize{$L^2(\bfM)$}};
\end{tikzpicture}
=
\begin{tikzpicture}[baseline = -.1cm]
	\draw (0,-1) -- (0,1);
	\draw (.2,-.2) arc (180:0:.2cm) -- (.6,-.8) arc (0:-180:.6cm) -- (-.6,.5);
	\filldraw (0,1) circle (.05cm) node [left] {\scriptsize{$\Omega^*$}};
	\filldraw (0,-1) circle (.05cm) node [right] {\scriptsize{$\Omega$}};
	\roundNbox{unshaded}{(0,-.5)}{.3}{.1}{.1}{$x^*$}
	\roundNbox{unshaded}{(0,.5)}{.3}{.5}{.5}{$\lambda(j_b(f))$}
\end{tikzpicture}
=
\begin{tikzpicture}[baseline = -.1cm]
	\draw (0,-1) -- (0,1);
	\draw (.2,-.2) arc (180:0:.2cm) -- (.6,-.8) arc (-180:0:.2cm) -- (1,.8) .. controls ++(90:.7cm) and ++(90:.7cm) .. (-1,.8) -- (-1,.2) arc (-180:0:.2cm);
	\filldraw (0,1) circle (.05cm) node [left] {\scriptsize{$\Omega^*$}};
	\filldraw (0,-1) circle (.05cm) node [right] {\scriptsize{$\Omega$}};
	\roundNbox{unshaded}{(0,-.5)}{.3}{.1}{.1}{$x^*$}
	\roundNbox{unshaded}{(0,.5)}{.3}{.5}{.5}{$\lambda(j_b(f))$}
\end{tikzpicture}
\,.
$$
The first equality follows from Lemma \ref{lem:ConjugateInnerProduct}, the third equality follows from commutation between $\bfM$ and $\bfM'$, and the last equality follows from sphericality, since we have balanced solutions to the conjugate equations.
Now we use the correspondence between the dagger structure of $\cB_\bfH$ and the definition of $j^{\bfM}$ and $j^{\bfM'}$ to see that the right hand side above is equal to
\begin{align*}
\begin{tikzpicture}[baseline = -.1cm, yscale=-1]
	\draw (-.2,1) -- (-.2,.5);
	\draw (.2,-1) -- (.2,-.5);
	\draw (0,.5) -- (0,-.5);
	\draw (.2,.8) arc (180:0:.4cm) -- (1,-.8) arc (0:-180:.6cm);
	\filldraw (-.2,1) circle (.05cm) node [left] {\scriptsize{$\Omega$}};
	\filldraw (.2,-1) circle (.05cm) node [right] {\scriptsize{$\Omega^*$}};
	\roundNbox{unshaded}{(0,.5)}{.3}{.5}{.5}{$j^{\bfM'}_b(x)$}
	\roundNbox{unshaded}{(0,-.5)}{.3}{.3}{.3}{$\lambda(f)^*$}
	\node at (-.5,0) {\scriptsize{$L^2(\bfM)$}};
\end{tikzpicture}
&=
\langle f\Omega | j_b^{\bfM'}(x)\Omega\rangle.
\qedhere
\end{align*}
\end{proof}

\begin{thm}
For all $x\in \bfM'(b)$ and $y\in \bfM'(\overline{a})$, 
$$
xJyJ 
= 
\begin{tikzpicture}[baseline = -.1cm]
	\draw (0,2) -- (0,-2);
	\draw (-.4,-2) -- (-.4,.5);
	\draw (.4,-2) -- (.4,-.5);
	\draw (.4,.5) -- (.4,1.5);
	\roundNbox{unshaded}{(.1,1.5)}{.3}{0}{.2}{$x$}
	\roundNbox{unshaded}{(0,.5)}{.3}{.3}{.3}{$J$}
	\roundNbox{unshaded}{(.1,-.5)}{.3}{0}{.2}{$y$}
	\roundNbox{unshaded}{(0,-1.5)}{.3}{.3}{.3}{$J$}
	\node at (0,-2.2) {\scriptsize{$\bfH$}};
	\node at (-.4,-2.2) {\scriptsize{$\mathbf{a}$}};
	\node at (.4,-2.2) {\scriptsize{$\mathbf{b}$}};
	\node at (-.6,-.5) {\scriptsize{$\overline{\mathbf{b}}$}};
	\node at (.6,-1) {\scriptsize{$\overline{\mathbf{a}}$}};
	\node at (.6,1) {\scriptsize{$\mathbf{b}$}};
	\node at (0,2.2) {\scriptsize{$\bfH$}};
\end{tikzpicture}
=
\begin{tikzpicture}[baseline = -.1cm]
	\draw (0,2) -- (0,-2);
	\draw (-.4,-2) -- (-.4,-.5);
	\draw (.4,-2) -- (.4,-1.5);
	\draw (.4,-.5) -- (.4,.5);
	\roundNbox{unshaded}{(0,1.5)}{.3}{.3}{.3}{$J$}
	\roundNbox{unshaded}{(.1,.5)}{.3}{0}{.2}{$y$}
	\roundNbox{unshaded}{(0,-.5)}{.3}{.3}{.3}{$J$}
	\roundNbox{unshaded}{(.1,-1.5)}{.3}{0}{.2}{$x$}
	\node at (0,-2.2) {\scriptsize{$\bfH$}};
	\node at (-.4,-2.2) {\scriptsize{$\mathbf{a}$}};
	\node at (.4,-2.2) {\scriptsize{$\mathbf{b}$}};
	\node at (.6,0) {\scriptsize{$\overline{\mathbf{a}}$}};
	\node at (0,2.2) {\scriptsize{$\bfH$}};
\end{tikzpicture}
=
JyJx.
$$
Thus $J\bfM'J \subseteq \bfM$, and together with Proposition \ref{prop:JMJinM'}, $J\bfM J = \bfM'$.
\end{thm}
\begin{proof}
The proof follows \eqref{eq:JM'JinM} almost exactly.
Fix $c\in\Irr(\cC)$, and suppose
$\alpha\boxtimes f\boxtimes \beta\in \cC(e, a\otimes d)\boxtimes \bfM(d)\boxtimes \cC(c, e\otimes b)$
and $g\in \bfM(c)$.
We calculate $\langle xJyJ (\alpha\boxtimes f\boxtimes \beta)\Omega, g\Omega\rangle_{L^2(\bfM)(c)}$ as follows.
To get the first line of \eqref{eq:JM'JinM}, we get the following equalities by Lemma \ref{lem:C-C-bilinearJ}, $Jf\Omega = j_d^\bfM(f)\Omega$, the fact that $\bfM$ and $\bfM'$ commute, and Lemma \ref{lem:ConjugateInnerProduct} respectively:
$$
\begin{tikzpicture}[baseline = -.1cm]
	\draw (0,3) -- (0,-3);
	\draw (.5,.5) -- (.5,1.5);
	\draw (-.6,.5) -- (-.6,-3) arc (-180:0:.2cm) -- (-.2,-2.8);
	\draw (-.4,-3.2) arc (-180:0:.5cm) -- (.6,-.5);
	\draw (-.3,2.8) -- (-.3,3) .. controls ++(90:.5cm) and ++(90:.5cm) .. (1,3) -- (1,-3.7) .. controls ++(270:.5cm) and ++(270:.5cm) .. (.1,-3.7);
	\filldraw[fill=white] (-.4,-3.2) circle (.05cm) node [above] {\scriptsize{$\alpha$}};
	\filldraw[fill=white] (.1,-3.7) circle (.05cm) node [above] {\scriptsize{$\beta$}};
	\filldraw (0,-3) circle (.05cm) node [right] {\scriptsize{$\Omega$}};
	\filldraw (0,3) circle (.05cm) node [right] {\scriptsize{$\Omega^*$}};
	\roundNbox{unshaded}{(0,2.5)}{.3}{.2}{.2}{$g^*$}
	\roundNbox{unshaded}{(.1,1.5)}{.3}{0}{.4}{$x$}
	\roundNbox{unshaded}{(0,.5)}{.3}{.5}{.5}{$J$}
	\roundNbox{unshaded}{(.1,-.5)}{.3}{0}{.4}{$y$}
	\roundNbox{unshaded}{(0,-1.5)}{.3}{.5}{.5}{$J$}
	\roundNbox{unshaded}{(0,-2.5)}{.3}{.05}{.05}{$f$}
	\node at (1.2,0) {\scriptsize{$\overline{\mathbf{c}}$}};
	\node at (-.8,-.5) {\scriptsize{$\overline{\mathbf{b}}$}};
	\node at (.8,-1) {\scriptsize{$\overline{\mathbf{a}}$}};
	\node at (.7,1) {\scriptsize{$\mathbf{b}$}};
	\node at (-.5,3) {\scriptsize{$\mathbf{c}$}};
\end{tikzpicture}
=
\begin{tikzpicture}[baseline = -.1cm]
	\draw (0,3) -- (0,-3);
	\draw (.5,.5) -- (.5,1.5);
	\draw (-.3,2.8) -- (-.3,3) .. controls ++(90:.5cm) and ++(90:.5cm) .. (1,3) -- (1,.5);
	\draw (-.6,.5) -- (-.6,-3) arc (-180:0:.2cm) -- (-.2,-2.8);
	\draw (-.4,-3.2) arc (-180:0:.5cm) -- (.6,-.5);
	\draw (-1,.5) -- (-1,-1.5);
	\draw (1,-1.5) -- (1,-3.7) .. controls ++(270:.5cm) and ++(270:.5cm) .. (.1,-3.7);
	\filldraw[fill=white] (-.4,-3.2) circle (.05cm) node [above] {\scriptsize{$\alpha$}};
	\filldraw[fill=white] (.1,-3.7) circle (.05cm) node [above] {\scriptsize{$\beta$}};
	\filldraw (0,-3) circle (.05cm) node [right] {\scriptsize{$\Omega$}};
	\filldraw (0,3) circle (.05cm) node [right] {\scriptsize{$\Omega^*$}};
	\roundNbox{unshaded}{(0,2.5)}{.3}{.2}{.2}{$g^*$}
	\roundNbox{unshaded}{(.1,1.5)}{.3}{0}{.4}{$x$}
	\roundNbox{unshaded}{(0,.5)}{.3}{.9}{.9}{$J$}
	\roundNbox{unshaded}{(.1,-.5)}{.3}{0}{.4}{$y$}
	\roundNbox{unshaded}{(0,-1.5)}{.3}{.9}{.9}{$J$}
	\roundNbox{unshaded}{(0,-2.5)}{.3}{.05}{.05}{$f$}
	\node at (-.8,-.5) {\scriptsize{$\overline{\mathbf{b}}$}};
	\node at (.8,-1) {\scriptsize{$\overline{\mathbf{a}}$}};
	\node at (-1.2,-.5) {\scriptsize{$\mathbf{c}$}};
	\node at (1.2,1) {\scriptsize{$\overline{\mathbf{c}}$}};
	\node at (.7,1) {\scriptsize{$\mathbf{b}$}};
	\node at (-.5,3) {\scriptsize{$\mathbf{c}$}};
\end{tikzpicture}
=
\begin{tikzpicture}[baseline = -.1cm]
	\coordinate (a) at (.3,-2.6);
	\coordinate (b) at ($ (a) + (-.5,-.5) $);
	\draw (0,3) -- (0,-2);
	\draw (.5,.5) -- (.5,1.5);
	\draw (-.3,2.8) -- (-.3,3) .. controls ++(90:.5cm) and ++(90:.5cm) .. (1,3) -- (1,.5);
	\draw ($ (a) + (-.6,.8) $) -- ($ (a) + (-.6,.6) $) arc (-180:0:.6cm) -- ($ (a) + (.6,.6) $) -- ($ (a) + (.6,1.8) $);
	\draw (a) arc (0:-180:.5cm) -- (-.7,.5);
	\draw (b) arc (0:-180:.4cm) -- (-1,.5);
	\filldraw[fill=white] (a) circle (.05cm) node [above] {\scriptsize{$\overline{\alpha}$}};
	\filldraw[fill=white] (b) circle (.05cm) node [above] {\scriptsize{$\overline{\beta}$}};
	\filldraw (0,-2) circle (.05cm) node [right] {\scriptsize{$\Omega$}};
	\filldraw (0,3) circle (.05cm) node [right] {\scriptsize{$\Omega^*$}};
	\roundNbox{unshaded}{(0,2.5)}{.3}{.2}{.2}{$g^*$}
	\roundNbox{unshaded}{(.1,1.5)}{.3}{0}{.4}{$x$}
	\roundNbox{unshaded}{(0,.5)}{.3}{.9}{.9}{$J$}
	\roundNbox{unshaded}{(.1,-.5)}{.3}{0}{.8}{$y$}
	\roundNbox{unshaded}{(0,-1.5)}{.3}{.25}{.25}{$j_d(f)$}
	\node at (-.5,-2) {\scriptsize{$\overline{\mathbf{d}}$}};
	\node at (1.05,-1) {\scriptsize{$\overline{\mathbf{a}}$}};
	\node at (-.85,-.5) {\scriptsize{$\overline{\mathbf{b}}$}};
	\node at (-1.15,-.5) {\scriptsize{$\mathbf{c}$}};
	\node at (1.2,1) {\scriptsize{$\overline{\mathbf{c}}$}};
	\node at (.7,1) {\scriptsize{$\mathbf{b}$}};
	\node at (-.5,3) {\scriptsize{$\mathbf{c}$}};
\end{tikzpicture}
=
\begin{tikzpicture}[baseline = -.1cm]
	\coordinate (a) at (.3,-2.6);
	\coordinate (b) at ($ (a) + (-.5,-.5) $);
	\draw (0,3) -- (0,-2);
	\draw (.5,.5) -- (.5,2.5);
	\draw (-.4,1.8) -- (-.4,3) .. controls ++(90:.5cm) and ++(90:.5cm) .. (1,3) -- (1,.5);
	\draw ($ (a) + (-.6,1.8) $) -- ($ (a) + (-.6,.6) $) arc (-180:0:.6cm) -- ($ (a) + (.6,.6) $) -- ($ (a) + (.6,.8) $);
	\draw (a) arc (0:-180:.5cm) -- (-.7,.5);
	\draw (b) arc (0:-180:.4cm) -- (-1,.5);
	\filldraw[fill=white] (a) circle (.05cm) node [above] {\scriptsize{$\overline{\alpha}$}};
	\filldraw[fill=white] (b) circle (.05cm) node [above] {\scriptsize{$\overline{\beta}$}};
	\filldraw (0,-2) circle (.05cm) node [right] {\scriptsize{$\Omega$}};
	\filldraw (0,3) circle (.05cm) node [right] {\scriptsize{$\Omega^*$}};
	\roundNbox{unshaded}{(.1,2.5)}{.3}{0}{.4}{$x$}
	\roundNbox{unshaded}{(0,1.5)}{.3}{.3}{0}{$g^*$}
	\roundNbox{unshaded}{(0,.5)}{.3}{.9}{.9}{$J$}
	\roundNbox{unshaded}{(0,-.5)}{.3}{.25}{.25}{$j_d(f)$}
	\roundNbox{unshaded}{(.1,-1.5)}{.3}{0}{.8}{$y$}
	\node at (1.05,-2) {\scriptsize{$\overline{\mathbf{a}}$}};
	\node at (-.5,-1.5) {\scriptsize{$\overline{\mathbf{d}}$}};
	\node at (-.85,-.5) {\scriptsize{$\overline{\mathbf{b}}$}};
	\node at (-1.15,-.5) {\scriptsize{$\mathbf{c}$}};
	\node at (1.2,1) {\scriptsize{$\overline{\mathbf{c}}$}};
	\node at (.7,1) {\scriptsize{$\mathbf{b}$}};
	\node at (-.6,3) {\scriptsize{$\mathbf{c}$}};
\end{tikzpicture}
=
\begin{tikzpicture}[baseline = -.1cm, yscale=-1]
	\coordinate (a) at (.3,-2.6);
	\coordinate (b) at ($ (a) + (-.5,-.5) $);
	\draw (0,3) -- (0,-2);
	\draw (.5,.5) -- (.5,2.5);
	\draw (-.4,1.8) -- (-.4,3) .. controls ++(90:.5cm) and ++(90:.5cm) .. (1,3) -- (1,.5);
	\draw ($ (a) + (-.6,1.8) $) -- ($ (a) + (-.6,.6) $) arc (-180:0:.6cm) -- ($ (a) + (.6,.6) $) -- ($ (a) + (.6,.8) $);
	\draw (a) arc (0:-180:.5cm) -- (-.7,.5);
	\draw (b) arc (0:-180:.4cm) -- (-1,.5);
	\filldraw[fill=white] (a) circle (.05cm) node [below] {\scriptsize{$\alpha^{\vee}$}};
	\filldraw[fill=white] (b) circle (.05cm) node [below] {\scriptsize{$\beta^{\vee}$}};
	\filldraw (0,-2) circle (.05cm) node [right] {\scriptsize{$\Omega^*$}};
	\filldraw (0,3) circle (.05cm) node [right] {\scriptsize{$\Omega$}};
	\roundNbox{unshaded}{(.1,2.5)}{.3}{0}{.4}{$x^*$}
	\roundNbox{unshaded}{(0,1.5)}{.3}{.3}{0}{$g$}
	\roundNbox{unshaded}{(0,.5)}{.3}{.9}{.9}{$J$}
	\roundNbox{unshaded}{(0,-.5)}{.3}{.3}{.3}{$j_d(f)^*$}
	\roundNbox{unshaded}{(.1,-1.5)}{.3}{0}{.8}{$y^*$}
	\node at (1.05,-2) {\scriptsize{$\overline{\mathbf{a}}$}};
	\node at (-.5,-1.5) {\scriptsize{$\overline{\mathbf{d}}$}};
	\node at (-.85,-.5) {\scriptsize{$\overline{\mathbf{b}}$}};
	\node at (-1.15,-.5) {\scriptsize{$\mathbf{c}$}};
	\node at (1.2,1) {\scriptsize{$\overline{\mathbf{c}}$}};
	\node at (.7,1) {\scriptsize{$\mathbf{b}$}};
	\node at (-.6,3) {\scriptsize{$\mathbf{c}$}};
\end{tikzpicture}
$$
We get the first two terms in the second line of \eqref{eq:JM'JinM} by applying isotopy, using the definition of $j^{\bfM'}$, applying Lemma \ref{lem:JxOmega=x*Omega}, and using Proposition \ref{prop:JMJinM'} respectively:
$$
\begin{tikzpicture}[baseline = -.1cm, yscale=-1]
	\coordinate (a) at (.3,-2.6);
	\coordinate (b) at ($ (a) + (-.5,-.5) $);
	\draw (0,3) -- (0,-2);
	\draw (.5,.5) -- (.5,2.5);
	\draw (-.4,1.8) -- (-.4,3) .. controls ++(90:.5cm) and ++(90:.5cm) .. (1,3) -- (1,.5);
	\draw ($ (a) + (-.6,1.8) $) -- ($ (a) + (-.6,.6) $) arc (-180:0:.6cm) -- ($ (a) + (.6,.6) $) -- ($ (a) + (.6,.8) $);
	\draw (a) arc (0:-180:.5cm) -- (-.7,.5);
	\draw (b) arc (0:-180:.4cm) -- (-1,.5);
	\filldraw[fill=white] (a) circle (.05cm) node [below] {\scriptsize{$\alpha^{\vee}$}};
	\filldraw[fill=white] (b) circle (.05cm) node [below] {\scriptsize{$\beta^{\vee}$}};
	\filldraw (0,-2) circle (.05cm) node [right] {\scriptsize{$\Omega^*$}};
	\filldraw (0,3) circle (.05cm) node [right] {\scriptsize{$\Omega$}};
	\roundNbox{unshaded}{(.1,2.5)}{.3}{0}{.4}{$x^*$}
	\roundNbox{unshaded}{(0,1.5)}{.3}{.3}{0}{$g$}
	\roundNbox{unshaded}{(0,.5)}{.3}{.9}{.9}{$J$}
	\roundNbox{unshaded}{(0,-.5)}{.3}{.3}{.3}{$j_d(f)^*$}
	\roundNbox{unshaded}{(.1,-1.5)}{.3}{0}{.8}{$y^*$}
	\node at (1.05,-2) {\scriptsize{$\overline{\mathbf{a}}$}};
	\node at (-.5,-1.5) {\scriptsize{$\overline{\mathbf{d}}$}};
	\node at (-.85,-.5) {\scriptsize{$\overline{\mathbf{b}}$}};
	\node at (-1.15,-.5) {\scriptsize{$\mathbf{c}$}};
	\node at (1.2,1) {\scriptsize{$\overline{\mathbf{c}}$}};
	\node at (.7,1) {\scriptsize{$\mathbf{b}$}};
	\node at (-.6,3) {\scriptsize{$\mathbf{c}$}};
\end{tikzpicture}
=
\begin{tikzpicture}[baseline = -.1cm, yscale=-1]
	\coordinate (a) at (.3,-2.6);
	\coordinate (b) at ($ (a) + (-.5,-.5) $);
	\draw (0,3) -- (0,-2);
	\draw (.2,2.2) arc (-180:0:.15cm) -- (.5,2.8) arc (180:0:.15cm) -- (.8,.5);
	\draw (-.4,1.8) -- (-.4,3) .. controls ++(90:.5cm) and ++(90:.5cm) .. (1,3) -- (1,.5);
	\draw ($ (a) + (-.6,1.8) $) -- ($ (a) + (-.6,.6) $) arc (-180:0:.6cm) -- ($ (a) + (.6,.6) $) -- ($ (a) + (.6,.8) $);
	\draw (a) arc (0:-180:.5cm) -- (-.7,.5);
	\draw (b) arc (0:-180:.4cm) -- (-1,.5);
	\filldraw[fill=white] (a) circle (.05cm) node [below] {\scriptsize{$\alpha^{\vee}$}};
	\filldraw[fill=white] (b) circle (.05cm) node [below] {\scriptsize{$\beta^{\vee}$}};
	\filldraw (0,-2) circle (.05cm) node [right] {\scriptsize{$\Omega^*$}};
	\filldraw (0,3) circle (.05cm) node [right] {\scriptsize{$\Omega$}};
	\roundNbox{unshaded}{(.1,2.5)}{.3}{0}{0}{$x^*$}
	\roundNbox{unshaded}{(0,1.5)}{.3}{.3}{0}{$g$}
	\roundNbox{unshaded}{(0,.5)}{.3}{.9}{.9}{$J$}
	\roundNbox{unshaded}{(0,-.5)}{.3}{.3}{.3}{$j_d(f)^*$}
	\roundNbox{unshaded}{(.1,-1.5)}{.3}{0}{.8}{$y^*$}
	\node at (1.05,-2) {\scriptsize{$\overline{\mathbf{a}}$}};
	\node at (-.5,-1.5) {\scriptsize{$\overline{\mathbf{d}}$}};
	\node at (-.85,-.5) {\scriptsize{$\overline{\mathbf{b}}$}};
	\node at (-1.15,-.5) {\scriptsize{$\mathbf{c}$}};
	\node at (1.2,1) {\scriptsize{$\overline{\mathbf{c}}$}};
	\node at (.6,1) {\scriptsize{$\mathbf{b}$}};
	\node at (-.6,3) {\scriptsize{$\mathbf{c}$}};
\end{tikzpicture}
=
\begin{tikzpicture}[baseline = -.1cm, yscale=-1]
	\coordinate (a) at (.3,-2.6);
	\coordinate (b) at ($ (a) + (-.5,-.5) $);
	\draw (0,3) -- (0,-2);
	\draw (.6,2.8) arc (180:0:.2cm) -- (1,.5);
	\draw (-.8,1.8) -- (-.8,3) .. controls ++(90:.7cm) and ++(90:.7cm) .. (1.2,3) -- (1.2,.5);
	\draw ($ (a) + (-.6,1.8) $) -- ($ (a) + (-.6,.6) $) arc (-180:0:.6cm) -- ($ (a) + (.6,.6) $) -- ($ (a) + (.6,.8) $);
	\draw (a) arc (0:-180:.5cm) -- (-.7,.5);
	\draw (b) arc (0:-180:.4cm) -- (-1,.5);
	\filldraw[fill=white] (a) circle (.05cm) node [below] {\scriptsize{$\alpha^{\vee}$}};
	\filldraw[fill=white] (b) circle (.05cm) node [below] {\scriptsize{$\beta^{\vee}$}};
	\filldraw (0,-2) circle (.05cm) node [right] {\scriptsize{$\Omega^*$}};
	\filldraw (0,3) circle (.05cm) node [right] {\scriptsize{$\Omega$}};
	\roundNbox{unshaded}{(.1,2.5)}{.3}{.4}{.4}{$j^{\bfM'}_{b}(x)$}
	\roundNbox{unshaded}{(0,1.5)}{.3}{.7}{0}{$g$}
	\roundNbox{unshaded}{(0,.5)}{.3}{1.1}{1.1}{$J$}
	\roundNbox{unshaded}{(0,-.5)}{.3}{.3}{.3}{$j_d(f)^*$}
	\roundNbox{unshaded}{(.1,-1.5)}{.3}{0}{.8}{$y^*$}
	\node at (1.05,-2) {\scriptsize{$\overline{\mathbf{a}}$}};
	\node at (-.85,-.5) {\scriptsize{$\overline{\mathbf{b}}$}};
	\node at (-.5,-1.5) {\scriptsize{$\overline{\mathbf{d}}$}};
	\node at (-1.15,-.5) {\scriptsize{$\mathbf{c}$}};
	\node at (1.4,1) {\scriptsize{$\overline{\mathbf{c}}$}};
	\node at (.8,1) {\scriptsize{$\mathbf{b}$}};
	\node at (-1,2.5) {\scriptsize{$\mathbf{c}$}};
\end{tikzpicture}
=
\begin{tikzpicture}[baseline = -.1cm, yscale=-1]
	\coordinate (a) at (.3,-2.6);
	\coordinate (b) at ($ (a) + (-.5,-.5) $);
	\draw (0,4) -- (0,-2);
	\draw (.6,.5) -- (.6,2.5);
	\draw (.2,3.8) -- (.2,3.9) arc (0:180:.4cm) -- (-.6,2.5);
	\draw (-.3,1.5) -- (-.3,2.5);
	\draw (-1,2.8) -- (-1,4) .. controls ++(90:.7cm) and ++(90:.7cm) .. (1,4) -- (1,.5);
	\draw ($ (a) + (-.6,1.8) $) -- ($ (a) + (-.6,.6) $) arc (-180:0:.6cm) -- ($ (a) + (.6,.6) $) -- ($ (a) + (.6,.8) $);
	\draw (a) arc (0:-180:.5cm) -- (-.7,.5);
	\draw (b) arc (0:-180:.4cm) -- (-1,.5);
	\filldraw[fill=white] (a) circle (.05cm) node [below] {\scriptsize{$\alpha^{\vee}$}};
	\filldraw[fill=white] (b) circle (.05cm) node [below] {\scriptsize{$\beta^{\vee}$}};
	\filldraw (0,-2) circle (.05cm) node [right] {\scriptsize{$\Omega^*$}};
	\filldraw (0,4) circle (.05cm) node [left] {\scriptsize{$\Omega$}};
	\roundNbox{unshaded}{(.1,3.5)}{.3}{0}{0}{$x$}
	\roundNbox{unshaded}{(0,2.5)}{.3}{1}{1}{$J$}
	\roundNbox{unshaded}{(0,1.5)}{.3}{.3}{0}{$g$}
	\roundNbox{unshaded}{(0,.5)}{.3}{1}{1}{$J$}
	\roundNbox{unshaded}{(0,-.5)}{.3}{.3}{.3}{$j_d(f)^*$}
	\roundNbox{unshaded}{(.1,-1.5)}{.3}{0}{.8}{$y^*$}
	\node at (1.05,-2) {\scriptsize{$\overline{\mathbf{a}}$}};
	\node at (-.85,-.5) {\scriptsize{$\overline{\mathbf{b}}$}};
	\node at (-.5,-1.5) {\scriptsize{$\overline{\mathbf{d}}$}};
	\node at (-1.15,-.5) {\scriptsize{$\mathbf{c}$}};
	\node at (1.2,1) {\scriptsize{$\overline{\mathbf{c}}$}};
	\node at (.8,1) {\scriptsize{$\mathbf{b}$}};
	\node at (-1.2,3.5) {\scriptsize{$\mathbf{c}$}};
	\node at (-.8,3.5) {\scriptsize{$\overline{\mathbf{b}}$}};
	\node at (-.5,2) {\scriptsize{$\mathbf{c}$}};
\end{tikzpicture}
=
\begin{tikzpicture}[baseline = -.1cm, yscale=-1]
	\coordinate (a) at (.3,-2.6);
	\coordinate (b) at ($ (a) + (-.5,-.5) $);
	\draw (0,4) -- (0,-2);
	\draw (.8,-.5) -- (.8,1.5);
	\draw (.2,3.8) -- (.2,3.9) .. controls ++(90:.5cm) and ++(90:.5cm) .. (-.8,3.9) -- (-.8,1.5);
	\draw (-.3,.5) -- (-.3,1.5);
	\draw (-.3,1.5) -- (-.3,2.5);
	\draw (.5,1.5) -- (.5,-.5);
	\draw (-1.1,1.8) -- (-1.1,4) .. controls ++(90:.7cm) and ++(90:.7cm) .. (1.1,4) -- (1.1,-.5);
	\draw ($ (a) + (-.6,1.8) $) -- ($ (a) + (-.6,.6) $) arc (-180:0:.6cm) -- ($ (a) + (.6,.6) $) -- ($ (a) + (.6,.8) $);
	\draw (a) arc (0:-180:.5cm) -- (-.7,-.5);
	\draw (b) arc (0:-180:.4cm) -- (-1,-.5);
	\filldraw[fill=white] (a) circle (.05cm) node [below] {\scriptsize{$\alpha^{\vee}$}};
	\filldraw[fill=white] (b) circle (.05cm) node [below] {\scriptsize{$\beta^{\vee}$}};
	\filldraw (0,-2) circle (.05cm) node [right] {\scriptsize{$\Omega^*$}};
	\filldraw (0,4) circle (.05cm) node [left] {\scriptsize{$\Omega$}};
	\roundNbox{unshaded}{(.1,3.5)}{.3}{0}{0}{$x$}
	\roundNbox{unshaded}{(0,2.5)}{.3}{.3}{.3}{$j_d(f)^*$}
	\roundNbox{unshaded}{(0,1.5)}{.3}{1}{1}{$J$}
	\roundNbox{unshaded}{(0,.5)}{.3}{.3}{0}{$g$}
	\roundNbox{unshaded}{(0,-.5)}{.3}{1}{1}{$J$}
	\roundNbox{unshaded}{(.1,-1.5)}{.3}{0}{.8}{$y^*$}
	\node at (1.05,-2) {\scriptsize{$\overline{\mathbf{a}}$}};
	\node at (-.85,-1.5) {\scriptsize{$\overline{\mathbf{b}}$}};
	\node at (-.5,-1.5) {\scriptsize{$\overline{\mathbf{d}}$}};
	\node at (-1.15,-1.5) {\scriptsize{$\mathbf{c}$}};
	\node at (1.25,.5) {\scriptsize{$\overline{\mathbf{c}}$}};
	\node at (.95,.5) {\scriptsize{$\mathbf{b}$}};
	\node at (.65,.5) {\scriptsize{$\mathbf{d}$}};
	\node at (-1.2,3.5) {\scriptsize{$\mathbf{c}$}};
	\node at (-.6,3.5) {\scriptsize{$\overline{\mathbf{b}}$}};
	\node at (-.5,1) {\scriptsize{$\mathbf{c}$}};
	\node at (-.5,2) {\scriptsize{$\overline{\mathbf{d}}$}};
\end{tikzpicture}
$$
Now conjugation by $J$ is a natural transformation, so we may pull $\alpha^{\vee}$ and $\beta^{\vee}$ through the pair of $J$'s to move them to the bottom.
We rotate $\alpha^{\vee}$ and $\beta^{\vee}$ and the $\overline{\mathbf{d}}$ string attached to $j_d(f)^*$ to obtain $\alpha$ and $\beta$ and $f$ again.
We then use Lemma \ref{lem:JxOmega=x*Omega} again follows by the definition of $j^{\bfM'}$ to finish the second line of \eqref{eq:JM'JinM}.
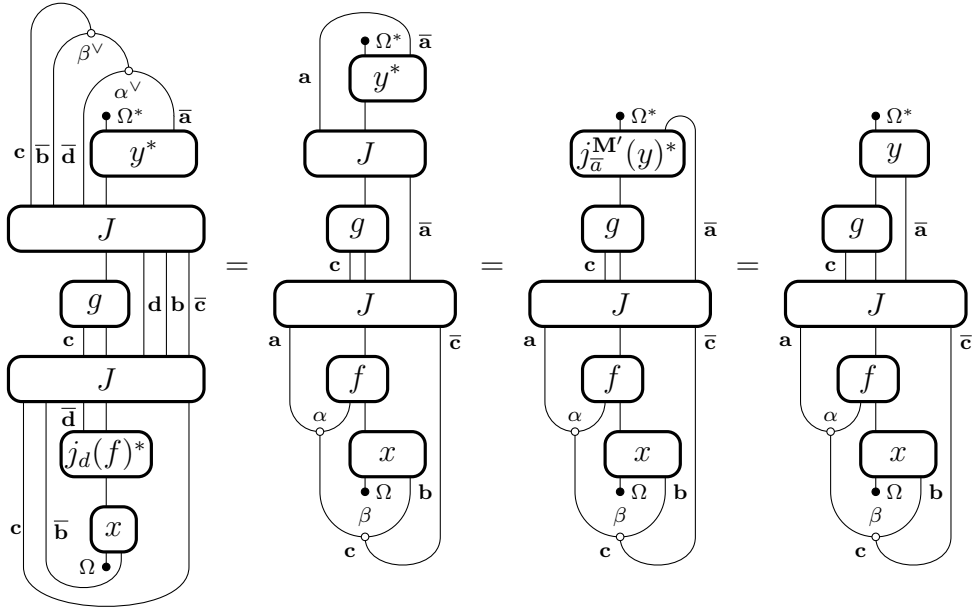
$$
\begin{tikzpicture}[baseline = -.1cm, yscale=-1]
	\coordinate (a) at (.3,-2.6);
	\coordinate (b) at ($ (a) + (-.5,-.5) $);
	\draw (0,4) -- (0,-2);
	\draw (.8,-.5) -- (.8,1.5);
	\draw (.2,3.8) -- (.2,3.9) .. controls ++(90:.5cm) and ++(90:.5cm) .. (-.8,3.9) -- (-.8,1.5);
	\draw (-.3,.5) -- (-.3,1.5);
	\draw (-.3,1.5) -- (-.3,2.5);
	\draw (.5,1.5) -- (.5,-.5);
	\draw (-1.1,1.8) -- (-1.1,4) .. controls ++(90:.7cm) and ++(90:.7cm) .. (1.1,4) -- (1.1,-.5);
	\draw ($ (a) + (-.6,1.8) $) -- ($ (a) + (-.6,.6) $) arc (-180:0:.6cm) -- ($ (a) + (.6,.6) $) -- ($ (a) + (.6,.8) $);
	\draw (a) arc (0:-180:.5cm) -- (-.7,-.5);
	\draw (b) arc (0:-180:.4cm) -- (-1,-.5);
	\filldraw[fill=white] (a) circle (.05cm) node [below] {\scriptsize{$\alpha^{\vee}$}};
	\filldraw[fill=white] (b) circle (.05cm) node [below] {\scriptsize{$\beta^{\vee}$}};
	\filldraw (0,-2) circle (.05cm) node [right] {\scriptsize{$\Omega^*$}};
	\filldraw (0,4) circle (.05cm) node [left] {\scriptsize{$\Omega$}};
	\roundNbox{unshaded}{(.1,3.5)}{.3}{0}{0}{$x$}
	\roundNbox{unshaded}{(0,2.5)}{.3}{.3}{.3}{$j_d(f)^*$}
	\roundNbox{unshaded}{(0,1.5)}{.3}{1}{1}{$J$}
	\roundNbox{unshaded}{(0,.5)}{.3}{.3}{0}{$g$}
	\roundNbox{unshaded}{(0,-.5)}{.3}{1}{1}{$J$}
	\roundNbox{unshaded}{(.1,-1.5)}{.3}{0}{.8}{$y^*$}
	\node at (1.05,-2) {\scriptsize{$\overline{\mathbf{a}}$}};
	\node at (-.85,-1.5) {\scriptsize{$\overline{\mathbf{b}}$}};
	\node at (-.5,-1.5) {\scriptsize{$\overline{\mathbf{d}}$}};
	\node at (-1.15,-1.5) {\scriptsize{$\mathbf{c}$}};
	\node at (1.25,.5) {\scriptsize{$\overline{\mathbf{c}}$}};
	\node at (.95,.5) {\scriptsize{$\mathbf{b}$}};
	\node at (.65,.5) {\scriptsize{$\mathbf{d}$}};
	\node at (-1.2,3.5) {\scriptsize{$\mathbf{c}$}};
	\node at (-.6,3.5) {\scriptsize{$\overline{\mathbf{b}}$}};
	\node at (-.5,1) {\scriptsize{$\mathbf{c}$}};
	\node at (-.5,2) {\scriptsize{$\overline{\mathbf{d}}$}};
\end{tikzpicture}
=
\begin{tikzpicture}[baseline = -.1cm]
	\coordinate (a) at (-.6,-2.2);
	\coordinate (b) at ($ (a) + (.6,-1.4) $);
	\draw (0,3) -- (0,-3);
	\draw (.6,1.5) -- (.6,-.5);
	\draw (.6,2.8) -- (.6,3) .. controls ++(90:.5cm) and ++(90:.5cm) .. (-.6,3) -- (-.6,1.5);
	\draw (-.2,-1.8) arc (0:-180:.4cm) -- (-1,-.8);
	\draw (-.2,.5) -- (-.2,-.5);
	\draw (a) -- ($ (a) + (0,-.8) $) arc (-180:0:.6cm) -- ($ (a) + (1.2,-.6) $);
	\draw (b) .. controls ++(270:.5cm) and ++(270:.5cm) .. (1,-3.6) -- (1,-.5);
	\filldraw[fill=white] (a) circle (.05cm) node [above] {\scriptsize{$\alpha$}};
	\filldraw[fill=white] (b) circle (.05cm) node [above] {\scriptsize{$\beta$}};
	\filldraw (0,-3) circle (.05cm) node [right] {\scriptsize{$\Omega$}};
	\filldraw (0,3) circle (.05cm) node [right] {\scriptsize{$\Omega^*$}};
	\roundNbox{unshaded}{(.1,2.5)}{.3}{0}{.4}{$y^*$}
	\roundNbox{unshaded}{(0,1.5)}{.3}{.5}{.5}{$J$}
	\roundNbox{unshaded}{(0,.5)}{.3}{.2}{0}{$g$}
	\roundNbox{unshaded}{(0,-.5)}{.3}{.9}{.9}{$J$}
	\roundNbox{unshaded}{(0,-1.5)}{.3}{.2}{0}{$f$}
	\roundNbox{unshaded}{(.1,-2.5)}{.3}{0}{.4}{$x$}
	\node at (.8,3) {\scriptsize{$\overline{\mathbf{a}}$}};
	\node at (-.8,2.5) {\scriptsize{$\mathbf{a}$}};
	\node at (.8,.5) {\scriptsize{$\overline{\mathbf{a}}$}};
	\node at (-.4,0) {\scriptsize{$\mathbf{c}$}};
	\node at (-1.2,-1) {\scriptsize{$\mathbf{a}$}};
	\node at (1.2,-1) {\scriptsize{$\overline{\mathbf{c}}$}};
	\node at (.8,-3) {\scriptsize{$\mathbf{b}$}};
	\node at ($ (b) + (-.2,-.2) $) {\scriptsize{$\mathbf{c}$}};
\end{tikzpicture}
=
\begin{tikzpicture}[baseline = -.1cm]
	\coordinate (a) at (-.6,-2.2);
	\coordinate (b) at ($ (a) + (.6,-1.4) $);
	\draw (0,2) -- (0,-3);
	\draw (.6,1.8) arc (180:0:.2cm) -- (1,-.5);
	\draw (-.2,-1.8) arc (0:-180:.4cm) -- (-1,-.8);
	\draw (-.2,.5) -- (-.2,-.5);
	\draw (a) -- ($ (a) + (0,-.8) $) arc (-180:0:.6cm) -- ($ (a) + (1.2,-.6) $);
	\draw (b) .. controls ++(270:.5cm) and ++(270:.5cm) .. (1,-3.6) -- (1,-.5);
	\filldraw[fill=white] (a) circle (.05cm) node [above] {\scriptsize{$\alpha$}};
	\filldraw[fill=white] (b) circle (.05cm) node [above] {\scriptsize{$\beta$}};
	\filldraw (0,-3) circle (.05cm) node [right] {\scriptsize{$\Omega$}};
	\filldraw (0,2) circle (.05cm) node [right] {\scriptsize{$\Omega^*$}};
	\roundNbox{unshaded}{(.1,1.5)}{.3}{.45}{.45}{$j_{\overline{a}}^{\bfM'}(y)^*$}
	\roundNbox{unshaded}{(0,.5)}{.3}{.2}{0}{$g$}
	\roundNbox{unshaded}{(0,-.5)}{.3}{.9}{.9}{$J$}
	\roundNbox{unshaded}{(0,-1.5)}{.3}{.2}{0}{$f$}
	\roundNbox{unshaded}{(.1,-2.5)}{.3}{0}{.4}{$x$}
	\node at (1.2,.5) {\scriptsize{$\overline{\mathbf{a}}$}};
	\node at (-.4,0) {\scriptsize{$\mathbf{c}$}};
	\node at (-1.2,-1) {\scriptsize{$\mathbf{a}$}};
	\node at (1.2,-1) {\scriptsize{$\overline{\mathbf{c}}$}};
	\node at (.8,-3) {\scriptsize{$\mathbf{b}$}};
	\node at ($ (b) + (-.2,-.2) $) {\scriptsize{$\mathbf{c}$}};
\end{tikzpicture}
=
\begin{tikzpicture}[baseline = -.1cm]
	\coordinate (a) at (-.6,-2.2);
	\coordinate (b) at ($ (a) + (.6,-1.4) $);
	\draw (0,2) -- (0,-3);
	\draw (.4,-.5) -- (.4,1.5);
	\draw (-.2,-1.8) arc (0:-180:.4cm) -- (-1,-.8);
	\draw (-.4,.5) -- (-.4,-.5);
	\draw (a) -- ($ (a) + (0,-.8) $) arc (-180:0:.6cm) -- ($ (a) + (1.2,-.6) $);
	\draw (b) .. controls ++(270:.5cm) and ++(270:.5cm) .. (1,-3.6) -- (1,-.5);
	\filldraw[fill=white] (a) circle (.05cm) node [above] {\scriptsize{$\alpha$}};
	\filldraw[fill=white] (b) circle (.05cm) node [above] {\scriptsize{$\beta$}};
	\filldraw (0,-3) circle (.05cm) node [right] {\scriptsize{$\Omega$}};
	\filldraw (0,2) circle (.05cm) node [right] {\scriptsize{$\Omega^*$}};
	\roundNbox{unshaded}{(.1,1.5)}{.3}{0}{.3}{$y$}
	\roundNbox{unshaded}{(-.1,.5)}{.3}{.3}{0}{$g$}
	\roundNbox{unshaded}{(0,-.5)}{.3}{.9}{.9}{$J$}
	\roundNbox{unshaded}{(0,-1.5)}{.3}{.2}{0}{$f$}
	\roundNbox{unshaded}{(.1,-2.5)}{.3}{0}{.4}{$x$}
	\node at (.6,.5) {\scriptsize{$\overline{\mathbf{a}}$}};
	\node at (-.6,0) {\scriptsize{$\mathbf{c}$}};
	\node at (-1.2,-1) {\scriptsize{$\mathbf{a}$}};
	\node at (1.2,-1) {\scriptsize{$\overline{\mathbf{c}}$}};
	\node at (.8,-3) {\scriptsize{$\mathbf{b}$}};
	\node at ($ (b) + (-.2,-.2) $) {\scriptsize{$\mathbf{c}$}};
\end{tikzpicture}
$$
We obtain the final line of \eqref{eq:JM'JinM} by again using that $\bfM$ and $\bfM'$ commute, then using the definition of $j^\bfM$, then using $Jj^\bfM_c(g)\Omega = g\Omega$, and finally Lemma \ref{lem:C-C-bilinearJ} again.
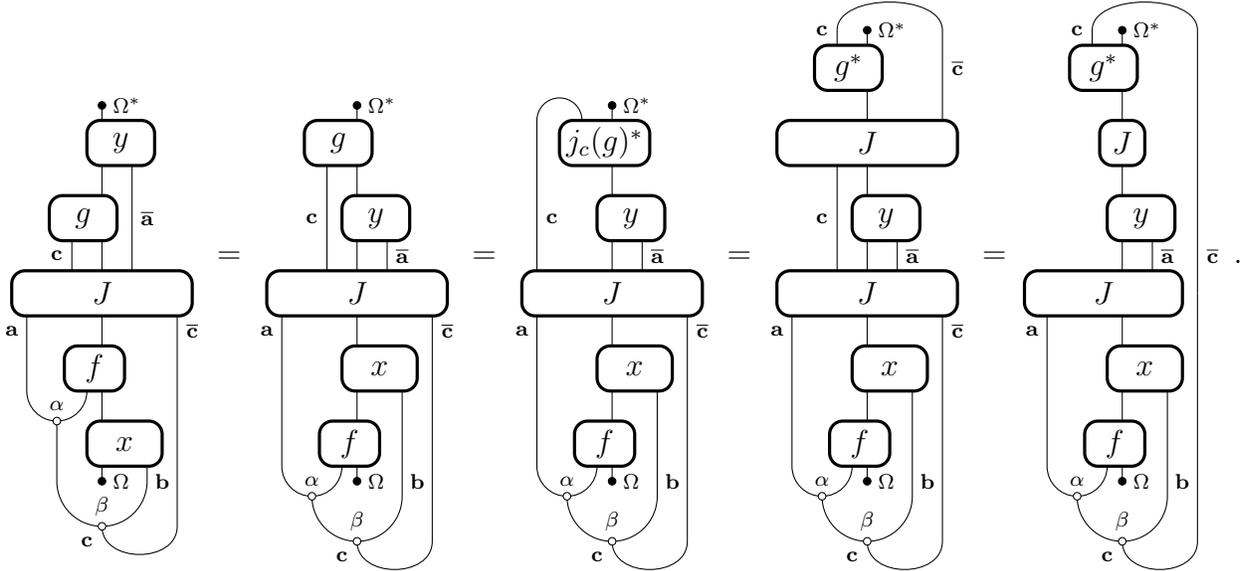
$$
\begin{tikzpicture}[baseline = -.1cm]
	\coordinate (a) at (-.6,-2.2);
	\coordinate (b) at ($ (a) + (.6,-1.4) $);
	\draw (0,2) -- (0,-3);
	\draw (.4,-.5) -- (.4,1.5);
	\draw (-.2,-1.8) arc (0:-180:.4cm) -- (-1,-.8);
	\draw (-.4,.5) -- (-.4,-.5);
	\draw (a) -- ($ (a) + (0,-.8) $) arc (-180:0:.6cm) -- ($ (a) + (1.2,-.6) $);
	\draw (b) .. controls ++(270:.5cm) and ++(270:.5cm) .. (1,-3.6) -- (1,-.5);
	\filldraw[fill=white] (a) circle (.05cm) node [above] {\scriptsize{$\alpha$}};
	\filldraw[fill=white] (b) circle (.05cm) node [above] {\scriptsize{$\beta$}};
	\filldraw (0,-3) circle (.05cm) node [right] {\scriptsize{$\Omega$}};
	\filldraw (0,2) circle (.05cm) node [right] {\scriptsize{$\Omega^*$}};
	\roundNbox{unshaded}{(.1,1.5)}{.3}{0}{.3}{$y$}
	\roundNbox{unshaded}{(-.1,.5)}{.3}{.3}{0}{$g$}
	\roundNbox{unshaded}{(0,-.5)}{.3}{.9}{.9}{$J$}
	\roundNbox{unshaded}{(0,-1.5)}{.3}{.2}{0}{$f$}
	\roundNbox{unshaded}{(.1,-2.5)}{.3}{0}{.4}{$x$}
	\node at (.6,.5) {\scriptsize{$\overline{\mathbf{a}}$}};
	\node at (-.6,0) {\scriptsize{$\mathbf{c}$}};
	\node at (-1.2,-1) {\scriptsize{$\mathbf{a}$}};
	\node at (1.2,-1) {\scriptsize{$\overline{\mathbf{c}}$}};
	\node at (.8,-3) {\scriptsize{$\mathbf{b}$}};
	\node at ($ (b) + (-.2,-.2) $) {\scriptsize{$\mathbf{c}$}};
\end{tikzpicture}
=
\begin{tikzpicture}[baseline = -.1cm]
	\coordinate (a) at (-.6,-3.2);
	\coordinate (b) at ($ (a) + (.6,-.6) $);
	\draw (0,2) -- (0,-3);
	\draw (-.4,-.5) -- (-.4,1.5);
	\draw (-.2,-2.8) arc (0:-180:.4cm) -- (-1,-.8);
	\draw (.4,.5) -- (.4,-.5);
	\draw (a) arc (-180:0:.6cm) -- ($ (a) + (1.2,1.6) $);
	\draw (b) .. controls ++(270:.5cm) and ++(270:.5cm) .. (1,-3.8) -- (1,-.5);
	\filldraw[fill=white] (a) circle (.05cm) node [above] {\scriptsize{$\alpha$}};
	\filldraw[fill=white] (b) circle (.05cm) node [above] {\scriptsize{$\beta$}};
	\filldraw (0,-3) circle (.05cm) node [right] {\scriptsize{$\Omega$}};
	\filldraw (0,2) circle (.05cm) node [right] {\scriptsize{$\Omega^*$}};
	\roundNbox{unshaded}{(-.1,1.5)}{.3}{.3}{0}{$g$}
	\roundNbox{unshaded}{(.1,.5)}{.3}{0}{.3}{$y$}
	\roundNbox{unshaded}{(0,-.5)}{.3}{.9}{.9}{$J$}
	\roundNbox{unshaded}{(.1,-1.5)}{.3}{0}{.4}{$x$}
	\roundNbox{unshaded}{(0,-2.5)}{.3}{.2}{0}{$f$}
	\node at (.6,0) {\scriptsize{$\overline{\mathbf{a}}$}};
	\node at (-.6,.5) {\scriptsize{$\mathbf{c}$}};
	\node at (-1.2,-1) {\scriptsize{$\mathbf{a}$}};
	\node at (1.2,-1) {\scriptsize{$\overline{\mathbf{c}}$}};
	\node at (.8,-3) {\scriptsize{$\mathbf{b}$}};
	\node at ($ (b) + (-.2,-.2) $) {\scriptsize{$\mathbf{c}$}};
\end{tikzpicture}
=
\begin{tikzpicture}[baseline = -.1cm]
	\coordinate (a) at (-.6,-3.2);
	\coordinate (b) at ($ (a) + (.6,-.6) $);
	\draw (0,2) -- (0,-3);
	\draw (-1,-.5) -- (-1,1.8) arc (180:0:.3cm);
	\draw (-.2,-2.8) arc (0:-180:.4cm) -- (-1,-.8);
	\draw (.4,.5) -- (.4,-.5);
	\draw (a) arc (-180:0:.6cm) -- ($ (a) + (1.2,1.6) $);
	\draw (b) .. controls ++(270:.5cm) and ++(270:.5cm) .. (1,-3.8) -- (1,-.5);
	\filldraw[fill=white] (a) circle (.05cm) node [above] {\scriptsize{$\alpha$}};
	\filldraw[fill=white] (b) circle (.05cm) node [above] {\scriptsize{$\beta$}};
	\filldraw (0,-3) circle (.05cm) node [right] {\scriptsize{$\Omega$}};
	\filldraw (0,2) circle (.05cm) node [right] {\scriptsize{$\Omega^*$}};
	\roundNbox{unshaded}{(-.1,1.5)}{.3}{.3}{.3}{$j_c(g)^*$}
	\roundNbox{unshaded}{(.1,.5)}{.3}{0}{.3}{$y$}
	\roundNbox{unshaded}{(0,-.5)}{.3}{.9}{.9}{$J$}
	\roundNbox{unshaded}{(.1,-1.5)}{.3}{0}{.4}{$x$}
	\roundNbox{unshaded}{(0,-2.5)}{.3}{.2}{0}{$f$}
	\node at (.6,0) {\scriptsize{$\overline{\mathbf{a}}$}};
	\node at (-.8,.5) {\scriptsize{$\mathbf{c}$}};
	\node at (-1.2,-1) {\scriptsize{$\mathbf{a}$}};
	\node at (1.2,-1) {\scriptsize{$\overline{\mathbf{c}}$}};
	\node at (.8,-3) {\scriptsize{$\mathbf{b}$}};
	\node at ($ (b) + (-.2,-.2) $) {\scriptsize{$\mathbf{c}$}};
\end{tikzpicture}
=
\begin{tikzpicture}[baseline = -.1cm]
	\coordinate (a) at (-.6,-3.2);
	\coordinate (b) at ($ (a) + (.6,-.6) $);
	\draw (0,3) -- (0,-3);
	\draw (-.4,-.5) -- (-.4,1.5);
	\draw (-.2,-2.8) arc (0:-180:.4cm) -- (-1,-.8);
	\draw (.4,.5) -- (.4,-.5);
	\draw (a) arc (-180:0:.6cm) -- ($ (a) + (1.2,1.6) $);
	\draw (b) .. controls ++(270:.5cm) and ++(270:.5cm) .. (1,-3.8) -- (1,-.5);
	\draw (-.4, 2.8) -- (-.4,3) .. controls ++(90:.5cm) and ++(90:.5cm) .. (1,3) -- (1,1.5);
	\filldraw[fill=white] (a) circle (.05cm) node [above] {\scriptsize{$\alpha$}};
	\filldraw[fill=white] (b) circle (.05cm) node [above] {\scriptsize{$\beta$}};
	\filldraw (0,-3) circle (.05cm) node [right] {\scriptsize{$\Omega$}};
	\filldraw (0,3) circle (.05cm) node [right] {\scriptsize{$\Omega^*$}};
	\roundNbox{unshaded}{(-.1,2.5)}{.3}{.3}{0}{$g^*$}
	\roundNbox{unshaded}{(0,1.5)}{.3}{.9}{.9}{$J$}
	\roundNbox{unshaded}{(.1,.5)}{.3}{0}{.3}{$y$}
	\roundNbox{unshaded}{(0,-.5)}{.3}{.9}{.9}{$J$}
	\roundNbox{unshaded}{(.1,-1.5)}{.3}{0}{.4}{$x$}
	\roundNbox{unshaded}{(0,-2.5)}{.3}{.2}{0}{$f$}
	\node at (-.6,3) {\scriptsize{$\mathbf{c}$}};
	\node at (1.2,2.5) {\scriptsize{$\overline{\mathbf{c}}$}};
	\node at (.6,0) {\scriptsize{$\overline{\mathbf{a}}$}};
	\node at (-.6,.5) {\scriptsize{$\mathbf{c}$}};
	\node at (-1.2,-1) {\scriptsize{$\mathbf{a}$}};
	\node at (1.2,-1) {\scriptsize{$\overline{\mathbf{c}}$}};
	\node at (.8,-3) {\scriptsize{$\mathbf{b}$}};
	\node at ($ (b) + (-.2,-.2) $) {\scriptsize{$\mathbf{c}$}};
\end{tikzpicture}
=
\begin{tikzpicture}[baseline = -.1cm]
	\coordinate (a) at (-.6,-3.2);
	\coordinate (b) at ($ (a) + (.6,-.6) $);
	\draw (0,3) -- (0,-3);
	\draw (-.2,-2.8) arc (0:-180:.4cm) -- (-1,-.8);
	\draw (.4,.5) -- (.4,-.5);
	\draw (a) arc (-180:0:.6cm) -- ($ (a) + (1.2,1.6) $);
	\draw (b) .. controls ++(270:.5cm) and ++(270:.5cm) .. (1,-3.8) -- (1,-.5);
	\draw (-.4, 2.8) -- (-.4,3) .. controls ++(90:.5cm) and ++(90:.5cm) .. (1,3) -- (1,-.5);
	\filldraw[fill=white] (a) circle (.05cm) node [above] {\scriptsize{$\alpha$}};
	\filldraw[fill=white] (b) circle (.05cm) node [above] {\scriptsize{$\beta$}};
	\filldraw (0,-3) circle (.05cm) node [right] {\scriptsize{$\Omega$}};
	\filldraw (0,3) circle (.05cm) node [right] {\scriptsize{$\Omega^*$}};
	\roundNbox{unshaded}{(-.1,2.5)}{.3}{.3}{0}{$g^*$}
	\roundNbox{unshaded}{(0,1.5)}{.3}{0}{0}{$J$}
	\roundNbox{unshaded}{(.1,.5)}{.3}{0}{.3}{$y$}
	\roundNbox{unshaded}{(0,-.5)}{.3}{1}{.5}{$J$}
	\roundNbox{unshaded}{(.1,-1.5)}{.3}{0}{.4}{$x$}
	\roundNbox{unshaded}{(0,-2.5)}{.3}{.2}{0}{$f$}
	\node at (-.6,3) {\scriptsize{$\mathbf{c}$}};
	\node at (1.2,0) {\scriptsize{$\overline{\mathbf{c}}$}};
	\node at (.6,0) {\scriptsize{$\overline{\mathbf{a}}$}};
	\node at (-1.2,-1) {\scriptsize{$\mathbf{a}$}};
	\node at (.8,-3) {\scriptsize{$\mathbf{b}$}};
	\node at ($ (b) + (-.2,-.2) $) {\scriptsize{$\mathbf{c}$}};
\end{tikzpicture}
\,.
$$
We conclude that $xJyJ(\alpha\boxtimes f\boxtimes \beta) = JyJx(\alpha\boxtimes f\boxtimes \beta)$ for all $\alpha\boxtimes f\boxtimes \beta\in \cC(e, a\otimes d)\boxtimes \bfM(d)\boxtimes \cC(c, e\otimes b)$, and thus $xJyJ=JyJx$ by a result similar to Lemma \ref{lem:ZeroInHilbC}.
\end{proof}

\subsection{Analytic properties}
\label{sec:AnalyticProperties}

We now consider connected W*-algebra objects $\bfM\in \Vec(\cC)$.
These algebras have natural definitions of analytic properties such as amenability, the Haagerup property, and property (T).  
We show that our definitions specialize to the usual definitions for two well-studied classes of examples: subfactors and rigid C*-tensor categories due to \cite{MR1729488,MR3406647}, and discrete (quantum) groups, studied by many authors.

Before we begin, we would like to remark on our specialization to connected algebras.  
For ordinary W*-algebras, the correct definitions for things like property (T), amenability, and the Haagerup property are undoubtably via Connes' correspondences.  
There, no traces or multipliers are required.  
While we strongly believe that a robust theory of correspondences exists in our setting, the development of this theory in full generality requires a generalization of Tomita-Takesaki theory, a task which would take us too far a field.  
We restrict our attention to connected algebras, which bear the strongest resemblance to the best studied examples: discrete (quantum) groups and rigid C*-tensor categories.

On the other hand, it would be relatively easy to modify our definitions to include algebras whose base algebra is a finite von Neumann algebra.
This would allow us to include the well studied case of actual ${\rm II}_1$ factors as examples.  
However, the definitions are not quite as clean and natural as they are for connected algebras, so we exclude them here, but we've set things up suggestively so that the interested reader may easily derive the correct generalization.

Let $\bfM\in \Vec(\cC)$ be a connected W*-algebra object, and let $\tau$ denote its unique state. 
Let $H_{\tau}:=\bigoplus_{a\in \Irr(\cC)} L^{2}(\bfM(a))_{\tau}$, where the inner product on $\bfM(a)$ is given by $\langle f|g\rangle_a = \tau(\langle f|g\rangle_{a})$.  
Then we can canonically identify $\End_{\Hilb(\cC)}(L^{2}(\bfM)_{\tau})\cong \bigoplus_{a\in\Irr(\cC)} B(L^{2}(\bfM)_{\tau})\subseteq B(H_{\tau})$.  
We define the von Neumann algebra $\ell^{\infty}(\bfM):= \bigoplus_{a\in\Irr(\cC)} B(L^{2}(\bfM(a))_{\tau})$.

Let $\Psi: \bfM\Rightarrow \bfM$ be a ucp map (note that all ucp maps are automatically normal since $\bfM$ is locally finite).
Then $m_{\Psi}:=\bigoplus_{a\in \Irr(\cC)} \Psi_{a}$ extends to a bounded norm one map in $\ell^{\infty}(\bfM)$ acting on $H_\tau$.  
Note conversely, for any $m\in \ell^{\infty}(\bfM)$, there exists some categorical multiplier $\Psi$ such that $m=m_{\Psi}$.
Thus we call an element $m_{\Psi}\in \ell^{\infty}(\bfM)$ a \emph{ucp-multiplier} if $\Psi$ is a ucp map.
(This is a slight conflict of notation with our earlier notion of cp-multiplier on $\cM_\bfM$, but they are equivalent notions by definition.)

Below, point-wise convergence of operators in $\ell^{\infty}(\bfM)\subseteq B(H_{\tau})$ means strong operator topology convergence in $B(H_{\tau})$. 
Also, a ucp multiplier $m_{\Psi}\in \ell^{\infty}(\bfM)\subseteq B(H_{\tau})$, so it naturally makes sense to talk about finite rank and compact multipliers.

\begin{defn}
\label{def:analytic} 
Let $\bfM\in \Vec(\cC)$ be a connected W*-algebra object.
Then $\bfM$
\begin{enumerate}[(1)]
\item
is \textit{amenable} if there exists a net of finite rank ucp multipliers converging to the identity point-wise in $\ell^{\infty}(\bfM)$.
\item
has the \textit{Haagerup property} if there exists a net of compact ucp multipliers converging to the identity point-wise in $\ell^{\infty}(\bfM)$.
\item
has \textit{property} (T) if every sequence of ucp multipliers which converges to the identity point-wise converges in the operator norm in $\ell^{\infty}(\bfM)$.
\end{enumerate}
\end{defn}

Notice that since connected implies all $\bfM(a)$ are finite dimensional, a ucp multiplier $m_{\Psi}$ is finite rank if and only if it is non-zero for at most finitely many $a\in \Irr(\cC)$. 
Similarly, $m_{\Psi}$ is compact if and only if for every $\varepsilon>0$, there is a finite set $F\subset \Irr(\cC)$ such that $\|\Psi_a\|<\varepsilon$ for all $a\in \Irr(\cC)\setminus F$.

We now consider several classes of examples which have recently generated a great deal of interest, and show that various notions of cp-multipliers in different settings correspond to our ucp maps.
We give a translation for the definition of analytic properties from well known examples to our context.

\begin{ex}[Discrete groups]
This is actually a special case of \ref{ex:dqg} below.  Suppose $G$ is a discrete group.
The group algebra $\bfG=\bbC[G]$ is a connected W*-algebra object in the rigid C*-tensor category $\cC=\fdHilb(G)$ of finite dimensional $G$-graded Hilbert spaces. 
Then $\ell^{\infty}(\bfG)=\ell^{\infty}(G)$.  
It follows the more general results from Example \ref{ex:dqg} below that cp-multipliers agree with the usual definition, and thus our definitions of analytic properties agree with the usual notions. 
\end{ex}{}


\begin{ex}[Symmetric enveloping algebra object]
\label{ex:SymEA}
Consider $\cC$ as a left $\cC\boxtimes \cC^{\text{mp}}$-module W*-category, where the action is given by $(a\boxtimes b^{\text{mp}})(c)=a\otimes c\otimes b$.  
Then the cyclic $(\cC\boxtimes \cC^{\text{mp}})$-module W*-category $(\cC, 1_\cC)$ yields an algebra $\bfM\in \Vec(\cC)$ called the symmetric enveloping algebra \cite{MR1302385,MR1729488}, or the quantum double.

For an object $A\in \cC\boxtimes \cC^{\text{mp}}$, $\bfM(A)=\bigoplus_{a\in \Irr(\cC)} (\cC\boxtimes \cC^{\text{mp}})(A, a\boxtimes \overline{a}^{\text{mp}})$.  
Note that we have a canonical isomorphism
$\gamma_{a,b}:(a\boxtimes \overline{a}^{\text{mp}})\otimes (b\boxtimes \overline{b}^{\text{mp}})\rightarrow (a\otimes b)\boxtimes \overline{(a\otimes b)}^{\text{mp}}$
given by the involutive structure on $\cC$.
The algebra structure on $\bfM$ is given for $f\in  (\cC\otimes \cC^{\text{mp}})(A, a\boxtimes \overline{a}^{\text{mp}})$ and  $g\in (\cC\otimes \cC^{\text{mp}})(B, b\boxtimes \overline{b}^{\text{mp}})$ by 
$$
f\cdot g= 
\sum_{\substack{c\in\Irr(\cC) \\
\alpha\in \ONB(a\otimes b, c)}} d_{c}(\alpha \boxtimes \overline{\alpha})\circ \gamma_{a,b}(f\otimes g).
$$ 

Using the identification of $(\cC, 1_\cC)$ with the category of free $\bfM$-modules, we see that $\cM(a\boxtimes b^{\text{mp}}, c\boxtimes d^{\text{mp}})\cong\cC(a\otimes b, c\otimes d)$.  
Now, in \cite[Def.~3.4]{MR3406647},  Popa and Vaes define a \textit{multiplier} on a rigid C*-tensor category as a family of maps $\Theta_{a,b}: \cC(a\otimes b, a\otimes b)\rightarrow \cC(a\otimes b, a\otimes b)$ satisfying certain compatibility conditions.
In \cite[Prop.~3.6]{MR3406647}, they show that these maps uniquely extend to maps $\Theta_{a\otimes b, c\otimes d}: \cC(a\otimes b, c\otimes d)\rightarrow \cC(a\otimes b, c\otimes d)$ satisfying the same type of compatibility conditions.  By inspection, these conditions are precisely the conditions for a family $\Theta_{a\otimes b, c\otimes d}$ to be a multiplier in the sense of Definition \ref{defn:multiplier}.
Furthermore, they define a \textit{cp-multiplier} to be a multiplier for which $\Theta_{a,b}$ is positive for all objects $a,b\in \cC$.  
This precisely corresponds to our definition of a ucp-multiplier.  
This yields a canonical bijection between ucp maps $\theta: \bfA\Rightarrow \bfA$ and cp-multipliers for $\cC$ in the sense of Popa and Vaes \cite{MR3406647}.
\end{ex}

From this discussion, we deduce the following result.

\begin{prop} 
The symmetric enveloping algebra object $\bfM$ has a property from Definition \ref{def:analytic} if and only if $\cC$ has the corresponding property in the sense of \cite[Def.~5.1]{MR3406647}.
\end{prop}

\begin{ex}[Discrete quantum groups]
\label{ex:dqg}
We rapidly recall the basics of discrete quantum groups from \cite{MR3204665}.
We refer the reader there for additional details.

Suppose $\cC$ admits a dagger tensor functor $(\bfF,\eta):\cC\rightarrow \fdHilb$.  
Using Tannaka-Krein-Woronowicz recontruction, we obtain a discrete quantum group $\bbG$, which has two canonical algebras associated to it. 
The first is the Hopf $*$-algebra $\bbC[\bbG]$ which, as in the case of group algebra $\bbC[G]$, has many C*-completions, each of which yield a \textit{compact} quantum group often thought of as the ``algebra of functions'' on the compact dual of $\bbG$.  
The second is the type ${\rm I}$ von Neumann algebra $\ell^{\infty}(\bbG)=\prod_{a\in \Irr(\cC)} B(\bfF(a))$, which has a non-trivial co-algebra structure making it into multiplier Hopf algebra.   
We describe both of these algebras more explicitly below.  
We remark that when viewed as a locally compact quantum group, one uses the latter von Neumann algebra as the fundamental object.

For convenience, for $c\in\cC$, we define the space $H_{c}=\fdHilb(\bbC, \bfF(c))$, which can canonically be identified with $\bfF(c)$ itself, and $H^{*}_{c}=\fdHilb(\bfF(c), \C)$ which is canonically identified with the dual space $\bfF(c)^*$.  
Then using the tensorator $\eta$, we can define maps $\eta_{c,d}:H_{c}\otimes H_{d}\rightarrow H_{c\otimes d}$  and $\eta^{*}_{c,d}: H^{*}_{c\otimes d}\rightarrow H^{*}_{c}\otimes H^{*}_{d}$.
For $\alpha\in \cC(a\otimes b, c)$, we set $\hat{\alpha}=\bfF(\alpha)\circ \eta_{a,b}: \bfF(a)\otimes \bfF(b)\rightarrow \bfF(c)$.   
In particular, since $\ev_{\overline{a}}\in \cC(a\otimes \overline{a}, 1_\cC)$ and $\ev_{a}\in \cC( \overline{a}\otimes a, 1_{\cC})$ are standard solutions to the duality equations in $\cC$, we see $\hat{\ev}_{a}$ and $\hat{\ev}_{\overline{a}}$ solve the duality equations for $\bfF(a), \bfF(\overline{a})$ in $\fdHilb$, since $\bfF$ is a dagger functor.  
It is an important point that these solutions in general are highly \textit{non-standard}.
For example, $\hat{\ev}^{*}_{a}\circ \hat{\ev}_{a}=d_{a}$, but in general, if $d_{a}$ is not an integer, we have $\dim_{\Hilb}(\bfF(a))< d_{a}$.  
It is also easy to check that in general they do \textit{not} induce a pivotal structure on $\fdHilb$.


Define $\ell^{\infty}(\bbG)=\bigoplus_{a\in \Irr(\cC)} B(\bfF(a))$, the von Neumann algebra direct sum.
While boring as an algebra, $\ell^{\infty}(\bbG)$ has a much more interesting co-algebra structure, which we will not describe in detail.
The \emph{group algebra}, also called the \emph{polynomial algebra on the compact dual}, is defined as $\bbC[\bbG]:=\bigoplus_{a\in \Irr(\cC)} H^{*}_{a}\otimes H_{a}$, the algebraic vector space direct sum, which we can view as the ``restricted dual" of the algebra $\ell^{\infty}(\bbG)$.  
This has a matrix co-algebra structure $\Delta$ (dual to the algebra structure on $\ell^{\infty}(\bbG)$.
The multiplication is defined for $x=\sum_{a\in\Irr(\cC)}x^{a}_{(1)}\otimes x^{a}_{(2)}$ and $y=\sum_{b\in \in\Irr(\cC)}y^{b}_{(1)}\otimes y^{b}_{(2)}$ (in Sweedler notation) by 
$$
xy  =\sum_{\substack{a,b\in\Irr(\cC)
\\
\alpha\in \ONB(a\otimes b, c)}}
d_c 
(x^{a}_{(1)}\otimes y^{b}_{(1)})\circ \hat{\alpha}^{*}\otimes \hat{\alpha}\circ (x^{a}_{(2)}\otimes y^{b}_{(2)}),
$$
under which $\bbC[\bbG]$ is an associative algebra.  
Define the \emph{antipode}
$S: H^{*}_{a}\otimes H_{a}\rightarrow H^{*}_{\overline{a}}\otimes H_{\overline{a}}$ 
by 
$$
S(x_{(1)}\otimes x_{(2)})
=
\widehat{\ev}^{*}_{a}\circ (1_{\overline{a}}\otimes x_{(2)}) \otimes (x_{(1)}\otimes 1_{\overline{a}})\circ \widehat{\ev}_{\overline{a}}
$$
which we extend to all of $\bbC[\bbG]$ linearly.  
It is easily verified that $S(xy)=S(y)S(x)$.  
Note that $S^{2}=1$ if and only if the $(\widehat{\ev}_{a}, \widehat{\ev}_{\overline{a}})$ induce a pivotal structure on $\fdHilb$, which is equivalent to asking that $\bbG$ is of \textit{Kac type}.

To define a $*$-structure, first note that the dagger structure on the category $\fdHilb$ gives us conjugate linear maps $*: H_{a}\rightarrow H^{*}_{a}$ and $*: H^{*}_{a}\rightarrow H_{a}$, yielding a conjugate linear map $j: H^{*}_{a}\otimes H_{a}\rightarrow H^{*}_{a}\otimes H_{a}$ given by $j(x_{(1)}\otimes x_{(2)})=x^{*}_{(2)}\otimes x^{*}_{(1)}$.  
The anit-linear involution on $\bbC[\bbG]$ is given by $x^{\#}=S\circ j(x)$.  
It is straightforward to verify that $x^{\# \#}=x$ and $(xy)^{\#}=y^{\#}x^{\#}$.
The maps $S$ and $\#$ satisfy the relation $S(S(x^{\#})^{\#})$, which makes $(\bbC[\bbG], \Delta, S)$ into a Hopf $*$-algebra.
In general, $S$ and ${\#}$ do not commute; rather, they commute precisely when $\bbG$ is Kac type.

To make the connection with the standard aspects of the theory of compact quantum groups, for $a\in \Irr(\cC)$, let $B_{a}:=\{e^{a}_{i}\}$ be an orthonormal basis for $\bfF(a)$.
Here, we again slightly abuse notation to identify $e^{a}_{i}$ with the morphism $\bbC \to \bfF(a)$ sending $1\mapsto e^{a}_{i}$.  
We use the notation $(e^{a}_{j})^*$ for dual basis element of $H^{*}_{a}$.
We define $u^{a}_{ij}=(e^{a}_{i})^*\otimes e^{a}_{j}\in H^{*}_{a}\otimes H_{a}\in \bbC[\bbG]$.  
It is well known (and easy to check) that $u^{a}:=(u^{a}_{ij})_{i,j}\in M_{n}(\bbC[\bbG])$ is unitary.  
In fact, the map $e^{a}_{i}\mapsto \sum_{j} e_{j} \otimes u^{a}_{i,j}$ extends to a linear map $\pi: \bfF(a)\rightarrow \bfF(a)\otimes \bbC[\bbG]$, and makes $\bfF(a)$ into a \textit{unitary co-representation} of $\bbC[\bbG]$.
One can show $\cC$ is equivalent to the rigid C*-tensor category of unitary co-representations of $\bbC[\bbG]$.
Again, we refer the reader to \cite{MR3204665} for more details on compact quantum groups.

\begin{defn} 
A \textit{state} on $\bbG$ is a linear functional $\phi$ on $\bbC[\bbG]$ such that $\phi(1_{\bbC[\bbG]})=1_\bbC$ where $1_{\bbC[\bbG]} = u^{1_\cC}_{1,1}$, and $\phi(x^{\#} x)\ge 0$ for all $x\in \bbC[\bbG]$.
\end{defn}

The pairing $H^{*}_{a}\otimes H_{a}\rightarrow \bbC$ given by $x_{(1)}\otimes x_{(2)}\rightarrow x_{(1)}(x_{(2)})$ is non-degenerate (since $H_a\in \fdHilb$).  
Thus for every linear functional $\phi_{a}: H^{*}_{a}\otimes H_{a}\rightarrow \C$, there is a unique $\Phi_{a}\in \End(\bfF(a))$ such that $\phi_{a}(x_{(1)}\otimes x_{(2)})=x_{(1)}(\Phi_{a}(x_{(2)}))$.  

Thus a linear functional $\phi$ on $\bbC[\bbG]$ is uniquely defined by a sequence $m_{\phi}=\Phi_{a}:\ \Phi_{a}\in \End(\bfF(a))\}$.  
If $\phi$ is a state, $\phi$ extends to a state on the universal $C^{*}$-algebra $C^{*}_{u}(\bbG)$
(for example, boundedness follows from \cite[Lem.~4.2]{MR1310296}.)
Then $\Phi_{a}$ is the image of $u^{a}$ under the amplification of the state $\phi$ (viewed as a ucp map from $C^{*}_{u}(\bbG)\rightarrow \bbC$), and thus $\|\Phi_{a}\|\leq 1$. 
Hence $m_{\phi}\in \ell^{\infty}(\bbG)$.

\begin{defn} 
A \textit{cp-multiplier} on $\bbG$ is an element $m_{\phi}\in \ell^{\infty}(\bbG)$ such that the corresponding functional $\phi$ is a state.
\end{defn}

We are now ready to give definitions for analytic properties for quantum groups.  
First, defining $H:=\bigoplus_{a\in \Irr(\cC)} H_{a}$, we can view $\ell^{\infty}(\bbG)$ as a von Neumann subalgebra of $B(H)$, which enables us to discuss finite rank/compact ucp-multipliers.
Again, point-wise convergence means strong operator topology convergence, viewing $\ell^\infty(\bbG)\subset B(H)$.
 
\begin{defn}  
Let $\bbG$ be a discrete quantum group.  
Then $\bbG$
\begin{enumerate}[(1)]
\item
is \textit{amenable} if there exists a net of finite rank ucp-multipliers converging point-wise to the identity in $\ell^{\infty}(\bbG)$.
\item
has the \textit{Haagerup property} if there exists a net of compact ucp-multipliers converging point-wise to the identity in $\ell^{\infty}(\bbG)$.
\item
has \textit{property} (T) if every net of ucp-multipliers converging point-wise to the identity in $\ell^{\infty}(\bbG)$ converges uniformly to the identity.
\end{enumerate}
\end{defn}

We remark that the definitions presented here are somewhat non-standard.  For a general overview of analytic properties (and many equivalent characterizations of the above properties) for locally compact quantum groups in general, we recommend the survey papers \cite{MR3456763} and \cite{1605.02800}, and the references therein.   Our definitions can be seen to be equivalent to the more usual definitions via the context of \textit{completely positive multipliers} for quantum groups (see \cite{MR3019431}).  
The equivalence of amenability to other definitions can be seen by applying the usual group-theory type arguments; for example, our definition clearly is equivalent to saying ``the trivial representation on $\C[\mathbbm{G}]$ can be approximated by finitely supported states'', so by standard group theory type arguments, the universal C*-norm equals the reduced C*-norm on $\C[\mathbbm{G}]$, hence the reducing map is injective.  Thus the compact dual is \textit{co-amenable}, which is equivalent to $\mathbbm{G}$  being amenable (see, for example \cite[Thm.~3.12, 3.13 and 3.15]{1605.01770}).
For the Haagerup property, this this can be seen directly from \cite[Thm.~5.5]{MR3456763}.
For property (T), this comes from \cite[Thm.~3.1]{1605.02800}.  We greatly thank Makoto Yamashita for very helpful discussions about the equivalent characterizations of these properties for discrete quantum groups.

We now unify these definitions with our categorical definitions. 
Recall the tensor functor $\bfF$ makes $\fdHilb$ into a $\cC$-module W*-category.  
Let $\bfG$ be the connected W*-algebra object corresponding to $\bbC\in \fdHilb$, given by $\bfG(a)=\fdHilb(\bfF(a), \bbC)=H^{*}_{a}$.  
A natural transformation $\Theta: \bfG\rightarrow \bfG$ is uniquely determined by a family of linear maps $\Phi_{a}: H^{*}_{a}\rightarrow H^{*}_{a}$.  
Each $\Phi_{a}$ can be viewed as an operator $\Phi_{a}\in B(H_{a})$, which acts on $H^{*}_{a}$ by precomposition. 
Such a $\Phi_{a}$ induces a categorical multiplier $\Phi_{a,b}: \fdHilb(\bfF(a), \bfF(b))\rightarrow \fdHilb(\bfF(a), \bfF(b))$, given by
$$
\Phi_{a,b}(f
):= 
\sum_{\substack{a,b\in\Irr(\cC)
\\
\alpha\in \ONB(a\otimes b, c)}} d_{c} \left(\id_{\bfF(a)}\otimes \left(\hat{\ev}_{a}\circ (\id_{\bfF(\overline{a})}\otimes f)\circ \hat{\alpha}^{*} \circ \Phi_{c}\circ \hat{\alpha}\right) \right)\circ (\hat{\ev}^{*}_{a}\otimes \id_{\bfF(b)})
$$
(compare with \eqref{eq:AmplifiedMultiplier}).
Conversely, every cp-multiplier is of this form and produces a family $\Phi_{a}$ and a state $\phi$ as described above.

Recall that if $A,B$ are $n\times n$ matrices, their \emph{Schur product} is the $n\times n$ matrix given by $(A\star B)_{ij}=A_{ij}B_{ij}$.  
The Schur Product Theorem states that the Schur product of two positive matrices is again positive. 
We have the following lemma that will be useful later.

\begin{lem}  
\label{lem:SchurProduct}
Let $A$ be an $n\times n$ matrix.  
Then $A$ is positive if and only if $v (B\star A) v^{*}\geq 0$ for all positive matrices $B$, where $v=(1,1, \dots, 1)$.
\end{lem}
\begin{proof}
By the Schur product theorem, if $A\geq 0$, then the condition follows.  
Conversely, assume our condition holds.  
Let $z=(z_1, \dots, z_{n})$ be an arbitrary row vector, and let $Z=\operatorname{diag}(z)$, so that $Z^{*}=\operatorname{diag}(\overline{z})$.  
Let $V$ be the $n\times n$ matrix with all entries equal to $1$.  
Then $z A z^{*}=v (Z (V\star A) Z^{*}) v^{*}=v\left((ZVZ^{*})\star A\right) v^{*}\geq 0$ since $ZVZ^{*}\geq 0$.
Here we have used the fact that since $Z, Z^{*}$ are diagonal, $Z(V\star A)Z^{*}=(ZVZ^{*})\star A$.  
Thus $A$ is positive. 
\end{proof}

\begin{prop} 
A natural transformation $\Phi: \bfG\rightarrow \bfG$ is ucp if and only if $\phi \in \bbC[\bbG]$ is a state. 
\end{prop}
\begin{proof}
Let $\set{\Phi_{a}\in \End(\bfF(a))}{ a\in \Irr(\cC)}$ be the sequence associated to $\Phi$.  
Let $B_{a}$ be an orthonormal basis for $\bfF(a)$ as above.
For $i,j\in B_{a}$ and $k,l\in B_{b}$, define the number $\Phi^{k,l;b}_{i,j;a}=\phi((u^{b}_{kl})^{\#} u^{a}_{ij})$.   
Here and throughout, the letter after the semi-colon indicates the component of the preceding index.

We want to find a nice characterization of positivity for ucp maps which we can compare to positivity of states on the algebra $\bbC[\bbG]$. 
First note that $\Phi$ is a ucp map if and only if for any $c\in \cC$, $f\in B(\bfF(c))$, and $t\in H_{c}$, the corresponding multiplier satisfies $t^{*}\circ \Phi_{c,c}(f)\circ t\ge 0$.  We will express this condition by choosing coordinates and writing this expression in terms of coefficients with respect to orthonormal bases.
First, we write $\bfF(c)\cong \bigoplus_{a\in \Lambda}\bfF(\cC(c, a))\otimes \bfF(a) =: K_{\Lambda}$, where $\Lambda$ is the finite subset of $\Irr(\cC)$ such that $a\prec c$.
We have a choice of orthonormal basis for $\bfF(a)$ (hence $H_{a}$), for $a\in \Irr(\cC)$ given above, and we choose an orthonormal basis $V_{a}$ of $\bfF(\cC(a, c))$ for each $a\in \Lambda$.  
Then viewing $f\in \End(K_{\Lambda})$, we define $f_{a,b}(v,i,w,k)$ as the coefficients of $f$ with respect to the tensor product basis, namely $\sum_{v\in V_{a} i\in B_{a}} f_{a,b}(v,i,w,k) w\otimes e^{b}_{k}=P_{b}(f(v\otimes e^{a}_{i}))$, where here $P_{b}$ is the projection of $K_{\Lambda}$ onto the component $ \bfF(\cC(c, b))\otimes \bfF(b)$.
Similarly, define the numbers $t_{a}(v,j)$ as the coefficient of $e^{a}_{j}$ in the expansion of $v\circ t$ for $v\in V_{a}$.   
Then positivity is equivalent to  
$$
\displaystyle 
\sum_{a,b\in\Lambda}
\sum_{w\in V_{b},\, v\in V_{a}}
\sum_{k,l\in B_{b},\, i,j\in B_{a}}  \overline{t_{b}(w,l)} f_{a,b}(v,i,w,k) t_{a}(v,j) \Phi^{k,l; b}_{i,j; a} 
\geq 0
$$ 

The above discussion leads to the following abstract characterizations of positivity in both cases: 
\begin{enumerate}[(1)]
\item
We see $\phi$ is a state on $\bbC[\bbG]$ if and only if for any finite set $\Lambda\subseteq \Irr(\cC)$, and any functions $\alpha_{a}: B_{a}\times B_{a}\rightarrow \C$, the sum 
$$
\sum_{a,b\in\Lambda}
\sum_{k,l\in B_{b},\ i,j\in B_{a}} 
\overline{\alpha_{b}(k,l)} \alpha_{a}(i, j) \Phi^{k,l;b}_{i,j; a}
\geq 0.
$$
\item
$\Phi:\bfG\Rightarrow\bfG$ is a ucp map if and only if for any finite set $\Lambda\subseteq \Irr(\cC)$ and arbitrary finite sets $V_{a}$ for $a\in \Lambda$, and for any maps $t_{a}: V_{a}\times B_{a} \rightarrow \bbC$ and $f_{a,b}: V_{a}\times B_{a}\times V_{b}\times B_{b}$ whose values $f_{a,b}(i,v,k,w)$ are the coefficients of an $n\times n$ positive matrix for  $n=|\bigcup_{a} V_{a}\times B_{a} |$ (as above)
we have 
$$
\sum_{a,b\in\Lambda}
\sum_{w\in V_{b},\, v\in V_{a}}
\sum_{k,l\in B_{b},\, i,j\in B_{a}} 
\overline{t_{b}(w,l)} f_{a,b}(v,i,w,k) t_{a}(v,j) \Phi^{k,l;b}_{i,j;a} 
\geq 0.
$$
\end{enumerate}

Now, we claim conditions $(1)$ and $(2)$ on the coefficients of $\Phi$ are equivalent.   
Let $\bbC[V_{a}]$ denote the Hilbert space with orthonormal basis given by $V_{a}$.   
Define $K_{\Lambda}:=\bigoplus_{a\in \Lambda} \C[V_{a}]\otimes H_{a}$ and $H_{\Lambda}:=\bigoplus_{a\in\Lambda} H^{*}_{a}\otimes H_{a}$.  
Then we can naturally view $f$ as a positive operator $F: K_{\Lambda}\rightarrow K_{\Lambda}$.  
Similarly we can view $T_{a}((e^{a}_{i})^*):=\sum_{v\in V_{a}} t_{a}(v,i) v$ as a linear map $H^{*}_{a}\rightarrow \bbC[V_{a}]$.
Setting $T:=\bigoplus_{a\in \Lambda} T_{a}\otimes \id_{a}: H_{\Lambda}\rightarrow K_{\Lambda}$, we have the positive operator $T^*\circ F\circ T$, which we view as a matrix with our distinguished orthonormal basis.  
Similarly, defining 
$\widehat{\Phi}((e^{a}_{i})^* \otimes e^{a}_{j}):=\sum_{b\in \Lambda,\  k,l\in B_{b}} \Phi^{k,l; b}_{i,j ; a} (e^{b}_{k})^*\otimes e^{b}_{l}$ 
yields an operator $\widehat{\Phi}: H_{\Lambda}\rightarrow H_{\Lambda}$.  

Its clear that condition (1) on $\Phi$ simply states that the operator $\widehat{\Phi}$ is positive, while condition (2) states that 
$$
v\left((T^{*}\circ F\circ T)\star \widehat{\Phi}\right) v^{*} \ge 0
$$ 
for all $K_{\Lambda}$ and positive $F: K_{\Lambda}\rightarrow K_{\Lambda}$, where $v=(1,1,\dots, 1)$.  
But $T^{*}\circ F \circ T\in \End(H_{\lambda})$ is always a positive operator, and every positive $P\in \End(H_{\lambda})$ arises this way for some $K_{\Lambda}$ (for example, pick $K_{\Lambda}=H_{\lambda}$, and set $T=\id$).  
Thus we can replace this condition with $v (P\star \widehat{\Phi}) v^{*}\ge 0$ for all positive operators $P\in \End(H_{\Lambda})$, which by Lemma \ref{lem:SchurProduct} above is equivalent to $\widehat{\Phi}$ being positive.  
Thus the conditions (1) and (2) are equivalent.
\end{proof}

We now deduce the following proposition.

\begin{prop} 
A discrete quantum group $\bbG$ is amenable, has the Haagerup property, or property (T) if and only if the corresponding \emph{W*}-algebra object $\bfG$ in $\Rep(\bbG)$ does.
\end{prop}

\end{ex}


\bibliographystyle{amsalpha}
{\footnotesize{
\bibliography{../../../Documents/research/penneys/bibliography}

\newcommand{\etalchar}[1]{$^{#1}$}
\providecommand{\bysame}{\leavevmode\hbox to3em{\hrulefill}\thinspace}
\providecommand{\MR}{\relax\ifhmode\unskip\space\fi MR }
\providecommand{\MRhref}[2]{%
  \href{http://www.ams.org/mathscinet-getitem?mr=#1}{#2}
}
\providecommand{\href}[2]{#2}
\begin{thebibliography}{DGNO10}

\bibitem[BDH14]{MR3342166}
Arthur Bartels, Christopher~L. Douglas, and Andr{\'e} Henriques,
  \emph{Dualizability and index of subfactors}, Quantum Topol. \textbf{5}
  (2014), no.~3, 289--345, \mathscinet{MR3342166} \doi{10.4171/QT/53}
  \arXiv{1110.5671}. \MR{3342166}

\bibitem[BdlHV08]{MR2415834}
Bachir Bekka, Pierre de~la Harpe, and Alain Valette, \emph{Kazhdan's property
  ({T})}, New Mathematical Monographs, vol.~11, Cambridge University Press,
  Cambridge, 2008, \mathscinet{MR2415834} \doi{10.1017/CBO9780511542749}.
  \MR{2415834}

\bibitem[BE16]{1603.05928}
Jonathan Brundan and Alexander~P. Ellis, \emph{Monoidal supercategories}, 2016,
  \arxiv{1603.05928}.

\bibitem[BGH{\etalchar{+}}16]{1603.09294}
Paul Bruillard, Cesar Galindo, Tobias Hagge, Siu-Hung Ng, Julia~Yael Plavnik,
  Eric~C. Rowell, and Zhenghan Wang, \emph{Fermionic modular categories and the
  16-fold way}, 2016, \arxiv{1603.09294}.

\bibitem[BHP12]{MR3405915}
Arnaud Brothier, Michael Hartglass, and David Penneys, \emph{Rigid {$C\sp
  *$}-tensor categories of bimodules over interpolated free group factors}, J.
  Math. Phys. \textbf{53} (2012), no.~12, 123525, 43, \mathscinet{MR3405915}
  \doi{10.1063/1.4769178} \arxiv{1208.5505}. \MR{3405915}

\bibitem[Bis97]{MR1424954}
Dietmar Bisch, \emph{Bimodules, higher relative commutants and the fusion
  algebra associated to a subfactor}, Operator algebras and their applications
  (Waterloo, ON, 1994/1995), 13-63, Fields Inst. Commun., 13, Amer. Math. Soc.,
  Providence, RI, 1997, \mathscinet{MR1424954}, \googlebooks{_InIRTO8Y7gC}.

\bibitem[BN11]{MR2863377}
Alain Brugui{\`e}res and Sonia Natale, \emph{Exact sequences of tensor
  categories}, Int. Math. Res. Not. IMRN (2011), no.~24, 5644--5705,
  \mathscinet{MR2863377} \doi{10.1093/imrn/rnq294} \arxiv{1006.0569}.
  \MR{2863377}

\bibitem[Bra16]{1605.01770}
Michael Brannan, \emph{Approximation properties for locally compact quantum
  groups}, 2016, \arxiv{1605.01770}.

\bibitem[Con80]{MR561983}
Alain Connes, \emph{On the spatial theory of von {N}eumann algebras}, J. Funct.
  Anal. \textbf{35} (1980), no.~2, 153--164, \mathscinet{MR561983}.

\bibitem[Daw12]{MR3019431}
Matthew Daws, \emph{Completely positive multipliers of quantum groups},
  Internat. J. Math. \textbf{23} (2012), no.~12, 1250132, 23,
  \mathscinet{MR3019431} \doi{10.1142/S0129167X12501327}. \MR{3019431}

\bibitem[Day70]{MR0272852}
Brian Day, \emph{On closed categories of functors}, Reports of the {M}idwest
  {C}ategory {S}eminar, {IV}, Lecture Notes in Mathematics, Vol. 137, Springer,
  Berlin, 1970, \mathscinet{MR0272852}, pp.~1--38. \MR{0272852}

\bibitem[DCY15]{MR3420332}
Kenny De~Commer and Makoto Yamashita, \emph{Tannaka-{K}re\u\i n duality for
  compact quantum homogeneous spaces {II}. {C}lassification of quantum
  homogeneous spaces for quantum {$\rm SU(2)$}}, J. Reine Angew. Math.
  \textbf{708} (2015), 143--171, \mathscinet{MR3420332}
  \doi{10.1515/crelle-2013-0074}. \MR{3420332}

\bibitem[DFSW16]{MR3456763}
Matthew Daws, Pierre Fima, Adam Skalski, and Stuart White, \emph{The {H}aagerup
  property for locally compact quantum groups}, J. Reine Angew. Math.
  \textbf{711} (2016), 189--229, 10.1515/crelle-2013-0113\mathscinet{MR3456763}
  \doi{}. \MR{3456763}

\bibitem[DGNO10]{MR2609644}
Vladimir Drinfeld, Shlomo Gelaki, Dmitri Nikshych, and Victor Ostrik, \emph{On
  braided fusion categories. {I}}, Selecta Math. (N.S.) \textbf{16} (2010),
  no.~1, 1--119, \mathscinet{MR2609644} \doi{10.1007/s00029-010-0017-z}
  \arxiv{0906.0620}. \MR{2609644 (2011e:18015)}

\bibitem[DK94]{MR1310296}
Mathijs~S. Dijkhuizen and Tom~H. Koornwinder, \emph{C{QG} algebras: a direct
  algebraic approach to compact quantum groups}, Lett. Math. Phys. \textbf{32}
  (1994), no.~4, 315--330, \mathscinet{MR1310296} \doi{10.1007/BF00761142}
  \arxiv{hep-th/9406042}. \MR{1310296}

\bibitem[DSPS13]{1312.7188}
Chris Douglas, Chris Schommer-Pries, and Noah Snyder, \emph{Dualizable tensor
  categories}, 2013, \arxiv{1312.7188}.

\bibitem[DSV16]{1605.02800}
Matthew Daws, Adam Skalski, and Ami Viselter, \emph{Around property {(T)} for
  quantum groups}, 2016, \arxiv{1605.02800}.

\bibitem[EK98]{MR1642584}
David~E. Evans and Yasuyuki Kawahigashi, \emph{Quantum symmetries on operator
  algebras}, Oxford Mathematical Monographs. Oxford Science Publications. The
  Clarendon Press, Oxford University Press, New York, 1998, xvi+829 pp. ISBN:
  0-19-851175-2, \mathscinet{MR1642584}.

\bibitem[FK98]{MR1642530}
Michael Frank and Eberhard Kirchberg, \emph{On conditional expectations of
  finite index}, J. Operator Theory \textbf{40} (1998), no.~1, 87--111,
  \mathscinet{MR1642530}. \MR{1642530}

\bibitem[GJ16]{MR3447719}
Shamindra~Kumar Ghosh and Corey Jones, \emph{Annular representation theory for
  rigid {$C^*$}-tensor categories}, J. Funct. Anal. \textbf{270} (2016), no.~4,
  1537--1584, \mathscinet{MR3447719} \doi{10.1016/j.jfa.2015.08.017}
  \arxiv{1502.06543}. \MR{3447719}

\bibitem[GLR85]{MR808930}
P.~Ghez, R.~Lima, and J.~E. Roberts, \emph{{$W\sp \ast$}-categories}, Pacific
  J. Math. \textbf{120} (1985), no.~1, 79--109, \mathscinet{MR808930}.
  \MR{808930 (87g:46091)}

\bibitem[GN43]{MR0009426}
I.~Gelfand and M.~Neumark, \emph{On the imbedding of normed rings into the ring
  of operators in {H}ilbert space}, Rec. Math. [Mat. Sbornik] N.S.
  \textbf{12(54)} (1943), 197--213, \mathscinet{MR0009426}. \MR{0009426}

\bibitem[HP15]{1511.05226}
Andr\'e Henriques and David Penneys, \emph{Bicommutant categories from fusion
  categories}, 2015, \arxiv{1511.05226} \doi{10.1007/s00029-016-0251-0}.

\bibitem[HPT]{uAPA}
Andr\'e Henriques, David Penneys, and James Tener, \emph{Unitary anchored
  planar algebras in $\mathcal{Z}(\mathcal{C})$}, In preparation.

\bibitem[HPT16a]{1509.02937}
Andr\'e Henriques, David Penneys, and James~E. Tener, \emph{Categorified trace
  for module tensor categories over braided tensor categories}, Doc. Math.
  \textbf{21} (2016), 1089--1149, \arXiv{1509.02937}.

\bibitem[HPT16b]{1607.06041}
\bysame, \emph{Planar algebras in braided tensor categories}, 2016,
  \arxiv{1607.06041}.

\bibitem[ILP98]{MR1622812}
Masaki Izumi, Roberto Longo, and Sorin Popa, \emph{A {G}alois correspondence
  for compact groups of automorphisms of von {N}eumann algebras with a
  generalization to {K}ac algebras}, J. Funct. Anal. \textbf{155} (1998),
  no.~1, 25--63, \mathscinet{MR1622812}.

\bibitem[JL16]{1602.02662}
Arthur Jaffe and Zhengwei Liu, \emph{Planar para algebras, reflection
  positivity}, 2016, \arxiv{1602.02662}.

\bibitem[Jon83]{MR0696688}
Vaughan F.~R. Jones, \emph{Index for subfactors}, Invent. Math. \textbf{72}
  (1983), no.~1, 1--25, \mathscinet{MR696688}, \doi{10.1007/BF01389127}.

\bibitem[Jon99]{math.QA/9909027}
\bysame, \emph{Planar algebras {I}}, 1999, \arXiv{math.QA/9909027}.

\bibitem[Kel05]{MR2177301}
G.~M. Kelly, \emph{Basic concepts of enriched category theory}, Repr. Theory
  Appl. Categ. (2005), no.~10, vi+137, \mathscinet{MR2177301}, Reprint of the
  1982 original [Cambridge Univ. Press, Cambridge; \mathscinet{MR0651714}].
  \MR{2177301}

\bibitem[KO02]{MR1936496}
Alexander Kirillov, Jr. and Viktor Ostrik, \emph{On a {$q$}-analogue of the
  {M}c{K}ay correspondence and the {ADE} classification of {$\mathfrak{sl}_2$}
  conformal field theories}, Adv. Math. \textbf{171} (2002), no.~2, 183--227,
  \mathscinet{MR1936496} \arXiv{math.QA/0101219} \doi{10.1006/aima.2002.2072}.
  \MR{MR1936496 (2003j:17019)}

\bibitem[LR97]{MR1444286}
R.~Longo and J.~E. Roberts, \emph{A theory of dimension}, $K$-Theory
  \textbf{11} (1997), no.~2, 103--159, \mathscinet{MR1444286}
  \doi{10.1023/A:1007714415067}. \MR{1444286}

\bibitem[MP]{Enriched}
Scott Morrison and David Penneys, \emph{Tensor categories enriched in braided
  tensor categories}, In preparation.

\bibitem[MPS10]{MR2559686}
Scott Morrison, Emily Peters, and Noah Snyder, \emph{Skein theory for the
  {$D_{2n}$} planar algebras}, J. Pure Appl. Algebra \textbf{214} (2010),
  no.~2, 117--139, \arXiv{0808.0764} \mathscinet{MR2559686}
  \doi{10.1016/j.jpaa.2009.04.010}. \MR{MR2559686}

\bibitem[M{\"u}g03]{MR1966525}
Michael M{\"u}ger, \emph{From subfactors to categories and topology. {II}.
  {T}he quantum double of tensor categories and subfactors}, J. Pure Appl.
  Algebra \textbf{180} (2003), no.~1-2, 159--219, \mathscinet{MR1966525}
  \doi{10.1016/S0022-4049(02)00248-7} \arXiv{math.CT/0111205}.

\bibitem[MvN43]{MR0009096}
F.~J. Murray and J.~von Neumann, \emph{On rings of operators. {IV}}, Ann. of
  Math. (2) \textbf{44} (1943), 716--808, \mathscinet{MR0009096}. \MR{0009096
  (5,101a)}

\bibitem[NT13]{MR3204665}
Sergey Neshveyev and Lars Tuset, \emph{Compact quantum groups and their
  representation categories}, Cours Sp\'ecialis\'es [Specialized Courses],
  vol.~20, Soci\'et\'e Math\'ematique de France, Paris, 2013,
  \mathscinet{MR3204665}. \MR{3204665}

\bibitem[NY15]{1511.06332}
Sergey Neshveyev and Makoto Yamashita, \emph{A few remarks on the tube algebra
  of a monoidal category}, 2015, \arxiv{1511.06332}.

\bibitem[NY16]{MR3509018}
Sergey Neshveyev and Makoto Yamashita, \emph{Drinfeld center and representation
  theory for monoidal categories}, Comm. Math. Phys. \textbf{345} (2016),
  no.~1, 385--434, \mathscinet{MR3509018} \doi{10.1007/s00220-016-2642-7}
  \arxiv{1501.07390}. \MR{3509018}

\bibitem[Ocn88]{MR996454}
Adrian Ocneanu, \emph{Quantized groups, string algebras and {G}alois theory for
  algebras}, Operator algebras and applications, Vol.\ 2, London Math. Soc.
  Lecture Note Ser., vol. 136, Cambridge Univ. Press, Cambridge, 1988,
  \mathscinet{MR996454}, pp.~119--172.

\bibitem[Ost03]{MR1976459}
Victor Ostrik, \emph{Module categories, weak {H}opf algebras and modular
  invariants}, Transform. Groups \textbf{8} (2003), no.~2, 177--206,
  \mathscinet{MR1976459} \arXiv{math/0111139}. \MR{MR1976459 (2004h:18006)}

\bibitem[Pop94]{MR1302385}
Sorin Popa, \emph{Symmetric enveloping algebras, amenability and {AFD}
  properties for subfactors}, Math. Res. Lett. \textbf{1} (1994), no.~4,
  409--425, \mathscinet{MR1302385}, \doi{10.4310/MRL.1994.v1.n4.a2}.
  \MR{1302385 (95i:46095)}

\bibitem[Pop95]{MR1334479}
Sorin Popa, \emph{An axiomatization of the lattice of higher relative
  commutants of a subfactor}, Invent. Math. \textbf{120} (1995), no.~3,
  427--445, \mathscinet{MR1334479} \doi{10.1007/BF01241137}.

\bibitem[Pop99]{MR1729488}
\bysame, \emph{Some properties of the symmetric enveloping algebra of a
  subfactor, with applications to amenability and property {T}}, Doc. Math.
  \textbf{4} (1999), 665--744 (electronic), \mathscinet{MR1729488}.

\bibitem[PP86]{MR860811}
Mihai Pimsner and Sorin Popa, \emph{Entropy and index for subfactors}, Ann.
  Sci. \'{E}cole Norm. Sup. (4) \textbf{19} (1986), no.~1, 57--106,
  \mathscinet{MR860811}.

\bibitem[PSV15]{1511.07329}
Sorin Popa, Dimitri Shlyakhtenko, and Stefaan Vaes, \emph{Cohomology and
  $l^2$-{B}etti numbers for subfactors and quasi-regular inclusions}, 2015,
  \arxiv{1511.07329}.

\bibitem[PV15]{MR3406647}
Sorin Popa and Stefaan Vaes, \emph{Representation theory for subfactors,
  {$\lambda$}-lattices and {$\rm C^*$}-tensor categories}, Comm. Math. Phys.
  \textbf{340} (2015), no.~3, 1239--1280, \mathscinet{MR3406647}
  \doi{10.1007/s00220-015-2442-5} \arxiv{1412.2732}. \MR{3406647}

\bibitem[Ush16]{1606.03466}
Robert Usher, \emph{Fermionic 6j-symbols in superfusion categories}, 2016,
  \arxiv{1606.03466}.

\bibitem[Yam04]{MR2091457}
Shigeru Yamagami, \emph{Frobenius duality in {$C^*$}-tensor categories}, J.
  Operator Theory \textbf{52} (2004), no.~1, 3--20, \mathscinet{MR2091457}.
  \MR{2091457 (2005f:46109)}

\end{thebibliography}
}}
\end{document}